\batchmode
\makeatletter
\def\input@path{{"D:/EPFL/Notes/PhD Thesis/"}}
\makeatother
\documentclass[11pt,english]{article}
\usepackage[T1]{fontenc}
\usepackage[latin9]{inputenc}
\usepackage{geometry}
\usepackage{verbatim}
\usepackage{amsmath}
\usepackage{amsthm}
\usepackage{amssymb}
\usepackage{setspace}
\makeatletter
\numberwithin{equation}{section}
\numberwithin{figure}{section}

\@ifundefined{date}{}{\date{}}
\usepackage[OT1]{fontenc}
\usepackage{amsmath}
\usepackage{amsthm}
\usepackage{amssymb}
\usepackage{amsfonts}
\usepackage{mathcomp}
\usepackage{mathrsfs}

\usepackage{enumitem}
\usepackage[all]{xy}
\usepackage{tikz-cd}
\usepackage{arraycols}
\usepackage{youngtab}
\usepackage{multirow}

\usepackage{imakeidx}
\makeindex[name=terms,title={Index}]
\makeindex[name=notations,title={Glossary of notations}]
\usepackage[linktoc=page]{hyperref}

\theoremstyle{definition}

\newtheorem{df}{Definition}[subsection]

\theoremstyle{plain}

\newtheorem{thm}[df]{Theorem}
\newtheorem*{thm*}{Theorem}
\newtheorem{cj}[df]{Conjecture}
\newtheorem*{cj*}{Conjecture}
\newtheorem{prop}[df]{Proposition}
\newtheorem*{prop*}{Proposition}
\newtheorem{lem}[df]{Lemma}
\newtheorem*{lem*}{Lemma}
\newtheorem{cor}[df]{Corollary}

\newtheorem{thmintro}{Theorem}
\newtheorem{propintro}[thmintro]{Proposition}
\newtheorem{corintro}[thmintro]{Corollary}

\theoremstyle{remark}

\newtheorem{exmp}[df]{Example}
\newtheorem{rmk}[df]{Remark}

\DeclareMathOperator{\Ker}{Ker}
\DeclareMathOperator{\Rea}{Re}

\DeclareMathOperator{\Cok}{Coker}
\DeclareMathOperator{\End}{End}

\DeclareMathOperator{\Aut}{Aut}
\DeclareMathOperator{\Id}{Id}
\DeclareMathOperator{\Hom}{Hom}
\DeclareMathOperator{\Ext}{Ext}
\DeclareMathOperator{\ext}{ext}
\DeclareMathOperator{\Tr}{Tr}
\DeclareMathOperator{\Mat}{Mat}
\DeclareMathOperator{\GL}{GL}

\DeclareMathOperator{\supp}{supp}
\DeclareMathOperator{\rk}{rk}
\DeclareMathOperator{\val}{val}

\DeclareMathOperator{\Rep}{Rep}
\DeclareMathOperator{\vol}{vol}
\DeclareMathOperator{\Res}{Res}
\DeclareMathOperator{\Exp}{Exp}

\DeclareMathOperator{\ask}{ask}
\DeclareMathOperator{\Hilb}{Hilb}
\DeclareMathOperator{\ch}{char}
\DeclareMathOperator{\sm}{sm}
\DeclareMathOperator{\sg}{sg}

\DeclareMathOperator{\Spec}{Spec}

\newcommand\dd{\mathbf{d}}
\newcommand\ee{\mathbf{e}}
\newcommand\mm{\mathbf{m}}
\newcommand\nn{\mathbf{n}}
\newcommand\rr{\mathbf{r}}
\newcommand\HH{\mathrm{H}}
\newcommand\NN{\mathbb{N}}
\newcommand\ZZ{\mathbb{Z}}
\newcommand\QQ{\mathbb{Q}}
\newcommand\FF{\mathbb{F}}
\newcommand\KK{\mathbb{K}}
\newcommand\RR{\mathbb{R}}
\newcommand\CC{\mathbb{C}}
\newcommand\git{/\!\!/}

\makeatother

\usepackage{babel}
\begin{document}
\title{Counting representations of quivers with multiplicities}
\author{PhD thesis \\ Tanguy Vernet \\ \\ Advisor: \\ Dimitri Wyss}

\maketitle
\pagebreak{}

\small

\tableofcontents{}

\normalsize

\pagebreak{}

\section*{Abstract}

\addcontentsline{toc}{section}{Abstract}

In this thesis, we study counts of quiver representations over finite
rings of truncated power series. We prove a plethystic formula relating
counts of quiver representations over these rings and counts of jets
on fibres of quiver moment maps. This solves a conjecture of Wyss'
and allows us to compute both counts on additional examples, using
local zeta functions. The relation between counts of representations
and counts of jets generalises the relation between Kac polynomials
and counts of points on preprojective stacks. Pursuing this analogy,
we establish further properties of our counts.

We show that, for totally negative quivers, counts of jets converge
to p-adic integrals on fibres of quiver moment maps. One expects a
relation between these p-adic integrals and BPS invariants of preprojective
algebras i.e. Kac polynomials.

For small rank vectors, we also prove that the polynomials counting
indecomposable quiver representations over finite rings have non-negative
coefficients. Moreover, we show that jet schemes of fibres of quiver
moment maps are cohomologically pure in that setting, so that their
Poincaré polynomials are given by the former counts. This is reminiscent
of the structure of cohomological Hall algebras, which are built from
the cohomology of preprojective stacks.

Finally, we compute the cohomology of jet spaces of preprojective
stacks explicitly for the $A_{2}$ quiver. Building on the structure
of the preprojective cohomological Hall algebra of $A_{2}$, we propose
a candidate analogue of the BPS Lie algebra and conjecture the existence
of a Hall product on the cohomology of these jet spaces.

\paragraph*{Keywords}

quiver representations - power series - quiver moment maps - jet spaces
- p-adic integration - Igusa local zeta functions - Kac polynomials
- cohomological Hall algebras

\pagebreak{}

\section*{Résumé}

\addcontentsline{toc}{section}{R\'{e}sum\'{e}}

Dans cette thèse, on étudie certains comptages de représentations
de carquois sur des anneaux finis de séries tronquées. Le premier
résultat est une formule pléthystique exprimant les comptages de jets
sur les fibres de l'application moment d'un carquois en fonction des
comptages de représentations du carquois sur les anneaux en question.
On prouve ainsi une conjecture de Wyss. Ce résultat permet aussi de
donner des formules explicites pour ces comptages dans quelques cas
supplémentaires, grâce à un calcul de fonctions zeta locales. Il généralise
une formule pléthystique plus ancienne, qui relie les polynômes de
Kac et les comptages de points sur le champ des représentations des
algèbres préprojectives. En suivant cette analogie, on arrive à démontrer
plusieurs autres propriétés des comptages étudiés ici.

On montre ainsi que, pour un carquois totalement négatif, les comptages
de jets convergent vers une intégrale p-adique sur les fibres de l'application
moment du carquois. On s'attend à ce que ces intégrales p-adiques
soient reliées aux invariants BPS de l'algèbre préprojective du carquois,
c'est-à-dire aux polynômes de Kac.

On prouve aussi que les polynômes qui comptent les représentations
indécomposables d'un carquois sur ces anneaux finis ont des coefficients
positifs, pour de petits vecteurs de dimensions. Dans ce cas, la cohomologie
des espaces de jets sur les fibres de l'application moment associée
est pure et ses polynômes de Poincaré peuvent être exprimés en fonction
des comptages précédents. Cette structure est très semblable à celle
des algèbres de Hall cohomologiques préprojectives, construites sur
la cohomologie des fibres de l'application moment.

Pour conclure, on calcule explicitement la cohomologie des espaces
de jets des champs de représentations de l'algèbre préprojective du
carquois $A_{2}$. \`{A} partir de la structure de l'algèbre de Hall
cohomologique préprojective de $A_{2}$, on propose une algèbre de
Lie qui généraliserait l'algèbre de Lie BPS. On conjecture également
qu'un produit de Hall peut être construit sur la cohomologie des espaces
de jets étudiés.

\paragraph*{Mots-clés}

représentations de carquois - séries tronquées - applications moment
de carquois - espaces de jets - intégration p-adique - fonctions zeta
locales d'Igusa - polynômes de Kac - algèbres de Hall cohomologiques

\pagebreak{}

\section*{Acknowledgements}

\addcontentsline{toc}{section}{Acknowledgements}

First of all, my sincere thanks go to my PhD advisor, Dimitri Wyss,
for his constant support over these four years. This thesis owes a
lot to his dedication and I am very grateful that I could benefit
from his mathematical guidance, professional advice and encouragements
all along.

I could not have written this work without my colleagues at the Chair
of Arithmetic Geometry either. I would like to express my gratitude
to my academic sisters: Ilaria Rossinelli and Elsa Maneval, for the
many chats, laughs and dinners we had during our time in Lausanne
(and all your teachings, Ila!). I would like to warmly thank our postdocs
as well: Francesca Carocci, Eric Chen, Oscar Kivinen, Sergej Monavari
and Giulio Orecchia, for their company both in and outside the lab
and all the experience they shared with me. Special thanks also go
to Monique Kiener and Laetitia Al-Sulaymaniyin for their friendly
support in all the administrative matters that I encountered at EPFL.

I would like to thank the members of my thesis jury: Philippe Michel,
Donna Testerman, Nero Budur and Emmanuel Letellier, for accepting
to review and examine my thesis and for their helpful comments. I
am also very grateful to my former professors in Paris: Olivier Schiffmann
and Emmanuel Letellier, for their constant support over the past years.

Over these four years, I had the chance to meet many mathematicians
at conferences, seminar talks and also sometimes in less mathematical
circumstances. Special thanks go to Loïs Faisant and Lucien Hennecart
for discussing cool math projects together. It is my pleasure to thank
Nero Budur, Loïs Faisant, Jonathan Gruber, Tam\'{a}s Hausel, Emmanuel
Letellier and Miriam Norris for inviting me to give a talk, as well
as the organisers of the conferences I attended during my PhD. I am
also grateful to Tommaso Botta, Ben Davison, Shivang Jindal, Thibault
Juillard, Sarunas Kaubrys, Fabian Korthauer, Mirko Mauri, Anton Mellit,
Tobias Rossmann, Francesco Sala, Matthew Satriano, Olivier Schiffmann,
Sebastian Schlegel-Mejia, Tommaso Scognamiglio and Paul Ziegler for
helpful and interesting discussions.

I am happy to thank my host institution, \'{E}cole Polytechnique
Fédérale de Lausanne, as well, for the great working conditions. My
work there was supported by the Swiss National Science Foundation
{[}No. 196960{]}.

Important and enjoyable though math research is, there was also time
for doing other things during these years! I would like to thank my
fellow math PhD representatives Katie, Léo, Maude and Micol for the
nice time we spent organising events for our doctoral school. My most
musical thanks go to all musicians from OChE as well. It was a real
pleasure to be part of the orchestra for so many seasons - I will
keep very fond memories of my time there. Finally, I am most grateful
to my friends from Lausanne and from France, for the many good moments
we shared over the past few years, and to my family for their loving
support all along: thank you for everything.

\pagebreak{}

\section*{Introduction \label{Intro}}

\addcontentsline{toc}{section}{Introduction}

The starting point of this thesis are certain counts of quiver representations
over rings of truncated power series. In his own PhD thesis \cite{Wys17b},
Wyss observed certain numerical relations between two families of
such counts. Our goal in this work is to broaden this connection and
provide a conceptual explanation for it, in light of the recent literature
on Kac polynomials, Hall algebras and enumerative invariants of quiver
moduli.

Let us start with an example. Let $Q$ be the cyclic quiver with 3
vertices:
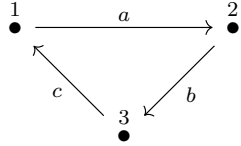
\[
\begin{tikzcd}[ampersand replacement=\&]
\overset{1}{\bullet} \ar[rr,"a"] \& \& \overset{2}{\bullet} \ar[dl,"b"] \\
\& \overset{3}{\bullet} \ar[ul,"c"] \& 
\end{tikzcd}
\]and consider locally-free representations of $Q$ of rank $\rr=(r_{1},r_{2},r_{3})=(1,1,1)$
over the ring $\mathcal{O}_{\alpha}:=\FF_{q}[t]/(t^{\alpha})$ ($\alpha\geq1$).
These are certain collections of linear maps labeled by arrows of
$Q$:
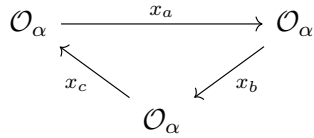
\[
\begin{tikzcd}[ampersand replacement=\&]
\mathcal{O}_{\alpha} \ar[rr,"x_a"] \& \& \mathcal{O}_{\alpha} \ar[dl,"x_b"] \\
\& \mathcal{O}_{\alpha} \ar[ul,"x_c"] \& 
\end{tikzcd}
.
\]Two representations are called isomorphic if they can be obtained
from one another by conjugating $x_{a},x_{b},x_{c}$ with invertible
matrices $g_{1},g_{2},g_{3}\in\mathcal{O}_{\alpha}^{\times}$ placed
at vertices of $Q$. A representation $x$ is called indecomposable
if removing the arrows where $x_{\blacksquare}=0$ does not disconnect
the cycle. The number of isomorphism classes of indecomposable representations
is a polynomial in $q$, which we call $A_{Q,\alpha}$. This is the
first family of counts studied by Wyss.

The other counts are related to a certain moment map $\mu_{Q}$ built
from moduli of quiver representations. The moment map takes as input
a representation $x$ of $Q$ and a representation $y$ of the opposite
quiver:
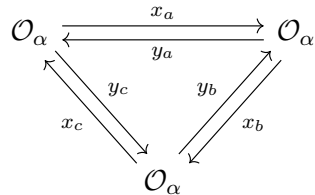
\[
\begin{tikzcd}[ampersand replacement=\&]
\mathcal{O}_{\alpha} \ar[rr,"x_a", shift left]\ar[ddr,"y_c", shift left] \& \& \mathcal{O}_{\alpha} \ar[ddl,"x_b", shift left]\ar[ll,"y_a", shift left] \\
\& \& \\
\& \mathcal{O}_{\alpha} \ar[uul,"x_c", shift left]\ar[uur,"y_b", shift left] \& 
\end{tikzcd}
\]and assigns to $(x,y)$ the matrix $\mu_{Q}(x,y)=(\mu_{1}(x,y),\mu_{2}(x,y),\mu_{3}(x,y))=(x_{c}y_{c}-y_{a}x_{a},x_{a}y_{a}-y_{b}x_{b},x_{b}y_{b}-y_{c}x_{c})$.
The second family of counts studied by Wyss consists of the polynomials
$B_{\mu_{Q},\alpha}$, which count the number of solutions to the
equation $\mu_{Q}(x,y)=0$ over $\mathcal{O}_{\alpha}$.

In this setting, Wyss observed that both counts converge to a rational
fraction in $q$ when $\alpha$ goes to infinity:\begin{align*}
& \underset{\alpha\rightarrow +\infty}{\lim}\left( q^{-\alpha}\cdot A_{Q,\alpha}\right) = \frac{q^2+4q+1}{(q-1)^2}, \\
& \underset{\alpha\rightarrow +\infty}{\lim}\left( q^{-4\alpha}\cdot B_{\mu_Q,\alpha}\right) = \frac{q^2+4q+1}{q^2}.
\end{align*}This numerical coincidence also occurs for longer cyclic quivers and
it was conjectured by Wyss that this holds for all quivers, provided
that the above counts converge when $\alpha$ goes to infinity \cite[Conj. 4.37.]{Wys17b}.
Wyss also conjectured that the numerators of these rational fractions
have non-negative coefficients \cite[Conj. 4.32.]{Wys17b}.

This raises several natural questions: does this generalise to higher
rank vectors $\rr>\underline{1}$? can we find similar formulas for
a fixed value of $\alpha$? and can we find a geometric explanation
for them? When the base field is $\FF_{q}$ (i.e. $\alpha=1$), these
questions have a very rich history in geometric representation theory.
The research involved spans several decades, starting with the introduction
of Kac polynomials in the early eighties and with very recent achievements
over the past five years concerning cohomological Hall algebras. We
will now review part of this story, insisting on the features which
inspired our own work.

\subsection*{Kac polynomials and geometric representation theory}

\paragraph*{Kac polynomials}

The first occurrence of counts of quiver representations over finite
fields goes back to two articles of Kac's \cite{Kac80,Kac82}. Given
a quiver $Q$ and a dimension vector $\dd\in\ZZ_{\geq0}^{Q_{0}}$,
we call $A_{Q,\dd}(q)$ the number of isomorphism classes of absolutely
indecomposable representations of $Q$ over $\FF_{q}$. When $Q$
has no loop arrows, Kac showed that $A_{Q,\dd}$ is a polynomial in
$q$ and $A_{Q,\dd}\ne0$ if, and only if, $\dd$ is a positive root
of $\mathfrak{g}_{Q}$, the Kac-Moody algebra associated to $Q$.
We refer to Section \ref{Subsect/QuivRep} for elementary definitions
on quivers and the associated Lie algebras. \index[terms]{Kac polynomials}\index[notations]{a@$A_{Q,\dd}$ - Kac polynomials}

The relation between quiver representations and Lie algebras was already
observed during the early developments of the theory, with Gabriel's
seminal paper \cite{Gab72}. When $Q$ is the Dynkin diagram of a
semisimple Lie algebra $\mathfrak{g}$, Gabriel showed that there
exists an indecomposable representation in dimension $\dd$ if, and
only if, $\dd$ is a positive root of $\mathfrak{g}$. In that case,
the indecomposable representation is unique (i.e. $A_{Q,\dd}=1$)
and the proof can be carried out with mostly algebraic methods. These
are the only quivers for which there exist finitely many indecomposable
representations (also called quivers of finite-type). Shortly after
that, Bernstein, Gelfand and Pomonarev gave a conceptual proof of
that result by introducing certain reflection functors, whose action
on representations categorifies the action of the Weyl group on dimension
vectors \cite{BGP73}.

Kac's generalisation of Gabriel's theorem was a decisive step in applying
geometric techniques to the representation theory of quivers. Indeed,
for arbitrary quivers, there are simply too many representations to
expect a complete classification. Instead, Kac used the polynomials
$A_{Q,\dd}$ to study the indecomposable locus in the moduli of quiver
representations. In particular, $A_{Q,\dd}$ does not depend on the
orientation of $Q$, which makes it possible to efficiently use reflection
functors for any quiver. Later, King brought further moduli-theoretic
tools to the study of quiver representations by introducing Mumford's
geometric invariant theory into the subject \cite{Kin94}.

Kac also conjectured that $A_{Q,\dd}$ has non-negative coefficients
and that $A_{Q,\dd}(0)=\dim\mathfrak{g}_{Q,\dd}$. These conjectures
sparked considerable work in the following decades and were solved
geometrically as well, by providing realisations of Kac polynomials
and associated Lie algebras in quiver moduli.

\paragraph*{Ringel-Hall algebras and perverse sheaves}

Lie algebras made another appearance in quiver representation theory
with Ringel's generalisation of the classical Hall algebra to ADE
quivers (i.e. Dynkin diagrams of semisimple Lie algebras of type ADE)
\cite{Rin90a}. This was later generalised to quivers without loop
arrows by Green \cite{Gre95}. Given a quiver $Q$, the Ringel-Hall
algebra $\mathbf{H}_{Q}$ is a convolution algebra of functions on
the set of isomorphism classes of representations of $Q$ over a finite
field $\FF_{q}$. Put geometrically, these are functions on $\FF_{q}$-points
of the moduli stack of representations of $Q$. This point of view
was developed by Lusztig in \cite{Lus91}, where he builds a certain
class of perverse sheaves on moduli of quiver representations which
recover functions in the Ringel-Hall algebra via the ``Faisceaux-Fonctions''
correspondence. Lusztig worked with quivers without loops and his
construction was later generalised to arbitrary quivers by Kang, Schiffmann
\cite{KS06} and Bozec \cite{Boz15}. Let us recall some results on
the structure of the Ringel-Hall algebra. \index[terms]{Ringel-Hall algebra}\index[notations]{hq@$\mathbf{H}_Q$ - Ringel-Hall algebra of $Q$}\index[terms]{Lusztig's perverse sheaves}

The first historical example of a Hall algebra is due to Steinitz
and later Hall (see \cite[Ch. II]{Mac95} for a detailed exposition).
We call it the classical Hall algebra. It is (a subalgebra of) the
Ringel-Hall algebra associated to the Jordan quiver:\[
Q=
\begin{tikzcd}[ampersand replacement=\&]
\bullet \arrow[loop, distance=2em, in=325, out=35]
\end{tikzcd}
.
\]The algebra $\mathbf{H}_{Q}$ consists of conjugacy-invariant functions
supported on orbits of nilpotent matrices - or equivalently, isomorphism
classes of $T$-primary $\FF_{q}[T]$-modules. These nilpotent orbits
are indexed by partitions, where $\lambda=(\lambda_{1}\geq\lambda_{2}\geq\ldots\geq\lambda_{r})$
corresponds to the module $M_{\lambda}:=\bigoplus_{i=1}^{r}\FF_{q}[T]/(T^{\lambda_{i}})$.
The Hall product of two classes $[M],[N]$\footnote{More precisely, the indicator function of $[M],[N]$.}
is defined as:\[
[M]*[N]:=\sum_{[P]}c_{[M],[N]}^{[P]}\cdot [P],
\]where $c_{[M],[N]}^{[P]}$ is the number of submodules $P'\subseteq P$
such that $P'\simeq N$ and $P/P'\simeq M$. Hall showed that $\mathbf{H}_{Q}$
is isomorphic to the ring of symmetric polynomials, where $[M_{(1^{r})}]$
corresponds to $q^{-\frac{r(r-1)}{2}}\cdot e_{r}$ ($e_{r}$ being
the $r$-th elementary symmetric function). This can be seen as a
linearised version of parabolic induction for class functions of $\GL_{n}(\FF_{q})$,
which is also controlled by symmetric polynomials.

The quiver-graded version of this result, where the Jordan quiver
is replaced with an arbitrary quiver, was established in the works
of Ringel, Green, Kang and Schiffmann mentioned above \cite{Rin90a,Gre95,KS06}.
Given a vertex $i\in Q_{0}$, call $E_{i}$ the one-dimensional zero
representation concentrated at $i$. Then the subalgebra of $\mathbf{H}_{Q}$
generated by $[E_{i}],\ i\in Q_{0}$ is isomorphic to $U_{\nu}^{-}(\mathfrak{g}_{Q})$,
the negative half of the quantum Borcherds algebra associated to $Q$
(suitably specialised). Lusztig categorified the Hall algebra using
perverse sheaves inspired from his character sheaves, which categorify
class functions of $\GL_{n}(\FF_{q})$ \cite{Lus85a,Lus91}. The induction
and restriction functors, which categorify Hall multiplication and
comultiplication, mirror the same functors for character sheaves,
which categorify parabolic induction and restriction. Moreover, the
simple constitutents of these perverse sheaves gave a geometric realisation
of the so-called canonical basis of $U_{\nu}^{-}(\mathfrak{g}_{Q})$.

The Ringel-Hall algebra was also key to later geometric realisations
of Kac polynomials, as we will see below.

\paragraph*{Moment maps and quiver varieties}

In the first decade of this century, significant progress was made
on Kac's conjectures using the geometry of quiver moment maps. Crawley-Boevey
and Van den Bergh first proved both conjectures for indivisible dimension
vectors in \cite{CBVB04}. Later Hausel \cite{Hau10} and Hausel,
Letellier, Rodriguez-Villegas \cite{HLRV13b} gave proofs of the conjectures
for arbitrary dimension vectors. All these proofs realise Kac polynomials
as certain Poincaré polynomials of quiver varieties and exploit the
interplay between arithmetic counts and Hodge structures on the cohomology
of complex algebraic varieties (see for instance the Appendix of \cite{HRV08}).
Let us give some details about quiver moment maps and their relations
to Kac polynomials and Lie algebras.

Moment maps $\mu_{Q,\dd}$ on quiver moduli arose in different guises
before being investigated systematically by Crawley-Boevey in the
early 2000s \cite{CB01,CB02,CB03a}. They appeared in the study of
singular supports of Lusztig's perverse sheaves \cite{Lus91,KS97}
and in works on Kleinian singularities and other hyperkähler quotients
\cite{Kro89,KN90,CBH98}, in particular in Nakajima's construction
of quiver varieties \cite{Nak94,Nak98}. Moreover, the equations $\mu_{Q,\dd}=0$
cut out the moduli stack of representations of the preprojective algebra
$\Pi_{Q}$ associated to $Q$, which appeared even earlier in several
places \cite{GP79,DR79,BGL87,CB99b}. The common point to all these
works is a certain cotangent construction: the moduli stack cut out
by moment map equations should be regarded as the cotangent bundle
of the moduli stack of quiver representations. Likewise, the preprojective
algebra $\Pi_{Q}$ should be seen as a noncommutative cotangent bundle
of the path algebra of $Q$ (for more details, see Section \ref{Subsect/MomMap}).

The relation between quiver moment maps and Kac polynomials owes to
the fact that moduli of quiver representations form (locally) a smooth
quotient stack $\left[R(Q,\dd)/\GL_{\dd}\right]$. Over a point $x\in R(Q,\dd)$,
the solutions of the moment map equation $\mu_{Q,\dd}(x,y)=0$ form
the conormal space at $x$ to the orbit $\GL_{\dd}\cdot x$. Since
$R(Q,\dd)$ is smooth, this conormal space has dimension $\dim(\GL_{\dd})_{x}+\dim\left[R(Q,\dd)/\GL_{\dd}\right]$.
So counting solutions to the equation $\mu_{Q,\dd}(x,y)=0$ amounts
to counting points $x\in R(Q,\dd)$, along with an endomorphism. This
roughly gives us the count of isomorphism classes of representations
of $Q$, by a refined application of Burnside's lemma. We discuss
this reasoning in more details in Sections \ref{Subsect/MomMap} and
\ref{Subsect/KrullSchmidt}.

The moment map $\mu_{Q,\dd}$ also gave rise to new relations between
Lie algebras and the representation theory of quivers. Kashiwara and
Saito realised the $\mathfrak{g}_{Q}$-crystal associated to the canonical
basis of $U_{\nu}^{-}(\mathfrak{g}_{Q})$ using certain lagrangian
cycles introduced by Lusztig \cite{Lus91,KS97} (see \cite{KKS09,Boz16}
for generalisations to quivers with loops). Around the same time,
Nakajima constructed an action of the Kac-Moody algebra $\mathfrak{g}_{Q}$
on the cohomology of quiver varieties (for quivers without loops)
\cite{Nak98}. The latter construction was key to Hausel's proof of
Kac's constant term conjecture \cite{Hau10}.

All these results found new interpretations in recent years due to
the development of cohomological Hall algebras, which we discuss briefly
below.

\paragraph*{Cohomological Hall algebras and BPS invariants}

Over the past decade, the relations between Kac polynomials, Lie algebras
and counts of solutions to the moment map equation $\mu_{Q,\dd}(x,y)=0$
were very much strengthened through the study of cohomological Hall
algebras (or CoHAs). These algebras appeared originally in the work
of Kontsevich and Soibelman on Donaldson-Thomas theory \cite{KS11}.
But other CoHAs were soon built from the preprojective algebra of
a quiver \cite{SV13b,YZ18a} and related to the Donaldson-Thomas theory
of certain quivers with potentials \cite{Moz11a,Dav17,RS17,YZ20}.

At the level of point-counts over finite fields, the following plethystic
formula was established by Mozgovoy \cite{Moz11a} using the properties
of quiver moment maps mentioned above:\[
\sum_{\dd\in\ZZ_{\geq0}^{Q_0}}
\frac{\sharp_{\FF_q}\mu_{Q,\dd}^{-1}(0)}{\sharp_{\FF_q}\GL_{\dd}}
\cdot q^{\langle\dd,\dd\rangle}t^{\dd}
=
\Exp_{q,t}\left(
\sum_{\dd\in\ZZ_{\geq0}^{Q_0}\setminus\{0\}}
\frac{A_{Q,\dd}(q)}{1-q^{-1}}\cdot t^{\dd}
\right).
\]This formula is a numerical shadow of the following isomorphism of
mixed Hodge structures:\[
\bigoplus_{\dd\geq0}
\HH_{\bullet}^{\mathrm{BM}}\left(\left[\mu_{Q,\dd}^{-1}(0)/\GL_{\dd}\right]\right)\otimes\mathbb{L}^{\otimes (-\langle\dd,\dd\rangle)}
\simeq
\mathrm{Sym}
\left(
\bigoplus_{\dd\geq0}\mathrm{BPS}_{Q,\dd}\otimes \HH^{\bullet}(\mathrm{B}\mathbb{G}_m)
\right).
\]This was established by Davison in \cite{Dav23c}, using structural
results on CoHAs of quivers with potentials obtained with Meinhardt
\cite{DM20}. The weight series of the graded mixed Hodge structure
$\mathrm{BPS}_{Q,\dd}$ recovers $A_{Q,\dd}$. Moreover, Davison proved
that $\HH_{\bullet}^{\mathrm{BM}}\left(\left[\mu_{Q,\dd}^{-1}(0)/\GL_{\dd}\right]\right)$,
hence also $\mathrm{BPS}_{Q,\dd}$ is pure, which reproved Kac's positivity
conjecture. Importantly, this cohomological identity came from a similar
plethystic formula involving mixed Hodge modules on quiver moduli.
In particular, $\mathrm{BPS}_{Q,\dd}$ is the hypercohomology of a
certain pure Hodge module on King's moduli space of quiver representations.

Moreover, the cohomology groups $\bigoplus_{\dd}\HH_{\bullet}^{\mathrm{BM}}\left(\left[\mu_{Q,\dd}^{-1}(0)/\GL_{\dd}\right]\right)$
are endowed with an algebra structure - the cohomological Hall algebra
\cite{SV13b,YZ18a}. For this algebra product, the mixed Hodge structure
$\bigoplus_{\dd}\mathrm{BPS}_{Q,\dd}$ is a $\ZZ\times\ZZ_{\geq0}^{Q_{0}}$-graded
Lie subalgebra $\mathfrak{n}_{Q}$ called the BPS Lie algebra \cite{DM20,Dav20}.
The isomorphism above is induced by the CoHA product and should be
thought of as a Poincaré-Birkhoff-Witt (PBW) isomorphism. In other
words, the CoHA is a noncommutative deformation of the symmetric algebra
$\mathrm{Sym}\left(\mathfrak{n}_{Q}[u]\right)$. This generalises
the Lie-theoretic interpretation of $A_{Q,\dd}(0)$: the coefficients
of $A_{Q,\dd}$ are the graded dimension of a Lie algebra, which can
be realised geometrically. \index[terms]{cohomological Hall algebra}\index[terms]{BPS Lie algebra}\index[terms]{Poincaré-Birkhoff-Witt (PBW) isomorphism}

The structure of $\mathfrak{n}_{Q}$ was recently computed by Davison,
Hennecart and Schlegel-Mejia \cite{DHSM22,DHSM23}. $\mathfrak{n}_{Q}$
is the positive half of a generalised Kac-Moody algebra built from
$Q$. This solved a generalisation of Kac's conjectures by Bozec and
Schiffmann \cite{BS19a}. The sheaf-theoretic upgrade of the CoHA
was a decisive tool in establishing this result and reduced the proof
to understanding the CoHA product on Lusztig's lagrangian cycles.
This problem was in turn solved by Hennecart in \cite{Hen24}, where
he showed that the classes of these lagrangian cycles are generated
by characteristic cycles of Lusztig's perverse sheaves under the CoHA
product.

These results shed a new light on previous occurrences of Lie algebras
in the representation theory of quivers. Indeed, the spherical Ringel-Hall
algebra defined above turns out to be the 0th-degree part of a larger
generalised Kac-Moody algebra, whose graded dimensions are given by
Kac polynomials. On the other hand, the action of Kac-Moody Lie algebras
on the cohomology of Nakajima's quiver varieties \cite{Nak98} can
be recovered from an action of the CoHA \cite{DHSM23}.

\subsection*{Representations of quivers with multiplicities}

Representations of quivers over rings of (truncated) power series
have appeared in a number of works over the past decade. Following
Yamakawa \cite{Yam10}, we will call these representations of quivers
with multiplicities (see Section \ref{Subsect/QuivRep} for complete
definitions). In geometry, certain moduli spaces of irregular meromorphic
connections can be identified with quiver varieties with multiplicities
\cite{Yam10,HWW23}. In representation theory, Geiss, Leclerc and
Schröer built a connection between root data of non-simply-laced Dynkin
diagrams and representations of quivers with multiplicities \cite{GLS17a,GLS17b,GLS16,GLS18a,GLS18b}.
Finally, Hausel, Letellier and Rodriguez-Villegas \cite{HLRV24} studied
analogues of Kac polynomials for certain quivers with multiplicities.
These were conjecturally related to local zeta functions of quiver
moment maps by Wyss in his PhD thesis \cite{Wys17b}. In view of the
literature on Kac polynomials, we identify two main motivations for
studying representations of quivers with multiplicities.

\paragraph*{Quivers with multiplicities and non-simply-laced root data}

The Cartan data associated to usual quivers are all symmetric. For
instance, quivers of finite type only give Dynkin diagrams of type
A, D or E. In order to obtain symmetrisable Cartan data, Geiss, Leclerc
and Schröer consider representations of quivers over rings of truncated
power series, where the length of power series varies along vertices
of the quiver. With this setup, the authors prove an analogue of Gabriel's
theorem \cite{GLS17a}: quivers with multiplicities of finite type
correspond to Dynkin diagrams of semisimple Lie algebras (including
types B,C,F and G). Moreover, the rank vectors of (a certain class
of) indecomposable representations are exactly positive roots. Their
theory includes reflection functors and preprojective algebras as
well. The case of affine symmetrisable Cartan data was also treated
in more recent works \cite{HLS23,Pfe23}.

Moreover, Geiss, Leclerc and Schröer constructed analogues of Ringel-Hall
algebras \cite{GLS16} and Kashiwara and Saito's geometric crystal
\cite{GLS18a}. In the first cited work, the Ringel-Hall algebra is
built from constructible functions on moduli of complex representations.
Its structure is completely known for quivers with multiplicities
of finite type. The second construction works for any quiver with
multiplicities which has no loop arrows and is shown to recover the
crystal of symmetrisable quantum Kac-Moody algebras. It is therefore
natural to ask if analogues of Kac polynomials for quivers with multiplicities
also admit a Lie-theoretic interpretation. For instance, one may hope
to build a cohomological Hall algebra for quivers with multiplicities,
whose weight series recovers the counts that we study in this thesis.
Related constructions appear in a few recent works, but to this day,
the connection to our counts remains elusive \cite{YZ22,VV23b}.

\paragraph*{Quivers with multiplicities and p-adic integrals}

Another motivation for studying counts of quiver representations over
rings of power series is p-adic integration. For instance, the counts
$B_{\mu_{Q},\alpha}$ mentioned above are counts of jets on the moment
map fibre $\mu_{Q}^{-1}(0)$ over $\FF_{q}$. When $\alpha$ goes
to infinity, one would expect that these counts converge to some p-adic
integral, seen as a ``count'' of arcs, or $\FF_{q}[[t]]$-points.
Besides, p-adic integration turns out to be a useful computational
tool: in \cite{Wys17b}, Wyss computed the generating series of all
$B_{\mu_{Q},\alpha},\ \alpha\geq1$ in the form of a local zeta function,
which is a parametric p-adic integral.

Recent works also suggest that a number of cohomological or enumerative
invariants can be recovered as p-adic integrals on certain moduli
spaces. This is the case for stringy Hodge numbers \cite{GWZ20a}
and BPS invariants of local del Pezzo surfaces \cite{COW21}, quivers
with zero potential and meromorphic Higgs bundles \cite{GWZ23a,GWZ23b}.
Since Kac polynomials are BPS invariants of certain quivers with potential,
one may expect that they can be computed as p-adic integrals on quiver
varieties. Additionally, one may hope that the limit of the counts
$B_{\mu_{Q},\alpha}$ can be expressed in terms of Kac polynomials.

\subsection*{Main results}

Let us now sum up our main results and the organisation of the thesis.

\paragraph*{Counting representations of quivers with multiplicities}

In Chapter \ref{Chap/KacPolynomials}, we prove a plethystic formula
relating two families of counts: (i) counts of solutions to the moment
map equation $\mu_{(Q,\nn),\rr}(x,y)=0$ of a quiver with multiplicities
$(Q,\nn)$ over $\FF_{q}$ and (ii) counts $A_{(Q,\nn),\rr}$ of absolutely
indecomposable (locally free) rank $\rr$ representations of $(Q,\nn)$
over $\FF_{q}$. This relation holds for arbitrary rank vectors $\rr\in\ZZ_{\geq0}^{Q_{0}}$
and \textit{fixed} multiplicities $\nn$. This directly generalises
the aforementioned formula of Mozgovoy's \cite{Moz11a}. The preprojective
algebra built by Geiss, Leclerc and Schröer satisfies the same geometric
properties as the classical one (see Chapter \ref{Chap/Preliminaries}),
so Mozgovoy's proof can be adapted without difficulty.

\begin{thmintro} \label{Thm/IntroExpFmlKacPol}

Let $(Q,\nn)$ be a quiver with multiplicities. Call $\langle\bullet,\bullet\rangle$
the Euler form of the associated Borcherds-Cartan matrix with symmetriser
and orientation. Then: \[
\sum_{\rr\in\ZZ_{\geq0}^{Q_0}}
\frac{\sharp_{\FF_q}\mu_{(Q,\nn),\rr}^{-1}(0)}{\sharp_{\FF_q}\GL_{\nn,\rr}}
\cdot q^{\langle\rr,\rr\rangle}t^{\rr}
=
\Exp_{q,t}\left(
\sum_{\rr\in\ZZ_{\geq0}^{Q_0}\setminus\{0\}}
\frac{A_{(Q,\nn),\rr}(q)}{1-q^{-1}}\cdot t^{\rr}
\right).
\]

\end{thmintro}

As an application, we obtain the \textit{asymptotic} relation between
$B_{\mu_{Q},\alpha}$ and $A_{Q,\alpha}$ conjectured by Wyss.

\begin{corintro} \label{Cor/IntroWyssConjAvsB}

Let $Q$ be a 2-connected quiver. Then: \[
\frac{B_{\mu_{Q}}}{(1-q^{-1})^{\sharp Q_0}}=\frac{A_{Q}}{1-q^{-1}}.
\]

\end{corintro}

We also prove that the numerator of $A_{Q}$ (hence that of $B_{\mu_{Q}}$
by Corollary \ref{Cor/IntroWyssConjAvsB}) has non-negative coefficients,
using Hilbert series techniques.

\begin{thmintro} \label{Thm/IntroWyssConjPositivity}

Let $Q$ be a 2-connected quiver. Consider the Stanley-Reisner ring
$\QQ[\Delta]$ associated to the order complex $\Delta$ of the poset
$(\Pi(Q_{1})\setminus\{\emptyset,Q_{1}\},\subseteq)$. Then:

\[
A_Q=\frac{(1-q^{-1})^{b(Q)}}{1-q^{-b(Q)}}\cdot\Hilb_{\Delta}\left(u_{E}=q^{-(b(Q)-b(Q\vert_{E}))}\right).
\] and $\Hilb_{\Delta}\left(u_{E}=q^{-(b(Q)-b(Q\vert_{E}))}\right)$
can be presented as a rational fraction whose numerator has non-negative
coefficients.

\end{thmintro}

Finally, we exploit Theorem \ref{Thm/IntroExpFmlKacPol} to compute
$A_{(Q,\nn),\rr}$ in several cases. When $\rr>\underline{1}$, little
is known on these counts. In particular, they are not known to be
polynomial in $q$. The usual counting methods \cite{Kac83,Hua00,S16,MS20,BSV20}
rely on variants of Burnside's lemma and Jordan decomposition in $\GL_{r}(\FF_{q})$.
However, conjugacy classes in $\GL_{r}(\FF_{q}[t]/(t^{n}))$ are known
only for small values of $r$ \cite{AOPV09,AKOV16}. From another
perspective, computing $A_{(Q,\nn),\rr}$ anounts to counting objects
in a category of homological dimension one (quiver representations)
along with a pair of commuting nilpotent endomorphisms, while the
methods cited above work for objects with one nilpotent endomorphism
\cite[Rem. 5.3.]{Moz19}.

Instead, when the base ring is $\FF_{q}[t]/(t^{\alpha})$ at each
vertex of $Q$, we propose to compute $A_{(Q,\nn),\rr}=A_{(Q,\alpha),\rr}$
by calculating the generating series of $\sharp_{\FF_{q}}\mu_{(Q,\alpha),\rr}^{-1}(0),\ \alpha\geq1$
in the form of a local zeta function. This zeta function is an example
of an ASK (Average Size of Kernels) zeta function \cite{Ros18,RV19,Ros20}.
We provide a formula for ASK zeta functions in terms of determinantal
ideals of certain matrices. This is amenable to computation in a few
more cases. Here is the statement for quiver moment maps:

\begin{propintro} \label{Prop/IntroZetaMomMap}

Let $Q$ be a quiver and $\dd\in\ZZ_{\geq0}^{Q_{0}}$ a dimension
vector.\[
Z_{\mu_{Q,\dd}}(s) = 
1-\sum_{i=0}^{r}\frac{q^{-i}\cdot(1-q^{-s})}{1-q^{-(s+i)}}\cdot \int_{R(Q,\dd)(\mathcal{O})} \left(\left(\frac{\Vert\Delta_{i}(x)\Vert}{\Vert\Delta_{i-1}(x)\Vert}\right)^{s+i}-\left(\frac{\Vert\Delta_{i+1}(x)\Vert}{\Vert\Delta_{i}(x)\Vert}\right)^{s+i} \right)\cdot\frac{dx}{\Vert\Delta_{i}(x)\Vert},
\]where $\Delta_{i}(x)$ is the vector of all $i\times i$ minors of
the matrix $\mu_{Q,\dd}(x,\bullet)$. We also fix the conventions
$\Vert\Delta_{-1}(x)\Vert=\Vert\Delta_{0}(x)\Vert=1$ and $\Vert\Delta_{r+1}(x)\Vert=0$,
where $r:=\underset{x\in R(Q,\dd)(\mathcal{O})}{\max}\{\rk(\mu_{Q,\dd}(x,\bullet))\}$.

\end{propintro}

In general, estimating the size of rank-loci for arbitrary matrices
is a wild problem \cite{Ros22}. An instance of such counting problems
is Higman's conjecture \cite{OBV15,Ros18,Ros20}. This concerns counts
of conjugacy classes of unitriangular matrices. As observed by Mozgovoy
in the aforementioned article \cite{Moz19}, this problem seems to
lie beyond the reach of current counting methods in quiver representation
theory, already over finite fields. The corresponding ASK zeta function
encodes the analogue counts over $\FF_{q}[t]/(t^{\alpha})$ for all
$\alpha$. Given that counts of quiver representations over $\FF_{q}$
are tame, one may expect that this tameness can be propagated to the
corresponding ASK zeta function. We hope to return to this question
in future works.

\paragraph*{Counts of jets and p-adic integrals}

In Chapter \ref{Chap/RatSg}, we study the asymptotic behaviour of
the counts $\sharp_{\FF_{q}}\mu_{(Q,\alpha),\dd}^{-1}(0)$, $\alpha\geq1$
i.e. counts of jets on $\mu_{Q,\dd}^{-1}(0)$ over finite fields.
When $\dd=\underline{1}$, Wyss showed that these counts, once suitably
normalised, converge as $\alpha$ goes to infinity if, and only if,
$Q$ is 2-connected. As was shown by Aizenbud, Avni and Glazer, this
is related to singularities of $\mu_{Q,\dd}^{-1}(0)$ \cite{AA18,Gla19}.
For $\dd>\underline{1}$, we show that a large class of quiver moment
maps have rational singularities (see Section \ref{Subsect/RatSgTotNegQuiv}
for a definition), which implies a convergence result for the counts
at hand. The proof closely follows methods introduced by Budur in
\cite{Bud21}.

\begin{thmintro} \label{Thm/IntroRatSgTotNeg}

Let $Q$ be a quiver and $\dd\in\ZZ_{\geq0}^{Q_{0}}\setminus\{0\}$
such that $(Q,\dd)$ has property (P). Then $\mu_{Q,\dd}^{-1}(0)$
has rational singularities.

\end{thmintro}

In addition to this, we provide a geometric interpretation for the
limit of these counts. Given a scheme $X$ with mild singularities,
we show that counts of jets converge to the p-adic volume of a certain
analytic manifold $X^{\natural}$ associated to $X$.

\begin{thmintro} \label{Thm/IntroJetCountCanMeas}

Let $X$ be a $\mathbb{Z}$-scheme of finite type and assume that
$X_{\bar{\QQ}}$ is locally complete intersection, of pure dimension
$d$ and has rational singularities. Let $F$ be a local field of
characteristic zero, with valuation ring $\mathcal{O}$ and residue
field $\FF_{q}$ (of characteristic $p$). Then if $p$ is large enough,
the sequence $q^{-nd}\cdot\sharp X(\mathbb{F}_{q}[t]/(t^{n})),\ n\geq1$
converges and its limit is given by:\[
\underset{n\rightarrow +\infty}{\lim}\frac{\sharp X(\mathbb{F}_{q}[t]/(t^{n}))}{q^{nd}}
=
\nu_{\mathrm{can}}(X^{\natural}).
\]\end{thmintro}

In the case of quiver moment maps, counts of jets converge to a p-adic
integral on $\mu_{Q,\dd}^{-1}(0)$ - alternatively, on the quotient
stack $\left[\mu_{Q,\dd}^{-1}(0)/\GL_{\dd}\right]$. In view of recent
works on p-adic integrals and BPS invariants \cite{COW21,GWZ23a,GWZ23b},
it is natural to look for a relation between integrals on the quiver
variety $\mu_{Q,\dd}^{-1}(0)\git\GL_{\dd}$ and the p-adic integral
we obtain. Moreover, we show that counts of jets on any moduli space
locally modelled on moment maps as in Theorem \ref{Thm/IntroRatSgTotNeg}
converge to a p-adic integral. One could thus ask a similar question
for moduli spaces of objects in 2-Calabi-Yau categories, as in \cite{Dav21a}.
Recently, Satriano and Usatine proved new change of variable formulas
for motivic integrals, which apply to good moduli spaces of Artin
stacks \cite{SU21,SU23c}. We hope to leverage these techniques in
future work to figure out the relation between the p-adic integrals
above and those which found applications in enumerative geometry.

\paragraph*{Positivity and purity: categorification results}

In Chapter \ref{Chap/Categorification}, we prove the analogue of
Kac's positivity conjecture for the counts $A_{(Q,\alpha),\rr}$ when
$\rr=\underline{1}$. Our proof uses an algorithm which contracts
or deletes edges of $Q$ and is inspired from a work by Abdelgadir,
Mellit and Rodriguez-Villegas on toric Kac polynomials \cite{AMRV22}.

\begin{thmintro} \label{Thm/IntroPosToricKacPol}

Let $\rr=\underline{1}$. Then $A_{(Q,\alpha),\rr}$ has non-negative
coefficients.

\end{thmintro}

Moreover, this contraction-deletion algorithm allows us to compute
the cohomology of the moduli stack $\left[\mu_{(Q,\alpha),\rr}^{-1}(0)/\GL_{\alpha,\rr}\right]$.
We obtain a plethystic formula which partially categorifies Theorem
\ref{Thm/IntroExpFmlKacPol}.

\begin{thmintro} \label{Thm/IntroPurityToricPreprojStack}

Let $\rr=\underline{1}$. Then:

\[
\HH_{\mathrm{c}}^{\bullet}\left(\left[\mu_{(Q,\alpha),\rr}^{-1}(0)/\GL_{\alpha,\rr}\right]\right)
\otimes\mathbb{L}^{\otimes\alpha\langle\rr,\rr\rangle}
\simeq
\bigoplus_{Q_0=I_1\sqcup\ldots\sqcup I_s}
\bigotimes_{j=1}^s
\left(
A_{(Q\vert_{I_j},\alpha),\rr\vert_{I_j}}(\mathbb{L})\otimes\mathbb{L}\otimes \HH_{\mathrm{c}}^{\bullet}(\mathrm{B}\mathbb{G}_m)
\right).
\]

In particular, $\HH_{\mathrm{c}}^{\bullet}\left(\left[\mu_{(Q,\alpha),\rr}^{-1}(0)/\GL_{\alpha,\rr}\right]\right)$
carries a pure Hodge structure.

\end{thmintro}

This result is reminiscent of the structure of $\HH_{\mathrm{c}}^{\bullet}\left(\left[\mu_{Q,\dd}^{-1}(0)/\GL_{\dd}\right]\right)$,
as exposed in \cite{Dav23c}. When $\rr=\underline{1}$, our computation
yields a posteriori a subspace of $\HH_{\mathrm{c}}^{\bullet}\left(\left[\mu_{(Q,\alpha),\rr}^{-1}(0)/\GL_{\alpha,\rr}\right]\right)$
analogous to $\mathrm{BPS}_{Q,\dd}\subseteq\HH_{\mathrm{c}}^{\bullet}\left(\left[\mu_{Q,\dd}^{-1}(0)/\GL_{\dd}\right]\right)$.
Theorem \ref{Thm/IntroPurityToricPreprojStack} can be seen as a first
piece of evidence that the plethystic formula from Theorem \ref{Thm/IntroExpFmlKacPol}
may be categorified. However, an important hurdle in this direction
is the absence of a good moduli theory for the quotient stacks $\left[\mu_{(Q,\alpha),\rr}^{-1}(0)/\GL_{\alpha,\rr}\right]$.
Since the group $\GL_{\alpha,\rr}$ is not reductive, King's moduli
space of quiver representations is unavailable and we cannot build
an analogue of $\mathrm{BPS}_{Q,\dd}$ from the hypercohomology of
a complex yet. However, recent work in non-reductive geometric invariant
theory \cite{BDHK18,BDHK20,HHJ24} suggests that this gap may now
be bridged. We hope to apply NRGIT techniques to this problem in future
works.

\paragraph*{Tentative cohomological Hall algebras for quivers with multiplicities}

In Chapter \ref{Chap/HallAlg}, we discuss some computations that
we carried out in an effort to build a preprojective CoHA in simple
cases. An important hurdle in this direction is the bad geometric
behaviour of flags of representations of quivers with multiplicities.
Let us give some details.

Hall algebra products are generally built from the following correspondence:\[
\begin{tikzcd}[ampersand replacement=\&]
\& \mathcal{E}xt_{\rr_1,\rr_2} \ar[dl,"q"']\ar[dr,"p"] \& \\
\mathfrak{M}_{\rr_1}\times\mathfrak{M}_{\rr_2} \& \& \mathfrak{M}_{\rr_1+\rr_2}
\end{tikzcd}
,
\]where $\mathfrak{M}_{\rr}$ is the moduli stack of representations
of rank vector $\rr$ and $\mathcal{E}xt_{\rr_{1},\rr_{2}}$ is the
moduli stack of extensions, which classifies short exact sequences
$M'\rightarrow M\rightarrow M''$, where $\rk(M')=\rr_{2}$ and $\rk(M'')=\rr_{1}$.
The morphisms $p,q$ are defined by $p:(M'\rightarrow M\rightarrow M'')\mapsto M$
and $q:(M'\rightarrow M\rightarrow M'')\mapsto(M'',M')$. In order
to build a Hall product on perverse sheaves following Lusztig \cite{Lus91}
or on cycles in Borel-Moore homology as for CoHAs, one needs $p$
to be proper. However, this is no longer the case for quivers with
multiplicities.

We discuss some consequences of this fact and a workaround in the
special case where $(Q,\nn)=(A_{2},\alpha)$. We first show that the
analogues of Lusztig sheaves in that setting are not semisimple. Therefore,
Lusztig's categorification of the Hall product does not directly generalise
to our setting. Likewise, we were, to this point, not able to build
a CoHA product using a correspondence diagram. Instead, we build a
partial CoHA product on the Borel-Moore homology of the preprojective
stack, for $(Q,\nn)=(A_{2},\alpha)$. To do this, we work with constructible
functions in the Ringel-Hall algebra and use their characteristic
cycles to mimic the CoHA product in the case $\alpha=1$, following
results of Hennecart \cite{Hen24}.

\begin{propintro} \label{Prop/IntroTopCoHA}

Consider the quiver with multiplicities $(Q,\nn)=(A_{2},\alpha)$,
where $\alpha\geq1$. For $\rr\in\ZZ_{\geq0}^{Q_{0}}$, consider the
$\QQ$-vector space $\mathbf{H}_{(A_{2},\alpha),\rr}$ of constructible
functions on $R(A_{2},\alpha;\rr)$ which are constant along $\GL_{\alpha,\rr}$-orbits.
Then the characteristic cycle map induces an isomorphism of $\ZZ_{\geq0}^{Q_{0}}$-graded
$\QQ$-vector spaces:\[
\mathrm{CC}:
\mathbf{H}_{(A_2,\alpha)}:=\bigoplus_{\rr\in\ZZ_{\geq0}^{Q_0}}\mathbf{H}_{(A_2,\alpha),\rr}
\overset{\sim}{\longrightarrow}
\bigoplus_{\rr\in\ZZ_{\geq0}^{Q_0}}\HH_{\mathrm{top}}^{\mathrm{BM}}\left(\mathfrak{M}_{\Pi_{(Q,\alpha)},\rr}\right)
,
\]where $\HH_{\mathrm{top}}^{\mathrm{BM}}\left(\mathfrak{M}_{\Pi_{(Q,\nn)},\rr}\right)$
is the equivariant Borel-Moore homology group of degree $2\dim\left(\mathfrak{M}_{\Pi_{(Q,\nn)},\rr}\right)$.

Moreover, when endowed with the Ringel-Hall algebra product, $\mathbf{H}_{(A_{2},\alpha)}$
is isomorphic to $U(\tilde{\mathfrak{n}})$, where $\tilde{\mathfrak{n}}$
is the trivial extension of the upper half of $\mathfrak{sl}_{3}$
by an abelian Lie algebra of dimension $\alpha-1$.

\end{propintro}

We consider this computation as a toy example of a putative preprojective
CoHA for quivers with multiplicities. We plan to extend this construction
to other root data (possibly non-simply-laced) in future works.

\subsection*{Preliminaries and notations}

We collected in Chapter \ref{Chap/Preliminaries} foundational material
on quiver representations and the geometric tools used throughout
the thesis. The reader willing to skip these preliminaries can refer
to the short \hyperref[Section/indexTerms]{Index} and to the \hyperref[Section/indexNotations]{Glossary of notations}
at the end of the thesis for quick reference to Chapter \ref{Chap/Preliminaries}.

\pagebreak{}

\section{Preliminaries \label{Chap/Preliminaries}}

In this chapter, we introduce the basic notions and tools used throughout
the thesis. In Sections \ref{Subsect/QuivRep}, \ref{Subsect/MomMap}
and \ref{Subsect/KrullSchmidt}, we cover our main objects of study:
representations of quivers with multiplicities, the associated moment
maps and indecomposable representations. We then turn to p-adic integrals
and their relation to resolutions of singularities (Section \ref{Subsect/p-adic}).
We exploit these in Chapter \ref{Chap/KacPolynomials} for computations
and in the study of singularities of quiver moment maps in Chapter
\ref{Chap/RatSg}. Finally, we collect in Section \ref{Subsect/CohAlgVar}
all the cohomological tools that we need in order to categorify counts
of quiver representations with multiplicities in Chapters \ref{Chap/Categorification}
and \ref{Chap/HallAlg} (mixed Hodge structures, perverse sheaves,
characteristic cycles). This chapter contains no original results;
facts stated without a precise reference are well-known to experts
and only the exposition here might be new.

\subsection{Quiver representations and root systems \label{Subsect/QuivRep}}

In this section, we introduce representations of quivers with multiplicities.
These are representations of quivers over rings of truncated power
series, that is $\KK[t]/(t^{\alpha}),\alpha\geq1$, where $\KK$ is
a field. We present quiver representations over these rings under
two levels of generality: quiver representations over the base ring
$\KK[t]/(t^{\alpha})$, as studied in \cite{Wys17b,HLRV24}, and quiver
representations where the multiplicity $\alpha$ varies along vertices
of the quiver, following the framework of Geiss, Leclerc and Schröer
\cite{GLS17a}. See also \cite{Yam10,HWW23} for applications of the
latter representations to the study of meromorphic connections.

\paragraph*{Quiver representations}

Let us first set notations for quivers.

\begin{df}[Quiver]\index[terms]{quiver} \index[notations]{q@$Q$ - quiver}\index[notations]{q@$Q_0$ - set of vertices}\index[notations]{q@$Q_1$ - set of arrows}\index[notations]{s@$s,t$ - source and target maps}

A quiver $Q=(Q_{0},Q_{1},s,t)$ is the datum of a finite set of vertices
$Q_{0}$, a finite set of arrows $Q_{1}$, along with source and target
maps $s,t:Q_{1}\rightarrow Q_{0}$.

\end{df}

A quiver is typically represented as an oriented graph, possibly with
parallel arrows and loop arrows. Vertices are indexed by $Q_{0}$
and the graph contains one arrow $s(a)\rightarrow t(a)$ for every
$a\in Q_{1}$.

A representation of a quiver over a ring $R$ can be thought of as
a realisation of $Q$ in the category of $R$-modules.

\begin{df}[Quiver representation]

Let $R$ be a ring.

A representation of $Q$ over $R$ is the datum of $R$-modules $M_{i},\ i\in Q_{0}$
and $R$-linear maps $f_{a}:M_{s(a)}\rightarrow M_{t(a)},\ a\in Q_{1}$.
We call this representation $M$. We say that $M$ is locally free
if $M_{i}$ is a free $R$-module for every $i\in Q_{0}$.\footnote{Some authors require that $M_{i}$ be projective over $R$, see for
instance \cite{Sch92}. Since for us $R$ will always be a local ring,
this will not make any difference.}

Let $M,M'$ be two representations of $Q$ over $R$. A morphism of
representations is the datum of $R$-linear maps $\varphi_{i}:M_{i}\rightarrow M'_{i}$
such that $f'_{a}\circ\varphi_{s(a)}=\varphi_{t(a)}\circ f_{a}$ for
all $a\in Q_{1}$. We denote such a morphism by $\varphi:M\rightarrow M'$.

Let $M$ be a locally free representation of $Q$ over $R$ and suppose
that $M_{i}$ is a finitely generated $R$-module for every $i\in Q_{0}$.
Then the rank vector of $M$ is the tuple $\rk(M):=(\rk_{R}(M_{i}))_{i\in Q_{0}}\in\ZZ_{\geq0}^{Q_{0}}$.
We denote by $\Rep_{R}(Q)$ (resp. $\Rep_{R}^{\mathrm{l.f.}}(Q)$)
the category of (locally free) representations of $Q$ over $R$.

\end{df}

\begin{rmk}

When $R$ is a field, we will also call rank vectors \textit{dimension}
vectors. We will typically denote dimension vectors by $\dd$ and
rank vectors by $\rr$.

\end{rmk}

\paragraph*{The root system of a quiver}

To any quiver, one can associate a root system $\Delta\subseteq\ZZ^{Q_{0}}$,
see \cite{Kac83}. Root systems of quivers include, but are not limited
to, root systems of semisimple Lie algebras. As we recall below, they
correspond to root systems of so-called Borcherds algebras.

The relation between representations of quivers and root systems of
Lie algebras was discovered by Gabriel in \cite{Gab72}. Gabriel showed
that quivers which only possess finitely many indecomposable representations
are precisely those obtained by orienting the Dynkin diagram of a
semisimple Lie algebra. Shortly after, Bernstein, Gelfand and Ponomarev
found a conceptual explanation of this result, by categorifying the
action of the Weyl group on dimension vectors through so-called reflection
functors \cite{BGP73}. The connection with root systems of infinite-dimensional
Lie algebras was established by Kac in \cite{Kac80}, who showed that
$\Delta$ is exactly the set of dimension vectors of indecomposable
representations of $Q$.

The root system $\Delta$ is generated by certain subsets of roots
under the action of a Weyl group, which we now introduce. In order
to define the reflections generating the Weyl group, we need a certain
bilinear form on $\ZZ^{Q_{0}}$, built from $Q$.

\begin{df}[Euler form]\index[terms]{Euler form}\index[notations]{eu@$\langle\bullet,\bullet\rangle$ - Euler form}

The Euler form of $Q$ is the bilinear form on $\ZZ^{Q_{0}}$ given
by:\[
\begin{array}{cccc}
\langle\bullet,\bullet\rangle_Q: & \ZZ^{Q_0}\times\ZZ^{Q_0} & \rightarrow & \ZZ \\
& (\dd,\ee) & \mapsto & \sum_{i\in Q_0}d_ie_i-\sum_{a\in Q_1}d_{s(a)}e_{t(a)}
\end{array}
.
\] The symmetrised Euler form of $Q$ is the bilinear form on $\ZZ^{Q_{0}}$
given by $(\dd,\ee)_{Q}:=\langle\dd,\ee\rangle_{Q}+\langle\ee,\dd\rangle_{Q}$.

\end{df}

As in the case of root systems of semisimple Lie algebras, $\Delta=\Delta_{+}\sqcup(-\Delta_{+})$,
where $\Delta_{+}=\Delta\cap\ZZ_{\geq0}^{Q_{0}}$. Moreover, roots
can be either real or imaginary. Real roots (resp. imaginary roots)
are characterised by the property $(\dd,\dd)_{Q}=2$ (resp. $(\dd,\dd)_{Q}\leq0$).
We denote by $\Delta^{\mathrm{re}}$ (resp. $\Delta^{\mathrm{im}}$,
$\Delta_{+}^{\mathrm{re}}$, $\Delta_{+}^{\mathrm{im}}$) the subset
of real (resp. imaginary, positive real, positive imaginary) roots.
We now explain how to generate $\Delta^{\mathrm{re}}$ and $\Delta^{\mathrm{im}}$
from simple roots and the so-called fundamental set. We will take
this as our definition of $\Delta$.

\begin{df}[Simple roots]

Let us call $\epsilon_{i},\ i\in Q_{0}$ the canonical $\ZZ$-basis
of $\ZZ^{Q_{0}}$. The simple roots of $Q$ are the dimension vectors
$\epsilon_{i},\ i\in Q_{0}$. A simple root $\epsilon_{i}$ is called
real if $i$ does not carry any loops and imaginary otherwise. We
call $\Pi$ (resp. $\Pi^{\mathrm{re}}$, $\Pi^{\mathrm{im}}$) the
set of simple roots (resp. real, imaginary simple roots).

\end{df}

The definition of real and imaginary simple roots is consistent with
the characterisation by the Euler form above, as $(\epsilon_{i},\epsilon_{i})=2(1-\sharp\{a\in Q_{1}\ \vert\ s(a)=t(a)=i\})$.
We can now define the Weyl group.

\begin{df}[Weyl group]

The Weyl group of $Q$ is the group of automorphisms of $\ZZ^{Q_{0}}$
generated by the following reflections:\[
\begin{array}{cccc}
r_i: & \ZZ^{Q_0} & \rightarrow & \ZZ^{Q_0} \\
& \dd & \mapsto & \dd-(\dd,\epsilon_i)_Q\cdot \epsilon_i
\end{array}
,\ i\in\Pi^{\mathrm{re}}.
\] Let us denote the Weyl group of $Q$ by $W$.

\end{df}

While real roots can be obtained from real simple roots under the
action of the Weyl group, imaginary roots are obtained from a larger
set, the fundamental set, which we now define. Given a dimension vector
$\dd$, the support of $\dd$ is $\supp(\dd):=\{i\in Q_{0}\ \vert\ d_{i}\ne0\}$.

\begin{df}[Fundamental set]

The fundamental set of $Q$ is the following subset of $\ZZ^{Q_{0}}$:\[
F_Q:=
\left\{
\dd\in\ZZ_{\geq0}^{Q_0}\setminus\{0\}\ \vert\ 
\substack{
\forall i\in Q_0,\ (\dd,\epsilon_i)_Q\leq0 \\
\supp(\dd)\ \mathrm{is}\ \mathrm{connected}
}
\right\},
\] where $\supp(\dd)$ is seen as the subquiver of $Q$ whose set of
vertices is $\supp(\dd)$ and whose arrows are those joining vertices
in $\supp(\dd)$.

\end{df}

We can finally define:

\begin{df}[Root system of a quiver]

The set of real (resp. imaginary) roots of $Q$ is defined as $\Delta^{\mathrm{re}}=W\cdot\Pi^{\mathrm{re}}$
(resp. $\Delta^{\mathrm{im}}=W\cdot(\pm F_{Q})$).

\end{df}

One can associate a symmetric Borcherds-Cartan matrix to $Q$, in
the sense of \cite{Bor88} (see also \cite[\S 11.13]{Kac95}):\[
C:=\Big( (\epsilon_i,\epsilon_j)\Big)_{i,j\in Q_0}.
\]Similarly to the construction of semisimple Lie algebras from their
Cartan matrix, one can construct a so-called Borcherds algebra $\mathfrak{g}(C)$
from $C$. The Borcherds algebra admits a decomposition into root
spaces as in the semisimple case, which yields a (possibly infinite)
set of roots contained in $\ZZ^{Q_{0}}$.

\begin{prop}{\cite[\S 11.13]{Kac95}}

The set of roots of $\mathfrak{g}(C)$ is $\Delta$.

\end{prop}

\begin{rmk}

Note that, for $i,j\in Q_{0}$, $2\delta_{i,j}-(\epsilon_{i},\epsilon_{j})$
is the number of arrows of $Q$ joining $i$ and $j$. Thus the generalised
Cartan matrix $C$ determines $Q$, up to orientation.

\end{rmk}

Before moving on to representations of quivers with multiplicities
in the sense of Geiss, Leclerc and Schröer, we introduce the path
algebra of a quiver, since quiver representations with multiplicities
are presented in \cite{GLS17a} as modules over a generalised path
algebra.

\paragraph*{Path algebra and preprojective algebra}

Let $R$ be a ring and $Q$ a quiver. The path algebra $RQ$ is designed
in such a way that the categories $\Rep_{R}(Q)$ and $RQ-\mathrm{Mod}$
are equivalent. Path algebras of quivers are ubiquitous in the representation
theory of finite-dimensional algebras, for any basic algebra over
a field is a quotient of some path algebra \cite{ARS97}. Let us review
their construction.

\begin{df}[Path algebra]

A non-trivial path in $Q$ is an ordered sequence $a_{r}\ldots a_{1}\in Q_{1}$
($r\geq1$) of edges in $Q$, such that $s(a_{i+1})=t(a_{i})$ for
$1\leq i\leq r-1$. Given $i\in Q_{0}$, the trivial (or lazy) path
$e_{i}$ is the path starting in $i$ and terminating in $i$ without
following any arrow.

Let $\hat{Q}$ be the set of paths in $Q$ and extend the source and
target maps to $\hat{Q}$ in the obvious way. The path algebra $RQ$
is the $R$-module $\bigoplus_{\rho\in\hat{Q}}R\rho$, endowed with
the following multiplication law:\[
x\cdot y=
\left\{
\begin{array}{ll}
xy & \mathrm{if}\ t(y)=s(x), \\
0 & \mathrm{else.}
\end{array}
\right.
\]

\end{df}

As mentioned above, the following fact can be checked easily:

\begin{prop}{\cite[\S 1]{CB92}} \label{Prop/QuivRepVSModPathAlg}

There is an equivalence of categories $\Rep_{R}(Q)\simeq RQ-\mathrm{Mod}$.

\end{prop}

\begin{rmk}

Using this equivalence, the usual representation-theoretic notions
of submodule, quotient module and direct sum of modules naturally
translate to the notions of subrepresentation, quotient representation
and direct sum of representations of $Q$.

\end{rmk}

We now turn to the preprojective algebra associated to a quiver. This
algebra can be presented as the quotient of the path algebra of the
double quiver $\overline{Q}$ by a certain ideal of quadratic relations.
Let us recall its definition.

\begin{df}[Double quiver]

Let $Q$ be a quiver. The double quiver of $Q$ is the quiver $\overline{Q}=(Q_{0},\overline{Q}_{1},\overline{s},\overline{t})$,
where $\overline{Q}_{1}:=Q_{1}\sqcup\{a^{*},\ a\in Q_{1}\}$ and $\overline{s},\overline{t}$
are defined by extending $s,t$ to $\overline{Q}_{1}$, so that $(\overline{s}(a^{*}),\overline{t}(a^{*}))=(t(a),s(a))$.

\end{df}

\begin{df}[Preprojective algebra]\index[terms]{preprojective algebra}\index[notations]{pi@$\Pi_Q$ - preprojective algebra of $Q$}

Let $Q$ be a quiver and $R$ be a ring. The preprojective algebra
of $Q$ over $R$ is defined as:\[
\Pi_Q:=R\overline{Q}/\left(\sum_{a\in Q_1}[a,a^*]\right).
\]

\end{df}

The preprojective algebra was introduced in the late seventies by
Gelfand and Ponomarev \cite{GP79}, who endowed the sum of indecomposable
preprojective modules over $RQ$ with an algebra structure - hence
the name ``preprojective algebra''. The equivalence between this
presentation and the definition above can be found in \cite{DR79,Rin98,CB99b}.
The construction of $\Pi_{Q}$ from the path algebra $RQ$ turns out
to be a truncated version of Keller's Calabi-Yau completion in the
context of differential-graded algebras \cite{Kel11}. In this sense,
$\Pi_{Q}$ should be thought as a noncommutative cotangent bundle
of $RQ$. We will return to this fundamental fact in the next sections
and explain how this connects moduli of $\Pi_{Q}$-modules and counts
of representations of $Q$.

\paragraph*{Representations of quivers with multiplicities after Geiss, Leclerc
and Schröer}

We now introduce a more general notion of quiver representations due
to Geiss, Leclerc and Schröer \cite{GLS17a}, which we call representations
of quivers with multiplicities, following Yamakawa \cite{Yam10}.

We saw earlier that the datum of a quiver $Q$ is, up to orientation,
equivalent to that of a symmetric generalised Cartan matrix. Moreover,
the category $\Rep_{R}(Q)$ is equivalent to the category of modules
over the path algebra $RQ$. The idea of Geiss, Leclerc and Schröer
is to associate a modified path algebra $H(C,D,\Omega)$ to a symmetrisable
generalised Cartan matrix $C$, a symmetriser $D$ and an orientation
$\Omega$. In what follows, we define a slight generalisation of their
construction, which shares all the properties of the initial construction
that we use in the thesis. We work over a base field $\KK$.

\begin{df}[Symmetrisable Borcherds-Cartan matrix and symmetriser]

Let $n\geq1$ be an integer. A symmetrisable Borcherds-Cartan matrix
with symmetriser is a pair $(C,D)$, where $C=(c_{ij})_{1\leq i,j\leq n}$
is a matrix with integer coefficients, $D$ is a diagonal matrix of
size $n$, with positive integer diagonal coefficients $c_{1},\ldots,c_{n}$,
which satisfy the following conditions:

\begin{enumerate}[align=left]

\item[(BC1)] For all $i\in\{1,\ldots,n\}$, $c_{ii}\in\{2,0,-2,\ldots\}$.

\item[(BC2)] For all $i,j\in\{1,\ldots,n\}$, $i\ne j\Rightarrow c_{ij}\leq0$.

\item[(BC3)] For all $i,j\in\{1,\ldots,n\}$, $c_{ij}=0\Rightarrow c_{ji}=0$.

\item[(BC4)] $DC$ is symmetric.

\end{enumerate}

We call $C$ a Cartan matrix (or symmetrisable Borcherds-Cartan matrix)
and $D$ a symmetriser of $C$.

\end{df}

Following Geiss, Leclerc and Schröer \cite{GLS17a}, we define $g_{ij}:=\mathrm{gcd}(\vert c_{ij}\vert,\vert c_{ji}\vert)$,
$f_{ij}:=\vert c_{ij}\vert/g_{ij}$ and $k_{ij}:=\mathrm{gcd}(c_{i},c_{j})$
for $i\ne j$. Since $DC$ is symmetric, we obtain $c_{i}c_{ij}=c_{j}c_{ji}$
for all $i\ne j$, hence $c_{i}=k_{ij}f_{ji}$. Let us now define
an orientation of $C$.

\begin{df}[Orientation]

Let $(C,D)$ be a symmetrisable Borcherds-Cartan matrix with symmetriser.
Consider \[
\tilde{Q_1}:=\{\alpha_{ij}^{(g)}:j\rightarrow i,\ 1\leq i\ne j\leq n,\ 1\leq g\leq g_{ij}\}
\sqcup \{\alpha_{ii}^{(g)}:i\rightarrow i,\ 1\leq g\leq 2-c_{ii}\}
.
\] An orientation of $C$ is a subset $\Omega\subseteq\tilde{Q_{1}}$
satisfying the following requirements:

\begin{enumerate}

\item[(i)] For all $i\in\{1,\ldots,n\}$, for all $g\in\{1,\ldots,2-c_{ii}\}$,
$\alpha_{ii}^{(g)}\in\Omega$.

\item[(ii)] For all distinct $i,j\in\{1,\ldots,n\}$, for all $g\in\{1,\ldots,g_{ij}\}$,
either $\alpha_{ij}^{(g)}\in\Omega$ or $\alpha_{ji}^{(g)}\in\Omega$.

\end{enumerate}

\end{df}

Note that the orientation above allows to define a quiver $Q=(Q_{0},Q_{1},s,t)$,
where $Q_{0}:=\{1,\ldots,n\}$, $Q_{1}:=\Omega$, $s(\alpha_{ij}^{(g)})=j$
and $t(\alpha_{ij}^{(g)})=i$. We can now introduce a variant of Geiss,
Leclerc and Schröer's algebra:

\begin{df}

Consider the quiver $Q^{+}:=(Q_{0},Q_{1}^{+},s,t)$, where $Q_{1}^{+}:=\Omega\sqcup\{\epsilon_{i}:i\rightarrow i,\ i\in Q_{0}\}$.
The algebra $H(C,D,\Omega)$ is the quotient of the path algebra $\KK Q^{+}$
by the (double-sided) ideal generated by the following relations:

\begin{enumerate}[align=left]

\item[(H1)] $\epsilon_{i}^{c_{i}}$, for $i\in Q_{0}$,

\item[(H2)] $\epsilon_{i}^{f_{ji}}\alpha_{ij}^{(g)}-\alpha_{ij}^{(g)}\epsilon_{j}^{f_{ij}}$,
for $i\ne j\in Q_{0}$ and $\alpha_{ij}^{(g)}\in\Omega$,

\item[(H2')] $\epsilon_{i}\alpha_{ii}^{(g)}-\alpha_{ii}^{(g)}\epsilon_{i}$,
for $i\in Q_{0}$ and $g\in\{1,\ldots,2-c_{ii}\}$.

\end{enumerate}

\end{df}

\begin{rmk}

Note that our assumption (BC1) differs from the assumption (C1) of
Geiss, Leclerc and Schröer, who assume that $c_{ii}=2$ for all $i\in\{1,\ldots,n\}$
\cite[\S 1.4.]{GLS17a}. Indeed, Geiss, Leclerc and Schröer work with
generalised Cartan data associated to symmetrisable Kac-Moody algebras.
These Cartan data correspond to quivers without loops. Since in this
thesis, we are interested in quivers which may have loops, we choose
to work with Cartan data of symmetrisable \textit{Borcherds} algebras.
The results from \cite{GLS17a} that we cite in this section generalise
to this setting with the same proofs.

\end{rmk}

\begin{rmk} \label{Rmk/QuivMulOrient}

Our definition of the orientation $\Omega$ also differs from that
of Geiss, Leclerc and Schröer. In \cite{GLS17a}, the authors only
consider orientations such that, given $i\ne j\in Q_{0}$, either
all the arrows of $Q$ joining $i$ and $j$ have target $j$ or they
all have target $i$. Moreover, Geiss, Leclerc and Schröer require
that $Q$ be acyclic. Here, we drop both assumptions. Again, this
does not invalidate the results from \cite{GLS17a} cited below.

\end{rmk}

Given an $H(C,D,\Omega)$-module $M$ and $i\in Q_{0}$, the $\KK$-vector
space $e_{i}M$ is naturally endowed with a structure of $e_{i}H(C,D,\Omega)e_{i}=\KK[\epsilon_{i}]/(\epsilon_{i}^{c_{i}})$-module,
because of relation (H1). Moreover, for $i,j\in Q_{0}$ and $g\in\{1,\ldots,g_{ij}\}$,
relation (H2) amounts to the fact that $\alpha_{ij}^{(g)}$ induces
a $\KK[t]/(t^{k_{ij}})$-linear map $e_{j}M\rightarrow e_{i}M$, where
$t$ acts on $e_{i}M$ (resp. $e_{j}M$) as $\epsilon_{i}^{f_{ji}}$
(resp. as $\epsilon_{j}^{f_{ij}}$). A similar statement holds for
the arrows $\alpha_{ii}^{(g)}$, due to relation (H2'). This leads
us to the notion of representation of a quiver with multiplicities.

\begin{df}[Quiver with multiplicities]\index[terms]{quiver!with multiplicities}\index[notations]{q@$(Q,\nn)$ - quiver with multiplicities}

A quiver with multiplicities is a pair $(Q,\mathbf{n})$, where $Q$
is a quiver and $\mathbf{n}\in\ZZ_{>0}^{Q_{0}}$. The entries of $\mathbf{n}$
are called the multiplicities of $(Q,\mathbf{n})$. For $i,j\in Q_{0}$,
let us call $n_{ij}:=\mathrm{gcd}(n_{i},n_{j})$.

A representation of $(Q,\mathbf{n})$ over $\KK$ is the datum of
a $\KK[t]/(t^{n_{i}})$-module $M_{i}$ for every $i\in Q_{0}$ and
of a $\KK[t]/(t^{n_{ij}})$-linear map $f_{a}:M_{i}\rightarrow M_{j}$
for every $Q_{1}\ni a:i\rightarrow j$ (using the ring morphisms $\KK[t]/(t^{n_{ij}})\rightarrow\KK[t]/(t^{n_{i}})$,
$t\mapsto t^{\frac{n_{i}}{n_{ij}}}$).

Let $M,M'$ be two representations of $(Q,\mathbf{n})$ over $\KK$.
A morphism of representations is the datum of a $\KK[t]/(t^{n_{i}})$-linear
map $\varphi_{i}:M_{i}\rightarrow M'_{i}$ for every $i\in Q_{0}$
such that $f'_{a}\circ\varphi_{s(a)}=\varphi_{t(a)}\circ f_{a}$ for
all $a\in Q_{1}$. We denote such a morphism by $\varphi:M\rightarrow M'$.

We denote by $\Rep_{\KK}(Q,\mathbf{n})$ the category of representations
of $(Q,\mathbf{n})$ over $\KK$.

\end{df}

We argue that the datum $(C,D,\Omega)$ of a Cartan matrix with symmetriser
and orientation is equivalent to the datum of a quiver with multiplicities
$(Q,\mathbf{n})$ and that there is an equivalence of categories $H(C,D,\Omega)-\mathrm{Mod}\simeq\Rep_{\KK}(Q,\mathbf{n})$.
Indeed, the discussion above led us to associate the quiver with multiplicities
$(Q,\mathbf{n})$ to $(C,D,\Omega)$, where $n_{i}:=c_{i}$. It follows
from the relation $c_{ij}=k_{ij}f_{ji}$ that $c_{ij}:=-g_{ij}\cdot\frac{n_{j}}{n_{ij}}$
for $i\ne j\in Q_{0}$ - recall that $g_{ij}$ is the number of arrows
of $Q$ joining $i$ and $j$. Similarly, $2-c_{ii}$ is the number
of loops of $Q$ at vertex $i$. The following equivalence of categories
generalises Proposition \ref{Prop/QuivRepVSModPathAlg}, with the
same proof (see \cite[\S 1.]{CB92}).

\begin{prop} \label{Prop/QuivMultVsGLSAlg}

Let $n\geq1$ be an integer. Then there is a one-to-one correspondence
between Cartan matrices of size $n$ with symmetriser and orientation
$(C,D,\Omega)$ and quivers with multiplicities $(Q,\mathbf{n})$
with set of vertices $Q_{0}=\{1,\ldots,n\}$. Moreover, if $(C,D,\Omega)$
corresponds to $(Q,\mathbf{n})$, then there is an equivalence of
categories:\[
H(C,D,\Omega)-\mathrm{Mod}\simeq\Rep_{\KK}(Q,\mathbf{n}).
\]

\end{prop}

There is also a notion of locally free representations in this setting.

\begin{df}[Locally free representations]

Let $(C,D,\Omega)$ be a Cartan matrix with symmetriser and orientation.
An $H(C,D,\Omega)$-module $M$ is called locally free if $M_{i}:=e_{i}M$
is free as a $\KK[\epsilon_{i}]/(\epsilon_{i}^{c_{i}})$-module for
every $i\in Q_{0}$. We call $H(C,D,\Omega)-\mathrm{Mod}^{\mathrm{l.f.}}$
the full subcategory of $H(C,D,\Omega)-\mathrm{Mod}$ whose objects
are locally free modules.

Let $(Q,\mathbf{n})$ be a quiver with multiplicities. A representation
$M$ of $(Q,\mathbf{n})$ is called locally free if $M_{i}$ is a
free $\KK[t]/(t^{n_{i}})$-module for every $i\in Q_{0}$. When all
$M_{i}$ have finite rank, we define the rank vector of $M$ as $\rk(M):=(\rk(M_{i}))_{i\in Q_{0}}$.
We call $\Rep_{\KK}^{\mathrm{l.f.}}(Q,\mathbf{n})$ the full subcategory
of $\Rep_{\KK}(Q,\mathbf{n})$ whose objects are locally free representations.

The equivalence of categories from Proposition \ref{Prop/QuivMultVsGLSAlg}
restricts to an equivalence of categories, which preserves rank vectors:\[
H(C,D,\Omega)-\mathrm{Mod}^{\mathrm{l.f.}}\simeq\Rep_{\KK}^{\mathrm{l.f.}}(Q,\mathbf{n}).
\]

\end{df}

\begin{exmp}

Fix a quiver $Q$, $\alpha\geq1$ and $R:=\KK[t]/(t^{\alpha})$. Let
$C$ be the symmetric Borcherds-Cartan matrix associated to $Q$ and
$\Omega\subseteq\tilde{Q_{1}}$ correspond to the orientation of $Q$.
Then $H(C,\alpha\cdot\mathrm{Id},\Omega)\simeq RQ$ and there is an
equivalence of categories $\Rep_{R}(Q)\simeq\Rep_{\KK}(Q,\mathbf{n})\simeq H(C,\alpha\cdot\mathrm{Id},\Omega)-\mathrm{Mod}$,
where $\nn=(\alpha,\ldots,\alpha)$. This equivalence also preserves
locally free representations and rank vectors. We will often write
$(Q,\mathbf{n})=(Q,\alpha)$ for short and call $(Q,\mathbf{n})$
a quiver with equal multiplicities.

\end{exmp}

\paragraph*{The root system of a quiver with multiplicities}

In \cite{GLS17a}, Geiss, Leclerc and Schröer generalised the connection
between indecomposable quiver representations and root systems of
Lie algebras to symmetrisable Cartan data. More specifically, they
built reflection functors à la Bernstein, Gelfand and Ponomarev in
their setting and proved an analogue of Gabriel's theorem for non
simply-laced Dynkin diagrams (i.e. also in types B,C,F and G).

Let us recall how to attach a root system to a quiver with multiplicities.
As in the case of quivers without multiplicities, this turns out to
be the root system of the Borcherds algebra associated to $C$.

\begin{df}[Euler form]\index[terms]{Euler form}\index[notations]{eu@$\langle\bullet,\bullet\rangle_H$ - Euler form}

Let $(C,D,\Omega)$ be a Cartan matrix with symmetriser and orientation.
The Euler form of $(C,D,\Omega)$ is the bilinear form on $\ZZ^{Q_{0}}$
given by:\[
\begin{array}{cccc}
\langle\bullet,\bullet\rangle_H: & \ZZ^{Q_0}\times\ZZ^{Q_0} & \rightarrow & \ZZ \\
& (\rr,\mathbf{s}) & \mapsto & \sum_{i\in Q_0}c_ir_is_i-\sum_{\substack{a\in Q_1 \\ a:i\rightarrow j}}c_i\vert c_{ij}\vert r_{i}s_{j}
\end{array}
.
\] The symmetrised Euler form of $(C,D,\Omega)$ is the bilinear form
on $\ZZ^{Q_{0}}$ given by $(\dd,\ee)_{H}:=\langle\dd,\ee\rangle_{H}+\langle\ee,\dd\rangle_{H}$.

\end{df}

Note that the symmetrised Euler form $(\bullet,\bullet)_{H}$ does
not depend from the orientation $\Omega$, as in the case of a quiver
without multiplicities. We now recall the definition of the root system
$\Delta\subseteq\ZZ^{Q_{0}}$ associated to $(C,D)$.

\begin{df}[Simple roots, Fundamental set]

The simple roots of $(C,D)$ are the rank vectors $\epsilon_{i},\ i\in Q_{0}$.
A simple root $\epsilon_{i}$ is called real if $c_{ii}>0$ and imaginary
otherwise. We call $\Pi$ (resp. $\Pi^{\mathrm{re}}$, $\Pi^{\mathrm{im}}$)
the set of simple roots (resp. real, imaginary simple roots).

The fundamental set of $(C,D)$ is the following subset of $\ZZ^{Q_{0}}$:\[
F:=
\left\{
\rr\in\ZZ_{\geq0}^{Q_0}\setminus\{0\}\ \vert\ 
\substack{
\forall i\in Q_0,\ (\rr,\epsilon_i)_H\leq0 \\
\supp(\rr)\ \mathrm{is}\ \mathrm{connected}
}
\right\},
\] where $\supp(\rr)$ is seen as the subquiver of $Q$ whose set of
vertices is $\supp(\rr)$ and whose arrows are those joining vertices
in $\supp(\rr)$.

\end{df}

\begin{df}[Weyl group]

The Weyl group associated to $(C,D)$ is the group of automorphisms
of $\ZZ^{Q_{0}}$ generated by the following reflections:\[
\begin{array}{cccc}
r_i: & \ZZ^{Q_0} & \rightarrow & \ZZ^{Q_0} \\
& \rr & \mapsto & \rr-\frac{2(\rr,\epsilon_i)_H}{(\epsilon_i,\epsilon_i)_H}\cdot \epsilon_i
\end{array}
,\ i\in\Pi^{\mathrm{re}}.
\] Let us denote the Weyl group associated to $(C,D)$ by $W$.

\end{df}

\begin{df}[Root system]

The root system associated to $(C,D)$ is the subset $\Delta:=\Delta^{\mathrm{re}}\cup\Delta^{\mathrm{im}}\subseteq\ZZ^{Q_{0}}$,
where:\[
\begin{array}{l}
\Delta^{\mathrm{re}}:=W\cdot\Pi^{\mathrm{re}} \\
\Delta^{\mathrm{im}}:=W\cdot (\pm F)
\end{array}
.
\]

\end{df}

As in the case of a symmetric Cartan matrix, real and imaginary roots
can be characterised in terms of the Euler form $(\bullet,\bullet)_{H}$
and the root system $\Delta$ coincides with the root system of the
Borcherds algebra $\mathfrak{g}(C)$ associated to $C$.

\begin{prop}{\cite[\S 11.13]{Kac95}}

The set of roots of $\mathfrak{g}(C)$ is $\Delta$; among all roots,
real roots (resp. imaginary roots) are characterised by the property
$(\rr,\rr)_{H}\in\{2c_{1},\ldots,2c_{n}\}$ (resp. $(\rr,\rr)_{H}\leq0$).

\end{prop}

\begin{rmk}

Note that for $i\ne j\in Q_{0}$, $\frac{2(\epsilon_{j},\epsilon_{i})_{H}}{(\epsilon_{i},\epsilon_{i})_{H}}=c_{ji}$.
Thus the Weyl group and the root system associated to $(C,D)$ only
depend on $C$, as follows from the previous proposition.

\end{rmk}

\paragraph*{Projective resolutions for representations of quivers with mutliplicities}

Let $Q$ be a quiver and $M$ a locally free representation of $Q$
over a ring $R$. A standard projective resolution of $M$ can be
written in terms of the path algebra $RQ$. The proof can be found
in \cite[\S 1.]{CB92} when $R$ is a field and works just as well
when $R$ is an arbitrary ring.

\begin{prop}{\cite[\S 1.]{CB92}}

The following sequence of left $RQ$-modules is exact:\[
0\rightarrow\bigoplus_{\substack{a\in Q_1 \\ a:i\rightarrow j}}RQe_j\otimes_Re_iM
\underset{f}{\rightarrow}\bigoplus_{i\in Q_0}RQe_i\otimes_Re_iM
\underset{g}{\rightarrow} M\rightarrow 0,
\]where $f$ and $g$ are determined by $f(\rho e_{j}\otimes e_{i}\cdot m)=\rho ae_{i}\otimes e_{i}\cdot m-\rho e_{j}\otimes e_{j}a\cdot m$
and $g(\rho e_{i}\otimes e_{i}\cdot m)=\rho e_{i}\cdot m$, for $\rho\in\hat{Q}$
and $m\in M$.

\end{prop}

Note that the two leftmost terms in the sequence are projective $RQ$-modules,
since they are direct sums of the modules $RQe_{i},\ i\in Q_{0}$
and these are direct summands of $RQ$. For quivers with mutliplicities,
Geiss, Leclerc and Schröer proved that a similar projective resolution
holds \cite[\S 7.]{GLS17a}, by presenting the algebra $H(C,D,\Omega)$
as a tensor algebra. Before stating this, we explain how representations
of $H(C,D,\Omega)$ can be seen as representations of a certain modulation
of the graph associated to $C$ \cite[\S 5.]{GLS17a}, in order to
set up some notations.

Let $(C,D,\Omega)$ be a Cartan matrix with symmetriser and orientation.
To alleviate notations, we call $H:=H(C,D,\Omega)$ and $H_{i}:=e_{i}He_{i}=\KK[\epsilon_{i}]/(\epsilon_{i}^{c_{i}})$
for $i\in\{1,\ldots,n\}$. Given $a\in\Omega$ and $(i,j)=(s(a),t(a))$,
we define the $H_{j}-H_{i}$-bimodule $H_{a}:=H_{j}\cdot a\cdot H_{i}\subseteq H$.
Note that, in the notations of \cite[\S 5.]{GLS17a}, $_{j}H_{i}=\bigoplus_{\substack{a\in\Omega\\
a:i\rightarrow j
}
}H_{a}$. We depart from the notations of Geiss, Leclerc and Schröer in order
to accomodate for the fact that, in our setting, the arrows of a quiver
with multiplicities joining two given vertices do not necessarily
share the same orientation - see Remark \ref{Rmk/QuivMulOrient}.
Again, this does not invalidate the results exposed below.

In the spirit of \cite[\S 5.2.]{GLS17a}, we say that the rings $H_{i},\ i\in\{1,\ldots,n\}$,
along with the bimodules $H_{a},\ a\in\Omega$ form a \textit{modulation}
$\mathcal{M}:=(H_{i},H_{a})_{\substack{i\in Q_{0}\\
a\in Q_{1}
}
}$ of the quiver $Q=(\{1,\ldots,n\},\Omega,s,t)$. See also \cite{Li12}.

\begin{df}[Representation of a modulation]

A representation of the modulation $\mathcal{M}$ over $\KK$ is the
datum of $H_{i}$-modules $M_{i},\ i\in Q_{0}$ and $H_{j}$-linear
maps $f_{a}:H_{a}\otimes_{H_{i}}M_{i}\rightarrow M_{j},\ Q_{1}\ni a:i\rightarrow j$.
We call this representation $M$.

Let $M,M'$ be two representations of $\mathcal{M}$. A morphism of
representations is the datum of $H_{i}$-linear maps $\varphi_{i}:M_{i}\rightarrow M'_{i}$
such that $f'_{a}\circ(\mathrm{id}\otimes\varphi_{s(a)})=\varphi_{t(a)}\circ f_{a}$
for all $a\in Q_{1}$. We denote such a morphism by $\varphi:M\rightarrow M'$.

The notions of locally free representations and rank vectors are defined
as for representations of quivers with multiplicities. We denote by
$\Rep_{\KK}(\mathcal{M})$ (resp. $\Rep_{\KK}^{l.f.}(\mathcal{M})$)
the category of (locally free) representations of $\mathcal{M}$.

\end{df}

The following result is adapted from \cite[Prop. 5.1.]{GLS17a}, with
the same proof.

\begin{prop}{\cite[Prop. 5.1.]{GLS17a}}

There is an equivalence of categories:\[
H-\mathrm{Mod}\simeq\Rep_{\KK}(\mathcal{M}),
\]which restricts to an equivalence of categories preserving rank vectors:\[
H-\mathrm{Mod}^{\mathrm{l.f.}}\simeq\Rep_{\KK}^{\mathrm{l.f.}}(\mathcal{M}).
\]

\end{prop}

Having defined the bimodules $H_{a},\ a\in\Omega$, we can now state
Geiss, Leclerc and Schröer's explicit projective resolution:

\begin{prop}{\cite[Cor. 7.2.]{GLS17a}} \label{Prop/ProjResH}

Let $M\in H-\mathrm{Mod}^{\mathrm{l.f.}}$. Then the following sequence
of left $H$-modules is exact:\[
0\rightarrow\bigoplus_{\substack{a\in\Omega \\ a:i\rightarrow j}}He_j\otimes_{H_j}H_a\otimes_{H_i}e_iM
\underset{f}{\rightarrow}\bigoplus_{i=1}^nHe_i\otimes_{H_i}e_iM
\underset{g}{\rightarrow} M\rightarrow 0,
\] where $f$ and $g$ are determined by $f(\rho e_{j}\otimes h\otimes e_{i}\cdot m)=\rho ahe_{i}\otimes e_{i}\cdot m-\rho e_{j}\otimes e_{j}ha\cdot m$
and $g(\rho e_{i}\otimes e_{i}\cdot m)=\rho e_{i}\cdot m$, for $\rho\in\hat{Q}$
and $m\in M$.

\end{prop}

Note that the proof of \cite[Cor. 7.2.]{GLS17a} holds without assuming
that $M$ is finitely generated. As in the case of quivers without
multiplicities, the two leftmost terms in the sequence are projective
$H$-modules. This entails that for any $M,N\in H-\mathrm{Mod}^{\mathrm{l.f.}}$,
$\Ext_{H}^{i}(M,N)=0$ for $i>1$. Similarly to \cite[\S 1.]{CB92},
we can then deduce the following exact sequence, which will be useful
later in the study of moment maps on quiver moduli.

\begin{prop} \label{Prop/HomExSeq}

Let $M,N\in H-\mathrm{Mod}^{\mathrm{l.f.}}$ and $f_{a}:H_{a}\otimes_{H_{i}}e_{i}M\rightarrow e_{j}M,\ a\in\Omega$
(resp. $g_{a}:H_{a}\otimes_{H_{i}}e_{i}N\rightarrow e_{j}N,\ a\in\Omega$)
be the corresponding maps. For all $i\in\{1,\ldots,n\}$ and $\Omega\ni a:i\rightarrow j$,
the following maps are isomorphisms of $\KK$-vector spaces:\[
\begin{array}{rcl}
\Hom_{H_i}(e_iM,e_iN) & \rightarrow & \Hom_{H}(He_i\otimes_{H_i}e_iM,N) \\
\varphi & \mapsto & \left( he_i\otimes e_i\cdot m\mapsto h\cdot \varphi(e_i\cdot m) \right) \\
& & \\
\Hom_{H_j}(H_a\otimes_{H_i}e_iM,e_jN) & \rightarrow & \Hom_{H}(He_j\otimes_{H_j}H_a\otimes_{H_i}e_iM,N) \\
\varphi & \mapsto & \left( he_i\otimes h'\otimes e_i\cdot m\mapsto h\cdot \varphi(h'\otimes e_i\cdot m) \right)
\end{array}
\] Under these identifications, applying $\Hom_{H}(\bullet,N)$ to the
projective resolution from Proposition \ref{Prop/ProjResH} yields
the following long exact sequence:\[
0\rightarrow \Hom_H(M,N)
\rightarrow\bigoplus_{i=1}^n\Hom_{H_i}(e_iM,e_iN)
\underset{c_M}{\rightarrow}\bigoplus_{\substack{a\in\Omega \\ a:i\rightarrow j}}\Hom_{H_j}(H_a\otimes_{H_i}e_iM,e_jN)
\rightarrow \Ext_H^1(M,N)\rightarrow0,
\] where, for $\xi=(\xi_{i})_{i\in Q_{0}}\in\bigoplus_{i=1}^{n}\Hom_{H_{i}}(e_{i}M,e_{i}N)$,
$c_{M}(\xi)=\left(\xi_{j}\circ f_{a}-g_{a}\circ(\mathrm{id}_{H_{a}}\otimes\xi_{i})\right)_{a\in\Omega}$.

\end{prop}

An easy corollary of this long exact sequence is the fact that the
Euler form computes the Euler characteristic of the RHom complex between
any two locally free representations.

\begin{cor}{\cite[Prop. 4.1.]{GLS17a}}

Let $M,N$ be two locally free $H$-modules of rank vector $\rr,\mathbf{s}$
respectively. Then:\[
\dim_{\KK}\Hom_H(M,N)-\dim_{\KK}\Ext^1_H(M,N)=\langle \rr,\mathbf{s}\rangle_H.
\]

\end{cor}

\paragraph*{The preprojective algebra of a quiver with multiplicities}

Finally, we recall Geiss, Leclerc and Schröer's construction of a
preprojective algebra for a quiver with multiplicities \cite{GLS17a}.
The preprojective algebra $\Pi(C,D)$ is constructed from a Borcherds-Cartan
matrix with symmetriser, as for the algebra $H(C,D,\Omega)$. By analogy
with the case of a symmetric Cartan datum, the definition goes by
quotienting the generalised path algebra of the double quiver by an
ideal of quadratic relations.

\begin{df}[Preprojective algebra of a quiver with multiplicities]

Let $(C,D,\Omega)$ be a symmetrisable Borcherds-Cartan matrix with
symmetriser and orientation. Let $Q$ be the quiver determined by
$\Omega$ and $\overline{Q}^{+}:=(Q_{0},\overline{Q_{1}}^{+},s,t)$,
where $\overline{Q_{1}}^{+}:=\Omega\sqcup\Omega^{*}\sqcup\{\epsilon_{i}:i\rightarrow i,\ i\in Q_{0}\}$
and $\Omega^{*}:=\{\beta_{ji}^{(g)}:i\rightarrow j,\ \alpha_{ij}^{(g)}\in\Omega\}$.
Then we define $\Pi(C,D,\Omega)$ as the quotient of $\KK\overline{Q}^{+}$
by the double-sided ideal generated by the following relations:

\begin{enumerate}[align=left]

\item[(P1)] $\epsilon_{i}^{c_{i}}$, for $i\in Q_{0}$,

\item[(P2)] $\epsilon_{i}^{f_{ji}}\alpha_{ij}^{(g)}-\alpha_{ij}^{(g)}\epsilon_{j}^{f_{ij}}$,
for $i\ne j\in Q_{0}$ and $\alpha_{ij}^{(g)}\in\Omega$,

\item[(P2')] $\epsilon_{j}^{f_{ij}}\beta_{ji}^{(g)}-\beta_{ji}^{(g)}\epsilon_{i}^{f_{ji}}$,
for $i\ne j\in Q_{0}$ and $\beta_{ji}^{(g)}\in\Omega^{*}$,

\item[(P2'')] $\epsilon_{i}\alpha_{ii}^{(g)}-\alpha_{ii}^{(g)}\epsilon_{i}$
and $\epsilon_{i}\beta_{ii}^{(g)}-\beta_{ii}^{(g)}\epsilon_{i}$,
for $i\in Q_{0}$ and $g\in\{1,\ldots,2-c_{ii}\}$,

\item[(P3)] $\sum_{\alpha_{ij}^{(g)}\in\Omega}\sum_{f=0}^{f_{ji}-1}\epsilon_{i}^{f}\alpha_{ij}^{(g)}\beta_{ji}^{(g)}\epsilon_{i}^{f_{ji}-1-f}-\sum_{\alpha_{ji}^{(g)}\in\Omega}\sum_{f=0}^{f_{ji}-1}\epsilon_{i}^{f}\beta_{ij}^{(g)}\alpha_{ji}^{(g)}\epsilon_{i}^{f_{ji}-1-f}$,
for $i\in Q_{0}$.

\end{enumerate}

\end{df}

One can easily check that, if $\Omega'$ is another orientation of
$C$, then $\Pi(C,D,\Omega)\simeq\Pi(C,D,\Omega')$. Therefore, we
often drop the mention of the orientation and define $\Pi(C,D):=\Pi(C,D,\Omega)$
for $\Omega$ some orientation of $C$. If $(Q,\nn)$ is the quiver
with multiplicities corresponding to $(C,D,\Omega)$, we also use
the notation $\Pi_{(Q,\nn)}:=\Pi(C,D,\Omega)$. \index[terms]{preprojective algebra}\index[notations]{pi@$\Pi_{(Q,\nn)}$ - preprojective algebra of $(Q,\nn)$}

Note that Geiss, Leclerc and Schröer show that $\Pi(C,D)$ carries
an analogue of the algebra structure built by Gelfand and Ponomarev
on the sum of all indecomposable preprojective $H$-modules \cite[Thm. 1.6.]{GLS17a}.
Thus, $\Pi(C,D)$ also deserves to be called a preprojective algebra.

\begin{rmk}

Fix a quiver $Q$, $R:=\KK[t]/(t^{\alpha})$ ($\alpha\geq1$) and
let $\Pi_{Q}$ be the preprojective algebra of $Q$ over $R$. Let
$C$ be the symmetric Borcherds-Cartan matrix associated to $Q$.
Then $\Pi(C,\alpha\cdot\mathrm{Id},\Omega)\simeq\Pi_{Q}$.

\end{rmk}

\subsection{Quiver moduli and moment maps \label{Subsect/MomMap}}

In this section, we introduce moduli spaces of quiver representations,
that is, certain algebraic varieties or stacks which parametrise representations
of quivers with multiplicities. We also define moment maps on quiver
moduli and recall their connection with preprojective algebras. In
the case of quivers without multiplicities, these facts have been
known for a long time and are exposed in \cite{CB99a,Rei08a}. The
generalisation to quivers with multiplicities is adapted from \cite{GLS16,GLS18a}.
The interpretation of Geiss, Leclerc and Schröer's preprojective algebra
in terms of moment maps is explained in unpublished notes by Yamakawa
\cite{Yam}. We give full proofs of this for completeness.

\paragraph*{Moduli of quiver representations}

For the rest of this section, we fix a base field $\KK$. We will
construct moduli of quiver representations from spaces of matrices
over $\KK$. Let us set some notations.

Given $n\geq1$, we will write for short $\KK_{n}:=\KK[t]/(t^{n})$.
For $n,r\geq1$, we denote by $\GL_{n,r}$ the group of invertible
square matrices of size $r$, with coefficients in $\KK_{n}$. In
other words, $\GL_{n,r}$ is the set of $r\times r$ matrices over
$\KK_{n}$ whose determinant is non-zero modulo $t$. Likewise, for
$n_{1},r_{1}\geq1$, $n_{2},r_{2}\geq1$ and $n_{12}:=\mathrm{gcd}(n_{1},n_{2})$,
let $\Hom_{\KK_{n_{12}}}(\KK_{n_{1}}^{\oplus r_{1}},\KK_{n_{2}}^{\oplus r_{2}})$
be the $\KK$-vector space of $\KK_{n_{12}}$-linear maps between
$\KK_{n_{1}}^{\oplus r_{1}}$ and $\KK_{n_{2}}^{\oplus r_{2}}$ (using
the ring morphisms $\KK_{n_{12}}\hookrightarrow\KK_{n_{i}}$, $t\mapsto t^{\frac{n_{i}}{n_{12}}}$
for $i=1,2$). We will consider $\GL_{n,r}$ and $\Hom_{\KK_{n_{12}}}(\KK_{n_{1}}^{\oplus r_{1}},\KK_{n_{2}}^{\oplus r_{2}})$
as algebraic varieties over $\KK$ (i.e. finite-type, separated irreducible
$\KK$-schemes). In what follows, we will only be considering closed
$\KK$-points of $\KK$-schemes. \index[notations]{g@$\GL_{\nn,\rr}$ - general linear group}\index[notations]{k@$\KK_n$ - ring of truncated power series}

\begin{df}[Locally free representation space]\index[terms]{quiver moduli}\index[notations]{r@$R(Q,\nn;\rr)$ - representation space}

Let $(Q,\nn)$ be a quiver with multiplicities. The locally free representation
space of $(Q,\nn)$ in rank vector $\rr$ is the following $\KK$-variety:\[
R(Q,\nn;\rr):=\prod_{\substack{a\in Q_1 \\ a:i\rightarrow j}}\Hom_{\KK_{n_{ij}}}\left(\KK_{n_i}^{\oplus r_i},\KK_{n_j}^{\oplus r_j}\right)
\]\end{df}

A closed point of $R(Q,\nn;\rr)$ yields a locally free representation
of $(Q,\nn)$ with rank vector $\rr$. Choosing bases for all $M_{i},\ i\in Q_{0}$,
any locally free, rank $\rr$ representation $M$ of $(Q,\nn)$ can
be identified with a closed point of $R(Q,\nn;\rr)$. However, the
choice of bases for $M_{i},\ i\in Q_{0}$ is not unique and so distinct
points of $R(Q,\nn;\rr)$ may correspond to isomorphic representations
of $(Q,\nn)$. This is encoded in the action of the following algebraic
group on $R(Q,\nn;\rr)$:\[
\GL_{\nn,\rr}:=\prod_{i\in Q_0}\GL_{n_i,r_i}.
\]The group $\GL_{\nn,\rr}$ acts on $R(Q,\nn;\rr)$ by changing bases.
Explicitly, for $g=(g_{i})_{i\in Q_{0}}\in\GL_{\nn,\rr}$ and $x=(x_{a})_{a\in Q_{1}}\in R(Q,\nn;\rr)$:\[
g\cdot x:=(g_{t(a)}x_ag_{s(a)}^{-1})_{a\in Q_1}.
\] Then isomorphism classes of locally free, rank $\rr$ representations
of $(Q,\nn)$ are in one-to-one correspondence with $\GL_{\nn,\rr}$-orbits
of $R(Q,\nn;\rr)$. This leads us to the following quotient stack.

\begin{df}[Moduli stack of locally free representations]\index[notations]{m@$\mathfrak{M}_{(Q,\nn)}$ - moduli stack of representations of $(Q,\nn)$}

Let $(Q,\nn)$ be a quiver with mutiplicities and $\rr\in\ZZ_{\geq0}^{Q_{0}}$
be a rank vector. We define the stack of locally free representations
of $(Q,\nn)$ with rank vector $\rr$ as:\[
\mathfrak{M}_{(Q,\nn),\rr}:=\left[R(Q,\nn;\rr)/\GL_{\nn,\rr}\right]
.
\] The stack of locally free representations of $(Q,\nn)$ is defined
as the following disjoint union:\[
\mathfrak{M}_{(Q,\nn)}:=\bigsqcup_{\rr\in\ZZ_{\geq0}^{Q_{0}}}\mathfrak{M}_{(Q,\nn),\rr}.
\]\end{df}

\begin{rmk}

For quivers without multiplicities ($\nn=\underline{1}$), we will
typically drop the subscript $\nn$ and denote by $R(Q,\dd)$ (resp.
$\GL_{\dd}$) the representation space (resp. the algebraic group
acting on it), where $\rr=\dd$ is called a dimension vector.

\end{rmk}

The geometric study of the action $\GL_{\dd}\circlearrowleft R(Q,\dd)$
was initiated by Kac in \cite{Kac80,Kac82} (for quivers without multiplicity).
Methods from geometric invariant theory were decisive in showing that
dimension vectors of indecomposable representations of $Q$ are exactly
the positive roots of $Q$. Later, King studied quotients of $R(Q,\dd)$
in the sense of geometric invariant theory and interpreted them in
terms of semisimple (or more generally, polystable) representations
\cite{Kin94}. Let us recall a particular case of this result.

\begin{df}[(Semi)simple representation]

Let $M$ be a representation of $Q$. The representation $M$ is called
simple if it corresponds to a simple $\KK Q$-module i.e. it contains
no strict, non-zero subrepresentations. The representation $M$ is
called semisimple if it is a direct sum of simple representations.

\end{df}

\begin{df}[Moduli space of representations]\index[terms]{King's moduli space}\index[notations]{m@$M_{Q,\dd}$ - moduli space of representations of $Q$}

Let $Q$ be a quiver and $\dd\in\ZZ_{\geq0}^{Q_{0}}$ be a dimension
vector. The moduli space of representations of $Q$ with dimension
vector $\dd$ is the following $\KK$-scheme:\[
M_{Q,\dd}:=R(Q,\dd)\git\GL_{\dd}=\Spec(\KK[R(Q,\dd)]^{\GL_{\dd}}).
\]

\end{df}

King's description of $M_{Q,\dd}$ relies on semisimple representations
of $Q$. His main result is actually more general and concerns quotients
of $R(Q,\dd)$ involving stability conditions. The case of $M_{Q,\dd}$
(i.e. trivial stability condition) was apparently well-known before
King's work (see \cite{LBP90} and \cite[\S 3.]{CB92}).

\begin{prop}{\cite[\S 2., Prop. 3.2.]{Kin94}}

Suppose that $\KK$ is algebraically closed. The closed points of
$M_{Q,\dd}$ correspond bijectively to isomorphism classes of semisimple
representations of $Q$ with dimension vector $\dd$ (resp. Zariski-closed
$\GL_{\dd}$-orbits in $R(Q,\dd)$). Under the canonical quotient
morphism $R(Q,\dd)\rightarrow M_{Q,\dd}$, a closed point $x\in R(Q,\dd)$
corresponds to the unique semisimple representation of $Q$ (up to
isomorphism) whose orbit is contained in $\overline{\GL_{\dd}\cdot x}\subseteq R(Q,\dd)$.

\end{prop}

\paragraph*{Moment maps on quiver moduli}

We now introduce moduli spaces of quiver representations associated
to the preprojective algebra $\Pi_{(Q,\nn)}$. It turns out that the
quadratic relations defining $\Pi_{(Q,\nn)}$ come from moment maps
associated to the symplectic representations spaces $R(\overline{Q},\nn;\rr)$.
We explain the construction of these moment maps, the interpretation
of their zero-loci as cotangent stacks and prove that they yield the
defining relations of $\Pi_{(Q,\nn)}$.

Consider a symplectic $\KK$-variety $X$ endowed with an action of
an algebraic group $G$ preserving the symplectic structure. Call
$\mathfrak{g}$ the Lie algebra of $G$. The moment map $\mu:X\rightarrow\mathfrak{g}^{\vee}$
is the defining feature of so-called Hamiltonian symplectic $G$-varieties
(see \cite[\S 1.4.]{CG97}). We will only consider symplectic varieties
of the form $X=\mathrm{T}^{*}V$, where $V$ is a smooth $G$-variety.
In that case, $X$ is automatically Hamiltonian and the moment map
enjoys the following description, which we briefly recall (see \textit{loc.
cit.}).

The infinitesimal action of $G$ on $V$ yields a morphism of Lie
algebras $\mathfrak{g}\rightarrow\Gamma(\mathrm{T}V)$ and thus a
morphism of algebras $\mathcal{U}(\mathfrak{g})\rightarrow\mathcal{D}(V)$.
Here $\Gamma(\mathrm{T}V)$ is the space of global vector fields on
$V$ and $\mathcal{D}(V)$ is the filtered algebra of differential
operators. This morphism is compatible with the Poincaré-Birkhoff-Witt
filtration on $\mathcal{U}(\mathfrak{g})$ and the order filtration
on $\mathcal{D}(V)$. Taking the associated graded algebras, we obtain
a morphism $\mathcal{O}(\mathfrak{g}^{\vee})\simeq\mathrm{Sym}(\mathfrak{g})\rightarrow\mathcal{O}(\mathrm{T}^{*}V)$,
which yields the desired moment map $\mu:\mathrm{T}^{*}V\rightarrow\mathfrak{g}^{\vee}$.
When $V$ is a linear representation of $G$, the moment map has the
following simple characterisation.

\begin{df}[Moment map]

Let $G$ be an algebraic group and $\mathfrak{g}$ its Lie algebra
Let $V$ be a finite-dimensional linear representation of $G$. The
moment map associated to $V$ is the morphism $\mu_{V,G}=\mu_{V}:V\oplus V^{\vee}\rightarrow\mathfrak{g}^{\vee}$
of algebraic varieties determined by:\[
\langle\xi\cdot x,y\rangle=\langle\xi,\mu_V(x,y)\rangle
\] for all $x\in V$, $y\in V^{\vee}$ and $\xi\in\mathfrak{g}$, where
$\xi\cdot x\in V$ is obtained by differentiating the action of $G$
on $V$.

\end{df}

One possible motivation for the moment map $\mu_{V}$ is the fact
that the quotient stack $\left[\mu_{V}^{-1}(0)/G\right]$ is the (truncated,
or underived) cotangent stack of $\left[V/G\right]$. This can be
made precise using notions of derived algebraic geometry, as in \cite[Exmp. 4.17.]{BCS20}.
Let us give heuristics for this construction. At a point $x\in V$,
the stack $\left[V/G\right]$ has a tangent complex instead of a tangent
space: $\mathbb{T}_{\overline{x}}[V/G]=\left[\mathfrak{g}\rightarrow\mathrm{T}_{x}V\right]$,
where $\overline{x}$ is the image of $x$ under the quotient morphism
$V\rightarrow\left[V/G\right]$. The cotangent complex of $\left[V/G\right]$
at $x$ is then $\mathbb{L}_{\overline{x}}[V/G]=\left[\mathrm{T}_{x}^{*}V\underset{\mu_{V}(x,\bullet)}{\longrightarrow}\mathfrak{g}^{\vee}\right]$,
concentrated in cohomological degrees 0 and 1. By truncating, we obtain
the fibre of the underived cotangent stack $\mathrm{T}^{*}[V/G]\rightarrow[V/G]$
at $\overline{x}$: $\tau_{\leq0}\mathbb{L}_{\overline{x}}[V/G]=\Ker(\mu_{V}(x,\bullet))$,
which is precisely the fibre over $\overline{x}$ of the projection
$\left[\mu_{V}^{-1}(0)/G\right]\rightarrow\left[V/G\right]$.

We now set $V=R(Q,\nn;\rr)$, $G=\GL_{\nn,\rr}$ and call $\mu_{(Q,\nn),\rr}$
the associated moment map.\index[terms]{quiver moment map}\index[notations]{mu@$\mu_{(Q,\nn),\rr}$ - quiver moment map}
Let also $(C,D)$ be the Cartan matrix with symmetriser associated
to $(Q,\nn)$. Following Geiss, Leclerc and Schröer's notations, for
$i,j\in Q_{0}$, we call $f_{ji}=\frac{n_{i}}{\mathrm{gcd}(n_{i},n_{j})}$.
This coincides with the notations from Section \ref{Subsect/QuivRep}.
We show that the quotient stack $\left[\mu_{(Q,\nn),\rr}^{-1}(0)/\GL_{\nn,\rr}\right]$
parametrises locally free, rank $\rr$ modules over the preprojective
algebra $\Pi_{(Q,\nn)}$. To do so, we need to make two identifications
using trace pairings:\[
\begin{array}{rcl}
R(Q,\nn;\rr)^{\vee} & \simeq & R(Q^{\mathrm{op}},\nn;\rr) \\
\mathfrak{gl}_{\nn,\rr}^{\vee} & \simeq & \mathfrak{gl}_{\nn,\rr}
\end{array}
\] Here, $Q^{\mathrm{op}}$ is the so-called oppposite quiver of $Q$,
that is, the quiver obtained from $Q$ by reversing all arrows. Let
us give details on these pairings. Given $n\geq1$, we equip the ring
$\KK_{n}$ with the following $\KK$-linear form:\[
\begin{array}{rccl}
r: & \KK_n & \rightarrow & \KK \\
 & \lambda(t)=\sum_k\lambda_k\cdot t^k & \mapsto &\underset{t=0}{\mathrm{Res}}\left( t^{-n}\lambda(t)\right):=\lambda_{n-1}
\end{array}
.
\] Then the identifications above are built using the following perfect
pairing, called the trace pairing. For $r_{1},r_{2}\geq1$, the bilinear
map\[
\begin{array}{ccccl}
\Hom_{\KK_n}(\KK_n^{\oplus r_1},\KK_n^{\oplus r_2}) & \times & \Hom_{\KK_n}(\KK_n^{\oplus r_2},\KK_n^{\oplus r_1}) & \rightarrow & \KK \\
A & , & B & \mapsto & r\left(\mathrm{Tr}(AB)\right)
\end{array}
\] induces an isomorphism $\Hom_{\KK_{n}}(\KK_{n}^{\oplus r_{1}},\KK_{n}^{\oplus r_{2}})^{\vee}\simeq\Hom_{\KK_{n}}(\KK_{n}^{\oplus r_{2}},\KK_{n}^{\oplus r_{1}})$.
Given an arrow $Q_{1}\ni a:i\rightarrow j$, there is an isomorphism
$\Hom_{\KK_{n_{ij}}}(\KK_{n_{i}}^{\oplus r_{i}},\KK_{n_{j}}^{\oplus r_{j}})\simeq\Hom_{\KK_{n_{ij}}}(\KK_{n_{ij}}^{\oplus f_{ji}r_{i}},\KK_{n_{ij}}^{\oplus f_{ij}r_{j}})$
and so we use the above pairing to identify $R(Q,\nn;\rr)^{\vee}$
and $R(Q^{\mathrm{op}},\nn;\rr)$. This yields an explicit formula
for $\mu_{(Q,\nn),\rr}$.

\begin{prop}{\cite{Yam}}

Under the above trace pairings, the moment map $\mu_{(Q,\nn),\rr}:R(Q,\nn;\rr)\times R(Q^{\mathrm{op}},\nn;\rr)\rightarrow\mathfrak{gl}_{\nn,\rr}$
has the following expression: for $(x,y)\in R(Q,\nn;\rr)\times R(Q^{\mathrm{op}},\nn;\rr)$,\[
\mu_{(Q,\nn),\rr}(x,y)=
\left(
\sum_{\substack{a\in Q_1 \\ a:j\rightarrow i}}
\sum_{f=0}^{f_{ji}-1}
t^{f}x_ay_at^{f_{ji}-1-f}
-
\sum_{\substack{a\in Q_1 \\ a:i\rightarrow j}}
\sum_{f=0}^{f_{ji}-1}
t^{f}y_ax_at^{f_{ji}-1-f}
\right)_{i\in Q_0}
.
\]\end{prop}

Note that, even though the maps $x_{a}y_{a},\ Q_{1}\ni a:i\rightarrow j$
are not $\KK_{n_{i}}$-linear, the sums $\sum_{f=0}^{f_{ji}-1}t^{f}x_{a}y_{a}t^{f_{ji}-1-f}$
are and can thus be considered as elements of $\mathfrak{gl}_{\nn,\rr}$.
One recognises here the defining relations of $\Pi_{(Q,\nn)}$ and
so $\left[\mu_{(Q,\nn),\rr}^{-1}(0)/\GL_{\nn,\rr}\right]$ parametrises
locally free, rank $\rr$ modules over $\Pi_{(Q,\nn)}$. Therefore
$\Pi(C,D)$ should be thought of as a noncommutative cotangent bundle
of $H(C,D,\Omega)$, meaning that the moduli stack of $\Pi(C,D)$-modules
is the cotangent bundle of the moduli stack of $H(C,D,\Omega)$-modules.
For a precise statement in the case of quivers without multiplicities,
see \cite[Rmk. 4.18.]{BCS20}.

\begin{proof}

We first reduce to the case of a quiver with only two vertices joined
by one arrow $a:i\rightarrow j$, using the following elementary facts
on moment maps:
\begin{enumerate}
\item If $G=G_{1}\times G_{2}$, then $\mu_{V,G}=\mu_{V,G_{1}}\times\mu_{V,G_{2}}:V\oplus V^{\vee}\rightarrow\mathfrak{g}_{1}^{\vee}\times\mathfrak{g}_{2}^{\vee}\simeq\mathfrak{g}^{\vee}$.
\item If $V=V_{1}\oplus V_{2}$, then $\mu_{V,G}=\mu_{V_{1},G}\oplus\mu_{V_{2},G}:(V_{1}\oplus V_{1}^{\vee})\oplus(V_{2}\oplus V_{2}^{\vee})\rightarrow\mathfrak{g}^{\vee}$.
\end{enumerate}
Thus, we are left with computing the moment map of $V=\Hom_{\KK_{n_{12}}}(\KK_{n_{1}}^{\oplus r_{1}},\KK_{n_{2}}^{\oplus r_{2}})$
under the action of $G=\GL_{n_{2},r_{2}}$ (recall that $n_{12}:=\mathrm{gcd}(n_{1},n_{2})$).
We use the following additional fact on moment maps:
\begin{enumerate}
\item[3.]  If $G\subseteq H$ is an algebraic subgroup and $\iota:\mathfrak{g}\hookrightarrow\mathfrak{h}$
is the associated linear map, then $\mu_{V,G}=\iota^{\vee}\circ\mu_{V,H}:V\oplus V^{\vee}\underset{\mu_{V,H}}{\longrightarrow}\mathfrak{h}^{\vee}\underset{\iota^{\vee}}{\longrightarrow}\mathfrak{g}^{\vee}$.
\end{enumerate}
Consider $H=\GL_{n_{12},f_{12}r_{2}}\supseteq G$. Then $\mu_{V,H}$
is simply the moment map associated to the $\GL_{n_{12},f_{12}r_{2}}$-representation
$\Hom_{\KK_{n_{12}}}(\KK_{n_{12}}^{\oplus f_{21}r_{1}},\KK_{n_{12}}^{\oplus f_{12}r_{2}})$,
that is $\mu_{V,H}(x,y)=xy$ (using the trace pairing above). The
proposition now follows from the lemma below. \end{proof}

\begin{lem}

Let $r\geq1$, and $n,F,k\geq1$ such that $n=Fk$. Consider the morphism
of rings $\KK[z]/(z^{k})\hookrightarrow\KK[t]/(t^{n})$ mapping $z$
to $t^{F}$ and the embedding $\iota:\mathfrak{gl}_{n,r}\hookrightarrow\mathfrak{gl}_{k,Fr}$
deduced from the decomposition\[
\KK[t]/(t^{n})=\bigoplus_{f=0}^{F-1}\KK[z]/(z^k)\cdot t^f.
\]Then under the trace pairing identifications $\mathfrak{gl}_{n,r}^{\vee}\simeq\mathfrak{gl}_{n,r}$
and $\mathfrak{gl}_{k,Fr}^{\vee}\simeq\mathfrak{gl}_{k,Fr}$, the
linear map $\iota^{\vee}$ has the following expression: for $\eta\in\mathfrak{gl}_{k,Fr}$,\[
\iota^{\vee}(\eta)=\sum_{f=0}^{F-1}t^f\eta t^{F-f-1}.
\]

\end{lem}

\begin{proof}

Let $\eta=\eta(z)\in\mathfrak{gl}_{k,Fr}$ and $\theta=\sum_{f=0}^{F-1}t^{f}\eta t^{F-f-1}$,
seen as an element $\theta(t)\in\mathfrak{gl}_{n,r}$. We want to
show that for all $\xi=\xi(t)\in\mathfrak{gl}_{n,r}$, $r(\mathrm{Tr}(\xi(t)\times\theta(t)))=r(\mathrm{Tr}(\iota(\xi)\times\eta(z)))$.
Let us compute in turn both sides of the equality.

We begin with the right-hand side. Let us expand $\xi(t)=\xi_{0}(t^{F})+\ldots+t^{F-1}\cdot\xi_{F-1}(t^{F})$
and consider $J(z)$ the matrix of multiplication by $t$ in $\KK_{n}^{\oplus r}=\bigoplus_{f=0}^{F-1}t\cdot\KK_{k}^{\oplus r}$:

\[
J(z)
=
\left(
\begin{array}{ccccc}
0 & 0 & \ldots & 0 & z\cdot\Id \\
\Id & 0 & \ldots & 0 & 0 \\
\vdots & \ddots & \ddots & \vdots & \vdots \\
0 & \ddots & \ddots & 0 & 0 \\
0 & 0 & \ldots & \Id & 0
\end{array}
\right)
.
\] Then $\iota(\xi)=\sum_{f=0}^{F-1}\mathrm{Diag}(\xi_{f}(z),\ldots,\xi_{f}(z))\times J(z)^{f}$,
where $\mathrm{Diag}(\xi_{f}(z),\ldots,\xi_{f}(z))$ is the block-diagonal
matrix composed of $F$ blocks $\xi_{f}(z)\in\mathfrak{gl}_{k,r}$.
Therefore, if we decompose $\eta(z)$ into blocks $\eta_{i,j}(z)\in\mathfrak{gl}_{k,r}$,
$0\leq i,j\leq F-1$ we obtain:\[
\begin{split}
r(\Tr(\iota(\xi)\times \eta(z))) & = r\left(\Tr (\xi_0(z)\times (\eta_{0,0}(z)+\ldots+\eta_{F-1,F-1}(z)))\right) \\
& +  r\left(\Tr (\xi_1(z)\times (\eta_{0,1}(z)+\ldots+\eta_{F-2,F-1}(z)+z\cdot\eta_{F-1,0}(z)))\right) \\
& +  \ldots  + \\
& + r\left(\Tr (\xi_{F-1}(z)\times (\eta_{0,F-1}(z)+z\cdot\eta_{1,0}(z)+\ldots+z\cdot\eta_{F-1,F-2}(z)))\right) .
\end{split}
\]

We now focus on the left-hand side. Let us expand $\theta(t)=\theta_{0}(t^{F})+\ldots+t^{F-1}\cdot\theta_{F-1}(t^{F})$
as above. Then using the same reasoning, we see that $\theta_{f}(z)$
is the $(f,0)$-block in the matrix $\sum_{f=0}^{F-1}J(z)^{f}\times\eta(z)\times J(z)^{F-1-f}$.
More explicitly:\[
\theta_f(z)=\eta_{f,F-1}(z)+\eta_{f-1,F-2}(z)+\ldots +\eta_{0,F-1-f}(z)+z\cdot\eta_{F-1,F-2-f}(z)+\ldots + z\cdot\eta_{F-1-f,0}(z).
\]Moreover, one can check that\[
r(\Tr(\xi(t)\times \theta(t)))
=
\sum_{f=0}^{F-1}r(\Tr(\xi_f(z)\times \theta_{F-1-f}(z)).
\]This yields:\[
\begin{split}
r(\mathrm{Tr}(\xi(t)\times\theta(t))) & = r\left(\Tr (\xi_0(z)\times (\eta_{0,0}(z)+\ldots+\eta_{F-1,F-1}(z)))\right) \\
& +  r\left(\Tr (\xi_1(z)\times (\eta_{0,1}(z)+\ldots+\eta_{F-2,F-1}(z)+z\cdot\eta_{F-1,0}(z)))\right) \\
& +  \ldots  + \\
& + r\left(\Tr (\xi_{F-1}(z)\times (\eta_{0,F-1}(z)+z\cdot\eta_{1,0}(z)+\ldots+z\cdot\eta_{F-1,F-2}(z)))\right) ,
\end{split}
\]which finishes the proof. \end{proof}

The connection between moment maps of $(Q,\nn)$ and the preprojective
algebra $\Pi_{(Q,\nn)}$ motivates the following definition:

\begin{df}[Moduli stack of locally free $\Pi_{(Q,\nn)}$-modules]\index[notations]{m@$\mathfrak{M}_{\Pi_{(Q,\nn)}}$ - moduli stack of $\Pi_{(Q,\nn)}$-modules}

Let $(Q,\nn)$ be a quiver with mutiplicities and $\rr\in\ZZ_{\geq0}^{Q_{0}}$
be a rank vector. We define the stack of locally free $\Pi_{(Q,\nn)}$-modules
with rank vector $\rr$ as:\[
\mathfrak{M}_{\Pi_{(Q,\nn)},\rr}:=\left[\mu_{(Q,\nn),\rr}^{-1}(0)/\GL_{\nn,\rr}\right]
.
\] The stack of locally free $\Pi_{(Q,\nn)}$-modules is defined as
the following disjoint union:\[
\mathfrak{M}_{\Pi_{(Q,\nn)}}:=\bigsqcup_{\rr\in\ZZ_{\geq0}^{Q_{0}}}\mathfrak{M}_{\Pi_{(Q,\nn)},\rr}.
\]

\end{df}

When $\nn=\underline{1}$, we can build a moduli space of $\Pi_{Q}$-modules
using geometric invariant theory, as we did for representations of
$Q$.

\begin{df}[Moduli space of $\Pi_Q$-modules]\index[notations]{m@$M_{\Pi_Q,\dd}$ - moduli space of $\Pi_Q$-modules}

Let $Q$ be a quiver and $\dd\in\ZZ_{\geq0}^{Q_{0}}$ be a dimension
vector. The moduli space of $\Pi_{Q}$-modules with dimension vector
$\dd$ is the following $\KK$-scheme:\[
M_{\Pi_Q,\dd}:=\mu_{Q,\dd}^{-1}(0)\git\GL_{\dd}=\Spec(\KK[\mu_{Q,\dd}^{-1}(0)]^{\GL_{\dd}}).
\]

\end{df}

\begin{prop}{\cite[Ch. 1.2, Proof of Thm. 1.1, Fact (3)]{MFK94}}

Suppose that $\KK$ is algebraically closed. Then $M_{\Pi_{Q},\dd}$
is a closed subscheme of $M_{\overline{Q},\dd}$, whose closed points
correspond bijectively to isomorphism classes of semisimple $\Pi_{Q}$-modules
with dimension vector $\dd$ i.e. those semisimple representations
of $\overline{Q}$ which satisfy the defining relations of $\Pi_{Q}$.

\end{prop}

Finally, the standard projective resolution of representations of
$(Q,\nn)$ allows us to compute fibres of the projection $\mu_{(Q,\nn),\rr}^{-1}(0)\rightarrow R(Q,\nn;\rr)$.
This is a crucial tool in counting isomorphism classes of representations
of $(Q,\nn)$, as we will see in Chapter \ref{Chap/KacPolynomials}.

\begin{prop} \label{Prop/MomMapExSeq}

Let $M$ be a locally free representation of $(Q,\nn)$ with rank
vector $\rr$. Let $x\in R(Q,\nn;\rr)$ belong to the orbit corresponding
to $M$. Then under the identifications:\[
\begin{array}{rcl}
R(Q,\nn;\rr) & \simeq & \bigoplus_{a:i\rightarrow j}\Hom_{H_j}(H_a\otimes_{H_i}e_iM,e_jM) \\
\mathfrak{gl}_{\nn,\rr} & \simeq & \bigoplus_{i=1}^n\Hom_{H_i}(e_iM,e_iM)
\end{array}
,
\] the exact sequence of Proposition \ref{Prop/HomExSeq}, once dualised,
yields:\[
0\rightarrow
\Ext_H^1(M,M)^{\vee}\rightarrow
R(Q,\nn;\rr)^{\vee}\underset{\mu_{(Q,\nn),\rr}(x)}{\longrightarrow}
\mathfrak{gl}_{\nn,\rr}^{\vee}\rightarrow
\End_H(M)^{\vee}\rightarrow 0.
\]

\end{prop}

\subsection{Counts of absolutely indecomposable representations and Krull-Schmidt
decomposition \label{Subsect/KrullSchmidt}}

Let $(Q,\nn)$ be a quiver with multiplicities and $\KK=\FF_{q}$
be a finite field. In this section, we introduce central objects of
study in this thesis: the counts $A_{(Q,\nn),\rr}$ of locally free,
rank $\rr$, absolutely indecomposable representations of $(Q,\nn)$
over $\KK$. These counts are natural analogues for quivers with multiplicities
of Kac polynomials, introduced in \cite{Kac82}. The goal of this
section is to introduce the notions required to make sense of the
following formula:\[
\sum_{\rr\in\NN^{Q_0}}
\vol (\mathrm{Pairs}(\mathrm{rep}_{\KK}^{\mathrm{l.f.}}(Q,\nn;\rr)))\cdot t^{\rr}
=
\Exp_{q,t}\left(\sum_{\rr\in\NN^{Q_0}\setminus\{0\}}\frac{A_{(Q,\nn),\rr}}{1-q^{-1}}\cdot t^{\rr}\right),
\]which is a crucial tool in relating $A_{(Q,\nn),\rr}$ with counts
of $\KK$-points on $\mu_{(Q,\nn),\rr}^{-1}(0)$. The formula relies
on Krull-Schmidt decompositions in $\Rep_{\KK}^{\mathrm{l.f.}}(Q,\nn)$
and standard Galois descent arguments, for which we follow Mozgovoy's
article \cite{Moz19}. We show that the category $\Rep_{\KK}^{\mathrm{l.f.}}(Q,\nn)$
fits in the framework of linear stacks developed in \textit{loc. cit.}
and introduce counts of absolutely indecomposable objects with value
in a suitable volume ring.

A typical example of linear stack is the stack of finitely generated
modules over a finite-dimensional $\KK$-algebra $A$ . To every finite
field $\KK'\supseteq\KK$ is attached a category $A\otimes_{\KK}\KK'-\mathrm{mod}$
and for every finite extension $\KK''/\KK'$, there is a base change
functor $\bullet\otimes_{\KK'}\KK'':A\otimes_{\KK}\KK'-\mathrm{mod}\rightarrow A\otimes_{\KK}\KK''-\mathrm{mod}$.
Whether a given module over $A\otimes_{\KK}\KK'$ is defined over
a subfield of $\KK'$ or splits as a direct sum over an extension
of $\KK'$ can be determined using Galois actions. These ideas of
Galois descent are encoded in the definition of (linear) stacks.

\begin{df}[Linear stack] \label{Def/LinearStack}

An additive, $\KK$-linear category $\mathcal{C}$ with finite Hom-sets
is called Krull-Schmidt if every object in $\mathcal{C}$ decomposes
into a finite direct sum of indecomposable objects and all indecomposable
objects of $\mathcal{C}$ have local rings of endomorphisms.

A linear stack $\mathcal{A}$ over $\KK$ is the datum of:
\begin{itemize}
\item a $\KK'$-linear Krull-Schmidt category $\mathcal{A}(\KK')$ for every
finite field $\KK'\supseteq\KK$,
\item a tensor product functor $\bullet\otimes_{\KK',\varphi}\KK'':\mathcal{A}(\KK')\rightarrow\mathcal{A}(\KK'')$
for every morphism of finite fields $\varphi:\KK'\rightarrow\KK''$,
\item a natural isomorphism $\theta_{\psi,\phi}:(\bullet\otimes_{\KK',\varphi}\KK'')\otimes_{\KK'',\psi}\KK'''\rightarrow\bullet\otimes_{\KK',\psi\circ\varphi}\KK'''$
for every two composable morphisms of finite fields $\KK'\overset{\varphi}{\rightarrow}\KK''\overset{\psi}{\rightarrow}\KK'''$,
\end{itemize}
satisfying the compatibilities of a lax 2-functor from the category
of finite fields to the 2-category of categories \cite[\S 3.1.2.]{Vis05}
and the following descent properties:
\begin{itemize}
\item Descent for objects: let $\KK''/\KK'$ be a finite field extension
and $X''\in\mathcal{A}(\KK'')$; if for all $\sigma\in\mathrm{Gal}(\KK''/\KK')$,
$\sigma\cdot X'':=X''\otimes_{\KK'',\sigma}\KK''\simeq X''$, then
there exists $X'\in\mathcal{A}(\KK')$ such that $X''\simeq X\otimes_{\KK'}\KK''$;
\item Descent for morphisms: let $\iota:\KK'\hookrightarrow\KK''$ be a
finite field extension and $X',Y'\in\mathcal{A}(\KK')$; then the
morphism $\bullet\otimes_{\KK'}\KK'':\Hom(X',Y')\rightarrow\Hom(X'\otimes_{\KK'}\KK'',Y'\otimes_{\KK'}\KK'')$
induces an isomorphism onto the subset of morphisms $f'':X'\otimes_{\KK'}\KK''\rightarrow Y'\otimes_{\KK'}\KK''$
satisfying $f''=\theta_{\sigma,\iota}\circ(\sigma\cdot f'')\circ\theta_{\sigma,\iota}^{-1}$
for all $\sigma\in\mathrm{Gal}(\KK''/\KK')$.
\end{itemize}
\end{df}

In more conceptual language, a linear stack over $\KK$ is a category
fibered in Krull-Schmidt categories with descent on the small étale
site of $\Spec(\KK)$ (see \cite[\S 4.1.3.]{Vis05}). Note that Mozgovoy
defines linear stack as categories fibered in additive, Karoubian
categories with finite Hom-sets \cite[\S 3.1.]{Moz19}. Since we work
over a field, these notions are equivalent \cite[\S 2.3.]{Moz19}.
Moreover, the descent property for objects is usually given in terms
of descent data. Since we work only with finite extensions of finite
fields, we can use the simpler condition in terms of Galois action
as a definition (see \cite[Lem. 3.5.]{Moz19}). Given a field extension
$\KK''/\KK'$ and $\sigma\in\mathrm{Gal}(\KK''/\KK')$, we will write
$\bullet\otimes_{\KK'}\KK''$ instead of $\bullet\otimes_{\KK',\subseteq}\KK''$
and $\sigma\cdot\bullet$ instead of $\bullet\otimes_{\KK'',\sigma}\KK''$.

We now show that locally free representations of quivers with multiplicities
fit into the framework of linear stacks. Given a quiver with multiplicities
$(Q,\nn)$, we call $\mathrm{rep}_{\KK}^{\mathrm{l.f.}}(Q,\nn)$ the
category of locally free representations of $(Q,\nn)$ with finite
rank vector.

\begin{prop} \label{Prop/LinearStack}

The assignment $\KK'\mapsto\mathrm{rep}_{\KK'}^{\mathrm{l.f.}}(Q,\nn)$,
along with the obvious tensor product functors, forms a linear stack,
which we also call $\mathrm{rep}_{\KK}^{\mathrm{l.f.}}(Q,\nn)$.

\end{prop}

\begin{proof}

We first check that $\mathrm{rep}_{\KK}^{\mathrm{l.f.}}(Q,\nn)$ is
Krull-Schmidt. It is clear, by induction on rank vectors, that any
locally free, finite-rank representation of $(Q,\nn)$ can be decomposed
into a direct sum of indecomposable representations.

Let $M\in\mathrm{rep}_{\KK}^{\mathrm{l.f.}}(Q,\nn)$ be indecomposable
and $\rr=\rk(M)$. We prove that any endomorphism of $M$ is either
invertible or nilpotent. Then nilpotent elements form a maximal double-sided
ideal, which is the unique maximal ideal of $\End(M)$. Let $\xi\in\End(M)$.
Since for every $i\in Q_{0}$, $\xi_{i}$ commutes with the action
of $t\in\KK[t]/(t^{n_{i}})$, its characteristic subspaces $M_{i}^{P},\ i\in Q_{0}$
($P$ an irreducible polynomial over $\KK$) are $\KK[t]/(t^{n_{i}})$-submodules,
which are preserved by the action of $f_{a},\ a\in Q_{1}$. Moreover,
because $M_{i}=\bigoplus_{P}M_{i}^{P}$, these modules are free over
$\KK[t]/(t^{n_{i}})$. We therefore obtain $M=\bigoplus_{P}M^{P}$
and since $M$ is indecomposable, $M=M^{P}$ for some $P$. Thus $\xi$
is either invertible ($P\ne X$) or nilpotent ($P=X$).

Let us now check that Galois descent holds for objects and morphisms.
Descent for objects follows from the fact that $\mathfrak{M}_{(Q,\nn),\rr}=\left[R(Q,\nn;\rr)/\GL_{\nn,\rr}\right]$
is an Artin stack ($\mathfrak{M}_{(Q,\nn),\rr}$ satisfies descent
for the étale topology \cite[Exmp. 8.1.12.]{Ols16}). Note that $\mathfrak{M}_{(Q,\nn),\rr}(\KK')$
is equivalent to the groupoid of $\mathrm{rep}_{\KK'}^{\mathrm{l.f.}}(Q,\nn;\rr)$
because, by Lang's theorem \cite[Thm. 4.4.17.]{Spr98}, all principal
$\GL_{\nn,\rr}$-bundles over $\Spec(\KK')$ are trivial. Indeed,
Lang's theorem implies that $\mathrm{H}_{\mathrm{\acute{e}t}}^{1}(\KK',\GL_{\nn,\rr})=0$
and $\mathrm{H}_{\mathrm{\acute{e}t}}^{1}(\KK',\GL_{\nn,\rr})$ classifies
principal $\GL_{\nn,\rr}$-bundles over $\Spec(\KK')$ up to isomorphism
\cite[Cor. 12.1.5.]{Ols16}.

Given a field extension $\KK''/\KK'$ and $M',N'\in\mathrm{rep}_{\KK'}^{\mathrm{l.f.}}(Q,\nn)$,
the Galois action on $\Hom(M'\otimes_{\KK'}\KK'',N'\otimes_{\KK'}\KK'')$
described in Definition \ref{Def/LinearStack} makes the following
isomorphism Galois-equivariant: $\Hom(M'\otimes_{\KK'}\KK'',N'\otimes_{\KK'}\KK'')\simeq\Hom(M,N)\otimes_{\KK'}\KK''$.
Then we obtain $\Hom(M',N')\simeq\Hom(M'\otimes_{\KK'}\KK'',N'\otimes_{\KK'}\KK'')^{\mathrm{Gal}(\KK''/\KK')}$
i.e. Galois descent for morphisms. This finishes the proof. \end{proof}

Let us now introduce the counts of quiver representations that we
are interested in. If these counts were polynomials in $q=\sharp\KK$,
we could consider our counts as objects in $\mathbb{Q}(q)$. Unfortunately,
this is not known to be the case in general for representations of
quivers with multiplicities, so we work with the following ring of
volumes introduced by Mozgovoy \cite{Moz19}.

\begin{df}[Ring of volumes]\index[terms]{ring of volumes}\index[notations]{v@$\mathcal{V}$ - ring of volumes}

We call $\mathcal{V}=\prod_{n\geq1}\mathbb{Q}$ the ring of volumes.
The Adams operators $\psi_{m},\ m\geq1$ on $\mathcal{V}$ are defined
by:\[
\psi_m(a)=(a_{mn})_{n\geq1},
\]where $a=(a_{n})_{n\geq1}\in\mathcal{V}$.

\end{df}

Note that there is an injective ring homomorphism $\iota:\QQ(q)\hookrightarrow\mathcal{V}$
mapping $q$ to $(q^{n})_{n\geq1}$. We will often abuse notations
and denote an element in the image of $\iota$ by the corresponding
rational fraction in $q$. With this convention, $\psi_{m}(q)=q^{m}$.
The count we are most interested in here concerns absolutely indecomposable
representations of quivers with multiplicities.

\begin{df}[Absolutely indecomposable representations]\index[notations]{A@$A_{(Q,\nn),\rr}$ - count of absolutely indecomposable representations}

Let $(Q,\nn)$ be a quiver with multiplicities and $M$ a representation
of $(Q,\nn)$. $M$ is called absolutely indecomposable if $M\otimes_{\KK}\bar{\KK}$
is indecomposable.

Given $\rr\in\ZZ_{\geq0}^{Q_{0}}$, we define $A_{(Q,\nn),\rr}:=\left(A_{(Q,\nn),\rr}(q^{n})\right)_{n\geq1}\in\mathcal{V}$,
where $A_{(Q,\nn),\rr}(q^{n})$ is the number of isomorphism classes
of locally free, absolutely indecomposable, rank $\rr$ representations
of $(Q,\nn)$ over $\FF_{q^{n}}$.

\end{df}

Note that an indecomposable, locally free representation of $(Q,\nn)$
is indecomposable as an object of $\Rep_{\KK}(Q,\nn)$ and of $\Rep_{\KK}^{\mathrm{l.f.}}(Q,\nn)$,
since a direct factor of a free $\KK[t]/(t^{n})$-module is free.
When $\nn=\underline{1}$, $A_{(Q,\nn),\rr}$ is essentially Kac's
polynomial \cite{Kac82} (seen as an element of $\mathcal{V}$).

\begin{exmp}[Volume of a stack]\index[notations]{v@$\sharp X$ - volume of a variety}\index[terms]{volume of a stack}\index[notations]{v@$\mathrm{vol}(\mathcal{A})$ - volume of a stack}

We will often consider volumes of stacks (linear stacks or algebraic
stacks) indexed by rank vectors of a quiver with mutliplicities $(Q,\nn)$:\[
\mathcal{A}=\bigsqcup_{\rr\in\ZZ_{\geq0}^{Q_0}}\mathcal{A}_{\rr}.
\] The main examples to keep in mind are the linear stack $\mathrm{rep}_{\KK}^{\mathrm{l.f.}}(Q,\nn)$
and the algebraic stacks $\mathfrak{M}_{(Q,\nn)}$ and $\mathfrak{M}_{\Pi_{(Q,\nn)}}$.
We define:\[
\vol(\mathcal{A}_{\rr}):=\left(\vol(\mathcal{A}_{\rr}(\FF_{q^n}))\right)_{n\geq1},
\] where for a category $\mathcal{C}$ with finitely many objects up
to isomorphism and finite Hom-sets:\[
\vol(\mathcal{C}):=\sum_{M\in\mathrm{Obj}(\mathcal{C})/\simeq}\frac{1}{\sharp\Aut(M)}.
\]In order to deal with the whole stack $\mathcal{A}$ (which typically
has infinite volume), we will work with power series with coefficients
in $\mathcal{V}$:\[
\sum_{\rr\in\ZZ_{\geq0}^{Q_{0}}}\vol(\mathcal{A}_{\rr})\cdot t^{\rr}\in\mathcal{V}[[t_{i},\ i\in Q_{0}]].
\]Note that the volume of a quotient stack $\left[X/G\right]$ is a
quite concrete count, as:\[
\vol\left([X/G](\KK)\right)=\frac{\sharp X(\KK)}{\sharp G(\KK)}.
\]Given a $\KK$-variety $X$ we will use the notation $\sharp X:=(\sharp X(\FF_{q^{n}}))_{n\geq1}$.

\end{exmp}

Absolutely indecomposable representations are characterised by their
rings of endomorphisms. We collect here some of their properties,
which will be useful in Chapter \ref{Chap/KacPolynomials}.

\begin{prop} \label{Prop/EndRings}

Let $M\in\Rep_{\KK}^{\mathrm{l.f.}}(Q,\nn)$ and $\rr=\rk(M)$. Call
$\mathrm{Rad}(M)$ the Jacobson radical of $\End(M)$ and $\mathrm{top}(\End(M)):=\End(M)/\mathrm{Rad}(M)$.
\begin{enumerate}
\item If $M$ is indecomposable, then $\mathrm{Rad}(M)$ consists of nilpotent
elements.
\item $M$ is absolutely indecomposable if, and only if, $\mathrm{top}(\End(M))=\KK$.
\item If $M$ is indecomposable and $\mathrm{top}(\End(M))=\KK'$, then
there exists $N\in\Rep_{\KK'}^{\mathrm{l.f.}}(Q,\nn)$ absolutely
indecomposable, such that $M\otimes_{\KK}\KK'=\bigoplus_{\sigma\in\mathrm{Gal}(\KK'/\KK)}\sigma\cdot N$.
In particular, if $\rr$ is indivisible, then $M$ is absolutely indecomposable.
\end{enumerate}
\end{prop}

\begin{proof}

1. follows from the proof of Proposition \ref{Prop/LinearStack},
since $\mathrm{Rad}(M)$ is the maximal ideal of $\End(M)$. 2. is
\cite[Lemma 3.8.2.]{Moz19}. The first part of 3. is \cite[Thm. 3.9.1.]{Moz19}.

If $M$ is indecomposable and $\rr$ is indivisible, then the decomposition
of $M_{\KK'}$ implies that $[\KK':\KK]\vert\rr$. Thus $\mathrm{top}(\End(M))=\KK$
and $M$ is absolutely indecomposable by 2. \end{proof}

The formula we mentioned at the beginning of this section involves
an operator $\Exp_{q,t}$ on the ring of power series $\mathcal{V}[[t_{i},\ i\in Q_{0}]]$.
This is the plethystic exponential, whose definition we recall below.

\begin{df}[Plethystic exponential]\index[terms]{plethystic exponential}\index[notations]{e@$\Exp_{q,t}$ - plethystic exponential}

For $m\geq1$, the Adams operators on $\mathcal{V}[[t_{i},\ i\in Q_{0}]]$
are defined by the following formula: \[
\psi_m\left(\sum_{\rr\in\NN^{Q_0}}f_{\rr}\cdot t^{\rr}\right):=\sum_{\rr\in\NN^{Q_0}}\psi_m(f_{\rr})\cdot t^{m\rr}.
\]

Let $\mathcal{V}[[t_{i},\ i\in Q_{0}]]_{+}=\{f\in\mathcal{V}[[t_{i},\ i\in Q_{0}]]\ \vert\ f_{0}=0\}$
be the augmentation ideal. The plethystic exponential is the operator
$\Exp_{q,t}:\mathcal{V}[[t_{i},\ i\in Q_{0}]]_{+}\rightarrow\mathcal{V}[[t_{i},\ i\in Q_{0}]]$
defined by:\[
\Exp_{q,t}(F)=\exp\left(\sum_{m\geq1}\frac{\psi_m(F)}{m}\right).
\]

\end{df}

We are now in a position to state the announced formula. Given a quiver
with multiplicities $(Q,\nn)$, let $\mathrm{Pairs}(\mathrm{rep}_{\KK}^{\mathrm{l.f.}}(Q,\nn;\rr))$
be the linear stack of pairs $(M,\phi)$, where $M\in\mathrm{rep}_{\KK}^{\mathrm{l.f.}}(Q,\nn;\rr)$
and $\phi\in\End(M)$. Then the formula follows from a general result
of Mozgovoy's on volumes of linear stacks. \index[notations]{p@$\mathrm{Pairs}(\mathrm{rep}_{\KK}^{\mathrm{l.f.}}(Q,\nn;\rr))$ - stack of pairs}

\begin{prop}{\cite[Thm. 4.6.]{Moz19}} \label{Prop/VolStackPairs}

Let $(Q,\nn$) be a quiver with multiplicities and $\rr\in\ZZ_{\geq0}^{Q_{0}}$
a rank vector. Then:\[
\sum_{\rr\in\NN^{Q_0}}
\vol\left(\mathrm{Pairs}(\mathrm{rep}_{\KK}^{\mathrm{l.f.}}(Q,\nn;\rr))\right)\cdot t^{\rr}
=
\Exp_{q,t}\left(\sum_{\rr\in\NN^{Q_0}\setminus\{0\}}\frac{A_{(Q,\nn),\rr}}{1-q^{-1}}\cdot t^{\rr}\right).
\]

\end{prop}

\subsection{p-adic integrals and Igusa's local zeta functions \label{Subsect/p-adic}}

In this section, we introduce the tools from p-adic integration used
throughout the thesis. We recall basic facts about integrals over
a non-archimedean local field, the change of variable formula and
integration over analytic manifolds. We then define Igusa's local
zeta function for a polynomial mapping and state Denef's formula (which
expresses the zeta function in terms of an embedded resolution of
singularities). We will use the connection between Igusa's local zeta
functions and resolutions of singularities in Chapter \ref{Chap/KacPolynomials},
for computational purposes and in Chapter \ref{Chap/RatSg}, in order
to analyse the asymptotic of jet-counts. A standard reference for
p-adic integrals and Igusa's local zeta function is \cite{Igu00}
(see also the Prologue of \cite{CLNS18}). Denef's formula and its
generalisation to arbitrary ideals are taken from \cite{Den87,VZG08}.

For the rest of this section, we fix a local field $F$ i.e. a complete,
discretely valued, non-archimedean field with finite residue field.
We call $\mathcal{O}$ the valuation ring of $F$, $\mathfrak{m}$
its maximal ideal and $\varpi\in\mathfrak{m}$ a uniformiser. We assume
that the residue field of $F$ has $q$ elements, hence $\mathcal{O}/\mathfrak{m}\simeq\FF_{q}$.
Finally, we call $v:F^{\times}\rightarrow\ZZ$ the discrete valuation
on $F$ and consider the non-archimedean absolute value $\vert.\vert:F\rightarrow\mathbb{R}$
normalised by $\vert\varpi\vert=q^{-1}$.

\begin{exmp}

Let $p$ be a prime. The two main examples of local fields are:
\begin{itemize}
\item $\mathbb{Q}_{p}$, the field of p-adic numbers; then $\mathcal{O}=\ZZ_{p}$
and the residue field is $\FF_{p}$; $\mathbb{Q}_{p}$ has mixed characteristic,
since $\ch(\mathbb{Q}_{p})=0$, while $\ch(\FF_{p})=p>0$;
\item $\FF_{p}((t))$, the field of Laurent series over $\FF_{p}$; then
$\mathcal{O}=\FF_{p}[[t]]$ and the residue field is $\FF_{p}$; $\FF_{p}((t))$
has equal characteristic, since $\ch(\FF_{p}((t)))=\ch(\FF_{p})=p>0$.
\end{itemize}
It turns out that any local field is a finite extension of either
$\mathbb{Q}_{p}$ or $\FF_{p}((t))$ \cite[Prop. II.5.2.]{Neu99}.

\end{exmp}

The absolute value $\vert.\vert$ makes $(F,+)$ a locally compact
topological group. We call $\nu=dx$ the associated Haar measure,
normalised by $\nu(\mathcal{O})=1$. Similarly, for $m\geq1$, we
endow $F^{m}$ with the non-archimedean norm $\Vert.\Vert:F^{m}\rightarrow\mathbb{R}$,
$x\mapsto\sup_{1\leq i\leq m}\vert x_{i}\vert$. This turns $(F^{m},+)$
into a locally compact group. Slightly abusing notations, we call
$\nu=dx_{1}\ldots dx_{m}$ the Haar measure on $F^{m}$, normalised
by $\nu(\mathcal{O}^{m})=1$. This is also the $m$-fold product of
the Haar measure on $F$. Then integrable functions $f:F^{m}\rightarrow\mathbb{C}$
can be assigned an integral with respect to the Haar measure. Since
we work with discrete valuations, it will often be the case that $f$
takes values in a countable subset $\{a_{n},\ n\geq0\}\subseteq\mathbb{R}_{\geq0}$.
Then:\[
\int_{F^m}f(x)dx=\sum_{n\geq0}a_n\cdot\nu(f^{-1}(a_n)).
\] We will refer to such integrals as non-archimedean or p-adic integrals.
These satisfy many properties of real integrals, such as a change
of variables formula, allowing to integrate on certain analytic manifolds.
Before we expose these facts, we need a notion of analytic functions
and mappings.

\begin{df}[Analytic functions, mappings]

Let $U\subseteq F^{m}$ be an open subset and consider a function
$f:U\rightarrow F$. The function $f$ is called analytic on $U$
if, for all $a\in U$, there exist $r>0$ and a power series\[
f_{a}(T_1,\ldots,T_m)=\sum_{n_1,\ldots,n_m\geq0}f_{a,n_1,\ldots,n_m}\cdot T_1^{n_1}\ldots T_m^{n_m}
\] converging on the polydisk $D(0,r)=\{x\in F^{m}\ \vert\ \Vert x\Vert<r\}$
such that $D(a,r)=a+D(0,r)\subseteq U$ and $f(a+x)=f_{a}(x)$ for
all $x\in D(0,r)$. Then the $i$-th partial derivative at $a$ is
defined as $\frac{\partial f}{\partial x_{i}}(a):=\frac{\partial f_{a}}{\partial x_{i}}(a)$.

Set $r\geq1$. A mapping $f:U\rightarrow F^{r}$ is called analytic
if all its components $f_{1},\ldots,f_{r}$ are analytic on $U$.
A mapping is called bianalytic if it is a homeomorphism onto its image
and its inverse is also analytic.

\end{df}

Note that, given an analytic function $f:U\rightarrow F$, the function
$x\mapsto\frac{\partial f}{\partial x_{i}}(x)$ is also analytic on
$U$. As mentioned above, p-adic integrals on open subset of $F^{m}$
enjoy a change of variables formula under analytic mappings.

\begin{prop}{\cite[Prop. 7.4.1.]{Igu00}}

Let $U\subseteq F^{m}$ be an open subset and $f:U\rightarrow F^{m}$
be an injective analytic mapping. Let us call $\mathrm{Jac}(f)(x):=\det\left(\frac{\partial f_{i}}{\partial x_{j}}(x)\right)_{1\leq i,j\leq m}$
for $x\in U$. Then for any integrable function $g:f(U)\rightarrow\mathbb{C}$:\[
\int_{f(U)}g(x)dx=\int_{U}(g\circ f)(x)\vert\mathrm{Jac}(f)(x)\vert dx.
\]

\end{prop}

As in real integration, the change of variables formula is a key fact
in the definition of integrals on manifolds. The definition of analytic
manifolds and differential forms in the non-archimedean case is adapted
from the real or complex setting in a straightforward way.

\begin{df}[Analytic manifold]

Let $M$ be a topological space. An analytic manifold structure on
$M$ (of dimension $m$) is the datum of an open covering $U_{i},\ i\in I$
and local charts $\varphi_{i}:U_{i}\rightarrow F^{m}$ (i.e. $\varphi_{i}:U_{i}\rightarrow\varphi_{i}(U_{i})$
is a homeomorphism) such that, for all $i,j\in I$, $\varphi_{j}\circ(\varphi_{i})^{-1}$
is a bianalytic mapping. The collection of all local charts is called
an atlas and two atlases are called equivalent if their union is also
an atlas.

Let $k\geq0$. A differential form of degree $k$ on $M$ is the datum
of a formal linear combination \[
\omega_{U_i}=\sum_{\substack{I\subseteq\{1,\ldots, m\} \\ I=\{i_1<\ldots <i_k\}}}\omega_{U_i,I}\cdot dx_{i_1}\wedge\ldots\wedge dx_{i_{k}}
\] for every local chart $(U_{i},\varphi_{i})$, where $\omega_{I}:\varphi_{i}(U_{i})\rightarrow F,\ I\subseteq\{1,\ldots,m\}$
are analytic functions and such that for all $i,j\in I$ $\omega_{U_{j}}=(\varphi_{j}\circ(\varphi_{i})^{-1})^{*}\omega_{U_{i}}$
under the usual chain rule.

\end{df}

\begin{exmp}[Analytification of a smooth scheme]

All the analytic manifolds that we will encounter in this thesis come
from algebraic geometry in the following sense. Given a finite-type,
smooth $F$-scheme $X$, the set $X(F)$ can be endowed with a canonical
analytic manifold structure (see \cite[Prologue, \S 1.6.3.]{CLNS18}).
This is essentially a consequence of the implicit function theorem
for analytic mappings. Then a global section of the sheaf $\bigwedge^{k}\Omega_{X/F}^{1}$
yields a differential form on $X(F)$.

\end{exmp}

Just like one can integrate differential forms of top degree on a
smooth real manifold, a top-form $\omega$ on an analytic manifold
$M$ over $F$ yields a measure $\nu_{\omega}$ on $M$.

\begin{prop}{\cite[\S 7.4.]{Igu00}}

Let $M$ be an analytic manifold of dimension $m$ and $\omega$ a
differential form of degree $m$ on $M$. There exists a unique measure
$\nu_{\omega}$ on $M$ such that for any measurable subset $A\subseteq M$
contained in a local chart $(U_{i},\varphi_{i})$:\[
\nu_{\omega}(A)=\int_{A}\vert\omega_{U_i}(x)\vert dx_1\ldots dx_m.
\]

\end{prop}

When $M$ is the analytification of a smooth $F$-scheme, there may
well be no global algebraic top-forms on $M$. One may however try
to glue measures built from local sections of $\bigwedge^{\mathrm{top}}\Omega_{X/F}^{1}$.
This works when these sections are ``defined over $\mathcal{O}$''.
We explain below how this works, following \cite[\S 4]{Yas17} and
\cite[\S 3]{COW21}.

\begin{exmp}[Canonical measure of an $\mathcal{O}$-scheme]\index[terms]{canonical measure}\index[notations]{nu@$\nu_{\mathrm{can}}$ - canonical measure}
\label{Exmp/CanMeas}

Let $X$ be a finite type $\ensuremath{\mathcal{O}}$-scheme and $d=\dim X_{F}$.
Let us define $X^{\natural}:=X^{\mathrm{sm}}(F)\cap X(\mathcal{O})$.
This is an open subset of $X^{\mathrm{sm}}(F)$, hence an analytic
manifold over $F$. Suppose that $X$ admits an invertible sheaf $\ensuremath{\mathcal{L}}$,
which restricts to $\Omega_{X^{\mathrm{sm}}/\mathcal{O}}^{d}$ on
the smooth locus $X^{\mathrm{sm}}$. Consider trivializing open subsets
$V_{k}\subseteq X$ for $\ensuremath{\mathcal{L}}$ and non-vanishing
sections $\omega_{k}\in\Gamma(V_{k},\ensuremath{\mathcal{L}})$. Note
that $X^{\natural}=\bigcup_{k}V_{k}^{\natural}$. Indeed, given $x\in X(\mathcal{O})$,
we have $x_{\mathbb{F}_{q}}\in V_{k}(\mathbb{F}_{q})\Rightarrow x\in V_{k}(\mathcal{O})$
since the pullback of $V_{k}$ along $x$ is a neighbourhood of the
closed point of $\Spec(\mathcal{O})$.

The sections $\omega_{k}$ induce measures $\nu_{k}$ on $V_{k}^{\natural}$
as above. The key observation is that for $x\in V_{k}^{\natural}\cap V_{l}^{\natural}$,
$\frac{\omega_{k}}{\omega_{l}}(x)\in\mathcal{O}^{\times}$. Then the
change of variables formula implies that, for any measurable subset
$A\subseteq V_{k}^{\natural}\cap V_{l}^{\natural}$, $\nu_{k}(A)=\nu_{l}(A)$.
Thus the measures $\nu_{k}$ glue to a measure $\nu_{\mathrm{can}}$.
A similar argument shows that $\nu_{\mathrm{can}}$ does not depend
on a choice of top-forms i.e. $\nu_{\mathrm{can}}$ is canonically
attached to $X$. When $X$ is a smooth $\mathcal{O}$-scheme, this
recovers Weil's canonical measure, as constructed in \cite[Ch. II, \S 2.2.]{Wei12}.

\end{exmp}

We now turn to Igusa's local zeta functions. These are parametric
p-adic integrals associated to a polynomial (or analytic) mapping.
These were initially studied by Igusa, in connection with Weil's adelic
formulation of the Smith--Minkowski--Siegel mass formula (see \cite[Ch. 8]{Igu00}).
It was soon realised that the local zeta function also gives information
on the singularities of the polynomial mapping. In particular, a formula
by Denef gives an explicit expression of the local zeta function in
terms of an embedded resolution of singularities of the mapping \cite{Den87}.
We recall this below.

\begin{df}[Igusa's local zeta function]\index[terms]{Igusa's local zeta function}\index[notations]{z@$Z_f(s)$ - Igusa's local zeta function}

Let $f=(f_{1},\ldots,f_{r})$, where $f_{i}\in\mathcal{O}[x_{1},\ldots,x_{m}]$.
The local zeta function associated to $f$ is the complex function
defined by:\[
Z_f(s):=\int_{\mathcal{O}^m}\Vert f(x)\Vert^sdx,\ s\in\CC.
\]

The Poincaré series associated to $f$ is the following power series:\[
P_f(T):=\sum_{n\geq0}N_{f,n}\cdot T^n,
\]where $N_{f,n}:=\{x\in(\mathcal{O}/\mathfrak{m}^{n})^{m}\ \vert\ f(x)=0\}$
for $n\geq1$ and $N_{f,0}=1$.

\end{df}

Given that $\Vert f(x)\Vert\leq1$ for $x\in\mathcal{O}^{m}$, the
local zeta function converges for $\Rea(s)\geq0$. Note that, since
$x\in\mathcal{O}^{m}$ and $f_{i}\in\mathcal{O}[x_{1},\ldots,x_{m}]$
for $1\leq i\leq m$, $\Vert f(x)\Vert$ does not depend on a choice
of generators of the ideal $I=(f_{1},\ldots,f_{m})$. Seeing $x$
as an $\mathcal{O}$-point $\Spec(\mathcal{O})\overset{x}{\rightarrow}\mathbb{A}_{\mathcal{O}}^{m}$,
$\Vert f(x)\Vert=q^{-\mathrm{ord}_{I}(x)}$, where $\mathrm{ord}_{I}(x)$
is characterised by $x^{-1}(I)\mathcal{O}=(\varpi^{\mathrm{ord}_{I}(x)})$.
Moreover, the local zeta function and the Poincaré series are related
as follows:

\begin{lem}{\cite[Thm. 8.2.2.]{Igu00}} \label{Lem/IgusaVSPoincar=0000E9}\[
P_f(q^{-m-s})=\frac{1-q^{-s}\cdot Z_f(s)}{1-q^{-s}}.
\]\end{lem}

Therefore, the local zeta function contains the data of all $N_{f,n},\ n\geq1$.
This will be important for computing counts of quiver representations
in Chapter \ref{Chap/KacPolynomials}.

Let us now assume that $f_{1},\ldots,f_{r}\in\KK[x_{1},\ldots,x_{m}]$,
where $\KK$ is a number field and $\mathcal{O}_{\KK}$ its ring of
integers. For any non-zero prime ideal $\mathfrak{p}\subseteq\mathcal{O}_{\KK}$
lying over $(p)\subseteq\ZZ$, the completion of $\KK$ with respect
to the $\mathfrak{p}$-adic valuation is a local field of residual
characteristic $p$. Call $X=V(f_{1},\ldots,f_{r})\subseteq\mathbb{A}_{\KK}^{m}$
the closed subscheme defined by the ideal $(f_{1},\ldots,f_{r})$.
Denef's computation relies on the fact that a given embedded resolution
of $X$ is well-behaved modulo a large enough prime (a property called
``good reduction'').

\begin{df}[Embedded resolution of singularities]\index[terms]{embedded resolution of singularities}\index[notations]{n@$(N_i,\nu_i)$ - numerical data of an embedded resolution of singularities}
\label{Def/EmbedRes}

Consider a closed subscheme $X\subsetneq\mathbb{A}_{\KK}^{m}$ defined
by an ideal $I=(f_{1},\ldots,f_{r})$ as above. An embedded resolution
of singularities of $X$ is a projective morphism $h:Y\rightarrow\mathbb{A}_{\KK}^{m}$
such that:
\begin{enumerate}
\item $Y$ is a smooth $\KK$-variety;
\item $h\vert_{h^{-1}(\mathbb{A}_{\KK}^{m}\setminus X)}:h^{-1}(\mathbb{A}_{\KK}^{m}\setminus X)\rightarrow\mathbb{A}_{\KK}^{m}\setminus X$
is an isomorphism;
\item $h^{-1}(X)$, seen as a reduced closed subscheme of $Y$, is a divisor
$E$ with simple normal crossings.
\end{enumerate}
We call $E_{i},\ 1\leq i\leq t$ the irreducible components of $E$.
For $1\leq i\leq t$, let $N_{i}$ be the multiplicity of $h^{-1}(I)\cdot\mathcal{O}_{Y}$
along $E_{i}$ and $\nu_{i}-1$ the order of the form $h^{*}(dx_{1}\wedge\ldots\wedge dx_{m})$
along $E_{i}$. We call $(N_{i},\nu_{i})_{1\leq i\leq t}$ the numerical
data of the resolution $h$.

\end{df}

By \cite[Main Thm. II.]{Hir64}, there always exists such an embedded
resolution of singularities and it can be obtained as a sequence of
blowing-ups. Let us now examine the behaviour of the resolution modulo
$p$.

\begin{df}[Good reduction modulo p] \label{Def/GoodRed}

Consider an embedded resolution of singularities $h:Y\rightarrow\mathbb{A}_{\KK}^{m}$
as above . Let $\mathfrak{p}\subseteq\mathcal{O}_{\KK}$ be a non-zero
prime ideal lying over $(p)\subseteq\ZZ$ with residue field $\mathcal{O}_{\KK}/\mathfrak{p}\simeq\FF_{q}$.
Let $F$ be the completion of $\KK$ with respect to the $\mathfrak{p}$-adic
valuation and $\mathcal{O}$ its valuation ring. For $p$ large enough,
we may assume that $X_{F}$ is defined over $\mathcal{O}$ i.e. $f_{i}\in\mathcal{O}[x_{1},\ldots,x_{m}]$
for all $1\leq i\leq r$.

Since $h$ is projective, we may consider $Y_{F}$ as a closed subscheme
of $\mathbb{P}_{F}^{N}\times_{F}X_{F}$. Consider $Y_{\mathcal{O}}$
(resp. $E_{\mathcal{O}}$, $E_{i,\mathcal{O}},\ 1\leq i\leq t$) the
closure of $Y_{F}$ (resp. $E_{F}$, $E_{i,F},\ 1\leq i\leq t$) in
$\mathbb{P}_{\mathcal{O}}^{N}\times_{\mathcal{O}}X_{\mathcal{O}}\supset\mathbb{P}_{F}^{N}\times_{F}X_{F}$.
We say that the resolution $h$ has good reduction modulo $\mathfrak{p}$
if:
\begin{enumerate}
\item $Y_{\FF_{q}}$ is smooth;
\item $E_{\mathbb{F}_{q}}$ is a divisor with simple normal crossings;
\item for $1\leq i\ne j\leq t$, $E_{i,\mathbb{F}_{q}}$ and $E_{j,\mathbb{F}_{q}}$
have no common irreducible components.
\end{enumerate}
\end{df}

By \cite[Thm. 2.4.]{Den87}, a given resolution has good reduction
modulo $\mathfrak{p}$ for almost all prime ideals of $\mathcal{O}_{K}$.
This leads to the following computational lemma, which we recall for
later computations in Chapter \ref{Chap/KacPolynomials}.

\begin{lem}{\cite[Proof of Thm. 3.1, Eqn. (4)]{Den87}} \label{Lem/DenefLemma}

Let $h:Y\rightarrow\mathbb{A}_{\KK}^{m}$ be a projective birational
morphism satisfying properties 1. and 3. of Definition \ref{Def/EmbedRes}.
Let $\mathfrak{p}\subseteq\mathcal{O}_{\KK}$ be a non-zero prime
ideal and suppose that $h$ satisfies propeties 1, 2. and 3. of Definition
\ref{Def/GoodRed}.

Consider $a\in Y_{\FF_{q}}(\FF_{q})$. Define $B_{a}=\{y\in Y_{\mathcal{O}}(\mathcal{O})\ \vert\ \overline{y}=a\}$
and $T_{a}=\{1\leq i\leq t\ \vert\ a\in E_{i,\FF_{q}}(\FF_{q})\}$.
Finally, consider the ideal sheaf $J=\mathcal{O}_{Y}(-\sum_{i}a_{i}E_{i})$,
where $a_{i}\geq0$. Then there is a bijective analytic mapping $\varphi:\varpi\mathcal{O}^{m}\rightarrow B_{a}$
such that:
\begin{itemize}
\item for all $z\in\varpi\mathcal{O}^{m}$, $q^{-\mathrm{ord}_{J}(\varphi(z))}=\prod_{i\in T_{a}}z_{i}^{a_{i}}$;
\item $(h\circ\varphi)^{*}dx_{1}\wedge\ldots\wedge dx_{m}=\prod_{i\in T_{a}}z_{i}^{\nu_{i}-1}\cdot dz_{1}\wedge\ldots\wedge dz_{m}$.
\end{itemize}
Moreover, for $u_{i}\in\ZZ_{\geq0}$, $v_{i}\in\ZZ$, $1\leq i\leq m$
and $s\in\CC$ of large enough real part, the following holds:\[
\int_{B_a}\vert z_1\vert^{u_1s+v_1}\ldots\vert z_m\vert^{u_ms+v_s}dz_1\wedge\ldots\wedge dz_m=q^{-m}\cdot\prod_{i=1}^m\frac{(q-1)q^{-(u_is+v_i+1)}}{1-q^{-(u_is+v_i+1)}}.
\]

\end{lem}

Using change of variables along $h$, this yields Denef's formula.

\begin{prop}{\cite[Thm. 3.1.]{Den87},\cite[Thm. 2.10.]{VZG08}}\index[terms]{Denef's formula}

Let $X=V(I)=V(f_{1},\ldots,f_{r})\subseteq\mathbb{A}_{\KK}^{m}$ and
$h:Y\rightarrow X$ be an embedded resolution as above, with numerical
data $(N_{i},\nu_{i})_{1\leq i\leq t}$. Let $\mathfrak{p}\subseteq\mathcal{O}_{\KK}$
be a non-zero prime ideal and suppose that $h$ has good reduction
modulo $\mathfrak{p}$. Then:\[
Z_f(s)=q^{-m}\cdot\sum_{I\subseteq\{ 1,\ldots,t\}}c_I\cdot\prod_{i\in I}\frac{(q-1)q^{-N_is-\nu_i}}{1-q^{-N_is-\nu_i}},
\]where $c_{I}=\sharp\{a\in Y_{\FF_{q}}(\FF_{q})\ \vert\ a\in E_{i,\FF_{q}}(\FF_{q})\Leftrightarrow i\in I\}$.

\end{prop}

\subsection{Cohomology of moduli stacks, Hodge structures and perverse sheaves
\label{Subsect/CohAlgVar}}

In this section, we review some cohomological techniques used in categorifying
counts of quiver representations. Sketching the construction of the
main objects involved (derived categories, étale cohomology, perverse
sheaves, Hodge structures...) would go well beyond the scope of this
thesis, so we will only outline their computational properties. We
first recall some facts about derived categories of sheaves (both
in the analytic and étale topologies), perverse sheaves and the ``Faisceaux-Fonctions''
correspondence. We then briefly discuss notions of weights for complexes
of sheaves: mixed Hodge structures and mixed Hodge modules on complex
varieties and weights of Frobenius actions for varieties over finite
fields. We also review a construction of equivariant cohomology using
Borel's approximation (including weights). Finally, we recall elementary
facts on characteristic cycles of constructible complexes. In what
follows, $\KK$ will denote either an algebraically closed field or
a finite field. We will call a quasi-projective scheme over $\KK$
(not necessarily irreducible) a $\KK$-variety.

\paragraph*{Derived categories of sheaves, perverse sheaves and the ``Faisceaux-Fonctions''
correspondence}

\subparagraph*{Categories of constructible complexes}

Let us begin with the derived category $D^{b}(X(\CC))$ of sheaves
on a complex variety $X$. Various (co)homology groups of $X(\CC)$
(with the usual analytic topology) can be expressed as hypercohomology
of certain complexes of sheaves. The category $D^{b}(X(\CC))$ is
the derived category of bounded complexes of sheaves on $X(\CC)$
(we consider sheaves of $\QQ$-vector spaces). It was first studied
by Verdier in the 1960s following ideas of Grothendieck, in order
to efficiently state generalisations of Poincaré duality. We refer
to \cite[Ch. 2,3.]{Dim04} for details about its construction.

Given a morphism of complex algebraic varieties $f:X\rightarrow Y$,
there are two pairs of adjoint functors $f^{*}\dashv\mathrm{R}f_{*}$
and $\mathrm{R}f_{!}\dashv f^{!}$:\[
\begin{tikzcd}[ampersand replacement=\&]
\mathrm{R}f_{*},\mathrm{R}f_{!}: \ D^b(X(\CC)) \ar[r, shift left] \& D^b(Y(\CC))  \ :f^*,f^! \ar[l, shift left]
\end{tikzcd}
\]and two derived bifunctors $\mathrm{R}\mathcal{H}om:D^{b}(X(\CC))^{\mathrm{op}}\times D^{b}(X(\CC))\rightarrow D^{b}(X(\CC))$,
$\otimes^{\mathbf{L}}:D^{b}(X(\CC))\times D^{b}(X(\CC))\rightarrow D^{b}(X(\CC))$
such that, for $K^{\bullet}\in D^{b}(X)$, there is an adjunction
$\bullet\otimes^{\mathbf{L}}K^{\bullet}\dashv\mathrm{R}\mathcal{H}om(K^{\bullet},\bullet)$
\cite[Prop. 2.6.3.]{KS90}. These form the six basic operations on
complexes of sheaves, which underly many computations in geometric
representation theory. Grothendieck-Verdier duality is formalised
by the dualising complex $\omega_{X}:=p^{!}\underline{\QQ}_{\mathrm{pt}}$
(where $p$ is the canonical morphism $X\rightarrow\mathrm{pt}$)
and the dualising functor $\mathbb{D}=\mathbb{D}_{X}:=\mathrm{R}\mathcal{H}om(\bullet,\omega_{X}):D^{b}(X(\CC))^{\mathrm{op}}\rightarrow D^{b}(X(\CC))$.
Then we recover the usual (co)homology groups of $X(\CC)$ as follows.
We call $\underline{\QQ}_{X}:=p^{*}\underline{\QQ}_{\mathrm{pt}}$.\[
\begin{array}{llll}
\text{Singular cohomology:} & \HH^{\bullet}(X,\QQ) & = & \HH^{\bullet}(\mathrm{R}p_*\underline{\QQ}_{X}) \\
\text{Compactly supported cohomology:} & \HH_{\mathrm{c}}^{\bullet}(X,\QQ) & = & \HH^{\bullet}(\mathrm{R}p_!\underline{\QQ}_{X}) \\
\text{Borel-Moore homology:} & \HH_{\bullet}^{\mathrm{BM}}(X,\QQ) & = & \HH^{-\bullet}(\mathrm{R}p_*\mathbb{D}\underline{\QQ}_{X}) \\
\text{Hypercohomology of a complex:} & \mathbb{H}^{\bullet}(X,K^{\bullet}) & = & \HH^{\bullet}(\mathrm{R}p_*K^{\bullet})
\end{array}
\]

We will always work with constructible complexes. These form a well-behaved
subcategory of $D^{b}(X(\CC))$, whose objects are compatible with
a stratification of $X$. Let us fix a notion of stratification for
algebraic varieties.

\begin{df}[Stratification]

Let $X$ be a $\KK$-variety. A stratification of $X$ is a locally
finite partition of $X$ into locally closed subschemes $X_{i},\ i\in I$
called strata, such that:
\begin{enumerate}
\item for all $i\in I$, $X_{i}$ is smooth;
\item for all $i\in I$, the Zariski closure of $X_{i}$ is a union of strata.
\end{enumerate}
\end{df}

A sheaf on $X(\CC)$ is called constructible with respect to $(X_{i})_{i\in I}$
if it is locally constant on each stratum and its stalks are finite-dimensional.
The full subcategory $D_{\mathrm{c}}^{b}(X(\CC))\subseteq D^{b}(X(\CC))$
of complexes with constructible cohomology sheaves (with respect to
some stratification) is called the constructible derived category
of $X$. It is stable under the six operations above, see \cite[Ch. 4.]{Dim04}
for details. Apart from their good behaviour under the six operations,
a historical motivation for studying constructible complexes is Kashiwara's
constructibility theorem, which asserts that the complex of solutions
of a regular holonomic D-module is constructible - see \cite[\S 5.3.]{Dim04}.
Solution complexes of D-modules form an abelian subcategory $\mathrm{Perv}(X(\CC))\subseteq D_{\mathrm{c}}^{b}(X(\CC))$,
the category of perverse sheaves, which we will describe below.

We will also work with the category of constructible l-adic complexes
on a $\KK$-variety $X$ (here, $\KK$ might be of positive characteristic
e.g. a finite field). Let $l$ be a prime number, which differs from
$\ch(\KK)$. Motivated by certain comparison theorems between perverse
sheaves on complex varieties and on varieties defined over a finite
field, Beilinson, Bernstein, Deligne and Gabber built an analogue
$D_{\mathrm{c}}^{b}(X,\bar{\QQ_{l}})$ of $D_{\mathrm{c}}^{b}(X(\CC))$
with l-adic coefficients, using étale cohomology \cite{BBDG18}. The
construction of $D_{\mathrm{c}}^{b}(X,\bar{\QQ_{l}})$ is subtle,
as it is actually not the derived category of an abelian category
of l-adic sheaves. Instead, it is constructed as a limit of derived
categories of étale sheaves with coefficients in $\ZZ/l^{n}\ZZ,\ n\geq1$,
which enjoy a six-operation formalism and duality, as $D_{\mathrm{c}}^{b}(X(\CC))$.
In particular, $D_{\mathrm{c}}^{b}(\Spec(\bar{\KK}),\bar{\QQ_{l}})$
is the derived category of the category of $\bar{\QQ_{l}}$-vector
spaces. We refer to \cite[Ch. 59, 63]{SP} for the properties of the
derived category of étale sheaves and to \cite{KW01} for the construction
of $D_{\mathrm{c}}^{b}(X,\bar{\QQ_{l}})$. We now collect useful computational
facts which hold in both categories of constructible complexes. \index[notations]{d@$D_{\mathrm{c}}^{b}(X,\bar{\QQ_{l}})$ - category of l-adic complexes}

\begin{prop} \label{Prop/FormulasD_c^b}

The following properties hold:
\begin{enumerate}
\item \cite[Prop. 3.3.7.]{Dim04}\cite[Cor. 7.5.]{KW01} Let $f:X\rightarrow Y$
be a morphism of complex varieties (resp. $\KK$-varieties). Then
there are natural isomorphisms:\[
\begin{array}{rcl}
\mathbb{D}^2 & \simeq & \mathrm{Id} \\
\mathbb{D}_Y\circ\mathrm{R}f_* & \simeq & \mathrm{R}f_!\circ\mathbb{D}_X \\
\mathbb{D}_X\circ f^* & \simeq & f^!\circ\mathbb{D}_Y.
\end{array}
\]
\item (Adjunction triangle) \cite[Prop. 5.2.2.]{Dim04}\cite[\href{https://stacks.math.columbia.edu/tag/0GKK}{Tag 0GKK}]{SP}\cite[App. A]{KW01}
Let $X$ be a complex variety (resp. $\KK$-variety) and $K^{\bullet}\in D_{\mathrm{c}}^{b}(X(\CC))$
(resp. $D_{\mathrm{c}}^{b}(X,\bar{\QQ_{l}})$). Let $i:Z\hookrightarrow X$
be a closed subscheme and $j:U=X\setminus Z\hookrightarrow X$ its
open complement. Then there is a distinguished triangle:\[
\mathrm{R}j_!j^*K^{\bullet}\rightarrow K^{\bullet}\rightarrow \mathrm{R}i_*i^*K^{\bullet}.
\]
\item (Proper base change) \cite[Thm. 2.3.26, Rmk. 2.3.27.]{Dim04}\cite[\href{https://stacks.math.columbia.edu/tag/0F7L}{Tag 0F7L}]{SP}\cite[App. A]{KW01}
Consider a cartesian diagram of complex varieties (resp. $\KK$-varieties):\[
\begin{tikzcd}[ampersand replacement=\&]
X' \ar[r,"f'"]\ar[d,"g'"] \& Y' \ar[d,"g"] \\
X \ar[r,"f"] \& Y.
\end{tikzcd}
\]Then there is a natural isomorphism of functors $g^{*}\circ\mathrm{R}f_{!}\simeq\mathrm{R}(f')_{!}\circ(g')^{*}$.
\item (Projection formula) \cite[Thm. 2.3.29.]{Dim04}\cite[\href{https://stacks.math.columbia.edu/tag/0GL5}{Tag 0GL5}]{SP}\cite[App. A]{KW01}
Let $f:X\rightarrow Y$ be a morphism of complex varieties (resp.
$\KK$-varieties), $K^{\bullet}\in D_{\mathrm{c}}^{b}(X(\CC))$ (resp.
$D_{\mathrm{c}}^{b}(X,\bar{\QQ_{l}})$) and $L^{\bullet}\in D_{\mathrm{c}}^{b}(Y(\CC))$
(resp. $D_{\mathrm{c}}^{b}(Y,\bar{\QQ_{l}})$). Then:\[
\mathrm{R}f_!K^{\bullet}\otimes^{\mathbf{L}}L^{\bullet}\simeq\mathrm{R}f_!(K^{\bullet}\otimes^{\mathbf{L}}f^*L^{\bullet}).
\] (Künneth isomorphism) Using proper base change, this implies, for
$K_{i}^{\bullet}\in D_{\mathrm{c}}^{b}(X_{i}(\CC)),\ i=1,2$ (resp.
$D_{\mathrm{c}}^{b}(X,\bar{\QQ_{l}})$):\[
\mathbb{H}_{\mathrm{c}}^{\bullet}(X_1\times X_2,p_1^*K_1^{\bullet}\otimes^{\mathbf{L}}p_2^*K_2^{\bullet})\simeq\mathbb{H}_{\mathrm{c}}^{\bullet}(X_1,K_1^{\bullet})\otimes^{\mathbf{L}}\mathbb{H}_{\mathrm{c}}^{\bullet}(X_2,K_2^{\bullet}),
\]where $p_{i}:X_{1}\times X_{2}\rightarrow X_{i},\ i=1,2$ are the
natural projections.
\end{enumerate}
\end{prop}

\subparagraph*{Perverse sheaves}

Both categories $D_{\mathrm{c}}^{b}(X(\CC))$ and $D_{\mathrm{c}}^{b}(X,\bar{\QQ_{l}})$
contain abelian subcategories $\mathrm{Perv}(X(\CC))$ and $\mathrm{Perv}(X,\bar{\QQ_{l}})$
of so-called perverse sheaves, originally built in \cite{BBDG18}
using t-structures. Perverse sheaves are ubiquitous in geometric representation
theory, as they are essential building blocks of constructible complexes
and provide geometric realisations of many ``simple'' representation-theoretic
objects. In what follows, for $i\in\ZZ$, we denote by $[i]:K^{\bullet}\mapsto K^{\bullet+i}$
the shift functor. Let us also call $\tau_{\leq i}=[-i]\circ\tau_{\leq0}\circ[i]$
the truncation functor of the standard t-structure, where\[
\tau_{\leq0}K^{\bullet}=(\ldots\rightarrow K^{-1}\rightarrow\Ker(d^0)\rightarrow0\rightarrow\ldots)
\] for $K^{\bullet}\in D_{\mathrm{c}}^{b}(X(\CC))$ (the l-adic version
is more subtle, see \cite[\S II.6.]{KW01}). \index[terms]{perverse sheaves}\index[notations]{s@$K^{\bullet}[i]$ - shift of $K^{\bullet}$}

\begin{exmp}[Smooth varieties]

If $X$ is smooth of dimension $d$, an important class of perverse
sheaves is given by locally constant sheaves on $X$ (also called
local systems in the complex case). If $\mathcal{L}$ is a locally
constant sheaf on $X$, then $\mathcal{L}[d]$ is perverse.

\end{exmp}

We recall below the decomposition of a constructible complex into
perverse sheaves (perverse cohomology, intersection complexes) and
Beilinson, Bernstein, Deligne and Gabber's decomposition theorem,
a case where the aforementioned decomposition is particularly nice.
Let us start with perverse cohomology. Just like a complex $K^{\bullet}\in D_{\mathrm{c}}^{b}(X(\CC))$
(resp. $D_{\mathrm{c}}^{b}(X,\bar{\QQ_{l}})$) has cohomology sheaves
$\HH^{i}(K^{\bullet}),\ i\in\ZZ$, which come from the standard t-structure
on $D_{\mathrm{c}}^{b}(X(\CC))$ (resp. $D_{\mathrm{c}}^{b}(X,\bar{\QQ_{l}})$),
$K^{\bullet}$ also has perverse cohomology sheaves $^{\mathrm{p}}\HH^{i}(K^{\bullet})\in\mathrm{Perv}(X(\CC))$
(resp. $^{\mathrm{p}}\HH^{i}(K^{\bullet})\in\mathrm{Perv}(X,\bar{\QQ_{l}})$),
$i\in\ZZ$, which come from the perverse t-structure (see \cite[\S 5.1.]{Dim04}\cite[\S III.1.]{KW01}
for its definition). These belong to the abelian category of perverse
sheaves. Now, perverse sheaves themselves admit a finite composition
series whose successive quotients are simple perverse sheaves (i.e.
perverse sheaves with no proper non-zero subobjects). These simple
constitutents are classified by certain geometric data. \index[notations]{h@$^{\mathrm{p}}\HH^{i}(K^{\bullet})$ - perverse cohomology sheaf of $K^{\bullet}$}

\begin{prop}{\cite[\S 5.2.]{Dim04}\cite[\S III.5.]{KW01}}\index[terms]{intersection complex}\index[notations]{ic@$\mathrm{IC}(X,\mathcal{L})$ - intersection complex}

Let $X$ be a complex variety (resp. a $\KK$-variety), $j:U\hookrightarrow X$
an open embedding and $K^{\bullet}\in\mathrm{Perv}(U(\CC))$ (resp.
$K^{\bullet}\in\mathrm{Perv}(U,\bar{\QQ_{l}})$). Then there exists
a unique $L^{\bullet}\in\mathrm{Perv}(X(\CC))$ (resp. $L^{\bullet}\in\mathrm{Perv}(X,\bar{\QQ_{l}})$)
such that:
\begin{enumerate}
\item $j^{*}L^{\bullet}\simeq K^{\bullet}$;
\item $L^{\bullet}$ has no subobjects or quotients in $\mathrm{Perv}(X(\CC))$
(resp. $\mathrm{Perv}(X,\bar{\QQ_{l}})$) supported on the complement
of $U$.
\end{enumerate}
Then $L^{\bullet}$ is called the intermediate extension of $K^{\bullet}$
and denoted by $j_{!*}K^{\bullet}$.

Moreover, any simple perverse sheaf on $X$ is of the form $\mathrm{IC}(V,\mathcal{L}):=i_{*}j_{!*}\mathcal{L}[\dim(V)]$,
where $V\overset{j}{\hookrightarrow}\overline{V}\overset{i}{\hookrightarrow}X$
is a smooth, irreducible, locally closed subvariety of $X$ and is
$\mathcal{L}$ is an irreducible $\QQ$-local system (resp. a locally
constant, irreducible $\bar{\QQ_{l}}$-sheaf) on $V$.

\end{prop}

We call $\mathrm{IC}(V,\mathcal{L})$ the intersection complex associated
to $\mathcal{L}$. Intersection complexes are the key building blocks
appearing in Beilinson, Bernstein, Deligne and Gabber's decomposition
theorem, which we state below.

\begin{thm}{\cite[Thm. 6.2.5.]{BBDG18}} \label{Thm/BBDGDecomp}

Let $f:X\rightarrow Y$ be a proper morphism of complex varieties
(resp. $\KK$-varieties). Consider $V\subseteq X$ a smooth, irreducible,
locally closed subvariety of $X$ and $\mathcal{L}$ a locally constant
sheaf on $V$. Then $\mathrm{R}f_{*}\mathrm{IC}(V,\mathcal{L})$ is
a semisimple complex i.e:
\begin{enumerate}
\item $\mathrm{R}f_{*}\mathrm{IC}(V,\mathcal{L})\simeq\bigoplus_{i\in\ZZ}\ ^{\mathrm{p}}\HH^{i}(\mathrm{R}f_{*}\mathrm{IC}(V,\mathcal{L}))[-i]$;
\item For all $i\in\ZZ$, $^{\mathrm{p}}\HH^{i}(\mathrm{R}f_{*}\mathrm{IC}(V,\mathcal{L}))$
is a direct sum of simple perverse sheaves.
\end{enumerate}
\end{thm}

\subparagraph*{Sheaf-function dictionary}

Here we describe certain numerical invariants of algebraic varieties,
which can be obtained by cohomological methods. A typical example
is the number of rational points of an algebraic variety $X$ over
a finite field $\KK$. This count can be obtained from Grothendieck's
trace formula \cite[Thm. 1.1.]{KW01}, as the trace of the Frobenius
morphism on the compactly supported l-adic cohomology of $X$. In
other words, we consider the cohomology of the complex $\mathrm{R}p_{!}\underline{\bar{\QQ_{l}}}_{X}$,
where $p$ is the canonical morphism $X\rightarrow\mathrm{pt}$.

More generally, if we consider the cohomology of a constructible complex
on $X$ and restrict it at various points $x\in X$, we obtain a function
on $X$. We will consider two flavours of functions: constructible
functions on a complex variety and functions on $X(\KK)$, where $\KK$
is a finite field. We start with constructible functions (see \cite[\S 4.1.]{Dim04}).

\begin{df}[Constructible function]\index[terms]{constructible function}\index[terms]{Faisceaux-Fonctions correspondence}\index[notations]{c@$\chi_{K^{\bullet}}$ - constructible function associated to $K^{\bullet}$}

Let $X$ be a complex variety. A function $\varphi:X(\CC)\rightarrow\QQ$
is called constructible if there exists a locally finite partition
of $X$ into locally closed subschemes $(X_{i})_{i\in I}$ such that
$\varphi$ is constant on $X_{i}(\CC)$ for every $i\in I$.

Given $K^{\bullet}\in D_{\mathrm{c}}^{b}(X(\CC))$, we define the
constructible function:\[
\begin{array}{rrcl}
\chi_{K^{\bullet}}: & X(\CC) & \rightarrow & \QQ \\
 & x & \mapsto & \sum_{i\in\ZZ}(-1)^i\cdot\dim\left(i_x^*\HH^i(K^{\bullet})\right),
\end{array}
\] where $i_{x}:\{x\}\hookrightarrow X$ is the inclusion of the point
$x$.

\end{df}

Now, let $\KK$ be a finite field with $q$ elements and $X$ a $\KK$-variety.
Then one can associate a function $\chi_{K^{\bullet}}:X(\KK)\rightarrow\CC$
to $K^{\bullet}\in D_{\mathrm{c}}^{b}(X,\bar{\QQ_{l}})$ in a similar
fashion\footnote{This assumes we choose an isomorphism $\bar{\QQ_{l}}\simeq\CC$ or,
more cautiously, an embedding of $\bar{\QQ}$ into $\bar{\QQ_{l}}$
and $\CC$ respectively. See \cite[Rem. 1.2.11.]{Del80}}. Given $x\in X(\KK)$, $\overline{x}\in X(\bar{\KK})$ a geometric
point above $x$ and $i\in\ZZ$, there is an action of $\mathrm{Gal}(\overline{\KK}/\KK)$
on the $\bar{\QQ_{l}}$-vector space $i_{\overline{x}}^{*}\HH^{i}(K^{\bullet})$
- see \cite[\S 1.1.]{KW01} for its construction. We call $F_{\overline{x}}$
the automorphism of $i_{\overline{x}}^{*}\HH^{i}(K^{\bullet})$ given
by the so-called geometric Frobenius, the inverse of the arithmetic
Frobenius $a\mapsto a^{q}$. Its trace does not depend on the choice
of $\overline{x}$. We refer to \cite[\S III.12.]{KW01} for elementary
facts about these functions.

\begin{df}[Function over $X(\KK)$]

Given $K^{\bullet}\in D_{\mathrm{c}}^{b}(X,\bar{\QQ_{l}})$, we define
the function:\[
\begin{array}{rrcl}
\chi_{K^{\bullet}}: & X(\KK) & \rightarrow & \CC \\
 & x & \mapsto & \sum_{i\in\ZZ}(-1)^i\cdot\Tr\left( F_{\overline{x}}\ \vert\ i_{\overline{x}}^*\HH^i(K^{\bullet})\right).
\end{array}
\]

\end{df}

The construction of functions $\chi_{K^{\bullet}}$ (in both settings)
is compatible with the operations on sheaves that we introduced earlier.
Let us first define the corresponding operations on functions:

\begin{df}[Pushforward, pullback of functions]

Let $f:X\rightarrow Y$ be a morphism of complex algebraic varieties
(resp. varieties over a finite field $\KK$). Given constructible
functions $\varphi_{X}:X(\CC)\rightarrow\QQ$ and $\varphi_{Y}:Y(\CC)\rightarrow\QQ$
(resp. functions $\varphi_{X}:X(\KK)\rightarrow\CC$ and $\varphi_{Y}:Y(\KK)\rightarrow\CC$),
we define the functions:\[
f_!\varphi_X:y\mapsto\sum_{a\in\QQ}a\cdot\chi_{\mathrm{c}}(f^{-1}(y)\cap\varphi_X^{-1}(a))
\left(
\text{resp. }
f_!\varphi_X:y\mapsto\sum_{x\in f^{-1}(y)}\varphi_X(x)
\right)
,
\]
\[
f^*\varphi_Y:x\mapsto\varphi_Y(f(x)),
\] where $\chi_{\mathrm{c}}$ denotes the Euler characteristic in compactly
supported cohomology.

\end{df}

Then we have the following compatibilities:

\begin{prop}{\cite[Thm. 4.1.22, Prop. 4.1.33.]{Dim04}\cite[Thm. 12.1.]{KW01}}

Let $f:X\rightarrow Y$ be a morphism of complex algebraic varieties
(resp. varieties over a finite field $\KK$). Then the following hold:
\begin{enumerate}
\item For all $K^{\bullet}\in D_{\mathrm{c}}^{b}(X(\CC))$ (resp. $D_{\mathrm{c}}^{b}(X,\bar{\QQ_{l}})$),
$\chi_{\mathrm{R}f_{!}K^{\bullet}}=f_{!}\chi_{K^{\bullet}}$.
\item For all $K^{\bullet}\in D_{\mathrm{c}}^{b}(Y(\CC))$ (resp. $D_{\mathrm{c}}^{b}(Y,\bar{\QQ_{l}})$),
$\chi_{f^{*}K^{\bullet}}=f^{*}\chi_{K^{\bullet}}$.
\end{enumerate}
\end{prop}

\paragraph*{Weights: mixed Hodge structures, mixed Hodge modules and mixed l-adic
complexes}

We now turn to notions of weights on the cohomology of algebraic varieties.
In l-adic cohomology, weights encode the absolute values of Frobenius
eigenvalues. Weil's last conjecture on the zeta function of smooth
projective varieties is ultimately a statement on weights of their
l-adic cohomology \cite{Del74b}. A parallel formalism, based on Hodge
theory, was developed around the same time by Deligne for complex
algebraic varieties \cite{Del70,Del71,Del74a}. The cohomology of
complex algebraic varieties is endowed with a mixed Hodge structure,
which includes a filtration by weights. We start by defining this
notion.

\begin{df}[Mixed Hodge structures]\index[terms]{mixed Hodge structure}\index[notations]{l@$\mathbb{L}$ - see mixed Hodge structure}

Let $k\in\ZZ$. A pure Hodge structure of weight $k$ is the datum
of a $\QQ$-vector space $V$ and a descending filtration $F^{\bullet}$
on $V_{\CC}$ such that, for all $p,q\in\ZZ$ satisfying $p+q=k+1$,
$F^{p}\cap\overline{F^{q}}=\{0\}$. Given $p,q\in\ZZ$ such that $p+q=k$,
the $(p,q)$-part of $V$ is $V^{p,q}:=F^{p}\cap\overline{F^{q}}$
and the Hodge numbers of $V$ are defined as $h^{p,q}=h^{p,q}(V):=\dim_{\CC}V^{p,q}$.
Then $V_{\CC}$ admits a Hodge decomposition:\[
V_{\CC}=\bigoplus_{p+q=k}V^{p,q}.
\]

A polarisation on $V$ is a $(-1)^{k}$-symmetric, $\QQ$-bilinear
form $Q:V\otimes_{\QQ}V\rightarrow\QQ$ such that $(F^{m})^{\perp}=F^{k-m+1}$
for all $m\in\ZZ$ and the hermitian form $(u,v)\mapsto Q(Cu,\overline{v})$
is positive definite. Here, $C$ is the Weil operator, which acts
diagonally along the Hodge decomposition, with eigenvalue $i^{p-q}$
on $V^{p,q}$. A pure Hodge structure is called polarisable if it
admits a polarisation.

A mixed Hodge structure is the datum of a $\QQ$-vector space $V$,
an increasing filtration $W_{\bullet}$ on $V$ and a decreasing filtration
$F^{\bullet}$ on $V_{\CC}$ such that, for all $k\in\ZZ$, the filtration
$F^{\bullet}$ induces a pure Hodge structure of weight $k$ on $W_{k}/W_{k-1}$.
The filtration $W_{\bullet}$ (resp. $F^{\bullet}$) is called the
weight filtration (resp. the Hodge filtration). For $p,q\in\ZZ$,
the Hodge numbers $h^{p,q}=h^{p,q}(V)$ are defined as the Hodge numbers
of $W_{p+q}/W_{p+q-1}$. A mixed Hodge structure $V$ is called graded-polarisable
if $W_{k}/W_{k-1}$ is polarisable for all $k\in\ZZ$.

For $m\in\ZZ$, we denote by $\QQ(m)$ the pure Hodge structure of
weight $-2m$, whose underlying vector space is one-dimensional and
lies in the $(-m,-m)$-part. We call $\QQ(m),\ m\in\ZZ$ the Hodge
structures of Tate.

A (cohomologically) graded mixed Hodge structure is a $\ZZ$-graded
mixed Hodge structure. We will refer to that grading as cohomological
degree. We denote by $\mathbb{L}$ the cohomologically graded mixed
Hodge structure $\QQ(-1)$, concentrated in cohomological degree $2$.
We say that a graded mixed Hodge structure $V^{\bullet}$ is of Tate
type if there exist $a_{m,n}\in\ZZ_{\geq0},\ m,n\in\ZZ$ such that
$V^{\bullet}=\bigoplus_{m,n}\left(\mathbb{L}^{\otimes n}[m]\right)^{\oplus a_{m,n}}$.

\end{df}

Note that there is a well-defined tensor product on (cohomologically
graded) mixed Hodge structures. Given a mixed Hodge structure $V$
and $m\in\ZZ$, we will write $V(m):=V\otimes\QQ(m)$. For a cohomologically
graded mixed Hodge structure $V^{\bullet}$ and $i\in\ZZ$, we will
write $V^{\bullet}[i]=V^{\bullet+i}$ as above for complexes of sheaves.
Then $\mathbb{L}=\QQ[-2](-1)$.

\begin{exmp}[Cohomology of an algebraic variety]

An important source of cohomologically graded mixed Hodge structures
is the cohomology of complex algebraic varieties. By \cite{Del71,Del74a},
if $X$ is a complex algebraic variety, then $\HH^{\bullet}(X,\QQ)$
carries a cohomologically graded mixed Hodge structure. This is also
true for compactly supported cohomology and Borel-Moore homology,
as can be seen either from Deligne's simplicial techniques \cite[Ch. 5,6]{PS08}
or using Saito's mixed Hodge modules \cite[Ch. 14]{PS08}. See \cite[Rmk. 14.10.]{PS08}
for a comparison between these approaches.

\end{exmp}

In the relative setting (i.e. working with sheaves over a base variety),
the corresponding notions are mixed l-adic complexes (see \cite{Del80})
and mixed Hodge modules (later developed by Saito \cite{Sai88,Sai90}).
Let us recall briefly what these objects are. We refer to \cite[\S II.12.]{KW01}
for mixed l-adic complexes and \cite{Sai89}\cite[Ch. 14]{PS08} for
mixed Hodge modules. See also \cite{Sch14} for a gentle and motivated
introduction.

Consider a variety $X$ over a finite field $\KK$ and an l-adic sheaf
$\mathcal{F}$ over $X$. Fix an isomorphism $\bar{\QQ_{l}}\simeq\CC$.
As mentioned above, for any closed point $x\in X$ with residue field
$\kappa(x)$, the Galois group $\mathrm{Gal}(\bar{\KK}/\kappa(x))$
acts on the $\bar{\QQ_{l}}$-vector space $\mathcal{F}_{\overline{x}}$.
Then $\mathcal{F}$ is called pure of weight $k$ if $\vert\alpha\vert^{2}=\left(\sharp\kappa(x)\right)^{k}$
for all closed points $x\in X$ and all eigenvalues $\alpha$ of the
geometric Frobenius on $\mathcal{F}_{\overline{x}}$. Likewise, $\mathcal{F}$
is called mixed if it admits a finite filtration whose composition
factors are pure and of increasing weight. Finally, a complex $K^{\bullet}\in D_{\mathrm{c}}^{b}(X,\bar{\QQ_{l}})$
is called mixed if its cohomology sheaves $\HH^{i}(K^{\bullet})$
are all mixed. Mixed complexes form a full subcategory $D_{\mathrm{m}}^{b}(X,\bar{\QQ_{l}})\subseteq D_{\mathrm{c}}^{b}(X,\bar{\QQ_{l}})$
which is preserved by the six operations introduced above. Purity
for complexes is more subtle: $K^{\bullet}\in D_{\mathrm{m}}^{b}(X,\bar{\QQ_{l}})$
is called pure of weight $k$ if, for all $i\in\ZZ$, for all closed
points $x\in X$, all eigenvalues $\alpha$ of the geometric Frobenius
on $i_{\overline{x}}\HH^{i}(K^{\bullet})$ (resp. $i_{\overline{x}}\HH^{i}(\mathbb{D}K^{\bullet})$)
satisfy $\vert\alpha\vert^{2}\leq\left(\sharp\kappa(x)\right)^{k+i}$
(resp. $\vert\alpha\vert^{2}\leq\left(\sharp\kappa(x)\right)^{-k+i}$).
We say that $K^{\bullet}$ has weight at most (resp. at least) $k$
if only the first (resp. only the second) condition is fulfilled.
Note that there are l-adic analogues $\underline{\bar{\QQ_{l}}}(m)$
of the Hodge structures of Tate. We will write $K^{\bullet}(m):=K^{\bullet}\otimes^{\mathbf{L}}\underline{\bar{\QQ_{l}}}(m)_{X}$
as well. \index[notations]{t@$K^{\bullet}(m)$ - Tate twist of $K^{\bullet}$}

In the complex setting, a relative analogue of a (graded-polarisable)
mixed Hodge structure is a (graded-polarisable) variation of mixed
Hodge structures on a complex manifold. Roughly speaking, this consists
of a $\QQ$-local system $\mathcal{L}$ and a complex vector bundle
$\mathcal{V}$ with connection, whose sheaf of horizontal sections
is $\mathcal{L}_{\CC}$. The local system carries an analogue of the
weight filtration, while the complex vector bundle carries an analogue
of the Hodge filtration. Saito's theory consists in generalising variations
of mixed Hodge structures in the same way that perverse sheaves generalise
local systems. The local system should be replaced by a perverse sheaf
$K^{\bullet}$, while the vector bundle with connection should be
replaced by a suitably filtered regular holonomic D-module $\mathcal{M}$,
such that $\mathrm{DR}(\mathcal{M})\simeq K^{\bullet}\otimes_{\QQ}\CC$\footnote{Here $\mathrm{DR}$ is the de Rham functor, which yields an equivalence
between the categories of regular holonomic D-modules and perverse
sheaves on $X$.}.

The outcome is the category of algebraic mixed Hodge modules $\mathrm{MHM}(X)$.
Mixed Hodge modules (and complexes thereof) enjoy the six operations
described above, which are compatible with the underlying constructible
complexes. In particular $\mathrm{MHM}(\mathrm{pt})$ is equivalent
to the category of graded-polarisable mixed Hodge structures. The
computational facts from Proposition \ref{Prop/FormulasD_c^b} generalise
to mixed Hodge modules (see \cite[\S 4.3-4.4.]{Sai89}\cite{MSS11}).
Any mixed Hodge module is endowed with a weight filtration generalising
that of mixed Hodge structures. Moreover, given a smooth, irreducible,
locally closed subvariety $V\subseteq X$ and a polarisable variation
of (pure) Hodge structures $\mathcal{L}$ on $V$, there is a (pure)
Hodge module $\mathrm{IC}^{H}(V,\mathcal{L})$ supported on $\overline{V}$
extending $\mathcal{L}$ and whose underlying perverse sheaf is $\mathrm{IC}(V,\mathcal{L})$.
These are exactly the simple Hodge modules \cite[Thm. 14.36.]{PS08}.
Finally, a complex of mixed Hodge modules $M^{\bullet}$ is called
pure of weight $k$ if, for all $i\in\ZZ$, $\HH^{i}(M^{\bullet})$
is pure\footnote{Note that, given a complex $M^{\bullet}$ of mixed Hodge modules,
the underlying complex of perverse sheaves corresponds to a constructible
complex $K^{\bullet}$. Then the perverse sheaf underlying $\HH^{i}(M^{\bullet})$
is $^{\mathrm{p}}\HH^{i}(K^{\bullet})$.} of weight $k+i$.

Besides categorification of counts, our interest in weights lies in
their relation to semisimplicity. It turns out that simple objects
in the categories of l-adic perverse sheaves and mixed Hodge modules
are pure, and controlling weights sometimes allows to derive vanishing
of certain extension groups. We will use this at times in computations.

\begin{prop} \label{Prop/WeightDecomp}

Let $X$ be a complex algebraic variety (resp. a variety over a finite
field $\KK$). Then the following hold:
\begin{itemize}
\item (Splitting of extensions) \cite[Cor. 1.10.]{Sai89}\cite[Prop. 12.4.]{KW01}
Consider $M^{\bullet},N^{\bullet}\in D^{b}(\mathrm{MHM}(X))$ (resp.
$K^{\bullet},L^{\bullet}\in D_{\mathrm{m}}^{b}(X,\bar{\QQ_{l}})$).
Suppose that the first complex has weight at most $k$ and that the
second complex has weight at least $k-i$ for some $k,i\in\ZZ$. Then
for all $j>i$:\[
\Hom(M^{\bullet},N^{\bullet}[j])=0 \  
(
\text{resp. }
\Hom(K^{\bullet},L^{\bullet}[j])=0
).
\]
\item (Purity implies semisimplicity) \cite[Cor. 1.11.]{Sai89}\cite[Thm. 10.6.]{KW01}
Consider $M^{\bullet}\in D^{b}(\mathrm{MHM}(X))$ (resp. $K^{\bullet}\in D_{\mathrm{m}}^{b}(X,\bar{\QQ_{l}})$).
Call $\bar{K}^{\bullet}$ the pullback of $K^{\bullet}$ to $X\times_{\KK}\bar{\KK}$.
If $M^{\bullet}$ (resp. $K^{\bullet}$) is pure, then it is semisimple
i.e:\[
M^{\bullet}\simeq\bigoplus_{i\in\ZZ}\HH^i(M^{\bullet})[-i]
\left(
\text{resp. }
\bar{K}^{\bullet}\simeq\bigoplus_{i\in\ZZ}\ ^{\mathrm{p}}\HH^i(\bar{K}^{\bullet})[-i]
\right)
\] and for all $i\in\ZZ$, $\HH^{i}(M^{\bullet})$ (resp. $^{\mathrm{p}}\HH^{i}(\bar{K}^{\bullet})$)
is a direct sum of simple mixed Hodge modules (resp. l-adic perverse
sheaves).
\item (Weight filtration) \cite[Thm. 14.31, Thm. 14.36.]{PS08}\cite[\S III.9 Lem. I, Lem. III.]{KW01}
Let $M$ be a mixed Hodge module (resp. $K^{\bullet}$ a mixed l-adic
perverse sheaf) on $X$. Then $M$ (resp. $K^{\bullet}$) admits a
finite increasing filtration whose composition factors are simple
(pure) Hodge modules (resp. simple pure perverse sheaves) with increasing
weight. Moreover, given $V\subseteq X$ a smooth, irreducible, locally
closed subvariety and $\mathcal{L}$ a variation of (pure) Hodge structures
(resp. a pure locally constant l-adic sheaf) of weight $k$, then
$\mathrm{IC}^{H}(V,\mathcal{L})$ (resp. $\mathrm{IC}(V,\mathcal{L})$)
is pure of weight $k+\dim(V)$.
\end{itemize}
\end{prop}

\paragraph*{Cohomology of quotient stacks}

In order to study the cohomology of quiver moduli, we need to build
cohomology groups and weight structures for quotient stacks. We work
with stacks of the form $\left[X/G\right]$, where $X$ is a complex
algebraic variety with an action of a linear algebraic group $G$.
The cohomology of $\left[X/G\right]$ is built using techniques of
equivariant cohomology. The main idea behind this construction is
to approximate the space $(X\times EG)/G$ by algebraic varieties
and dates back to works of Totaro and Edidin, Graham \cite{Tot99,EG98a}.
We follow \cite[\S 2.2.]{DM20} for the construction of a mixed Hodge
structure on the cohomology of $\left[X/G\right]$. We then explain
how weights in compactly supported cohomology relate to counts of
points over finite fields, following \cite[Appendix, Thm. 6.1.2.]{HRV08}.
Finally, we gather some computational lemmas which we will use in
Chapter \ref{Chap/Categorification} to compute the cohomology of
$\mathfrak{M}_{\Pi_{(Q,\nn)},\rr}$ in certain cases.

We recall Davison and Meinhardt's construction in the case where $G=\GL_{\nn,\rr}$.
We first need to build certain geometric quotients, inspired from
complex Stiefel manifolds. For $n,r\geq1$ and $N\geq1$ sufficiently
large, we consider the open subset $U_{r,N,n}\subseteq\Hom_{\KK_{n}}(\KK_{n}^{\oplus r},\KK_{n}^{\oplus N})$
formed by injective morphisms (recall the notation $\KK_{n}:=\KK[t]/(t^{n})$).
Viewing points $M\in\Hom_{\KK_{n}}(\KK_{n}^{\oplus r},\KK_{n}^{\oplus N})$
as $N\times r$ matrices with coefficients in $\KK_{n}$, $U_{r,N,n}$
is defined by the condition that not all $r\times r$ minors of $M$
are zero modulo $t$. Given $\Delta$ a subset of $r$ lines, we denote
by $U_{\Delta,n}\subseteq U_{r,N,n}$ the open subset of matrices
for which the minor associated to $\Delta$ is non-zero modulo $t$.
The variety $U_{r,N,n}$ is endowed with a free action of $\GL_{n,r}$
by right multiplication, for which the $U_{\Delta,n}$ form a $\GL_{n,r}$-invariant
open cover. We argue that this action has a geometric quotient (in
the sense of \cite[Def. 0.6.]{MFK94}), although the usual methods
of invariant theory are not available ($\GL_{n,r}$ being non-reductive
when $n>1$).

\begin{lem} \label{Lem/BorelApprox}

The variety $U_{r,N,n}$ admits a geometric quotient $U_{r,N,n}\rightarrow\mathrm{Gr}_{r,N,n}$,
where $\mathrm{Gr}_{r,N,n}$ is the $(n-1)$-th jet scheme of the
grassmannian $\mathrm{Gr}_{r,N}$.

\end{lem}

\begin{proof}

When $n=1$, $\mathrm{Gr}_{r,N}$ is the geometric quotient $U_{r,N,1}/\GL_{r}$,
which can be defined using geometric invariant theory. The quotient
morphism $q:U_{r,N}=U_{r,N,1}\rightarrow\mathrm{Gr}_{r,N}$ is a $\GL_{r}$-principal
bundle, which is trivial in restriction to the open subsets $V_{\Delta}=q(U_{\Delta})\subseteq\mathrm{Gr}_{r,N}$.
The quotient $\mathrm{Gr}_{r,N}$ can thus be constructed by glueing
the open subsets $V_{\Delta}$. Since the construction of jet schemes
is functorial and compatible with open immersions (see \cite[\S 2.1.]{CLNS18}),
the scheme $\mathrm{Gr}_{r,N,n}$ can be obtained by glueing the jet
schemes $V_{\Delta,n}$. We also obtain a principal $\GL_{n,r}$-bundle
$U_{r,N,n}\rightarrow\mathrm{Gr}_{r,N,n}$, which is trivial in restriction
to the open subsets $V_{\Delta,n}$. This shows that $U_{r,N,n}\rightarrow\mathrm{Gr}_{r,N,n}$
is a geometric quotient. \end{proof}

Now, let $G=\GL_{\nn,\rr}$ and $X$ a complex $G$-variety. Then
$U_{\rr,N,\nn}:=\prod_{i\in Q_{0}}U_{r_{i},N,n_{i}}$ admits a geometric
quotient as well under the action of $G$. Following \cite[\S 2.2.]{DM20},
we make the following:

\begin{df}[Homology, cohomology of a quotient stack]\index[terms]{cohomology!compactly supported cohomology}\index[terms]{cohomology!Borel-Moore homology}\index[notations]{h@$\HH^{\bullet}$ - cohomology}\index[notations]{h@$\HH_{\mathrm{c}}^{\bullet}$ - compactly supported cohomology}\index[notations]{h@$\HH_{\bullet}^{\mathrm{BM}}$ - Borel-Moore homology}

Let $k\in\ZZ$. The $k$-th cohomology and compactly supported cohomology
groups of $\left[X/G\right]$ are defined as the mixed Hodge structures:\begin{align*}
& \HH^{k}([X/G]):=\HH^{k}\left(X\times^{G}U_{\rr,N,\nn}\right), \\
& \HH_{\mathrm{c}}^{k}([X/G]):=\HH_{\mathrm{c}}^{k+2\dim(U_{\rr,N,\nn})}\left(X\times^{G}U_{\rr,N,\nn}\right)\otimes\QQ(\dim(U_{\rr,N,\nn})),
\end{align*} for $N$ large enough. Dually, we define the $k$-th Borel-Moore
homology group of $\left[X/G\right]$ as:\[
\HH_{k}^{\mathrm{BM}}([X/G]):=\HH_{\mathrm{c}}^{k}([X/G])^{\vee}.
\]

\end{df}

This is shown to be independent of the choice of $U_{\rr,N,\nn}$,
as long as a certain codimension assumption is satisfied, which is
the case for $N$ large enough. The variety $X\times^{G}U_{\rr,N,\nn}$
is the geometric quotient of $X\times U_{\rr,N,\nn}$ under the diagonal
action. This quotient is well-defined by \cite[Prop. 23]{EG98a} (note
that we are in the case where the principal bundle $U_{\rr,N,\nn}\rightarrow U_{\rr,N,\nn}/\GL_{\nn,\rr}$
is Zariski-locally trivial).

\begin{rmk}

There are also $G$-equivariant versions of the categories of constructible
complexes (both for the analytic and the étale topology) and mixed
Hodge modules. We do not dwell upon these as they appear only marginally
in the thesis. A lot of care is required to obtain a triangulated
category with the usual six operations. The construction ultimately
relies on approximation by varieties with a free $G$-action, as for
equivariant cohomology. See \cite{BL94} for categories of equivariant
complexes, \cite{LO08} for l-adic complexes on Artin stacks, \cite{Ach13}
for equivariant complexes of mixed Hodge modules and \cite[\S 3.2.2.]{DHSM22}
for an adaptation of this construction to certain stacks.

\end{rmk}

Let us now recall some results relating Hodge numbers of an algebraic
variety to counts of its points over finite fields. These results
rely on a theorem by Katz \cite[Appendix, Thm. 6.1.2.]{HRV08}. Throughout,
we assume that $G=\GL_{\nn,\rr}$.

\begin{df}[E-series]

Let $X$ be a complex $G$-variety. The E-series of $[X/G]$ is defined
as:\[
E([X/G];x,y):=\sum_{p,q\in\ZZ}\left(\sum_{k\in\ZZ}(-1)^k\cdot h^{p,q}\left(\HH_{\mathrm{c}}^{k}([X/G])\right)\right)\cdot x^py^q\in\ZZ((x^{-1},y^{-1})).
\]

\end{df}

Let us briefly justify that the E-series is a well-defined formal
Laurent series in $x^{-1},y^{-1}$. For a fixed couple $(p,q)$, $h^{p,q}\left(\HH_{\mathrm{c}}^{k}([X/G])\right)\ne0$
only if $p+q\leq k$ (see \cite[Ch. 5, Prop. 5.54.]{PS08}). As $\HH_{\mathrm{c}}^{\bullet}([X/G])$
is supported in degree at most $\dim\left([X/G]\right)$, the coefficient
of $x^{p}y^{q}$ boils down to a finite sum. Moreover, one can check
from \cite[Thm. 8.2.4.]{Del74a} and \cite[Ch. 5, Def. 5.52.]{PS08}
that $\sum_{k}(-1)^{k}\cdot h^{p,q}\left(\HH_{\mathrm{c}}^{k}([X/G])\right)\ne0$
only if $p,q\leq\dim\left([X/G]\right)$, so we indeed obtain a Laurent
series. A similar reasoning shows that the $E$-series of a variety
$X$ is actually a polynomial in $x,y$.

Given a complex $G$-variety $X$, we may count points over finite
fields of some spreading-out of $X$ i.e. an $R$-scheme $\mathcal{X}$,
where $R\subseteq\CC$ is a finitely generated $\ZZ$-algebra and
$\mathcal{X}\otimes_{R}\CC\simeq X$. Such a ring $R$ admits ring
homomorphisms $R\rightarrow\FF_{q}$ for finite fields of large enough
characteristics (see for instance \cite[Ch. 1, Lem. 2.2.6.]{CLNS18}).
Following \cite[Appendix]{HRV08}, we call $X$ polynomial-count if
there exists a spreading-out $\mathcal{X}$ and a polynomial $P\in\mathbb{Q}[T]$
such that for any ring homomorphism $\varphi:R\rightarrow\FF_{q}$,
$\sharp\mathcal{X}_{\varphi}(\FF_{q^{r}})=P(q^{r})$ (for all $r\geq1$).
Note that $G$ is polynomial-count, with counting polynomial $P_{G}(T)=\prod_{i\in Q_{0}}t^{n_{i}r_{i}^{2}}(1-T^{-1})\ldots(1-T^{-(r_{i}-1)})$.
The following proposition is a straightforward generalisation of \cite[Thm. 6.1.2.]{HRV08}.
See also \cite[\S 2.5.]{SV20} and \cite[\S 2.2]{LRV23} for analogous
results in Borel-Moore homology and l-adic cohomology.

\begin{prop} \label{Prop/E-seriesVSCountF_q}

If $X$ is polynomial-count, with counting polynomial $P_{X}$, then:\[
E([X/G];x,y)=\frac{P_X(xy)}{P_G(xy)}.
\]

\end{prop}

\begin{proof}

By construction,\[
E([X/G];x,y)=
\underset{N\rightarrow +\infty}{\lim}
\left(
\frac{E\left( X\times^GU_{\rr,N,\nn};x,y\right)}{(xy)^{\dim(U_{\rr,N,\nn})}}
\right)
.
\] One can check that $U_{\rr,N,\nn}$ is polynomial-count and:\[
\frac{P_{U_{\rr,N,\nn}}(T)}{T^{\dim(U_{\rr,N,\nn})}}
=
\prod_{i\in Q_0}(1-T^{-N})\ldots(1-T^{-(N-r_i+1)})
=1+\mathrm{O}(T^{-(N-\max_i\{r_i\}+1)}).
\] Now, by \cite[Thm. 6.1.2.]{HRV08} and from the construction of $X\times^{G}U_{\rr,N,\nn}$,
we obtain:\[
\frac{E\left( X\times^GU_{\rr,N,\nn};x,y\right)}{(xy)^{\dim(U_{\rr,N,\nn})}}
=
\frac{P_{U_{\rr,N,\nn}}(xy)}{(xy)^{\dim(U_{\rr,N,\nn})}}
\cdot
\frac{P_X(xy)}{P_G(xy)}
\underset{N\rightarrow +\infty}{\longrightarrow}
\frac{P_X(xy)}{P_G(xy)}.
\]\end{proof}

Finally, if $X$ is polynomial-count and $\HH_{\mathrm{c}}^{\bullet}([X/G])$
is pure, then it follows from Proposition \ref{Prop/E-seriesVSCountF_q}
that $\HH_{\mathrm{c}}^{\bullet}([X/G])$ is of Tate type, concentrated
in even degrees and:\[
E([X/G];x,y)=\sum_{k\in\ZZ}\dim\left(\HH_{\mathrm{c}}^{2k}([X/G])\right)\cdot (xy)^k.
\]In other words, $\HH_{\mathrm{c}}^{\bullet}([X/G])$ is determined
by its E-series, since $\HH_{\mathrm{c}}^{\bullet}([X/G])=P_{X}(\mathbb{L})\otimes F_{G}(\mathbb{L})$.
Here for a Laurent series $F(T)=\sum_{i}a_{i}\cdot T^{i}\in\ZZ_{\geq0}((T^{-1}))$,
we define $F(\mathbb{L}):=\bigoplus_{i}\left(\mathbb{L}^{\otimes i}\right)^{\oplus a_{i}}$.
Note that $\HH_{\mathrm{c}}^{\bullet}(\mathrm{BGL}_{\nn,\rr})$=$\HH_{\mathrm{c}}^{\bullet}([\mathrm{pt}/\GL_{\nn,\rr}])$
is pure, with E-series $F_{G}(xy)=\prod_{i\in Q_{0}}\frac{(xy)^{-n_{i}r_{i}^{2}}}{(1-(xy)^{-1})\ldots(1-(xy)^{-(r_{i}-1)})}$.

We now collect some lemmas on equivariant cohomology, which will prove
useful in Chapter \ref{Chap/Categorification} in computing $\HH_{\mathrm{c}}^{\bullet}(\mathfrak{M}_{\Pi_{(Q,\nn)},\rr})$.
Unless specified otherwise, we assume that $G=\GL_{\nn,\rr}$.

\begin{lem} \label{Lem/AffFib}

Suppose that $f:X\rightarrow Y$ is a $G$-equivariant (Zariski-locally
trivial) affine fibration of dimension $d$. Then:

\[
\HH_{\mathrm{c}}^{\bullet}([X/G])
\simeq
\mathbb{L}^{\otimes d}
\otimes
\HH_{\mathrm{c}}^{\bullet}([Y/G]).
\]

\end{lem}

\begin{proof}

For an affine fibration of algebraic varieties $f:X\rightarrow Y$,
the result is obtained by applying $\mathrm{R}p_{!}$ to the isomorphism
$\mathrm{R}f_{!}\underline{\QQ}_{X}\simeq\underline{\QQ}_{Y}\otimes\mathbb{L}_{Y}^{\otimes d}$
in $D^{b}(\mathrm{MHM}(Y))$, where $p$ is the canonical morphism
$Y\rightarrow\mathrm{pt}$ and $\mathbb{L}_{Y}:=p^{*}\mathbb{L}$.
The isomorphism is obtained as follows: one obtains a morphism $\varphi:\mathrm{R}f_{!}\underline{\QQ}_{X}\rightarrow\underline{\QQ}_{Y}\otimes\mathbb{L}_{Y}^{\otimes d}$
from the adjunction $\mathrm{R}f_{!}\dashv f^{!}$ and the natural
isomorphism $f^{!}\simeq f^{*}[2d](d)$\footnote{This isomorphism exists essentially by construction in the case of
affine fibrations, see \cite[\S 4.4.]{Sai89}.}. Then one can check on trivialising opens of $f$ that $\varphi$
is an isomorphism, using the projection formula (Proposition \ref{Prop/FormulasD_c^b}).

In the equivariant setting, one can check, using the construction
of $X\times^{G}U_{\rr,N,\nn}$ from \cite[Prop. 23]{EG98a}, that
$f$ induces an affine fibration $X\times^{G}U_{\rr,N,\nn}\rightarrow Y\times^{G}U_{\rr,N,\nn}$
of dimension $d$. \end{proof}

\begin{lem} \label{Lem/DepthChg}

Let $\nn'=(n_{i}-1)_{i\in Q_{0}}$. Let $X$ be a $\GL_{\nn,\rr}$-variety
such that the normal subgroup $K_{\nn,\rr}=\Ker(\GL_{\nn,\rr}\twoheadrightarrow\GL_{\nn',\rr})$
acts trivially on $X$. Then:

\[
\HH_{\mathrm{c}}^{\bullet}([X/\GL_{\nn,\rr}])
\simeq
\mathbb{L}^{\otimes (-\rr\cdot\rr)}
\otimes
\HH_{\mathrm{c}}^{\bullet}([X/\GL_{\nn',\rr}]).
\]

\end{lem}

\begin{proof}

For simplicity, assume that $Q$ has only one vertex. We claim that,
for $N$ large enough, the geometric quotient $U'_{r,N,n}:=U_{r,N,n}/K_{n,r}$
is well defined and that the projection $U_{r,N,n}\rightarrow U_{r,N,n-1}$
induces an affine fibration \[
X\times^{\GL_{n,r}}U_{r,N,n}\simeq X\times^{\GL_{n-1,r}}U'_{r,N,n}\rightarrow X\times^{\GL_{n-1,r}}U_{r,N,n-1}
\] of dimension $r(N-r)$. Here, the isomorphism $X\times^{\GL_{n,r}}U_{r,N,n}\simeq X\times^{\GL_{n',r}}U'_{r,N,n}$
comes from the fact that $K_{n,r}$ acts trivially on $X$. Then,
since $\dim(U_{r,N,n})=\dim(U_{r,N,n-1})+rN$, we obtain by Lemma
\ref{Lem/AffFib}:\[
\HH_{\mathrm{c}}^{\bullet}(X\times^{\GL_{n,r}}U_{r,N,n})
\otimes \mathbb{L}^{-\dim(U_{r,N,n})}
\simeq
\HH_{\mathrm{c}}^{\bullet}(X\times^{\GL_{n-1,r}}U_{r,N,n-1})
\otimes \mathbb{L}^{-\dim(U_{r,N,n-1})} \otimes \mathbb{L}^{-r^2},
\] which yields the results when $N$ goes to infinity.

We now prove the claims. First, we construct the geometric quotient
$U'_{r,N,n}$ by glueing quotients of the open subsets $U_{\Delta,n}$
as above. Given a set $\Delta$ of $r$ lines, consider the subset
$U'_{\Delta,n}\subset U_{\Delta,n}$ of matrices $M(t)=\sum_{k}M_{k}\cdot t^{k}$
such that the lines of $M_{n-1}$ indexed by $\Delta$ are zero. Then
the action of $K_{n,r}$ induces an isomorphism $U_{\Delta,n}\simeq U'_{\Delta,n}\times K_{n,r}$.
One can check that the varieties $U'_{\Delta,n}$ glue, since $K_{n,r}$
acts freely on $U_{r,N,n}$. This yields a geometric quotient $U'_{r,N,n}$
as in the proof of Lemma \ref{Lem/BorelApprox}.

Finally, we describe the affine fibration $X\times^{\GL_{n-1,r}}U'_{r,N,n}\rightarrow X\times^{\GL_{n-1,r}}U_{r,N,n-1}$.
The isomorphisms $U'_{\Delta,n}\simeq U_{\Delta,n-1}\times\mathbb{A}^{r(N-r)}$
glue to a $\GL_{n-1,r}$-equivariant affine fibration $U'_{r,N,n}\rightarrow U_{r,N,n-1}$
of dimension $r(N-r)$, induced by the $K_{n,r}$-invariant map $U_{r,N,n}\rightarrow U_{r,N,n-1}$.
One can then check that the following diagram \[
\begin{tikzcd}[ampersand replacement=\&]
X\times U'_{r,N,n} \ar[r]\ar[d] \& X\times^{\GL_{n-1,r}}U'_{r,N,n} \ar[d] \\
X\times U_{r,N,n-1} \ar[r] \& X\times^{\GL_{n-1,r}}U_{r,N,n-1}
\end{tikzcd}
\] restricts to \[
\begin{tikzcd}[ampersand replacement=\&]
X\times V_{\Delta,n-1}\times\GL_{n-1,r}\times\mathbb{A}^{r(N-r)}  \ar[r]\ar[d] \& X\times V_{\Delta,n-1}\times\mathbb{A}^{r(N-r)} \ar[d] \\
X\times V_{\Delta,n-1}\times\GL_{n-1,r} \ar[r] \& X\times V_{\Delta,n-1}
\end{tikzcd}
\] over $X\times V_{\Delta,n-1}$ (see the proof of Lemma \ref{Lem/BorelApprox}
for the definition of $V_{\Delta,n-1}$). \end{proof}

\begin{lem} \label{Lem/GrpChg}

Let $n,r\geq1$ and $\tilde{\nn}:=(n_{i},\ i\in Q_{0};n)$, $\tilde{\rr}:=(r_{i},\ i\in Q_{0};r)$.
Let $X$ be a $\GL_{\nn,\rr}$-variety. Then:\[
\HH_{\mathrm{c}}^{\bullet}([(X\times^{\GL_{\nn,\rr}}\GL_{\tilde{\nn},\tilde{\rr}}/\GL_{\tilde{\nn},\tilde{\rr}}])
\simeq
\HH_{\mathrm{c}}^{\bullet}([X/\GL_{\nn,\rr}]).
\]

\end{lem}

\begin{proof}

Using the isomorphism\[
\left(
X\times^{\GL_{\nn,\rr}}\GL_{\tilde{\nn},\tilde{\rr}}
\right)
\times^{\GL_{\tilde{\nn},\tilde{\rr}}}U_{\tilde{\rr},N,\tilde{\nn}}
\simeq
\left(
X\times^{\GL_{\nn,\rr}}U_{\rr,N,\nn}
\right)
\times U_{r,N,n},
\] and the Künneth isomorphism (see Proposition \ref{Prop/FormulasD_c^b}
and its generalisation to mixed Hodge modules), we obtain:\[
\begin{split}
\HH_{\mathrm{c}}^{\bullet}((X\times^{\GL_{\nn,\rr}}\GL_{\tilde{\nn},\tilde{\rr}})\times^{\GL_{\tilde{\nn},\tilde{\rr}}}U_{\tilde{\rr},N,\tilde{\nn}}))
\otimes \mathbb{L}^{-\dim(U_{\tilde{\rr},N,\tilde{\nn}})}
\simeq & \left(\HH_{\mathrm{c}}^{\bullet}(X\times^{\GL_{\nn,\rr}}U_{\rr,N,\nn})
\otimes \mathbb{L}^{-\dim(U_{\rr,N,\nn})}
\right) \\
& \otimes
\left(
\HH_{\mathrm{c}}^{\bullet}(U_{r,N,n})\otimes\mathbb{L}^{-\dim(U_{r,N,n})}
\right)
\end{split}
\] Let us examine $\HH_{\mathrm{c}}^{\bullet}(U_{r,N,n})\otimes\mathbb{L}^{-\dim(U_{r,N,n})}$.
The open-closed decomposition\[
\Hom_{\KK_{n}}(\KK_{n}^{\oplus r},\KK_{n}^{\oplus N})=U_{r,N,n}\sqcup\left(\Hom_{\KK_{n}}(\KK_{n}^{\oplus r},\KK_{n}^{\oplus N})\setminus U_{r,N,n}\right)
\] yields a long exact sequence in compactly supported cohomology (this
is a consequence of the adjunction triangle for mixed Hodge modules,
see Proposition \ref{Prop/FormulasD_c^b}).

The point-count of $U_{r,N,n}$ over finite fields (see the proof
of Proposition \ref{Prop/E-seriesVSCountF_q}) yields the following
facts: (i) $\HH_{\mathrm{c}}^{\bullet}(U_{r,N,n})\otimes\mathbb{L}^{-\dim(U_{r,N,n})}$
is concentrated in nonpositive degrees (ii) its graded piece in degree
0 is $\QQ$ (with trivial mixed Hodge structure) (iii) it vanishes
in cohomological degrees $-2(N-\max_{i}\{r_{i}\}+1)$ to $-1$.

Since $\HH_{\mathrm{c}}^{\bullet}(X\times^{\GL_{\nn,\rr}}U_{\rr,N,\nn}))\otimes\mathbb{L}^{-\dim(U_{\rr,N,\nn})}$
is supported in cohomological degree at most $\dim\left[X/\GL_{\nn,\rr}\right]$,
we get from the long exact sequence that $\HH_{\mathrm{c}}^{j}([(X\times^{\GL_{\nn,\rr}}\GL_{\tilde{\nn},\tilde{\rr}}/\GL_{\tilde{\nn},\tilde{\rr}}])\simeq\HH_{\mathrm{c}}^{j}([X/\GL_{\nn,\rr}])$
for any given $j$, by taking $N$ large enough. \end{proof}

\begin{lem} \label{Lem/Kunneth}

Let $X$ be a $\GL_{\nn,\rr_{1}}$-variety and $Y$ be a $\GL_{\nn,\rr_{2}}$-variety.
Then:

\[
\HH_{\mathrm{c}}^{\bullet}([(X\times Y)/(\GL_{\nn_1,\rr_1}\times\GL_{\nn_2,\rr_2})])
\simeq
H_{\mathrm{c}}^{\bullet}([X/\GL_{\nn_1,\rr_1}])
\otimes
H_{\mathrm{c}}^{\bullet}([Y/\GL_{\nn_2,\rr_2}]).
\]

\end{lem}

\begin{proof}

The isomorphism can be checked directly in each cohomological degree
- taking $N$ large enough - from the isomorphism\[
(X\times Y)\times^{\GL_{\nn_1,\rr_1}\times\GL_{\nn_2,\rr_2}}(U_{\rr_1,N,\nn_1}\times U_{\rr_2,N,\nn_2})
\simeq
(X\times^{\GL_{\nn_1,\rr_1}}U_{\rr_1,N,\nn_1})
\times
(Y\times^{\GL_{\nn_2,\rr_2}}U_{\rr_2,N,\nn_2})
\] and the Künneth isomorphism (see Proposition \ref{Prop/FormulasD_c^b}
and its generalisation to mixed Hodge modules). \end{proof}

\begin{lem} \label{Lem/Strat}

Let $X$ be a $G$-variety, $Z\subseteq X$ a $G$-invariant closed
subvariety and $U=X\setminus Z$. Then there is a long exact sequence
of mixed Hodge structures:

\[
\ldots\rightarrow
\HH_{\mathrm{c}}^{j-1}([Z/G])\rightarrow
\HH_{\mathrm{c}}^j([U/G])\rightarrow
\HH_{\mathrm{c}}^j([X/G])\rightarrow
\HH_{\mathrm{c}}^j([Z/G])\rightarrow
\HH_{\mathrm{c}}^{j+1}([U/G]) \rightarrow \ldots
\] 

Moreover, if both $\HH_{\mathrm{c}}^{\bullet}([U/G])$ and $\HH_{\mathrm{c}}^{\bullet}([Z/G])$
are pure, then $\HH_{\mathrm{c}}^{\bullet}([X/G])$ is also pure and
$\HH_{\mathrm{c}}^{\bullet}([X/G])\simeq\HH_{\mathrm{c}}^{\bullet}([U/G])\oplus\HH_{\mathrm{c}}^{\bullet}([Z/G])$.

\end{lem}

\begin{proof}

The following long exact sequence is a consequence of the adjunction
triangle for mixed Hodge modules (see Proposition \ref{Prop/FormulasD_c^b}):
\[
\ldots\rightarrow
\HH_{\mathrm{c}}^{j-1}(Z\times^GU_{\rr,N,\alpha})\rightarrow
\HH_{\mathrm{c}}^j(U\times^GU_{\rr,N,\alpha})\rightarrow
\HH_{\mathrm{c}}^j(X\times^GU_{\rr,N,\alpha})\rightarrow
\HH_{\mathrm{c}}^j(Z\times^GU_{\rr,N,\alpha})\rightarrow
\HH_{\mathrm{c}}^{j+1}(U\times^GU_{\rr,N,\alpha}) \rightarrow \ldots
\] We obtain the long exact sequence in equivariant cohomology by taking
$N$ large enough for each cohomological step.

If $\HH_{\mathrm{c}}^{\bullet}([U/G])$ and $\HH_{\mathrm{c}}^{\bullet}([Z/G])$
are pure, then for all $j\in\ZZ$, $\HH_{\mathrm{c}}^{j}([U/G])$
and $\HH_{\mathrm{c}}^{j}([Z/G])$ are pure of weight $j$. This implies
that the connecting morphisms of the long exact sequence vanish and
the short exact sequences obtained in each cohomological degree split
(see Proposition \ref{Prop/WeightDecomp}). Therefore, $\HH_{\mathrm{c}}^{j}([X/G])\simeq\HH_{\mathrm{c}}^{j}([U/G])\oplus\HH_{\mathrm{c}}^{j}([Z/G])$
is pure of weight $j$, which proves the claim. \end{proof}

\paragraph*{Characteristic cycles}

Finally, we collect a few basic facts on characteristic cycles of
constructible complexes on a smooth complex algebraic variety $X$.
These are certain lagrangian cycles in $\mathrm{T}^{*}X$, which are
also related to constructible functions. We give a rough idea of their
construction and state the formal properties that we will use to build
a Hall algebra in a specific example (see Chapter \ref{Chap/HallAlg}).
We refer to \cite[\S 4.3.]{Dim04} for an overview.

Let $X$ be a smooth complex algebraic variety of dimension $n$.
In order to state the results below more simply, we work with complexes
which are constructible with respect to a fixed stratification $(X_{i})_{i\in I}$.
We require that the stratification be Whitney-regular (i.e. it satisfies
Whitney's condition B, see \cite[Def. 1.6.]{Dim92}). In practice,
we will work with finite Whitney stratifications and use the following
criterion to obtain them:

\begin{prop}{\cite[Prop. 1.14.]{Dim92}} \label{Prop/WhitneyStrat}

Suppose that $G$ is a complex algebraic group acting on a smooth
complex algebraic variety $X$ with finitely many orbits $(X_{i})_{i\in I}$.
Then $(X_{i})_{i\in I}$ form a Whitney-regular stratification of
$X$.

\end{prop}

There are two notions of characteristic cycles: characteristic cycles
of D-modules and characteristic cycles of constructible complexes.
The two notions are compatible through the Riemann-Hilbert correspondence
\cite[Prop. 5.3.2]{Dim04}.

Every coherent holonomic D-module $M$ on $X$ admits a filtration
which is compatible with the order filtration on $\mathcal{D}_{X}$,
the sheaf of differential operators on $X$. Then the characteristic
cycle of $M$ is the weighted support of $\mathrm{gr}(M)$ as a coherent
$\mathrm{gr}(\mathcal{D}_{X})\simeq\mathcal{O}_{\mathrm{T}^{*}X}$-module:\[
\mathrm{CC}_{\mathrm{D-mod}}(M):=\sum_{Z\in\supp(\mathrm{gr}(M))}\mathrm{mult}(Z)\cdot [Z].
\]Holonomicity guarantees that irreducible components of $\mathrm{CC}_{\mathrm{D-mod}}(M)$
all have dimension $n$. We refer to \cite[\S 2.2.]{HTT08} for details.

On the constructible side, consider a complex $K^{\bullet}$ on $X$,
which is constructible with respect to a finite Whitney stratification
$\mathcal{S}=(X_{i})_{i\in I}$. Then the characteristic cycle of
$K^{\bullet}$ is supported on $\bigsqcup_{i\in I}\overline{\mathrm{T}_{X_{i}}^{*}X}$
and measures, roughly speaking, the extent to which the complex $K^{\bullet}$
fails to be a local system:\[
\mathrm{CC}_{\mathrm{cons}}(K^{\bullet}):=\sum_{i\in I}m_i(K^{\bullet})\cdot [\overline{\mathrm{T}_{X_{i}}^{*}X}].
\] Note that $\overline{\mathrm{T}_{X_{i}}^{*}X}$ also has dimension
$n$. The construction of $\mathrm{CC}_{\mathrm{cons}}(K^{\bullet})$
is delicate and involves microlocalisation techniques. We therefore
refer to \cite[Ch. IX]{KS90} for a detailed construction. \index[terms]{characteristic cycles}\index[notations]{c@$\mathrm{CC}$ - characteristic cycles}

Let us write $\mathrm{CC}=\mathrm{CC}_{\mathrm{cons}}$ for short.
Given a finite Whitney stratification $\mathcal{S}=(X_{i})_{i\in I}$,
we denote by $D_{\mathrm{c},\mathcal{S}}^{b}(X(\CC))$ the triangulated
subcategory of $D_{\mathrm{c}}^{b}(X(\CC))$ formed by complexes which
are constructible along $\mathcal{S}$ and $K_{0}(D_{\mathrm{c},\mathcal{S}}^{b}(X(\CC)))$
its Grothendieck group, tensored with $\QQ$. We also define the group
of lagrangian cycles associated to $\mathcal{S}$ as $\mathrm{Lag}_{\mathcal{S}}(X):=\bigoplus_{i\in I}\QQ\cdot[\overline{\mathrm{T}_{X_{i}}^{*}X}]$.
Since characteristic cycles add up along distinguished triangles in
$D_{\mathrm{c},\mathcal{S}}^{b}(X(\CC))$ \cite[Cor. 4.3.22.]{Dim04},
we obtain a group homomorphism $\mathrm{CC}:K_{0}(D_{\mathrm{c},\mathcal{S}}^{b}(X(\CC)))\rightarrow\mathrm{Lag}_{\mathcal{S}}(X)$.
Characteristic cycles admit further compatibilities with constructible
functions. Let us call $C_{\mathcal{S}}(X)$ the group of $\QQ$-valued
constructible functions which are constant along strata in $\mathcal{S}$.
Then the assignment $K^{\bullet}\mapsto\chi_{K^{\bullet}}$ yields
a group homomorphism $\chi:K_{0}(D_{\mathrm{c},\mathcal{S}}^{b}(X(\CC)))\rightarrow C_{\mathcal{S}}(X)$
\cite[\S 4.1.]{Dim04}. The compatibility between characteristic cycles
and constructible functions goes through MacPherson's Euler obstruction
\cite[\S 3.]{Mac74}, which gives an isomorphism $\mathrm{Eu}\circ\pi:\mathrm{Lag}_{\mathcal{S}}(X)\rightarrow C_{\mathcal{S}}(X)$,
$[\overline{\mathrm{T}_{X_{i}}^{*}X}]\mapsto(-1)^{\dim(X_{i})}\cdot\mathrm{Eu}(\overline{X_{i}})$.

\begin{prop}{\cite[Thm. 4.3.25.]{Dim04}}

Let $X$ be a smooth complex algebraic variety and $\mathcal{S}$
a finite Whitney stratification of $X$. Then the following diagram
of $\QQ$-vector spaces commutes:\[
\begin{tikzcd}[ampersand replacement=\&]
K_{0}(D_{\mathrm{c},\mathcal{S}}^{b}(X(\CC))) \ar[d,"\chi"]\ar[r,"\mathrm{CC}"] \& \mathrm{Lag}_{\mathcal{S}}(X) \ar[dl,"\mathrm{Eu}\circ\pi", rotate=-45,"\sim"'] \\
C_{\mathcal{S}}(X) \& 
\end{tikzcd}
\]\end{prop}

Then the compatibility between $\mathrm{CC}_{\mathrm{D-mod}}$ and
$\mathrm{CC}_{\mathrm{cons}}$ follows from Kashiwara's local index
theorem \cite[Thm. 4.6.7.]{HTT08}. Since the group homomorphism $\chi$
is surjective, we will sometimes abuse notations and denote by $\mathrm{CC}$
the isomorphism $(\mathrm{Eu}\circ\pi)^{-1}$ (see \cite[Rmk. 4.3.27.]{Dim04}).

\begin{rmk}

Characteristic cycles also enjoy nice functorial properties with respect
to proper pushforwards and smooth pullbacks (see \cite[Ch. IX.]{KS90}\cite{SV96}).
This is crucially used in \cite{Hen24} to relate several Hall algebras
and inspires our construction in Chapter \ref{Chap/HallAlg}. We will
only use compatibility of characteristic cycles with open immersions
in the thesis (see below). For more details on functoriality of characteristic
cycles, we refer to loc. cit.

\end{rmk}

We finish this section with two computational facts, which will be
of use in Chapter \ref{Chap/HallAlg}.

\begin{prop}{\cite[Prop. 9.4.3.]{KS90} \cite[Eqn. (9.4.8)]{KS90}}
\label{Prop/CCcomp}

Let $X$ be a smooth complex algebraic variety and $\mathcal{S}$
a finite Whitney stratification of $X$. Let $j:U\hookrightarrow X$
be an open union of strata. Consider the restriction morphism $\bullet\vert_{U}:\mathrm{Lag}_{\mathcal{S}}(X)\rightarrow\mathrm{Lag}_{\mathcal{S}\vert_{U}}(U)$
defined by:\[
[\overline{\mathrm{T}_{X_{i}}^{*}X}]\vert_{U}:=
\left\{
\begin{array}{ll}
[\overline{\mathrm{T}_{X_{i}}^{*}U}] & \text{if } X_i\subseteq U, \\
0 & \text{else.}
\end{array}
\right.
\] Then for any $K^{\bullet}\in D_{\mathrm{c},\mathcal{S}}^{b}(X(\CC))$,
$\mathrm{CC}(j^{*}K^{\bullet})=\mathrm{CC}(K^{\bullet})\vert_{U}$.

Moreover, if $i:Y\hookrightarrow X$ is a union of strata and a closed,
smooth subvariety, then:\[
\mathrm{CC}(i_*\underline{\QQ}_Y)=[\overline{\mathrm{T}_{Y}^{*}X}].
\]

\end{prop}

\pagebreak{}

\section{Kac polynomials with multiplicities and jets of quiver moment maps
\label{Chap/KacPolynomials}}

In this Chapter, we establish a plethystic formula relating the counts
$\vol(\mathfrak{M}_{\Pi_{(Q,\nn)},\rr})=\frac{\sharp\mu_{(Q,\nn),\rr}^{-1}(0)}{\sharp\GL_{\nn,\rr}}$
and $A_{(Q,\nn),\rr}$, for a quiver with multiplicities $(Q,\nn)$.
Recall that $A_{(Q,\nn),\rr}$ is the count of absolutely indecomposable,
locally free representations of $(Q,\nn)$ with rank vector $\rr$.
A relation between these counts was conjectured by Wyss when $\nn=(\alpha)_{i\in Q_{0}}$,
$\rr=\underline{1}$ and $\alpha$ goes to infinity \cite[Conj. 4.37.]{Wys17b}.
When $Q$ is a quiver without multiplicities, our result recovers
Mozgovoy's plethystic formula \cite{Moz11a} for Kac polynomials.
With the foundations from Sections \ref{Subsect/QuivRep} and \ref{Subsect/MomMap}
at hand, our formula is a straightforward generalisation of Mozgovoy's,
with the same proof.

\begin{thm} \label{Thm/Ch2ExpFmlKacPol}

Let $(Q,\nn)$ be a quiver with multiplicities. Call $\langle\bullet,\bullet\rangle$
the Euler form of the associated Borcherds-Cartan matrix with symmetriser
and orientation. Then: \[
\sum_{\rr\in\ZZ_{\geq0}^{Q_0}}
\frac{\sharp\mu_{(Q,\nn),\rr}^{-1}(0)}{\sharp\GL_{\nn,\rr}}
\cdot q^{\langle\rr,\rr\rangle}t^{\rr}
=
\Exp_{q,t}\left(
\sum_{\rr\in\ZZ_{\geq0}^{Q_0}\setminus\{0\}}
\frac{A_{(Q,\nn),\rr}}{1-q^{-1}}\cdot t^{\rr}
\right).
\]

\end{thm}

As an application, we prove two conjectures of Wyss' \cite[Conj. 4.32-4.37.]{Wys17b}
on the limit of both counts, when $\nn=(\alpha)_{i\in Q_{0}}$, $\rr=\underline{1}$
and $\alpha$ goes to infinity. Let us denote by $\langle\bullet,\bullet\rangle$
the Euler form of $Q$ (without multiplicities). Wyss showed that
the following two sequences converge if, and only if, $Q$ is 2-connected
(see Section \ref{Subsect/WyssConj} for the definition) and that
their limits are rational fractions in $q$:\begin{align*}
& A_{Q}(q):=
\underset{\alpha\rightarrow +\infty}{\lim}
\left(q^{-\alpha (1-\langle\rr,\rr\rangle)}\cdot A_{(Q,\alpha),\underline{1}}(q)\right), \\
& B_{\mu_{Q}}(q):=
\underset{\alpha\rightarrow +\infty}{\lim}
\left(q^{-\alpha(2\sharp Q_1-\sharp Q_0+1)}\cdot\sharp\mu_{(Q,\alpha),\underline{1}}^{-1}(0)(\FF_q)\right).
\end{align*} He further conjectured a simple relation between $A_{Q}$ and $B_{\mu_{Q}}$
\cite[Conj. 4.37.]{Wys17b} and that the numerator of $B_{\mu_{Q}}$
has non-negative coefficients \cite[Conj. 4.32.]{Wys17b}. From Theorem
\ref{Thm/Ch2ExpFmlKacPol}, we deduce:

\begin{cor} \label{Cor/Ch2WyssConjAvsB}

Let $Q$ be a 2-connected quiver. Then: \[
\frac{B_{\mu_{Q}}}{(1-q^{-1})^{\sharp Q_0}}=\frac{A_{Q}}{1-q^{-1}}.
\]

\end{cor}

Moreover, using Wyss' explicit formula for $A_{Q}$ \cite[Cor. 4.35.]{Wys17b}
and Hilbert series techniques, we prove:

\begin{thm} \label{Thm/Ch2WyssConjPositivity}

Let $Q$ be a 2-connected quiver. Consider the Stanley-Reisner ring
$\QQ[\Delta]$ associated to the order complex $\Delta$ of the poset
$(\Pi(Q_{1})\setminus\{\emptyset,Q_{1}\},\subseteq)$. Then:

\[
A_Q=\frac{(1-q^{-1})^{b(Q)}}{1-q^{-b(Q)}}\cdot\Hilb_{\Delta}\left(u_{E}=q^{-(b(Q)-b(Q\vert_{E}))}\right).
\] and $\Hilb_{\Delta}\left(u_{E}=q^{-(b(Q)-b(Q\vert_{E}))}\right)$
can be presented as a rational fraction whose numerator has non-negative
coefficients.

\end{thm}

Finally, Theorem \ref{Thm/Ch2ExpFmlKacPol} yields a new method to
compute $A_{(Q,\alpha),\rr}$, where $(Q,\alpha):=(Q,\nn)$ and $\nn=(\alpha)_{i\in Q_{0}}$.
Indeed, the datum of all $A_{(Q,\alpha),\rr}$ ($\alpha\geq1$, $\rr\in\ZZ_{\geq0}^{Q_{0}}$)
is equivalent to the local zeta functions $Z_{\mu_{Q,\rr}}$ for all
$\rr\in\ZZ_{\geq0}^{Q_{0}}$. This is interesting as the counts $A_{(Q,\alpha),\rr}$
are not known to be polynomials in $q$ for $\alpha>1$ and $\rr>\underline{1}$.
The counting method used by Kac, Stanley and Hua \cite{Kac83,Hua00},
which relies on Burnside's lemma, shows strong limitations, as the
conjugacy classes of $\GL_{n,r}$ remain unknown for $r>3$. As an
alternative, we prove the following formula for $Z_{\mu_{Q,\rr}}$
and use change of variables to do further computations for quivers
with multiplicities.

\begin{prop} \label{Prop/Ch2ZetaMomMap}

Let $Q$ be a quiver and $\dd\in\ZZ_{\geq0}^{Q_{0}}$ a dimension
vector.\[
Z_{\mu_{Q,\dd}}(s) = 
1-\sum_{i=0}^{r}\frac{q^{-i}\cdot(1-q^{-s})}{1-q^{-(s+i)}}\cdot \int_{R(Q,\dd)(\mathcal{O})} \left(\left(\frac{\Vert\Delta_{i}(x)\Vert}{\Vert\Delta_{i-1}(x)\Vert}\right)^{s+i}-\left(\frac{\Vert\Delta_{i+1}(x)\Vert}{\Vert\Delta_{i}(x)\Vert}\right)^{s+i} \right)\cdot\frac{dx}{\Vert\Delta_{i}(x)\Vert},
\]where $\Delta_{i}(x)$ is the vector of all $i\times i$ minors of
the matrix $\mu_{Q,\dd}(x,\bullet)$. We also fix the conventions
$\Vert\Delta_{-1}(x)\Vert=\Vert\Delta_{0}(x)\Vert=1$ and $\Vert\Delta_{r+1}(x)\Vert=0$,
where $r:=\underset{x\in R(Q,\dd)(\mathcal{O})}{\max}\{\rk(\mu_{Q,\dd}(x,\bullet))\}$.

\end{prop}

Throughout this chapter, we will use the following simplified notations.
Given a base field $\KK$ and $\alpha\geq1$, we define $\mathcal{O}_{\alpha}:=\KK[t]/(t^{\alpha})$,
as in \cite{Wys17b}. We denote by $(Q,\alpha)$ the quiver with multiplicities
$(Q,\nn)$, where $\nn=(\alpha)_{i\in Q_{0}}$. We also write $\GL_{\alpha,\rr}:=\GL_{\nn,\rr}$,
$R(Q,\alpha;\rr):=R(Q,\nn;\rr)$ and $\mu_{(Q,\alpha),\rr}:=\mu_{(Q,\nn),\rr}$.
From Section \ref{Subsect/WyssConj} on, we call $\langle\bullet,\bullet\rangle$
the Euler form of $Q$ (without multiplicities). Note that $\langle\bullet,\bullet\rangle_{H}=\alpha\cdot\langle\bullet,\bullet\rangle$,
where $H=H(C,D,\Omega)$ is the path algebra of the Cartan matrix
(with symmetriser and orientation) associated to $(Q,\alpha)$. \index[notations]{o@$\mathcal{O}_{\alpha}$ - ring of truncated power series}

We prove Theorem \ref{Thm/Ch2ExpFmlKacPol} in Section \ref{Subsect/PlethFml}.
In Section \ref{Subsect/SRrings} we gather some results on Stanley-Reisner
rings and their Hilbert series, which we apply in Section \ref{Subsect/WyssConj}
to prove Wyss' conjectures. Finally, we carry out computations of
$A_{(Q,\alpha),\rr}$ in Section \ref{Subsect/KacPol_p-adic}, using
in turn Burnside's lemma and Proposition \ref{Prop/Ch2ZetaMomMap}.
Theorem \ref{Thm/Ch2ExpFmlKacPol} (in the case where $(Q,\nn)=(Q,\alpha)$),
Corollary \ref{Cor/Ch2WyssConjAvsB} and Theorem \ref{Thm/Ch2WyssConjPositivity}
all appear in the preprint \cite{Ver23}.

\subsection{A plethystic formula \label{Subsect/PlethFml}}

In this Section, we briefly prove Theorem \ref{Thm/Ch2ExpFmlKacPol}.
We work over a finite field $\KK=\FF_{q}$. Let us recall the statement:

\begin{thm} \label{Thm/ExpFmlKacPol}

Let $(Q,\nn)$ be a quiver with multiplicities. Call $\langle\bullet,\bullet\rangle$
the Euler form of the associated Borcherds-Cartan matrix with symmetriser
and orientation. Then: \[
\sum_{\rr\in\ZZ_{\geq0}^{Q_0}}
\frac{\sharp\mu_{(Q,\nn),\rr}^{-1}(0)}{\sharp\GL_{\nn,\rr}}
\cdot q^{\langle\rr,\rr\rangle}t^{\rr}
=
\Exp_{q,t}\left(
\sum_{\rr\in\ZZ_{\geq0}^{Q_0}\setminus\{0\}}
\frac{A_{(Q,\nn),\rr}}{1-q^{-1}}\cdot t^{\rr}
\right).
\]

\end{thm}

\begin{proof}

By Proposition \ref{Prop/MomMapExSeq}, for a given $x\in R(Q,\nn;\rr)$
corresponding to a locally free representation $M$, we have $\sharp\{y\in R(Q^{\mathrm{op}},\nn;\rr))\ \vert\ \mu_{(Q,\nn),\rr}(x,y)=0\}=\sharp\Ext^{1}(M,M)$.
Summing over all isomorphism classes $[M]$ of locally free representations
in rank $\rr$, we obtain:\[
\frac{\sharp\mu_{(Q,\nn),\rr}^{-1}(0)}{\sharp\GL_{\nn,\rr}}
=
\sum_{[M]}\frac{\sharp\Ext^{1}(M,M)}{\sharp\Aut(M)}.
\] Using Proposition \ref{Prop/MomMapExSeq} again, we obtain that $\sharp\Ext^{1}(M,M)=q^{-\langle\rr,\rr\rangle}\cdot\sharp\End(M)$
and so:\[
q^{\langle\rr,\rr\rangle}\cdot\frac{\sharp\mu_{(Q,\nn),\rr}^{-1}(0)}{\sharp\GL_{\nn,\rr}}
=
\sum_{[M]}\frac{\sharp\End(M)}{\sharp\Aut(M)}=\vol (\mathrm{Pairs}(\mathrm{rep}_{\KK}^{\mathrm{l.f.}}(Q,\nn;\rr))).
\] Thus the formula follows from Proposition \ref{Prop/VolStackPairs}.
\end{proof}

An immediate consequence of Theorem \ref{Thm/ExpFmlKacPol} is that
$A_{(Q,\nn),\rr}$ does not depend on the orientation of $Q$. When
$(Q,\nn)=(Q,\alpha)$, this was already proved in \cite[Thm. 1.1.]{HLRV24}
using Fourier transform.

\begin{cor}

Let $(Q,\nn)$ be a quiver with multiplicities and $\rr\in\ZZ_{\geq0}^{Q_{0}}$.
Then $A_{(Q,\nn),\rr}$ does not depend on the orientation of $Q$.

\end{cor}

\subsection{Stanley-Reisner rings and their Hilbert series \label{Subsect/SRrings}}

In this section, we recall some results on Stanley-Reisner rings of
simplicial complexes and their Hilbert series. We focus on shellable
and Cohen-Macaulay simplicial complexes, as their Hilbert series enjoy
positivity properties that we exploit in our proof of Theorem \ref{Thm/Ch2WyssConjPositivity}.
More specifically, we state shellability results by Björner on order
complexes of posets \cite{Bjo80}. We refer to \cite{Sta07} for more
details on Hilbert series of simplicial complexes.

Let us first define simplicial complexes and their Stanley-Reisner
rings.

\begin{df}[Simplicial complex]

A(n abstract) simplicial complex $\Delta$ on a set of vertices $V=\{x_{1},\ldots,x_{n}\}$
is a collection of subsets of $V$ (called faces of $\Delta$) such
that: 
\begin{enumerate}
\item for all $x\in V$, $\{x\}\in\Delta$;
\item for any $F\in\Delta$, $G\subseteq F\Rightarrow G\in\Delta$.
\end{enumerate}
The dimension of a face $F\in\Delta$ is $\sharp F-1$ and faces of
maximal dimension are called facets.

\end{df}

\begin{df}[Stanley-Reisner ring]

Given a field $\KK$, the Stanley-Reisner ring $\KK[\Delta]$ associated
to $\Delta$ is:\[
\KK[x_1,\ldots,x_n]/I_{\Delta},
\] where $I_{\Delta}$ is the ideal generated by $x_{i_{1}}\ldots x_{i_{r}}$,
for $i_{1}<\ldots<i_{r}$ and $\{i_{1},\ldots,i_{r}\}\notin\Delta$.

\end{df}

Note that a $\KK$-basis of $\KK[\Delta]$ is given by monomials $x^{\mm}=x_{1}^{m_{1}}\ldots x_{n}^{m_{n}}$
whose support (i.e. $\supp(x^{\mm}):=\{1\leq i\leq n\ \vert\ m_{i}\ne0\}$)
is a face of $\Delta$.

Given a $\ZZ_{\geq0}$-grading of $\KK[\Delta]$, one can consider
its Hilbert series. We will consider $\ZZ_{\geq0}$-gradings for which
the Hilbert series is a specialisation of the fine Hilbert series
$\Hilb_{\Delta}\in\KK[[u_{1},\ldots,u_{n}]]$, associated to the natural
$\ZZ_{\geq0}^{n}$-grading of $\KK[\Delta]$ given by $\deg(x_{1}^{m_{1}},\ldots x_{n}^{m_{n}})=(m_{1},\ldots,m_{n})$. 

\begin{df}[Fine Hilbert series]

The fine Hilbert series is defined as:\[
\Hilb_{\Delta}(u):=\sum_{\mm\in\ZZ_{\geq0}^{n}}\dim_{\KK}\KK[\Delta]_{\mm}\cdot u^{\mm}
=\sum_{F\in\Delta}\prod_{Substack{1\leq i\leq n \\ x_i\in F}}\frac{u_i}{1-u_i}.
\]

\end{df}

By \cite[Ch. I, Thm. 2.3.]{Sta07}, there exists $\mathbf{n}\in\ZZ_{\geq0}^{n}$
and $P(u)\in\ZZ[u_{1},\ldots,u_{n}]$ such that:\[
\Hilb_{\Delta}(u)=u^{\mathbf{n}}\cdot\frac{P(u)}{\prod_{i=1}^n(1-u_i)}.
\]

We now discuss some properties of simplicial complexes which imply
that, when specialised according to any $\ZZ_{\geq0}$-grading of
$\KK[\Delta]$, $\Hilb_{\Delta}$ admits a numerator with non-negative
coefficients. A simplicial complex $\Delta$ is called Cohen-Macaulay
if the graded ring $\KK[\Delta]$ is Cohen-Macaulay (see \cite[Ch. I, \S5]{Sta07}
for background). Consider the $\ZZ_{\geq0}$-grading of $\KK[\Delta]$
given by $\deg(x_{i})=d_{i}$. Then the Cohen-Macaulay property implies
that the specialised Hilbert series can be presented as a rational
fraction whose numerator has non-negative coefficients:

\begin{prop}{\cite[Ch. I, \S 5.2. - Thm. 5.10.]{Sta07}} \label{Prop/CMHilb}

In the above setting, if $\Delta$ is Cohen-Macaulay, then there exist
$Q(q)\in\ZZ_{\geq0}[q]$ and $e_{1},\ldots,e_{s}\geq1$ such that:\[
\Hilb_{\Delta}(q^{d_1},\ldots,q^{d_n})=\frac{Q(q)}{\prod_{i=1}^s(1-q^{e_i})}.
\]

\end{prop}

\begin{rmk}

Note that, in Proposition \ref{Prop/CMHilb}, $Q(q)$ may not be a
specialisation of $P(u)$. Indeed, such a presentation of $\Hilb_{\Delta}$
as a rational fraction depends on the choice of some elements in $\KK[\Delta]$.
In the first presentation (with numerator $P(u)$), we choose generators
$x_{i},\ 1\leq i\leq n$. In the second presentation (with numerator
$Q(q)$), we consider a regular sequence of homogeneous elements (of
degrees $e_{i},\ 1\leq i\leq s$) in $\KK[\Delta]$. It may be the
case that generators $x_{i},\ 1\leq i\leq n$ do not form a regular
sequence (the depth of $\KK[\Delta]$ may even be smaller than $n$).
See \cite[Ch. I, \S5]{Sta07} for the relation between $Q(q)$ and
the choice of a regular sequence.

\end{rmk}

While the Cohen-Macaulay condition has an abstract algebraic formulation
in terms of $\KK[\Delta]$, there is a combinatorial condition on
$\Delta$ itself which implies that $\Delta$ is Cohen-Macaulay.

\begin{df}[Shellability]

A simplicial complex $\Delta$ is called shellable if:
\begin{enumerate}
\item $\Delta$ is pure i.e. its facets all have the same dimension $d$;
\item Facets of $\Delta$ can be ordered (let us call them $F_{1},\ldots,F_{s}$)
in such a way that for $2\leq i\leq s$, $F_{i}\cap\left(\bigcup_{j=1}^{i-1}F_{j}\right)$
is a nonempty union of faces of dimension $d-1$.
\end{enumerate}
The sequence of facets $F_{1},\ldots,F_{s}$ is called a shelling
of $\Delta$.

\end{df}

\begin{prop}{\cite[Ch. III, Thm. 2.5.]{Sta07}} \label{Prop/ShellCM}

If $\Delta$ is shellable, then it is Cohen-Macaulay.

\end{prop}

We will work with a class of simplicial complexes with well-studied
shellings.

\begin{df}[Order complex]

Let $\Pi$ be a poset. The order complex of $\Pi$ is the simplicial
complex $\mathcal{O}(\Pi)$ with set of vertices $\Pi$ and whose
faces are chains i.e. totally ordered subsets of $\Pi$.

\end{df}

In \cite{Bjo80}, Björner studies shellings of order complexes through
the stronger notion of lexicographic shellability. We recall some
of his results on order complexes of lattices. Recall that a lattice
is a poset where any pair $x,y\in\Pi$ admits a greatest lower bound
$x\wedge y$ and a lowest upper bound $x\vee y$ (see \cite[Ch. I.]{Bir73}
for background). A lattice is called modular if, for all $x,y,z\in\Pi$
such that $x\leq z$, $x\vee(y\wedge z)=(x\wedge y)\vee z$. A lattice
is called graded if there exists a strictly increasing function $\rho:\Pi\rightarrow\ZZ_{\geq0}$
such that for $x,y\in\Pi$, we have: $y$ covers\footnote{In a poset $\Pi$, $y$ is said to cover $x$ if $x\leq y$ and there
exist no $z\in\Pi$ such that $x<z<y$.} $x\ \Rightarrow\ \rho(y)=\rho(x)+1$. A modular lattice is graded
(see \cite[Ch II, \S 8]{Bir73}).

\begin{exmp} \label{Exmp/OrderCpx}

Given a set $S$, let us call $\Pi(S)$ the poset of subsets of $S$,
ordered by inclusion. The poset $\Pi(S)$ is a modular lattice. Indeed,
for $A,B\in\Pi(S)$, we have $A\wedge B=A\cap B$ and $A\vee B=A\cup B$;
moreover, if $A\subseteq C$, then $A\cup(B\cap C)=(A\cup B)\cap C$.
Note that $\Pi(S)$ is graded by cardinality $\rho:A\mapsto\sharp A$.

\end{exmp}

We collect the following results from \cite{Bjo80}:

\begin{prop} \label{Prop/LattShell}

Let $L$ be a finite lattice. If $L$ is graded by $\rho$, define
$L_{S}:=\{x\in L\ \vert\ \rho(x)\in S\}$, for $S\subset\ZZ_{\geq0}$.
Then the following hold:
\begin{enumerate}
\item \cite[Thm. 3.1.]{Bjo80} If $L$ is modular, then $\mathcal{O}(L)$
is shellable.
\item \cite[Thm. 4.1.]{Bjo80} If $\mathcal{O}(L)$ is shellable, then $\mathcal{O}(L_{S})$
is shellable for any $S\subset\ZZ_{\geq0}$.
\end{enumerate}
\end{prop}

\subsection{Proof of Wyss' conjectures \label{Subsect/WyssConj}}

In this Section, we prove Corollary \ref{Cor/Ch2WyssConjAvsB} and
Theorem \ref{Thm/Ch2WyssConjPositivity}. We first set up some notations
and collect basic facts on graphs. We then apply the results from
Sections \ref{Subsect/PlethFml} and \ref{Subsect/SRrings} to prove
Wyss' conjectures.

\paragraph*{Graph-theoretic setup}

Let $Q$ be a quiver. We define subquivers of $Q$ obtained by restricting
to a subset of vertices or arrows of $Q$.

\begin{df}\index[notations]{q@$Q\vert_I$ - restriction of a quiver}

Let $I\subseteq Q_{0}$ be a subset of vertices and $J\subseteq Q_{1}$
be a subset of arrows of $Q$.
\begin{enumerate}
\item $Q\vert_{I}$ is the quiver with set of vertices $I$ and set of arrows
$Q_{1,I}:=\{a\in Q_{1}\ \vert\ s(a),t(a)\in I\}$. The source and
target maps of $Q\vert_{I}$ are obtained from those of $Q$ by restriction.
\item $Q\vert_{J}$ is the quiver with set of vertices $Q_{0}$ and set
of arrows $J$. The source and target maps of $Q\vert_{J}$ are obtained
from those of $Q$ by restriction.
\end{enumerate}
\end{df}

We will also need a basic fact on the Betti number of a graph. Given
a graph $\Gamma$ (not necessarily oriented) with $C$ connected components,
$V$ vertices and $E$ edges, the Betti number of $\Gamma$ is $b(\Gamma):=C-V+E$.
This should be thought of as the number of independent cycles of $\Gamma$.
Given a quiver $Q$, we define its Betti number $b(Q)$ as the Betti
number of the underlying graph. We also denote by $c(Q)$ the number
of connected components of $Q$. Recall that a quiver $Q$ (or its
underlying graph $\Gamma$) is said to be 2-connected if $\Gamma$
is connected and removing any edge from $\Gamma$ does not disconnect
it. \index[term]{Betti number of a graph}\index[terms]{2-connected quiver}\index[notations]{b@$b(Q)$ - Betti number}\index[notations]{c@$c(Q)$ - number of connected components}

\begin{prop} \label{Prop/BettiNb}

Let $Q$ be a quiver. The following hold:
\begin{enumerate}
\item If $Q$ is 2-connected, then $b(Q)>0$.
\item Let $Q_{0}=I_{1}\sqcup\ldots\sqcup I_{s}$ be a partition of the set
of vertices and $Q'$ the quiver obtained from $Q$ by contracting
the arrows of subquivers $Q\vert_{I_{1}},\ldots,Q\vert_{I_{s}}$.
Then $b(Q)=b(Q')+b(Q\vert_{I_{1}})+\ldots+b(Q\vert_{I_{s}})$.
\item $Q$ is 2-connected if, and only if, for all partitions $Q_{0}=I_{1}\sqcup\ldots\sqcup I_{s}$,
we have $b(Q)>b(Q\vert_{I_{1}})+\ldots+b(Q\vert_{I_{s}})$.
\end{enumerate}
\end{prop}

\begin{proof}

1. By contraposition (and leaving out the trivial case where $Q$
is not connected), suppose that $b(Q)=0$ and $Q$ is connected. Then
$\sharp Q_{1}=\sharp Q_{0}-1$, which is the minimum number of arrows
required to get a connected quiver with $\sharp Q_{0}$ vertices.
Hence removing one arrow disconnects $Q$ i.e. it is not 2-connected.

2. This follows from a straightforward computation. Suppose that $Q$
(resp. $Q\vert_{I_{k}}$) has $C$ (resp. $C_{k}$) connected components,
$V$ (resp. $V_{k}$) vertices and $A$ (resp. $A_{k}$) arrows. Then
$Q'$ has $C$ connected components, $C_{1}+\ldots+C_{s}$ vertices
and $A-A_{1}-\ldots-A_{s}$ arrows.

3. Suppose that $Q$ is 2-connected. Let $Q_{0}=I_{1}\sqcup\ldots\sqcup I_{s}$
be a partition of the set of vertices and $Q'$ be the quiver obtained
from $Q$ by contracting the arrows of subquivers $Q\vert_{I_{1}},\ldots,Q\vert_{I_{s}}$.
Then $Q'$ is also 2-connected and the wanted inequality follows by
1. and 2. On the other hand, if $Q$ is not 2-connected, consider
an arrow $a\in Q_{1}$ which disconnects $Q$ when removed. Then $Q\vert_{Q_{1}\setminus\{a\}}$
splits into two connected components with sets of vertices $I_{1},I_{2}$
and one can check that $b(Q)=b(Q\vert_{I_{1}})+b(Q\vert_{I_{2}})$.
\end{proof}

\paragraph*{Proof of Corollary \ref{Cor/Ch2WyssConjAvsB}}

We now prove Corollary \ref{Cor/Ch2WyssConjAvsB}. This is a direct
consequence of Theorem \ref{Thm/Ch2ExpFmlKacPol}. Let us recall the
statement.

\begin{cor} \label{Cor/WyssConjAvsB}

Let $Q$ be a 2-connected quiver. Then: \[
\frac{B_{\mu_{Q}}}{(1-q^{-1})^{\sharp Q_0}}=\frac{A_{Q}}{1-q^{-1}}.
\]

\end{cor}

\begin{proof}

Let us spell out Theorem \ref{Thm/ExpFmlKacPol} for $(Q,\nn)=(Q,\alpha)$
and $\rr=\underline{1}$. Recall that we denote by $\langle\bullet,\bullet\rangle$
the Euler form of $Q$ (without multiplicities). Then:\[
q^{-\alpha(b(Q)+\sharp Q_0-\langle\rr,\rr\rangle)}
\cdot
\frac{\sharp\mu_{(Q,\alpha),\underline{1}}^{-1}(0)}{\prod_{i\in Q_0}(1-q^{-1})}
=
q^{-\alpha b(Q)}\cdot\sum_{Q_0=I_1\sqcup\ldots\sqcup I_s}\prod_{j=1}^s\frac{A_{(Q,\alpha),\underline{1}\vert_{I_j}}}{1-q^{-1}}.
\] Note that $b(Q)+\sharp Q_{0}-\langle\rr,\rr\rangle=2\sharp Q_{1}-\sharp Q_{0}+1$.
Moreover, for $I\subseteq Q_{0}$, the count $A_{(Q,\alpha),\underline{1}_{I}}$
is a polynomial of degree $\alpha b(Q\vert_{I})$ (see \cite[Cor. 4.35.]{Wys17b}).
By Proposition \ref{Prop/BettiNb}, only the term corresponding to
$I_{1}=Q_{0}$ contributes to the limit in the right-hand side. Therefore
we conclude:\[
\frac{B_{\mu_{Q}}}{(1-q^{-1})^{\sharp Q_0}}=\frac{A_{Q}}{1-q^{-1}}.
\] \end{proof}

\begin{rmk}

Wyss' conjecture actually relates the numerators of $A_{Q}$ and $B_{\mu_{Q}}$.
However, the conjecture does not hold as stated. The following quiver
is a counter-example (regardless of orientation):
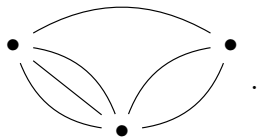
\[
\begin{tikzcd}[ampersand replacement=\&]
\bullet \ar[rr, dash, bend left] \& \& \bullet \ar[dl, dash, bend left]\ar[dl, dash, bend right] \\
\& \bullet \ar[ul, dash]\ar[ul, dash, bend left]\ar[ul, dash, bend right] \& 
\end{tikzcd}
.
\]

\end{rmk}

\paragraph*{Proof of Theorem \ref{Thm/Ch2WyssConjPositivity}}

We now turn to the proof of Theorem \ref{Thm/Ch2WyssConjPositivity}.
Recall the following explicit expression for $A_{Q}$ (see \cite[Cor. 4.35.]{Wys17b}):\[
A_Q=
(1-q^{-1})^{b(Q)}\cdot
\sum_{E_1\subsetneq E_2\ldots\subsetneq E_s=Q_1}\prod_{j=1}^{s-1}\frac{1}{q^{b(Q)-b(Q\vert_{E_j})}-1}.
\] The appearance of chains of subsets of $Q_{1}$ is reminiscent of
the order complex of $\Pi(Q_{1})$ and indeed, we show that $A_{Q}$
is essentially the Hilbert series of the Stanley Reisner ring associated
to $\Delta:=\mathcal{O}(\Pi(Q_{1}))$, suitably specialised. This
approach was inspired from \cite[Prop. 1.9.]{MV22}.

\begin{thm} \label{Thm/WyssConjPositivity}

Let $Q$ be a 2-connected quiver. Consider the Stanley-Reisner ring
$\QQ[\Delta]$ associated to the order complex $\Delta$ of the poset
$\Pi(Q_{1})\setminus\{\emptyset,Q_{1}\}$. Then:

\[
A_Q=\frac{(1-q^{-1})^{b(Q)}}{1-q^{-b(Q)}}\cdot\Hilb_{\Delta}\left(u_E=q^{-(b(Q)-b(Q\vert_E))}\right).
\] Moreover, $\Hilb_{\Delta}\left(u_{E}=q^{-(b(Q)-b(Q\restriction_{E}))}\right)$
can be written as a rational fraction whose numerator has non-negative
coefficients.

\end{thm}

\begin{proof}

The above formula for $A_{Q}$ \cite[Cor. 4.35.]{Wys17b} can be rewritten
as:\[
\begin{split}
A_Q & =
(1-q^{-1})^{b(Q)}\cdot\sum_{\emptyset\subsetneq E_1\subsetneq E_2\ldots\subsetneq E_{s-1}\subsetneq E_s=Q_1}\prod_{j=1}^{s-1}\frac{1}{q^{b(Q)-b(Q\vert_{E_j})}-1}\times\left\{ 1+\frac{1}{q^{b(Q)}-1}\right\} \\
& =
\frac{(1-q^{-1})^{b(Q)}}{1-q^{-b(Q)}}\cdot
\sum_{\emptyset\subsetneq E_1\subsetneq E_2\ldots\subsetneq E_{s-1}\subsetneq E_s=Q_1}\prod_{j=1}^{s-1}\frac{(q^{-1})^{b(Q)-b(Q\vert_{E_j})}}{1-(q^{-1})^{b(Q)-b(Q\vert_{E_j})}} \\
& = \frac{(1-q^{-1})^{b(Q)}}{1-q^{-b(Q)}}\cdot\Hilb_{\Delta}\left(u_E=q^{-(b(Q)-b(Q\vert_E))}\right).
\end{split}
\] Moreover, the order complex of $\Pi(Q_{1})\setminus\{\emptyset,Q_{1}\}$
is shellable, by Proposition \ref{Prop/LattShell} (see also Example
\ref{Exmp/OrderCpx}), so it is Cohen-Macaulay, by Proposition \ref{Prop/ShellCM}.
Therefore, up to substituting $q^{-1}$ for $q$, the second part
of the statement follows directly from Proposition \ref{Prop/CMHilb}.
\end{proof}

\begin{rmk}

Wyss' initial conjecture concerns the numerator of $B_{\mu_{Q}}$,
which is explicitly defined in \cite[\S 4.6.]{Wys17b}. Using Corollary
\ref{Cor/WyssConjAvsB} and Theorem \ref{Thm/WyssConjPositivity},
one can show that this conjecture is true. We leave out the detailed
computations to avoid introducing too many additional notations. As
the reader may have noticed, $A_{Q}$ can be related more directly
to the order complex of $\Pi(Q_{1})\setminus\{Q_{1}\}$ (instead of
$\Pi(Q_{1})\setminus\{\emptyset,Q_{1}\}$), which is also shellable.
We chose to work with $\Pi(Q_{1})\setminus\{\emptyset,Q_{1}\}$ in
order to extract a factor $(1-q^{-b(Q)})$ from the Hilbert series,
which turns out to be useful in the proof of \cite[Conj. 4.32.]{Wys17b}.

\end{rmk}

\subsection{Kac polynomials of quivers with multiplicities via p-adic integrals
\label{Subsect/KacPol_p-adic}}

In this Section, we compute $A_{(Q,\alpha),\rr}$ in some examples
(small quivers and rank vectors $\rr$). We first follow the strategy
implemented by Kac, Stanley and later Hua in \cite{Kac83,Hua00},
for ranks at most 3. Then, we propose another approach using Igusa's
local zeta functions $Z_{\mu_{Q,\dd}}$ and a change of variables
similar to Denef's in \cite{Den87}. This technique is not specific
to quiver moment maps and applies to so-called ASK zeta functions
\cite{Ros18,RV19,Ros20,CR22}. We run computations for a few other
quivers and small rank vectors with this p-adic method.

\paragraph*{Orbit-counting method}

In the first approach, we use Burnside's lemma to compute $A_{(Q,\alpha),\rr}$
\cite{Kac83,Hua00}. The computation turns out to be rather involved,
already for low ranks and simple quivers, as a good knowledge of conjugacy
classes of $\GL_{\alpha,\rr}$ is required. We exploit results of
Avni, Onn, Prasad, Vaserstein and Avni, Klopsch, Onn, Voll \cite{AOPV09,AKOV16}
to obtain closed formulas for $g$-loop quivers in ranks 2 and 3.
This involves solving a linear recurrence relation over $\alpha\geq1$.
Throughout, we assume that $\KK=\FF_{q}$.

Let us first apply Burnside's lemma. Consider $M_{(Q,\alpha),\rr}\in\mathcal{V}$
the count of all isomorphism classes of locally free quiver representations
over $\mathcal{O}_{\alpha}$, in rank $\rr$. This is the orbit-count
for the action $\GL_{\alpha,\rr}\circlearrowleft R(Q,\alpha;\rr)$
and so we obtain:\[
M_{(Q,\alpha),\rr}=\sum_{[\gamma]\in\mathrm{Cl}_{\alpha,\rr}}\frac{\sharp R(Q,\alpha;\rr)^{\gamma}}{\sharp Z_{\GL_{\alpha,\rr}}(\gamma)},
\] where the sum runs over the set $\mathrm{Cl}_{\alpha,\rr}$ of conjugacy
classes of $\GL_{\alpha,\rr}$ and we denote by $R(Q,\alpha;\rr)^{\gamma}$
(resp. $Z_{\GL_{\alpha,\rr}}(\gamma)$) the set of points which are
fixed by $\gamma$ (resp. the centraliser of $\gamma$ in $\GL_{\alpha,\rr}$).
Then we can recover $A_{(Q,\alpha),\rr}$ from $M_{(Q,\alpha),\rr}$
using the following formula \cite[Thm. 1.2.]{Moz19}:\[
\sum_{\rr\in\ZZ_{\geq0}^{Q_0}}M_{(Q,\alpha),\rr}\cdot t^{\rr}=\Exp_{q,t}\left(\sum_{\rr\in\ZZ_{\geq0}^{Q_0}\setminus\{0\}}A_{(Q,\alpha),\rr}\cdot t^{\rr}\right).
\]

Conjugacy classes of $\GL_{\alpha,\rr}$ are completely classified
in ranks 2 and 3 \cite{AOPV09,AKOV16}. However, as noted in \cite[\S 4]{AOPV09},
computing the space of intertwiners between two distinct conjugacy
classes proves untractable. Therefore, we reduce to the case of a
quiver $Q$ with one vertex and $g$ loops. One could then compute
the centralisers in $\GL_{\alpha,\rr}$ and $\mathfrak{gl}_{\alpha,\rr}$
for representatives of all conjugacy classes in $\GL_{\alpha,\rr}$.
However, we prefer to compute summands in Burnside's formula by induction
on $\alpha$, using the branching rules for conjugacy classes described
in \cite{AOPV09,AKOV16}. Let us give details on this when $r=2$.

Recall from \cite[\S 2]{AOPV09} that matrices in $\mathfrak{gl}_{\alpha,2}$
are either scalar (we call these of type I) or conjugate to a unique
matrix of the form \[
d\cdot\mathrm{I}_2+t^j\cdot
\left(
\begin{array}{cc}
0 & a_0 \\
1 & a_1
\end{array}
\right),
\] where $0\leq j\leq\alpha-1$, $d\in\bigoplus_{l=0}^{j-1}\FF_{q}\cdot t^{l}$
and $a_{0},a_{1}\in\mathcal{O}_{\alpha-j}$ (we call these of type
II). We split the set of conjugacy classes in $\GL_{\alpha,2}$ in
four disjoint subsets $\mathrm{Cl}_{\alpha,2}=(\mathrm{Cl}_{\alpha,2})_{\mathrm{I}}\sqcup(\mathrm{Cl}_{\alpha,2})_{\mathrm{II_{1}}}\sqcup(\mathrm{Cl}_{\alpha,2})_{\mathrm{II_{2}}}\sqcup(\mathrm{Cl}_{\alpha,2})_{\mathrm{II_{3}}}$
and compute the sums\[
S_{\bullet,\alpha}:=\sum_{[\gamma]\in(\mathrm{Cl}_{\alpha,2})_{\bullet}}\frac{\sharp R(Q,\alpha;2)^{\gamma}}{\sharp Z_{\GL_{\alpha,2}}(\gamma)}
\] recursively, for $\bullet\in\{\mathrm{I},\mathrm{II}_{1},\mathrm{II}_{2},\mathrm{II}_{3}\}$.
The subsets $(\mathrm{Cl}_{\alpha,2})_{\bullet}$ are defined as follows:
\begin{itemize}
\item $(\mathrm{Cl}_{\alpha,2})_{\mathrm{I}}$ consists of conjugacy classes
of scalar matrices;
\item $(\mathrm{Cl}_{\alpha,2})_{\mathrm{II_{1}}}$ consists of conjugacy
classes of type II, where $T^{2}-a_{1}\cdot T-a_{0}$ is split in
$\FF_{q}[T]$ and has two distinct simple roots;
\item $(\mathrm{Cl}_{\alpha,2})_{\mathrm{II_{2}}}$ consists of conjugacy
classes of type II, where $T^{2}-a_{1}\cdot T-a_{0}$ is split in
$\FF_{q}[T]$ and has a double root;
\item $(\mathrm{Cl}_{\alpha,2})_{\mathrm{II_{3}}}$ consists of conjugacy
classes of type II, where $T^{2}-a_{1}\cdot T-a_{0}$ is irreducible
in $\FF_{q}[T]$.
\end{itemize}
Conjugacy classes belonging to the same subset $(\mathrm{Cl}_{\alpha,2})_{\bullet}$
have similar centralisers in $\GL_{\alpha,2}$ and $\mathfrak{gl}_{\alpha,2}$.
Let $\GL_{\alpha,2}^{1}\subseteq\GL_{\alpha,2}$ be the kernel of
reduction modulo $t$ i.e. $\GL_{\alpha,2}^{1}=\mathrm{I}_{2}+t\cdot\mathfrak{gl}_{\alpha,2}$.
Following \cite[\S 2.1.]{AKOV16}, for $\sigma\in\{\mathrm{I},\mathrm{II}_{1},\mathrm{II}_{2},\mathrm{II}_{3}\}$
and $\gamma\in\GL_{\alpha,2}$ of type $\sigma$, we call $\Vert\sigma\Vert$
the cardinality of $Z_{\GL_{\alpha,2}}(\gamma)/\left(Z_{\GL_{\alpha,2}}(\gamma)\cap\GL_{\alpha,2}^{1}\right)$
and $\dim(\sigma)$ its dimension as an algebraic group. These do
not depend on the choice of $\gamma$.

Consider $\gamma\in\GL_{\alpha+1,2}$, of type $\sigma$, and its
reduction $\overline{\gamma}\in\GL_{\alpha,2}$, of type $\tau$.
Then, arguing as in the proof of \cite[Prop. 2.5.]{AKOV16}, we easily
obtain:\begin{align*}
\frac{\sharp Z_{\GL_{\alpha+1,2}}(\gamma)}{\sharp Z_{\GL_{\alpha,2}}(\overline{\gamma})}
&=
q^{\dim(\tau)}\cdot\frac{\Vert\sigma\Vert}{\Vert\tau\Vert}, \\
\frac{\sharp Z_{\mathfrak{gl}_{\alpha+1,2}}(\gamma)}{\sharp Z_{\mathfrak{gl}_{\alpha,2}}(\overline{\gamma})}
&=
q^{\dim(\sigma)}.
\end{align*} Moreover, one can deduce the following branching rules from the description
of conjugacy classes in $\GL_{\alpha,2}$. Let $a_{\tau,\sigma}(q)$
be the number of conjugacy classes of type $\sigma$ in $\GL_{\alpha+1,2}$
whose reduction in $\GL_{\alpha,2}$ is of type $\tau$. Then:\[
\left(a_{\tau,\sigma}(q)\right)_{\substack{\sigma\in\{\mathrm{I},\mathrm{II}_{1},\mathrm{II}_{2},\mathrm{II}_{3}\} \\ \tau\in\{\mathrm{I},\mathrm{II}_{1},\mathrm{II}_{2},\mathrm{II}_{3}\}}}
=
\left(
\begin{array}{cccc}
q & 0 & 0 & 0 \\
\frac{q(q-1)}{2} & q^2 & 0 & 0 \\
q & 0 & q^2 & 0 \\
\frac{q(q-1)}{2} & 0 & 0 & q^2
\end{array}
\right)
,
\] where rows are labeled by $\sigma$ and columns are labeled by $\tau$.
Putting everything together, we obtain:\[
\left(
\begin{array}{c}
S_{\mathrm{I},\alpha+1} \\
S_{\mathrm{II}_1,\alpha+1} \\
S_{\mathrm{II}_2,\alpha+1} \\
S_{\mathrm{II}_3,\alpha+1}
\end{array}
\right)
=
\left(
\begin{array}{cccc}
q^{4g-3} & 0 & 0 & 0 \\
\frac{1}{2}q^{2g-2}(q-1)(q+1) & q^{2g} & 0 & 0 \\
q^{2g-3}(q-1)(q+1) & 0 & q^{2g} & 0 \\
\frac{1}{2}q^{2g-2}(q-1)^2 & 0 & 0 & q^{2g}
\end{array}
\right)
\times
\left(
\begin{array}{c}
S_{\mathrm{I},\alpha} \\
S_{\mathrm{II}_1,\alpha} \\
S_{\mathrm{II}_2,\alpha} \\
S_{\mathrm{II}_3,\alpha}
\end{array}
\right)
\ ;\ 
\left(
\begin{array}{c}
S_{\mathrm{I},1} \\
S_{\mathrm{II}_1,1} \\
S_{\mathrm{II}_2,1} \\
S_{\mathrm{II}_3,1}
\end{array}
\right)
=
\left(
\begin{array}{c}
\frac{q^{4g}}{q(q-1)(q+1)} \\
\frac{q^2(q-2)}{2(q-1)} \\
q^{2g-1} \\
\frac{q^{2g+1}}{2(q+1)}
\end{array}
\right)
,
\] where the $(\sigma,\tau)$-entry of the transition matrix is $q^{g\dim(\sigma)-\dim(\tau)}\cdot\frac{\Vert\tau\Vert}{\Vert\sigma\Vert}\cdot a_{\tau,\sigma}(q)$.
This matrix can easily be diagonalised and with the help of a computer,
we finally get:\[
A_{(Q,\alpha),2}= \frac{q^{2\alpha g-1}(q^{2g}-1)(q^{\alpha(2g-3)}-1)}{(q^2-1)(q^{2g-3}-1)}.
\]

In rank $r=3$, we split $\mathrm{Cl}_{\alpha,3}$ into ten types
$\mathcal{G},\mathcal{L},\mathcal{J},\mathcal{T}_{1},\mathcal{T}_{2},\mathcal{T}_{3},\mathcal{M},\mathcal{N},\mathcal{K}_{0},\mathcal{K}_{\infty}$,
following \cite[\S 2.2.]{AKOV16}, and we obtain the following recursive
relation:\[
\begin{array}{l}
\left(
\begin{array}{c}
S_{\mathcal{G},\alpha+1} \\
S_{\mathcal{L},\alpha+1} \\
S_{\mathcal{J},\alpha+1} \\
S_{\mathcal{T}_{1},\alpha+1} \\
S_{\mathcal{T}_{2},\alpha+1} \\
S_{\mathcal{T}_{3},\alpha+1} \\
S_{\mathcal{M},\alpha+1} \\
S_{\mathcal{N},\alpha+1} \\
S_{\mathcal{K}_0,\alpha+1} \\
S_{\mathcal{K}_{\infty},\alpha+1}
\end{array}
\right)
=
M
\times
\left(
\begin{array}{c}
S_{\mathcal{G},\alpha} \\
S_{\mathcal{L},\alpha} \\
S_{\mathcal{J},\alpha} \\
S_{\mathcal{T}_{1},\alpha} \\
S_{\mathcal{T}_{2},\alpha} \\
S_{\mathcal{T}_{3},\alpha} \\
S_{\mathcal{M},\alpha} \\
S_{\mathcal{N},\alpha} \\
S_{\mathcal{K}_0,\alpha} \\
S_{\mathcal{K}_{\infty},\alpha}
\end{array}
\right)
\ ; \ 
\left(
\begin{array}{c}
S_{\mathcal{G},1} \\
S_{\mathcal{L},1} \\
S_{\mathcal{J},1} \\
S_{\mathcal{T}_{1},1} \\
S_{\mathcal{T}_{2},1} \\
S_{\mathcal{T}_{3},1} \\
S_{\mathcal{M},1} \\
S_{\mathcal{N},1} \\
S_{\mathcal{K}_0,1} \\
S_{\mathcal{K}_{\infty},1}
\end{array}
\right)
=
\left(
\begin{array}{c}
\frac{q^{9g-3}}{(q^2-1)(q^3-1)} \\
\frac{q^{5g-1}(q-2)}{(q-1)(q^2-1)} \\
\frac{q^{5g-3}}{q-1} \\
\frac{q^{3g}(q-2)(q-3)}{6(q-1)^2} \\
\frac{q^{3g+1}}{2(q+1)} \\
\frac{q^{3g+1}(q^2-1)}{3(q^3-1)} \\
\frac{q^{3g-1}(q-2)}{q-1} \\
q^{3g-2} \\
0 \\
0
\end{array}
\right)
\text{, where:}
\\ \\
M=
\left(
\begin{array}{cccccccccc}
q^{9g-8} & 0 & 0 & 0 & 0 & 0 & 0 & 0 & 0 & 0 \\
q^{5g-6}(q^3-1) & q^{5g-3} & 0 & 0 & 0 & 0 & 0 & 0 & 0 & 0 \\
\frac{q^{5g-8}(q^2-1)(q^3-1)}{q-1} & 0 & q^{5g-3} & 0 & 0 & 0 & 0 & 0 & 0 & 0 \\
\frac{q^{3g-5}(q-2)(q^2-1)(q^3-1)}{6(q-1)} & \frac{q^{3g-2}(q^2-1)}{2} & 0 & q^{3g} & 0 & 0 & 0 & 0 & 0 & 0 \\
\frac{q^{3g-4}(q-1)(q^3-1)}{2} & \frac{q^{3g-2}(q-1)^2}{2} & 0 & 0 & q^{3g} & 0 & 0 & 0 & 0 & 0 \\
\frac{q^{3g-5}(q-1)(q^2-1)^2}{3} & 0 & 0 & 0 & 0 & q^{3g} & 0 & 0 & 0 & 0 \\
q^{3g-6}(q^2-1)(q^3-1) & q^{3g-3}(q^2-1) & q^{3g-1}(q-1) & 0 & 0 & 0 & q^{3g} & 0 & 0 & 0 \\
q^{3g-7}(q^2-1)(q^3-1) & 0 & q^{3g-3}(q-1)^2 & 0 & 0 & 0 & 0 & q^{3g} & 0 & 0 \\
0 & 0 & q^{3g-3}(q-1) & 0 & 0 & 0 & 0 & 0 & q^{3g} & 0 \\
0 & 0 & 0 & 0 & 0 & 0 & 0 & 0 & 0 & q^{3g}
\end{array}
\right)
.
\end{array}
\]

Again, the transition matrix $M$ can be diagonalised and we get the
following closed formula.

\begin{prop} \label{Prop/KacPolHuaFml}

Let $Q$ be the quiver with one vertex and $g\geq1$ loops. Then:\[
\begin{split}
A_{(Q,\alpha),2}= & \frac{q^{2\alpha g-1}(q^{2g}-1)(q^{\alpha(2g-3)}-1)}{(q^2-1)(q^{2g-3}-1)}, \\
A_{(Q,\alpha),3}= &
\frac{q^{3\alpha g-2}(q^{2g}-1)(q^{2g-1}-1)}{(q^2-1)(q^3-1)(q^{2g-3}-1)(q^{6g-8}-1)(q^{4g-5}-1)} \\
 & \cdot
\big(
q^{\alpha(6g-8)-1}(q^{6g-7}-1)(q^{2g}+1)
-q^{\alpha(6g-8)+2g-4}(q^2-1)(q^{4g-3}+1) \\
 & +q^{\alpha(2g-3)-1}(q^2+q+1)(q^{2g-1}-1)(q^{6g-8}-1)+(q+1)(q^{8g-10}-1) \\
 & +q^{2g-4}(q^4+1)(q^{4g-5}-1) \big)
.
\end{split}
\]In particular, $A_{(Q,\alpha),2}$ and $A_{(Q,\alpha),3}$ are polynomials
and $A_{(Q,\alpha),2}$ has non-negative coefficients.

\end{prop}

Unfortunately, it is not obvious from the above expression that $A_{(Q,\alpha),3}$
is a polynomial in $q$ with non-negative coefficients. But we can
still compute $A_{(Q,\alpha),3}$ for small values of $g$ and $\alpha$
and check that this is indeed the case. We write $A_{g,\alpha,3}:=A_{(Q,\alpha),3}$
when $Q$ is the $g$-loop quiver. Using a computer, we get:

\[
\begin{array}{l}
g=1: \\
\begin{split}
A_{1,1,3} & = q \\
A_{1,2,3} & = q^4 + q^3 + 2  q^2 \\
A_{1,3,3} & = q^7 + q^6 + 3  q^5 + 2  q^4 + 2  q^3 \\
A_{1,4,3} & = q^{10} + q^9 + 3  q^8 + 3  q^7 + 4  q^6 + 2  q^5 + 2  q^4 \\
A_{1,5,3} & = q^{13} + q^{12} + 3  q^{11} + 3  q^{10} + 5  q^9 + 4  q^8 + 4  q^7 + 2  q^6 + 2  q^5
\end{split} \\
\\ 
g=2: \\
\begin{split}
A_{2,1,3} & = q^{10} + q^8 + q^7 + q^6 + q^5 + q^4 \\
A_{2,2,3} & = q^{20} + q^{18} + 2q^{17} + 3q^{16} + 3q^{15} + 4q^{14} + 3q^{13} + 3q^{12} + 2q^{11} + 2q^{10} \\
A_{2,3,3} & = q^{30} + q^{28} + 2q^{27} + 3q^{26} + 3q^{25} + 5q^{24} + 5q^{23} + 7q^{22} + 6q^{21} + 7q^{20} + 5q^{19} + 4q^{18} + 3q^{17} + 2q^{16} \\
A_{2,4,3} & = q^{40} + q^{38} + 2q^{37} + 3q^{36} + 3q^{35} + 5q^{34} + 5q^{33} + 7q^{32} + 7q^{31} \\
& + 9q^{30} + 9q^{29} + 10q^{28} + 9q^{27} + 9q^{26} + 6q^{25} + 5q^{24} + 3q^{23} + 2q^{22} \\
A_{2,5,3} & = q^{50} + q^{48} + 2q^{47} + 3q^{46} + 3q^{45} + 5q^{44} + 5q^{43} + 7q^{42} + 7q^{41} + 9q^{40} + 9q^{39} + 11q^{38} \\
& + 11q^{37} + 13q^{36} + 12q^{35} + 13q^{34} + 11q^{33} + 10q^{32} + 7q^{31} + 5q^{30} + 3q^{29} + 2q^{28}
\end{split} \\
\\
g=3: \\
\begin{split}
A_{3,1,3} & = q^{19} + q^{17} + q^{16} + q^{15} + q^{14} + 2q^{13} + q^{12} + 2q^{11} + 2q^{10} + q^{9} + q^{8} + q^{7} \\
A_{3,2,3} & = q^{38} + q^{36} + q^{35} + q^{34} + q^{33} + 2q^{32} + 2q^{31} + 3q^{30} + 4q^{29} + 4q^{28} + 4q^{27} \\
& + 5q^{26} + 4q^{25} + 4q^{24} + 4q^{23} + 5q^{22} + 3q^{21} + 4q^{20} + 3q^{19} + 2q^{18} + q^{17} + q^{16} \\
A_{3,3,3} & = q^{57} + q^{55} + q^{54} + q^{53} + q^{52} + 2q^{51} + 2q^{50} + 3q^{49} + 4q^{48} + 4q^{47} + 4q^{46} + 5q^{45} + 4q^{44} + 5q^{43} + 5q^{42}  \\
& + 7q^{41} + 6q^{40} + 8q^{39} + 8q^{38} + 8q^{37} + 7q^{36} + 8q^{35} + 7q^{34} + 6q^{33} + 6q^{32} + 6q^{31} + 4q^{30} + 4q^{29} + 3q^{28} \\
& + 2q^{27} + q^{26} + q^{25} \\
A_{3,4,3} & = q^{76} + q^{74} + q^{73} + q^{72} + q^{71} + 2q^{70} + 2q^{69} + 3q^{68} + 4q^{67} + 4q^{66} + 4q^{65} + 5q^{64} + 4q^{63} + 5q^{62} + 5q^{61}  \\
& + 7q^{60} + 6q^{59} + 8q^{58} + 8q^{57} + 8q^{56}  + 8q^{55} + 9q^{54} + 9q^{53} + 9q^{52} + 10q^{51} + 11q^{50} + 10q^{49} + 11q^{48} \\
& + 11q^{47} + 11q^{46} + 9q^{45} + 10q^{44} + 8q^{43} + 7q^{42} + 6q^{41} + 6q^{40} + 4q^{39} + 4q^{38} + 3q^{37} + 2q^{36} + q^{35} + q^{34} \\
A_{3,5,3} & = q^{95} + q^{93} + q^{92} + q^{91} + q^{90} + 2q^{89} + 2q^{88} + 3q^{87} + 4q^{86} + 4q^{85} + 4q^{84} + 5q^{83} + 4q^{82} + 5q^{81}  + 5q^{80} \\
& + 7q^{79} + 6q^{78} + 8q^{77} + 8q^{76} + 8q^{75} + 8q^{74} + 9q^{73} + 9q^{72} + 9q^{71} + 10q^{70} + 11q^{69} + 10q^{68}  + 12q^{67} \\
& + 12q^{66} + 13q^{65} + 12q^{64} + 14q^{63} + 13q^{62} + 13q^{61} + 13q^{60} + 14q^{59} + 13q^{58} + 13q^{57} + 13q^{56} + 12q^{55} \\
& + 10q^{54} + 10q^{53} + 8q^{52} + 7q^{51} + 6q^{50} + 6q^{49} + 4q^{48} + 4q^{47} + 3q^{46} + 2q^{45} + q^{44} + q^{43}
\end{split}
\end{array}
\] 

From the data above, it seems reasonable to expect that $A_{(Q,\alpha),\rr}$
has non-negative coefficients for any quiver and any rank vector.
We hope to adress this conjecture in future works:

\begin{cj} \label{Conj/PosKacPolMult}

Let $Q$ be a quiver, $\rr\in\ZZ_{\geq0}^{Q_{0}}$ and $\alpha\geq1$.
Then $A_{(Q,\alpha),\rr}\in\ZZ_{\geq0}[q]$ .

\end{cj} 

\paragraph*{p-adic method: ASK zeta functions}

Let us now explain our second approach to computing $A_{(Q,\alpha),\rr}$.
As we pointed out in the introduction to this Chapter, Theorem \ref{Thm/Ch2ExpFmlKacPol}
implies that the datum of all $A_{(Q,\alpha),\rr}$, for $\alpha\geq1$,
$\rr\in\ZZ_{\geq0}^{Q_{0}}$ and a fixed quiver $Q$ is equivalent
to the datum of all $Z_{\mu_{Q,\rr}}$, for $\rr\in\ZZ_{\geq0}^{Q_{0}}$.
We prove Proposition \ref{Prop/Ch2ZetaMomMap} in the wider context
of Rossmann's ASK zeta functions \cite{Ros18,Ros20}. We then compute
$Z_{\mu_{Q,\rr}}$ and higher-rank generalisations of $B_{\mu_{Q}}$
in some examples, using Proposition \ref{Prop/Ch2ZetaMomMap} and
an adequate change of variables. The latter is inspired from embedded
resolutions of determinantal varieties (see for instance \cite{Vai84}).

We start with basic facts on ASK zeta functions. Let $F$ be a local
field of characteristic zero, with valuation ring $\mathcal{O}$ and
residue field $\FF_{q}$.

\begin{df}[ASK zeta function]

Let $A,B,C$ be three free $\mathcal{O}$-modules of respective ranks
$r_{A},r_{B},r_{C}\geq1$. Consider an $\mathcal{O}$-linear map $\theta:A\rightarrow\Hom_{\mathcal{O}}(B,C)$.
For $n\geq1$, call $\theta_{n}$ the $\mathcal{O}/(\varpi^{n})$-linear
map $A\otimes\mathcal{O}/(\varpi^{n})\rightarrow\Hom_{\mathcal{O}/(\varpi^{n})}(B\otimes\mathcal{O}/(\varpi^{n}),C\otimes\mathcal{O}/(\varpi^{n}))$
induced by $\theta$. The Average Size of Kernels of $\theta$ over
$\mathcal{O}/(\varpi^{n})$ is defined as:\[
\ask_n(\theta):=q^{-nr_A}\cdot\sum_{a\in A\otimes\mathcal{O}/(\varpi^n)}\sharp\Ker(\theta_n(a)).
\]The ASK zeta function of $\theta$ is the following function:\[
Z_{\theta}^{\ask}(s):=\sum_{n\geq0}\ask_n(\theta)\cdot q^{-ns},
\]with the convention that $\ask_{0}(\theta)=1$.

\end{df}

\begin{prop} \label{Prop/ASKvsIgusa}

Let $\theta$ be an $\mathcal{O}$-linear map $A\rightarrow\Hom_{\mathcal{O}}(B,C)$
as above. Let $\mu:\mathbb{A}_{\mathcal{O}}^{r_{A}}\times\mathbb{A}_{\mathcal{O}}^{r_{B}}\rightarrow\mathbb{A}_{\mathcal{O}}^{r_{C}}$
(resp. $f:\mathbb{A}_{\mathcal{O}}^{r_{A}}\times\mathbb{A}_{\mathcal{O}}^{r_{B}}\times\mathbb{A}_{\mathcal{O}}^{r_{C}}\rightarrow\mathbb{A}_{\mathcal{O}}^{1}$)
be the polynomial mapping given in coordinates by $\mu(a,b):=\theta(a)(b)$
(resp. $f(a,b,c):=\langle\theta(a)(b),c\rangle$, where $\langle\bullet,\bullet\rangle$
is given by a perfect pairing on $C$). Then:\begin{align*}
& Z_{\theta}^{\ask}(s)=P_\mu(q^{-(s+r_A)})=\frac{1-q^{-(s-r_B)}\cdot Z_{\mu}(s-r_B)}{1-q^{-(s-r_B)}}, \\
& Z_{f}(s)=\frac{1-q^{-1}}{1-q^{-(s+1)}}\cdot Z_{\mu}(s).
\end{align*}

\end{prop}

\begin{proof}

The first equality follows from the observation that $\ask_{n}(\theta)=q^{-nr_{A}}\cdot N_{\mu,n}$
and Lemma \ref{Lem/IgusaVSPoincar=0000E9}. The second equality is
essentially \cite[Thm. 1.4.]{Mus22a}. We give a short proof for the
reader's convenience. By Fubini's theorem, the left-hand side reads:\[
Z_f(s)=\int_{\mathcal{O}^{\oplus r_A}\times\mathcal{O}^{\oplus r_B}}\left(\int_{\mathcal{O}^{\oplus r_C}}\vert\langle\mu(x,y),z\rangle\vert^s dz \right) dxdy.
\]For $(x,y)\in\mathcal{O}^{\oplus r_{A}}\times\mathcal{O}^{\oplus r_{B}}$
fixed, there exists $g\in\GL_{r_{C}}(\mathcal{O})$ such that $g\cdot\mu(x,y)=\Vert\mu(x,y)\Vert\cdot e_{1}$,
where $e_{1}$ is the first vector in the canonical basis of $\mathcal{O}^{\oplus r_{C}}$.
Using change of variables the inner integral reduces to:\[
\int_{\mathcal{O}^{\oplus r_C}}\vert\langle\mu(x,y),z\rangle\vert^s dz=\Vert\mu(x,y)\Vert^s\cdot\int_{\mathcal{O}}\vert z_1\vert^s dz_1=\frac{1-q^{-1}}{1-q^{-(s+1)}}\cdot\Vert\mu(x,y)\Vert^s.
\]This finishes the proof. \end{proof}

\begin{exmp}

Here is the example we keep in mind for applications. Let $Q$ be
a quiver and $\dd\in\ZZ_{\geq0}^{Q_{0}}$ a dimension vector. Consider
the linear map $\theta:R(Q,\dd)\rightarrow\Hom(R(Q,\dd)^{\vee},\mathfrak{gl}_{\dd})$,
$x\mapsto\mu_{Q,\dd}(x,\bullet)$. Then $\mu=\mu_{Q,\dd}$.

\end{exmp}

\begin{rmk} \label{Rmk/KnuthDuals}

Proposition \ref{Prop/ASKvsIgusa} also recovers \cite[Cor. 5.6.]{Ros20}:
consider the Knuth dual maps $\theta^{\circ}:B\rightarrow\Hom(A,C)$
and $\theta^{\bullet}:C^{\vee}\rightarrow\Hom(B,A^{\vee})$. Then
the polynomial mappings $f,f^{\circ},f^{\bullet}$ associated to $\theta,\theta^{\circ},\theta^{\bullet}$
respectively all coincide. Thus $Z_{\theta}^{\ask}(s)=Z_{\theta^{\circ}}^{\ask}(s-r_{B}+r_{A})=Z_{\theta^{\bullet}}^{\ask}(s)$.

\end{rmk}

\begin{rmk}

The equality\[
Z_{f}(s)=\frac{1-q^{-1}}{1-q^{-(s+1)}}\cdot Z_{\mu}(s).
\]can be seen as an incarnation of Behrend, Bryan and Szendr\H{o}i's
computation of motivic nearby cycles for equivariant functions (see
\cite{BBS13} and \cite[Thm. 3.1.]{DM15b} for a formulation adapted
to our setup). Behrend, Bryan and Szendr\H{o}i's result relates the
motivic nearby cycles $[\psi_{f}]$ to the class $[f^{-1}(1)]$ in
the Grothendieck ring of varieties. On the other hand, starting from
Proposition \ref{Prop/ASKvsIgusa}, a simple computation shows that:\[
\underset{\Rea(s)\rightarrow-\infty}{\lim}Z_{f}(s)=(q-1)\cdot\frac{\sharp_{\FF_q}f^{-1}(1)-\sharp_{\FF_q}[\psi_f]}{q^{r_A+r_B+r_C}}=0.
\] Note that the motivic analogue of $Z_{f}(s)$ is a naive zeta function,
while motivic nearby cycles of $f$ are actually computed by Denef
and Loeser's refined zeta function - see \cite[Ch. 7, \S 4.]{CLNS18}.
The connection between the two owes to the fact that $f$ is linear
in $z\in C$.

\end{rmk}

\begin{prop} \label{Prop/ZetaMomMap}

Let $\theta$ be an $\mathcal{O}$-linear map $A\rightarrow\Hom_{\mathcal{O}}(B,C)$
and $\mu$ the associated polynomial mapping as above. Then:\[
Z_{\mu}(s) = 
1-\sum_{i=0}^{r}\frac{q^{-i}\cdot(1-q^{-s})}{1-q^{-(s+i)}}\cdot \int_{\mathcal{O}^{\oplus r_A}} \left(\left(\frac{\Vert\Delta_{i}(x)\Vert}{\Vert\Delta_{i-1}(x)\Vert}\right)^{s+i}-\left(\frac{\Vert\Delta_{i+1}(x)\Vert}{\Vert\Delta_{i}(x)\Vert}\right)^{s+i} \right)\cdot\frac{dx}{\Vert\Delta_{i}(x)\Vert},
\]where $\Delta_{i}(x)$ is the vector composed of all $i\times i$
minors of the matrix $\theta(x)$, $r:=\underset{x\in\mathcal{O}^{\oplus r_{A}}}{\max}\{\rk(\theta(x))\}$
and by convention, $\Vert\Delta_{-1}(x)\Vert=\Vert\Delta_{0}(x)\Vert=1$
and $\Vert\Delta_{r+1}(x)\Vert=0$.

\end{prop}

\begin{proof}

By Proposition \ref{Prop/ASKvsIgusa}, we may compute $Z_{f}(s)$
instead. By Fubini's theorem:\[
Z_f(s)=\int_{\mathcal{O}^{\oplus r_A}}\left(\int_{\mathcal{O}^{\oplus r_B}\times\mathcal{O}^{\oplus r_C}}\vert\langle\theta(x)(y),z\rangle\vert^s dydz \right) dx.
\]For $(y,z)\in\mathcal{O}^{\oplus r_{B}}\times\mathcal{O}^{\oplus r_{C}}$,
the matrix $\theta(x)$ admits a Smith normal form (see for instance
\cite{Sta16}): there exists $(g_{1},g_{2})\in\GL_{r_{C}}(\mathcal{O})\times\GL_{r_{B}}(\mathcal{O})$
such that\[
g_{1}\theta(x)g_{2}=
\left(
\begin{array}{ccccccc}
\varpi^{\gamma_1(x)} & 0 & \ldots & 0 & 0 & \ldots & 0 \\
0 & \varpi^{\gamma_2(x)} & \ldots & 0 & 0 & \ldots & 0 \\
\vdots & \vdots & \ddots & \vdots & \vdots & \ddots & \vdots \\
0 & 0 & \ldots & \varpi^{\gamma_{r(x)}(x)} & 0 & \ldots & 0 \\
0 & 0 & \ldots & 0 & 0 & \ldots & 0 \\
\vdots & \vdots & \ddots &  \vdots & \vdots & \ddots & \vdots \\
0 & 0 & \ldots & 0 & 0 & \ldots & 0
\end{array}
\right),
\]where $r(x):=\rk(\theta(x))$ and $0\leq\gamma_{1}(x)\leq\ldots\leq\gamma_{r(x)}(x)$.
Note that $r(x)=r$ over an open subset $U\subseteq\mathcal{O}^{\oplus r_{A}}$
such that $\nu(\mathcal{O}^{\oplus r_{A}}\setminus U)=0$ - see for
instance \cite[Prop. 1.4.3.]{CLNS18}. Using change of variables,
the inner integral yields (for $x\in U$):\[
\int_{\mathcal{O}^{\oplus r_B}\times\mathcal{O}^{\oplus r_C}}\vert\langle\theta(x)(y),z\rangle\vert^s dydz
=
\int_{\mathcal{O}^{\oplus r}\times\mathcal{O}^{\oplus r}}\left\vert\sum_{i=1}^r\varpi^{\gamma_i(x)}u_iv_i\right\vert^s dudv.
\]The latter integrals were computed in \cite[\S 4.]{Wys17b} using
Fourier transform. By \cite[Cor. 4.14.]{Wys17b}, we obtain:\[
\int_{\mathcal{O}^{\oplus r}\times\mathcal{O}^{\oplus r}}\left\vert\sum_{i=1}^r\varpi^{\gamma_i(x)}u_iv_i\right\vert^s dudv
=
\frac{q^{s+1}-q^s}{q^{s+1}-1}
+
\frac{q-q^{s+1}}{q^{s+1}-1}
\cdot
\int_{F\setminus\mathfrak{m}}\frac{a_{\rho}(w)}{\vert w\vert^{s+1}}dw,
\]where\begin{align*}
& \rho(w):=w\cdot
\left(
\begin{array}{ccc}
\varpi^{\gamma_1(x)} & \ldots & 0 \\
\vdots & \ddots & \vdots \\
0 & \ldots & \varpi^{\gamma_r(x)}
\end{array}
\right) , \\
& a_{\rho}(w):=\int_{\mathcal{O}^r}\textbf{1}_{\Vert\rho(w)u\Vert<1}du
= \prod_{i=1}^r\min\left\{1,\frac{q^{\gamma_i(x)}}{q\vert w\vert}\right\}.
\end{align*}Putting everything together, we obtain:\begin{align*}
& \int_{F\setminus\mathfrak{m}}\frac{a_{\rho}(w)}{\vert w\vert^{s+1}}dw=(1-q^{-1})\cdot\sum_{i=0}^{r}\frac{q^{-i}}{1-q^{-(s+i)}}\cdot q^{\gamma_{1}(x)+\ldots+\gamma_{i}(x)}\cdot(q^{-\gamma_i(x)\cdot(s+i)}-q^{-\gamma_{i+1}(x)\cdot(s+i)}), \\
& Z_{\mu}(s)=
1-\sum_{i=0}^{r}\frac{q^{-i}\cdot(1-q^{-s})}{1-q^{-(s+i)}}\cdot\int_{U} q^{\gamma_{1}(x)+\ldots+\gamma_{i}(x)}\cdot(q^{-\gamma_i(x)\cdot(s+i)}-q^{-\gamma_{i+1}(x)\cdot(s+i)}) dx,
\end{align*} where by convention, $0=\gamma_{0}(x)<\gamma_{r+1}(x)=+\infty$.
Finally, by \cite[Thm. 2.4.]{Sta16}, $\Vert\Delta_{i}(x)\Vert=q^{-\gamma_{1}(x)-\ldots-\gamma_{i}(x)}$,
so $q^{-\gamma_{i}(x)}=\frac{\Vert\Delta_{i}(x)\Vert}{\Vert\Delta_{i-1}(x)\Vert}$
with the conventions from the theorem statement. This finishes the
proof. \end{proof}

\begin{rmk}

The ideals generated by $i\times i$ minors of $\theta(x)$ already
appeared in the works of Rossmann \cite{Ros18,Ros20} and Carnevale,
Rossmann \cite{CR22} on ASK zeta functions. In \cite{CR22}, the
authors consider Fitting ideals of the module $\Cok(\theta^{\circ}(y))$,
where coordinates of $y$ are treated as formal variables (see Remark
\ref{Rmk/KnuthDuals} for notations on Knuth duals). These ideals
are generated by minors of the matrix $\theta^{\circ}(y)$, with a
fixed size. Thus, up to taking Knuth duals, Proposition \ref{Prop/ZetaMomMap}
gives a closed formula for ASK zeta functions in terms of Fitting
ideals of $\Cok(\theta(x))$.

\end{rmk}

\begin{prop} \label{Prop/JetCountsASK}

Let $\theta$ be an $\mathcal{O}$-linear map $A\rightarrow\Hom_{\mathcal{O}}(B,C)$
and $\mu$ the associated polynomial mapping as above. For $x\in\mathcal{O}^{\oplus r_{A}}$,
consider $\Delta(x):=\Delta_{r}(x)$, where $r:=\underset{x\in\mathcal{O}^{\oplus r_{A}}}{\max}\{\rk(\theta(x))\}$
and $\Delta_{r}(x)$ is the vector composed of all $r\times r$ minors
of the matrix $\theta(x)$. Assume that:\[
\int_{x\in\mathcal{O}^{\oplus r_{A}}}\frac{dx}{\Vert\Delta(x)\Vert}<+\infty.
\]Then:\[
\underset{n\rightarrow +\infty}{\lim}\left(q^{-n(r_A+r_B-r)}\cdot\sharp\mu^{-1}(0)(\mathcal{O}/(\varpi^n))\right)=\int_{x\in\mathcal{O}^{\oplus r_{A}}}\frac{dx}{\Vert\Delta(x)\Vert}.
\]

\end{prop}

\begin{proof}

Let $n\geq1$. As above, we call $\theta_{n}$ the $\mathcal{O}/(\varpi^{n})$-linear
map $A\otimes\mathcal{O}/(\varpi^{n})\rightarrow\Hom_{\mathcal{O}/(\varpi^{n})}(B\otimes\mathcal{O}/(\varpi^{n}),C\otimes\mathcal{O}/(\varpi^{n}))$
induced by $\theta$. Then:\[
\sharp\mu^{-1}(0)(\mathcal{O}/(\varpi^n))=\sum_{x\in (\mathcal{O}/\varpi^n)^{\oplus r_A}}\sharp\Ker(\theta_n(x)).
\]Using a Smith normal form of $\theta_{n}(x)$ as in the proof of Proposition
\ref{Prop/ZetaMomMap} (with the same notations), one can easily check
that:\[
\sharp\Ker(\theta_n(x))=q^{n(r_B-r(x))}\cdot\prod_{i=1}^{r(x)}q^{\min\{n,\gamma_i(x)\}}.
\]Summing up, we obtain:\[
\begin{split}
\sharp\mu^{-1}(0)(\mathcal{O}/(\varpi^n))
& =
q^{-nr_A}\sum_{x\in (\mathcal{O}/\varpi^n)^{\oplus r_A}}\left( q^{n(r_A+r_B-r(x))}\cdot\prod_{i=1}^{r(x)}q^{\min\{n,\gamma_i(x)\}}\right) \\
& =
q^{n(r_A+r_B-r)}\int_{U}\left(\prod_{i=1}^{r}q^{\min\{n,\gamma_i(x)\}}\right) dx,
\end{split}
\]where $U\subseteq\mathcal{O}^{\oplus r_{A}}$ is the open subset where
$\rk(\theta(x))=r$. Let us call $\Vert\Delta(x)\Vert_{\leq n}:=\prod_{i=1}^{r}q^{-\min\{n,\gamma_{i}(x)\}}$.
Then $\Vert\Delta(x)\Vert_{\leq n},\ n\geq1$ is a non-increasing
sequence and converges to $\Vert\Delta(x)\Vert$ as $n$ goes to infinity.
From the monotone convergence theorem, we obtain:\[
q^{-n(r_A+r_B-r)}\cdot\sharp\mu^{-1}(0)(\mathcal{O}/(\varpi^n))=\int_U\frac{dx}{\Vert\Delta(x)\Vert_{\leq n}}\underset{n\rightarrow +\infty}{\longrightarrow}\int_U\frac{dx}{\Vert\Delta(x)\Vert}.
\]

\end{proof}

\begin{rmk}

Although Proposition \ref{Prop/JetCountsASK} is about counts of $\mathcal{O}/(\varpi^{n})$-points,
we can use the result to study the asymptotic behaviour of $\FF_{q}[t]/(t^{n})$-points,
provided that $\theta$ is defined over $\ZZ$. This is a consequence
of transfer principles for p-adic integrals (see \cite{CL05,CL10}
and \cite[Prop. 3.0.2.]{AA18}). We will return to this transfer principle
later, in Chapter \ref{Chap/RatSg}.

\end{rmk}

Our strategy now is to compute counts of jets over $\mu^{-1}(0)$
and their limit using a change of variables akin to Denef's computation
of local zeta functions (see Proposition \ref{Lem/DenefLemma}). We
are looking for a birational map to $\mathbb{A}_{F}^{r_{A}}$ which
simultaneously monomialises all the ideals $J_{i}$ generated by $i\times i$
minors of $\theta(x)$. Let us recall such a construction for determinantal
ideals of generic matrices i.e. in the case where $A=\Hom(B,C)$ and
$\theta=\mathrm{id}$ (see for instance \cite{Vai84}).

\begin{exmp}[Embdedded resolutions of generic determinantal ideals]

Let $\mathbb{A}_{F}^{m\times n}$ be the affine space of $m\times n$
matrices (say, $m\leq n$). Let $J_{r}$ be the ideal generated by
all $(r+1)\times(r+1)$ minors of the generic matrix $(x_{i,j})_{\substack{1\leq i\leq m\\
1\leq j\leq n
}
}$.

Consider first $\pi_{0}:X_{0}\rightarrow\mathbb{A}_{F}^{m\times n}$
the blowing-up of the origin in $\mathbb{A}_{F}^{m\times n}$ - in
other words, blowing up along $J_{0}$. Let us look at the strict
transform of $J_{r}$ along $\pi_{0}$. $X_{0}$ can be covered with
affine charts $U_{i,j}\simeq\mathbb{A}_{F}^{m\times n}$ ($1\leq i\leq m$,
$1\leq j\leq n$), such that for $y=(y_{k,l})_{\substack{1\leq k\leq m\\
1\leq l\leq n
}
}\in U_{i,j}$:\[
\pi_0(y)
=
\left(
\begin{array}{ccccc}
y_{i,j}y_{1,1} & \ldots & \ldots & \ldots & y_{i,j}y_{1,n} \\
\vdots & \ddots & \ddots & \ddots & \vdots \\
\vdots & \ddots & y_{i,j} & \ddots & \vdots \\
\vdots & \ddots & \ddots & \ddots & \vdots \\
y_{i,j}y_{m,1} & \ldots & \ldots & \ldots & y_{i,j}y_{m,n}
\end{array}
\right)
.
\]For simplicity, we compute the strict transform of $J_{r}$ in the
chart $U_{1,1}$. The $r\times r$ minors of $\pi_{0}(y)$ are all
of the form $y_{1,1}^{r}\cdot\Delta$, where $\Delta$ is a $r\times r$
minor of the matrix:\[
M=
\left(
\begin{array}{cccc}
1 & y_{1,2} & \ldots & y_{1,n} \\
y_{2,1} & \ddots  & \ddots & \vdots \\
\vdots & \ddots & \ddots  & \vdots \\
y_{m,1} & \ldots & \ldots  & y_{m,n}
\end{array}
\right)
.
\]By echeloning the matrix above, we obtain:\[
\left(
\begin{array}{cccc}
1 & y_{1,2} & \ldots & y_{1,n} \\
0 &  &  &  \\
\vdots & & M'=(y'_{k,l})  &  \\
0 & & &
\end{array}
\right)
,
\]where $y'_{k,l}:=y_{k,l}-y_{k,1}y_{1,l}$ for $2\leq k\leq m$ and
$2\leq l\leq n$. Now, if $\Delta$ contains row 1 and column 1, it
reduces to a $(r-1)\times(r-1)$ minor of $M'$. If $\Delta$ contains
row 1, but not column 1, then by developing $\Delta$ along row 1,
one checks that it lies in the ideal generated by the $(r-1)\times(r-1)$
minors of $M'$. The same reasoning works if $\Delta$ contains column
1, but not row 1 (then one develops $\Delta$ along column 1). Finally,
if $\Delta$ contains neither row 1 nor column 1, consider the minor
$\Delta'$ containing the corresponding rows and columns in $M'$.
Then one checks that both $\Delta'$ and $\Delta-\Delta'$ lie in
the ideal generated by $(r-1)\times(r-1)$ minors of $M'$. Summing
up, we have obtained that the strict transform of $J_{r}$ is generated
by the $(r-1)\times(r-1)$ minors of $M'$.

The polynomials $y_{1,1}$, $y_{k,1},y_{1,k}$ and $y'_{k,l}$ form
a set of coordinates of $U_{1,1}$, so the strict transform of $J_{r}$
in $U_{1,1}$ is the ideal $J_{r-1}$ in the coordinates $y'_{k,l}$.
In particular, blowing-up the strict transform of $J_{1}$ amounts
to repeating the above procedure with $M'$ in each chart. Iterating
$r$ times, we obtain that the total transform of $J_{r}$ is generated
in charts by minors of echelonned matrices (which are monomial) and
we are done.

\end{exmp}

We now apply a similar procedure to determinantal ideals of the matrix
$\mu_{Q,\dd}(x,\bullet)$, for certain families of quivers and small
dimension vectors. This allows us to compute $Z_{\mu_{Q,\dd}}(s)$
as wanted.

\begin{exmp}[g-loop quivers]

Set $g\geq2$ and $d=2$. Let $Q$ be the g-loop quiver i.e. the quiver
with one vertex and $g$ loop arrows. Consider $x\in R(Q,d)$, that
is a tuple of matrices $(x_{i})_{1\leq i\leq g}$ of size $2$. Set
coordinates:\[
x_i=
a_i
\left(
\begin{array}{cc}
1 & 0 \\
0 & 1
\end{array}
\right)
+
b_i
\left(
\begin{array}{cc}
0 & 1 \\
0 & 0
\end{array}
\right)
+
c_i
\left(
\begin{array}{cc}
0 & 0 \\
1 & 0
\end{array}
\right)
+
d_i
\left(
\begin{array}{cc}
1 & 0 \\
0 & -1
\end{array}
\right)
.
\]Recall that $\mu_{Q,d}(x,y)=\sum_{i=1}^{g}[x_{i},y_{i}]$. Then in
the basis above:\[
\mu_{Q,d}(x,\bullet)
=
\left(
\ldots
\left\vert
\begin{array}{cccc}
0 & 0 & 0 & 0 \\
0 & 2d_i & 0 & -2b_i \\
0 & 0 & -2d_i & 2c_i \\
0 & -c_i & b_i & 0
\end{array}
\right\vert
\ldots
\right)
.
\]Let us compute the possible ranks of $\mu_{Q,d}(x,\bullet)$. By Proposition
\ref{Prop/MomMapExSeq}, $d^{2}=\dim(\GL_{d})_{x}+\rk\mu_{Q,d}(x)$.
One can easily check that:
\begin{itemize}
\item if all $x_{i}$, $1\leq i\leq g$ are scalar, then $\dim(\GL_{d})_{x}=4$;
\item otherwise, if all $x_{i}$, $1\leq i\leq g$ commute, then $\dim(\GL_{d})_{x}=2$;
\item otherwise, $\dim(\GL_{d})_{x}=1$.
\end{itemize}
In those cases, $\mu_{Q,d}(x,\bullet)$ has rank 0, 2 and 3 respectively.
Let us now compute an embedded resolution of singularities which monomialises
$J_{0}$, $J_{1}$ and $J_{2}$. After blowing-up along $J_{0}$,
we obtain the following echelonned matrix in the chart corresponding
to $d_{1}$ (we remove rows and columns containing only zeroes for
convenience):\[
d_1\cdot
\left(
\left.
\begin{array}{ccc}
2 & 0 & -2b_1 \\
0 & -2 & 2c_1 \\
0 & 0 & 0
\end{array}
\right\vert
\ldots
\left\vert
\begin{array}{ccc}
2d_i & 0 & -2b_i \\
0 & -2d_i & 2c_i \\
c_1d_i-c_i & b_i-b_1d_i & b_1c_i-b_ic_1
\end{array}
\right\vert
\ldots
\right)
.
\]Then the strict transform of $J_{2}$ in this chart is the ideal generated
by $c_{i}-c_{1}d_{i}$ and $b_{i}-b_{1}d_{i}$, for $2\leq i\leq g$.
Call $E$ the exceptional divisor and $Z_{2}$ the strict transform
of $J_{2}$ in $\mathrm{Bl}_{J_{0}}(R(Q,d))$. By further blowing
up along $Z_{2}$, we obtain two smooth divisors $E_{1}$, $E_{2}$
with normal crossings. $E_{1}$ is the strict transform of $E$ (under
the second blowing-up). $E_{2}$ is the exceptional divisor of the
second blowing-up. Let us call $Y\rightarrow R(Q,d)$ the composition
of these two blowing-ups. Working in the above chart, one checks that
$Y\rightarrow R(Q,d)$ is an embedded resolution of $J_{2}$, with
data $(N_{1},\nu_{1})=(3,3g)$ and $(N_{2},\nu_{2})=(1,2g-2)$.

We also need to compute $\sharp_{\FF_{q}}(E_{1}\cap E_{2})$, $\sharp_{\FF_{q}}E_{1}^{\circ}$,
$\sharp_{\FF_{q}}E_{2}^{\circ}$ and $\sharp_{\FF_{q}}U$, where $E_{i}^{\circ}=E_{i}\setminus(E_{1}\cap E_{2})$
for $i=1,2$ and $U=Y\setminus(E_{1}\cup E_{2})$. We have $E\simeq\mathbb{A}^{g}\times\mathbb{P}^{3g-1}$
and $Z_{2}$ intersects $E$ in $\mathbb{A}^{g}\times V$, where $V\subseteq\mathbb{P}^{3g-1}$
is the projective variety cut out by the equations:\[
\left\{
\begin{array}{l}
b_ic_j-b_jc_i=0 \\
b_id_j-b_jd_i=0 \\
c_id_j-c_jd_i=0
\end{array}
\right.
,\ 1\leq i,j\leq g.
\] Thus $V$ is isomorphic to the image of $\mathbb{P}^{2}\times\mathbb{P}^{g-1}$
in $\mathbb{P}^{3g-1}$ under the Segre embedding. Then $E_{1}\cap E_{2}$
is the exceptional divisor of the blowing-up of $\mathbb{A}^{g}\times V$
in $E$, so it is isomorphic to a $\mathbb{P}^{2g-3}$-bundle over
$\mathbb{A}^{g}\times V$. Likewise, $E_{2}^{\circ}$ is a $\mathbb{P}^{2g-3}$-bundle
over $\mathbb{A}^{g}\times C$, where $C$ is the affine cone of $V$,
minus its vertex. Summing up, we get:
\begin{itemize}
\item $\sharp_{\FF_{q}}E=q^{g}\cdot\frac{q^{3g}-1}{q-1}$ and $\sharp_{\FF_{q}}V=\frac{q^{3}-1}{q-1}\cdot\frac{q^{g}-1}{q-1}$;
\item $\sharp_{\FF_{q}}(E_{1}\cap E_{2})=q^{g}\cdot\frac{q^{3}-1}{q-1}\cdot\frac{q^{g}-1}{q-1}\cdot\frac{q^{2g-2}-1}{q-1}$;
\item $\sharp_{\FF_{q}}E_{1}^{\circ}=\sharp_{\FF_{q}}E-\sharp_{\FF_{q}}(\mathbb{A}^{g}\times V)=q^{g}\cdot\frac{q^{3g}-1}{q-1}-q^{g}\cdot\frac{q^{3}-1}{q-1}\cdot\frac{q^{g}-1}{q-1}$;
\item $\sharp_{\FF_{q}}E_{2}^{\circ}=(q-1)\cdot q^{g}\cdot\sharp_{\FF_{q}}V\cdot\frac{q^{2g-2}-1}{q-1}=(q^{2g-2}-1)\cdot q^{g}\cdot\frac{q^{3}-1}{q-1}\cdot\frac{q^{g}-1}{q-1}$;
\item $\sharp_{\FF_{q}}U=q^{4g}-q^{g}\cdot(1+(q-1)\cdot\sharp_{\FF_{q}}V)=q^{4g}-q^{g}\cdot(1+(q^{3}-1)\cdot\frac{q^{g}-1}{q-1})$.
\end{itemize}
We can now apply Lemma \ref{Lem/DenefLemma} and Proposition \ref{Prop/ZetaMomMap}
to compute $Z_{\mu_{Q,d}}(s)$. Using a computer, we obtain:\[
Z_{\mu_{Q,d}}(s)
=
\frac{(q^3-1)(q^{2g}-1)}{(q^{3}-q^{-s})(q^{2g}-q^{-s})}.
\]This is consistent with our previous computation of $A_{Q,d,\alpha}$.
Indeed, by Theorem \ref{Thm/ExpFmlKacPol} and Proposition \ref{Prop/KacPolHuaFml},
we obtain:\begin{align*}
& \sharp \mu_{(Q,\alpha),d}^{-1}(0) = \frac{q^{2g}-1}{q^{3}(q^{2g-3}-1)}\cdot q^{\alpha(8g-3)}-\frac{q^{3}-1}{q^{3}(q^{2g-3}-1)}\cdot q^{6\alpha g}, \\
& P_{\mu_{Q,d}}(T) = \frac{q^{2g}-1}{q^{3}(q^{2g-3}-1)(1-q^{8g-3}T)}-\frac{q^{3}-1}{q^{3}(q^{2g-3}-1)(1-q^{6g}T)},
\end{align*}and thus, by Lemma \ref{Lem/IgusaVSPoincar=0000E9}:\[
Z_{\mu_{Q,d}}(s)
=
\frac{(q^3-1)(q^{2g}-1)}{(q^{3}-q^{-s})(q^{2g}-q^{-s})}.
\]

Using Proposition \ref{Prop/JetCountsASK} (or alternatively, Proposition
\ref{Prop/JetCountLim}), we obtain:\[
\underset{\alpha\rightarrow +\infty}{\lim}\left( q^{-\alpha (8g-3)}\cdot\sharp\mu_{(Q,\alpha),d}^{-1}(0)(\FF_q)\right)=\frac{q^{2g}-1}{q^3(q^{2g-3}-1)}.
\]

\end{exmp}

This technique can also be applied to quivers with more than one vertex,
for small dimension vectors.

\begin{exmp}[r-Kronecker quiver]

Set $r\geq3$ and $\dd=(1,2)$. Let $Q$ be the $r$-Kronecker quiver
i.e. the quiver with $Q_{0}=\{1,2\}$ and $r$ arrows from 1 to 2.
Consider $x\in R(Q,\dd)$, that is a tuple of vectors $(x_{i})_{1\leq i\leq r}\in(\mathbb{A}^{2})^{r}$.
Set coordinates:\[
x_i=
\left(
\begin{array}{c}
a_i \\
b_i
\end{array}
\right)
.
\]The moment map $\mu_{Q,\dd}$ is given by $\mu_{Q,\dd}(x,y)=(-\sum_{i}y_{i}^{\mathrm{T}}x_{i},\sum_{i}x_{i}y_{i}^{\mathrm{T}})\in\mathbb{A}^{1}\times\mathfrak{gl}_{2}$.
Thus, using the canonical basis of $(\mathbb{A}^{2})^{r}$ and the
following basis of $\mathbb{A}^{1}\times\mathfrak{gl}_{2}$:\[
\left(
1,
\left(
\begin{array}{cc}
1 & 0 \\
0 & 0
\end{array}
\right)
,
\left(
\begin{array}{cc}
0 & 0 \\
1 & 0
\end{array}
\right)
,
\left(
\begin{array}{cc}
0 & 1 \\
0 & 0
\end{array}
\right),
\left(
\begin{array}{cc}
0 & 0 \\
0 & 1
\end{array}
\right)
\right),
\]the matrix of $\mu_{Q,\dd}(x,\bullet)$ is:\[
\left(
\ldots
\left\vert
\begin{array}{cc}
a_i & b_i \\
a_i & 0 \\
b_i & 0 \\
0 & a_i \\
0 & b_i
\end{array}
\right\vert
\ldots
\right)
.
\]As in the previous example, we compute the possible ranks of $\mu_{Q,\dd}(x,\bullet)$:
\begin{itemize}
\item if $x=0$, then $\dim(\GL_{\dd})_{x}=5$;
\item otherwise, if all $x_{i},\ 1\leq i\leq r$ are collinear, then $\dim(\GL_{\dd})_{x}=3$;
\item otherwise, the $x_{i},\ 1\leq i\leq r$ span $\mathbb{A}^{2}$ and
$\dim(\GL_{\dd})_{x}=1$.
\end{itemize}
In those cases, $\mu_{Q,\dd}(x,\bullet)$ has rank 0, 2 and 4 respectively.
We compute as above an embedded resolution of singularities which
monomialises all $J_{k}$, $0\leq k\leq3$. After blowing up along
$J_{0}$, we obtain the following echelonned matrix in the chart corresponding
to $a_{1}$:\[
a_1\cdot
\left(
\left.
\begin{array}{cc}
1 & b_i \\
0 & 1 \\
0 & 0 \\
0 & 0 \\
0 & 0
\end{array}
\right\vert
\ldots
\left\vert
\begin{array}{cc}
a_i & b_i \\
0 & a_i \\
b_i-b_1a_i & -b_1(b_i-b_1a_i) \\
0 & b_i-b_1a_i \\
0 & 0
\end{array}
\right\vert
\ldots
\right)
.
\]Then the strict transform of $J_{2}$ in this chart is the ideal generated
by $b_{i}-b_{1}a_{i}$, for $2\leq i\leq r$. Let us call $E$ the
exceptional divisor and $Z_{2}$ the strict transform of $J_{2}$
in $\mathrm{Bl}_{J_{0}}(R(Q,d))$. By blowing-up $Z_{2}$, we obtain
two smooth divisors $E_{1},E_{2}$ with normal crossings. $E_{1}$
is the strict transform of $E$ (under the second blowing-up) and
$E_{2}$ is the exceptional divisor of the second blowing-up. We call
again $Y\rightarrow R(Q,d)$ the composition of these two blowing-ups.
It is an embedded resolution of $J_{3}$, with data $(N_{1},\nu_{1})=(4,2r)$
and $(N_{2},\nu_{2})=(2,r-1)$.

We now compute $\sharp_{\FF_{q}}(E_{1}\cap E_{2})$, $\sharp_{\FF_{q}}E_{1}^{\circ}$,
$\sharp_{\FF_{q}}E_{2}^{\circ}$ and $\sharp_{\FF_{q}}U$ as in the
previous example. In $\mathrm{Bl}_{J_{0}}(R(Q,d))$, $Z_{2}$ intersects
$E$ in the projective variety $V$ cut out by the equations $a_{i}b_{j}-a_{j}b_{i}=0$,
$1\leq i,j\leq r$. So $V\simeq\mathbb{P}^{1}\times\mathbb{P}^{r-1}\subseteq E\simeq\mathbb{P}^{2r-1}$
via the Segre embedding. Then $E_{1}\cap E_{2}$ is the exceptional
divisor of the blowing up of $V$ inside $E$, so it is a $\mathbb{P}^{r-2}$-bundle
over $V$. Likewise, $E_{2}^{\circ}$ is a $\mathbb{P}^{r-2}$-bundle
over the affine cone over $V$, minus its vertex. Summing up, we obtain:
\begin{itemize}
\item $\sharp_{\FF_{q}}E=\frac{q^{2r}-1}{q-1}$ and $\sharp_{\FF_{q}}V=(q+1)\cdot\frac{q^{r}-1}{q-1}$;
\item $\sharp_{\FF_{q}}(E_{1}\cap E_{2})=\sharp_{\FF_{q}}V\times\frac{q^{r-1}-1}{q-1}=(q+1)\cdot\frac{q^{r}-1}{q-1}\cdot\frac{q^{r-1}-1}{q-1}$;
\item $\sharp_{\FF_{q}}E_{1}^{\circ}=\sharp_{\FF_{q}}E-\sharp_{\FF_{q}}V=\frac{q^{2r}-1}{q-1}-(q+1)\cdot\frac{q^{r}-1}{q-1}$;
\item $\sharp_{\FF_{q}}E_{2}^{\circ}=(q-1)\cdot\sharp_{\FF_{q}}V\cdot\frac{q^{r-1}-1}{q-1}=(q^{r-1}-1)\cdot(q+1)\cdot\frac{q^{r}-1}{q-1}$;
\item $\sharp_{\FF_{q}}U=q^{2r}-(1+(q-1)\cdot\sharp_{\FF_{q}}V)=q^{2r}-(q^{2}-1)\cdot\frac{q^{r}-1}{q-1}-1$.
\end{itemize}
We can now apply Proposition \ref{Prop/ZetaMomMap} and Lemma \ref{Lem/DenefLemma}
as above, to obtain:\[
Z_{\mu_{Q,\dd}}(s)
=
\frac{(q^2-1)(q^{r}-1)(q^{r}(q-1)(q^{2}+q^{-s})+(q^{2}+1)(q^{2r+1}-q^{-s}))}{(q^{4}-q^{-s})(q^{2r}-q^{-s})(q^{r+1}-q^{-s})}
\]and:\[
\underset{\alpha\rightarrow +\infty}{\lim}\left( q^{-2\alpha (r-2)}\cdot\sharp\mu_{(Q,\alpha),\dd}^{-1}(0)(\FF_q)\right)
=
\frac{(q^r-1)(q^{r-1}-1)}{q^4(q^{r-2}-1)(q^{r-3}-1)}.
\]Finally, using Lemma \ref{Lem/IgusaVSPoincar=0000E9}, we can compute
the first terms of $P_{\mu_{Q,\dd}}$ and thereby compute $A_{(Q,\alpha),\dd}$
for small values of $\alpha$, using Theorem \ref{Thm/ExpFmlKacPol}.
We obtain:\[
\begin{split}
A_{(Q,1),\dd} = & \  \frac{(q^{r-1}-1) (q^{r}-1)}{(q-1)^{2} (q+ 1)} \\
A_{(Q,2),\dd} = & \ \frac{(q^{r-1}-1) (q^{r}-1)}{(q-1)^{2} (q+ 1)}\cdot (q^{2r-4}+q^{r-1}+q^{r-2}+1) \\
A_{(Q,3),\dd} = & \ \frac{(q^{r-1}-1) (q^{r}-1)}{(q-1)^{2} (q+ 1)}\cdot (q^{4r-8}+q^{3r-5}+q^{3r-6}+q^{2r-2}+q^{2r-3}+q^{2r-4}+q^{r-1}+q^{r-2}+1) \\
A_{(Q,4),\dd} = & \ \frac{(q^{r-1}-1) (q^{r}-1)}{(q-1)^{2} (q+ 1)}\cdot (q^{6r-12}+q^{5r-9}+q^{5r-10}+q^{4r-6}+q^{4r-7}+q^{4r-8}+q^{3r-3}+q^{3r-4} \\
& \ +q^{3r-5}+q^{3r-6}+q^{2r-2}+q^{2r-3}+q^{2r-4}+q^{r-1}+q^{r-2}+1) \\
A_{(Q,5),\dd} = & \ \frac{(q^{r-1}-1) (q^{r}-1)}{(q-1)^{2} (q+ 1)}\cdot (q^{8r-16}+q^{7r-13}+q^{7r-14}+q^{6r-10}+q^{6r-11}+q^{6r-12}+q^{5r-7}+q^{5r-8} \\
& \ +q^{5r-9}+q^{5r-10}+q^{4r-4}+q^{4r-5}+q^{4r-6}+q^{4r-7}+q^{4r-8}+q^{3r-3}+q^{3r-4} \\
& \ +q^{3r-5}+q^{3r-6}+q^{2r-2}+q^{2r-3}+q^{2r-4}+q^{r-1}+q^{r-2}+1) .
\end{split}
\]These are all polynomials with non-negative coefficients, so Conjecture
\ref{Conj/PosKacPolMult} holds for this example as well.

\end{exmp}

\pagebreak{}

\section{Counts of jets, rational singularities and p-adic integrals \label{Chap/RatSg}}

In this Chapter, we study the asymptotic behaviour of jet-counts on
$\mu_{Q,\dd}^{-1}(0)$. We pursue two goals: to generalise Wyss' results
when $\dd=\underline{1}$ \cite{Wys17b} and to provide a geometric
interpretation of the limit of jet-counts (when they converge). The
results in this Chapter all appear in the preprint \cite{Ver22}.

It turns out that both questions are related to singularities of $\mu_{Q,\dd}^{-1}(0)$,
as was shown by Aizenbud, Avni and Glazer in a series of articles
\cite{AA16,AA18,Gla19}. More precisely, when $\mu_{Q,\dd}^{-1}(0)$
is locally complete intersection, one needs to prove that $\mu_{Q,\dd}^{-1}(0)$
has rational singularities. Building on the connection between rational
singularities and p-adic integrals uncovered by Aizenbud and Avni
\cite{AA16}, we prove that jet-counts on mildly singular schemes
converge to the canonical p-adic volume introduced in Section \ref{Subsect/p-adic}
(see Example \ref{Exmp/CanMeas}).

\begin{thm} \label{Thm/Ch3JetCountCanMeas}

Let $X$ be a $\mathbb{Z}$-scheme of finite type and assume that
$X_{\bar{\QQ}}$ is locally complete intersection, of pure dimension
$d$ and has rational singularities. Let $F$ be a local field of
characteristic zero, with valuation ring $\mathcal{O}$ and residue
field $\FF_{q}$ (of characteristic $p$). Then if $p$ is large enough,
the sequence $q^{-nd}\cdot\sharp X(\mathbb{F}_{q}[t]/(t^{n})),\ n\geq1$
converges and its limit is given by:\[
\underset{n\rightarrow +\infty}{\lim}\frac{\sharp X(\mathbb{F}_{q}[t]/(t^{n}))}{q^{nd}}
=
\nu_{\mathrm{can}}(X^{\natural}).
\]\end{thm}

We then prove that $\mu_{Q,\dd}^{-1}(0)$ has l.c.i. and rational
singularities for a large class of pairs $(Q,\dd)$, satisfying a
certain property (P) (see Definition \ref{Def/Prop(P)}). As a consequence,
we obtain that the jet-counts studied by Wyss converge to a p-adic
integral in many cases where $\dd>\underline{1}$.

\begin{thm} \label{Thm/Ch3RatSgTotNeg}

Let $Q$ be a quiver and $\dd\in\ZZ_{\geq0}^{Q_{0}}\setminus\{0\}$
such that $(Q,\dd)$ has property (P). Then $\mu_{Q,\dd}^{-1}(0)$
has rational singularities.

\end{thm}

Finally, we extend these results to a larger class of moduli spaces
which are locally modelled on $\mu_{Q,\dd}^{-1}(0)$, based on work
of Davison \cite{Dav21a}. These are quotient stacks parametrising
objects in 2-Calabi-Yau categories. Examples include categories of
modules over multiplicative preprojective algebras or coherent sheaves
on K3 surfaces. Property (P) of the local model $\mu_{Q,\dd}^{-1}(0)$
then translates into a homological property of the category at hand,
which we call total negativity.

\begin{thm} \label{Thm/Ch3RatSg2CYMod}

Let $\mathfrak{M}=[X/G]$ be a quotient stack parametrising objects
in a totally negative, 2-Calabi-Yau category. Then $X$ is locally
complete intersection and has rational singularities.

\end{thm}

\begin{thm} \label{Thm/Ch3p-adicVol2CYMod}

Consider a quotient stack $[X/G]$, defined over $\QQ$ and which
parametrises objects in a totally negative, 2-Calabi-Yau category.
Let $F$ be a local field of characteristic zero, with valuation ring
$\mathcal{O}$ and residue field $\FF_{q}$ (of characteristic $p$).

Then for $p$ large enough, $X$ is defined over $\mathcal{O}$ and
the canonical measure $\nu_{\mathrm{can}}$ on $X^{\natural}=X^{\mathrm{sm}}(F)\cap X(\mathcal{O})$
is well-defined. Moreover, the sequence $q^{-nd}\cdot\sharp X(\mathbb{F}_{q}[t]/t^{n}),\ n\geq1$
converges and its limit is given by:\[
\underset{n\rightarrow +\infty}{\lim}\frac{\sharp X(\mathbb{F}_{q}[t]/(t^{n}))}{q^{nd}}
=
\nu_{\mathrm{can}}(X^{\natural}).
\]\end{thm}

Throughout this Chapter, we will denote by $X\git G$ the good categorical
quotient (or GIT quotient) \cite[\S 0.2, Rem. 6]{MFK94}\cite[\S 6.1.]{Dol03}
of a $\KK$-variety $X$ by a reductive group $G$ in the sense of
Geometric Invariant Theory, when it exists. Here $\KK$ is algebraically
closed. When $X=R(Q,\dd)$ or $X=\mu_{Q,\dd}^{-1}(0)$ and $G=\GL_{\dd}$,
the good categorical recovers the moduli spaces introduced in Section
\ref{Subsect/MomMap}. \index[notations]{x@$X\git G$ - good categorical quotient}

\subsection{Singularities and counts of jets \label{Subsect/Sing}}

In this section, we recall the definitions of certain classes of singularities
and of jet schemes. We also recall criteria for a scheme to have rational
singularities, in terms of its jet schemes. These criteria come in
two flavours: one, due to Musta\c{t}\v{a} \cite{Mus01}, is geometric
and relies on the dimensions of jet schemes, while the other, due
to Aizenbud, Avni and Glazer \cite{AA18,Gla19}, is arithmetic and
relies on jet-counts over finite fields.

\paragraph*{Singularities and jet schemes}

We begin with the definition of local complete intersection and rational
singularities, assuming $\KK$ is algebraically closed. Then we recall
the notion of jet schemes and state a criterion of Musta\c{t}\v{a}
\cite{Mus01} for a local complete intersection to have rational singularities.
Throughout, $X$ will be a $\KK$-scheme of finite type.

\begin{df}

The scheme $X$ is locally complete intersection (or l.c.i. for short)
if it can be covered by affine open subsets which are complete intersections
in some affine space. In that case, we say that $X$ has l.c.i. singularities.

\end{df}

Having l.c.i. singularities is a property of local rings: $X$ is
locally complete intersection if, and only if, for all $x\in X$,
$\mathcal{O}_{X,x}$ is a complete intersection ring. Moreover, this
is a local property for the smooth topology:

\begin{prop}{\cite[\href{https://stacks.math.columbia.edu/tag/069P}{Tag 069P}]{SP}}
\label{Lem/LciSgDesc}

Let $f:X\rightarrow Y$ be a smooth morphism between $\KK$-schemes
of finite type. If $Y$ has l.c.i. singularities, then $X$ has l.c.i.
singularities. The converse holds if $f$ is surjective.

\end{prop}

We now turn to rational singularities. We further assume that $\KK$
is of characteristic zero.

\begin{df}

$X$ has rational singularities if for some (hence for all) resolution
of singularities $p:\tilde{X}\rightarrow X$, the natural morphism
$\mathcal{O}_{X}\rightarrow\textbf{R}p_{*}\mathcal{O}_{\tilde{X}}$
is an isomorphism. In other words, the canonical morphism $\mathcal{O}_{X}\rightarrow p_{*}\mathcal{O}_{\tilde{X}}$
is an isomorphism and $\textbf{R}^{i}p_{*}\mathcal{O}_{\tilde{X}}=0$
for all $i>0$. A point $x\in X$ is called a rational singularity
if there exists a Zariski-open neighbourhood $U\ni x$ which has rational
singularities.

\end{df}

We will also use a relative notion of rational singularities, introduced
by Aizenbud and Avni \cite{AA16}, in our proof of Theorem \ref{Thm/Ch3JetCountCanMeas}.

\begin{df}{\cite[Def. II.]{AA16}}

Let $f:X\rightarrow Y$ be a morphism between smooth, irreducible
varieties over $\KK$ (not necessarily algebraically closed). The
morphism $f$ is called FRS (flat with rational singularities) if
it is flat and for every $y\in Y(\bar{\KK})$, the fibre $X\times_{Y}y$
has rational singularities.

\end{df}

If $X$ has rational singularities, then it is normal (hence reduced)
and Cohen-Macaulay - see \cite[II.1.]{Elk78}. While the above definition
might seem quite abstract, one can show directly that having rational
singularities is a local property with respect to smooth morphisms.
This is surely common knowledge for the experts, but we include a
proof for the reader's convenience, as we could not find one in the
literature.

\begin{lem} \label{Lem/RatSgDesc}

Let $f:X\rightarrow Y$ be a smooth morphism between schemes of finite
type. If $Y$ has rational singularities, then $X$ has rational singularities.
The converse holds if $f$ is surjective.

\end{lem}

\begin{proof}

Consider $p_{Y}:\tilde{Y}\rightarrow Y$ a resolution of $Y$ and
$\tilde{X}:=X\times_{Y}\tilde{Y}$. Then, since $f$ is smooth, the
canonical morphism $p_{X}:\tilde{X}\rightarrow X$ is also a resolution
of singularities, so that we get the following cartesian diagram,
where horizontal maps are resolutions of singularities and the vertical
maps are smooth (hence flat):\[
\begin{tikzcd}[ampersand replacement = \&]
\tilde{X}\ar[r, "p_X"]\ar[d, "\tilde{f}"] \& X \ar[d, "f"] \\
\tilde{Y}\ar[r, "p_Y"] \& Y
\end{tikzcd}
\]Suppose that $Y$ has rational singularities. Then flat base change
yields:\[
f^{*}\textbf{R}(p_{Y})_*\mathcal{O}_{\tilde{Y}}\simeq\textbf{R}(p_{X})_*\tilde{f^{*}}\mathcal{O}_{\tilde{Y}}\simeq\textbf{R}(p_{X})_*\mathcal{O}_{\tilde{X}}.
\]Since by assumption $\textbf{R}(p_{Y})_{*}\mathcal{O}_{\tilde{Y}}\simeq\mathcal{O}_{Y}$,
we obtain $\mathcal{O}_{X}\simeq f^{*}\mathcal{O}_{Y}\simeq\textbf{R}(p_{X})_{*}\mathcal{O}_{\tilde{X}}$,
hence $X$ has rational singularities.

Conversely, suppose that $X$ has rational singularities and $f$
is surjective. Then flat base change and the rational singularities
assumption for $X$ yield:\[
f^{*}\textbf{R}(p_{Y})_*\mathcal{O}_{\tilde{Y}}\simeq\textbf{R}(p_{X})_*\tilde{f^{*}}\mathcal{O}_{\tilde{Y}}\simeq\textbf{R}(p_{X})_*\mathcal{O}_{\tilde{X}}\simeq\mathcal{O}_{X}\simeq f^{*}\mathcal{O}_{Y}.
\]Since $f^{*}$ is exact, then we obtain, by taking cohomology sheaves:\[
f^*\textbf{R}^{i}(p_Y)_{*}\mathcal{O}_{\tilde{Y}}\simeq
\left\{
\begin{array}{ll}
f^*\mathcal{O}_{Y} & \text{, if }i=0 \ ; \\
0 & \text{, else.}
\end{array}
\right.
\]By fpqc descent ($f$ is surjective), we finally obtain:\[
\textbf{R}^{i}(p_Y)_{*}\mathcal{O}_{\tilde{Y}}\simeq
\left\{
\begin{array}{ll}
\mathcal{O}_{Y} & \text{, if }i=0 \ ; \\
0 & \text{, else.}
\end{array}
\right.
\]Thus $Y$ has rational singularities, which finishes the proof. \end{proof}

The above lemma shows in particular that having rational singularities
is an étale-local property. We will use this fact many times in Section
\ref{Subsect/RatSgTotNegQuiv} in order to transfer rational singularities
from $\mu_{Q,\dd}^{-1}(0)$ to its étale slices.

Finally, we recall the definition of jet schemes. These spaces are
strongly related to singularities of algebraic varieties, via motivic
integration (and as we will see below, p-adic integration). See for
example \cite{Mus01,Mus02,ELM04}.

\begin{df}{\cite[\S 3.2.]{CLNS18}}

Let $\KK$ be a field (of any characteristic) and $m\geq0$. Denote
by $\KK-\mathrm{CAlg}$ the category of commutative $\KK$-algebras.
Let $X$ be a finite-type $\KK$-scheme. The $m$-th jet scheme of
$X$ is the $\KK$-scheme $X_{m}$ representing the following functor
of points:\[
\begin{array}{rcl}
\KK-\mathrm{CAlg} & \rightarrow & \mathrm{Sets} \\
R & \mapsto & X(R[t]/(t^{m+1})).
\end{array}
\]

\end{df}

We now state a criterion by Musta\c{t}\v{a}, which characterises
rational singularities for locally complete intersection varieties.
This criterion is key to Budur's proof of rational singularities for
moment maps of g-loop quivers.

\begin{prop}{\cite[Thm. 0.1. \& Prop. 1.4.]{Mus01}} \label{Prop/MustCrit}

Let $X$ be a locally complete intersection variety, $X_{\sg}$ be
its singular locus and $\pi_{m}:X_{m}\rightarrow X$ its $m$-th jet
scheme. Then $X$ has rational singularities if, and only if, for
all $m\geq1$, $\dim\pi_{m}^{-1}(X_{\sg})<(m+1)\cdot\dim(X)$.

\end{prop}

\paragraph*{Counts of jets and resolutions of singularities}

We now turn to counts of jets over finite fields and their relation
to embedded resolutions of singularities of the underlying scheme.
We first exploit Denef's formula to find the growth rate of jet-counts,
following \cite{VZG08,Wys17b}. Then we recall results of Aizenbud,
Avni and Glazer \cite{AA18,Gla19}, which relate the growth of jet-counts
with rational singularities.

Consider $X=V(f_{1},\ldots,f_{m})\subseteq\mathbb{A}_{\ZZ}^{r}$ a
$\ZZ$-scheme of finite type. Fix $F$ a local field of characteristic
zero with valuation ring $\mathcal{O}$ and residue field $\FF_{q}$.
Then we know from Section \ref{Subsect/p-adic} that the counts $\sharp X(\mathcal{O}/\mathfrak{m}^{n})$
are encoded in Igusa's local zeta function $Z_{f}(s)$. Of course,
we are interested in the counts $\sharp X(\mathbb{F}_{q}[t]/(t^{n}))$
rather than $\sharp X(\mathcal{O}/\mathfrak{m}^{n})$. The following
result by Aizenbud and Avni\footnote{This result follows from transfer results for p-adic integrals, see
\cite{CL05,CL10}. Therefore, it holds for any choice of $F$, as
opposed to \cite{AA18}, where the authors only work with unramified
extensions of $\mathbb{Q}_{p}$.} tells us that, up to working in large enough characteristic, the
two counts actually coincide.

\begin{prop}{\cite[Prop. 3.0.2.]{AA18}} \label{Prop/TsfPrinc}

Let $X$ be a $\mathbb{Z}$-scheme of finite type. There is a finite
set of primes $S$ such that, for any $p\notin S$, for any $q$ power
of $p$ and for any $n\geq1$, $\sharp X(\mathbb{F}_{q}[t]/(t^{n}))=\sharp X(\mathcal{O}/\mathfrak{m}^{n})$.

\end{prop}

The radius of convergence $\rho=q^{-R}>0$ of $P_{f}(T)$ should be
thought of as the growth rate of the jet-counts $\sharp X(\mathbb{F}_{q}[t]/(t^{n}))$.
Consider $s_{0}=\max\{\Rea(s)\ \vert\ q^{-s}\text{ is a pole of }Z_{f}\}$.
Here, $Z_{f}$ is seen as a rational fraction in $T=q^{-s}$. By Denef's
formula (see Section \ref{Subsect/p-adic}), if $q^{-s}$ is a pole
of $Z_{f}$, then $s$ has negative real part - hence $s_{0}$ is
well-defined. Since $(1-q^{-s})\cdot P_{f}(q^{-r-s})=1-q^{-s}\cdot Z_{f}(s)$,
we obtain $R=r+s_{0}$. It turns out that $s_{0}$ coincides with
an important invariant in singularity theory. Given an embedded resolution
of $X\subseteq\mathbb{A}_{\QQ}^{r}$ with numerical data $(N_{i},\nu_{i})_{1\leq i\leq t}$,
the log-canonical threshold of the pair $(\mathbb{A}_{\QQ}^{r},X)$
is defined as $\mathrm{lct}(\mathbb{A}_{\mathbb{Q}}^{r},X_{\mathbb{Q}})=\min_{1\leq i\leq t}\left\{ \frac{\nu_{i}}{N_{i}}\right\} $.
This does not depend on the choice of an embedded resolution (see
for instance \cite{Mus02}). Now suppose that $X$ admits an embedded
resolution with good reduction modulo $p$. By \cite[Thm. 2.7.]{VZG08},
for $q$ a large enough power of $p$, $s_{0}=-\mathrm{lct}(\mathbb{A}_{\mathbb{Q}}^{r},X_{\mathbb{Q}})$.
Note that the proof of \cite[Thm. 2.7.]{VZG08} requires the existence
of a generic $F$-point on a divisor $(E_{i})_{F}$ satisfying $\mathrm{lct}(\mathbb{A}_{\mathbb{Q}}^{r},X_{\mathbb{Q}})=\frac{\nu_{i}}{N_{i}}$
(see also \cite[Rmk. 2.8.2.]{VZG08}). We can assume this is true
by taking a finite extension of $F$, which explains our assumption
on $q$.

Thus the growth rate of $\sharp X(\mathbb{F}_{q}[t]/(t^{n}))$ is
given by $R=r-\mathrm{lct}(\mathbb{A}_{\mathbb{Q}}^{r},X_{\mathbb{Q}})$.
We will focus below on cases where $R=\dim X_{\mathbb{Q}}$. This
can be thought of as a bound on the dimension of jet schemes (see
\cite[Cor. 0.1.]{Mus02}). In \cite{Wys17b}, Wyss interprets the
limit when $n$ goes to infinity of $q^{-n\dim X_{\mathbb{Q}}}\cdot\sharp X(\mathbb{F}_{q}[t]/(t^{n}))$
as a residue of $Z_{f}(s)$. Summing up results of \cite{Wys17b,AA18,Gla19},
we obtain:

\begin{prop} \label{Prop/JetCountLim}

Set $X=V(f_{1},\ldots,f_{m})\subseteq\mathbb{A}_{\mathbb{Z}}^{r}$
as above and assume that $X_{\bar{\QQ}}$ is an equidimensional local
complete intersection, of codimension $c$ in $\mathbb{A}_{\bar{\QQ}}^{r}$.
Then the following are equivalent:
\begin{enumerate}
\item $X_{\bar{\QQ}}$ has rational singularities,
\item For almost all primes $p$, the sequence $q^{-n\dim X_{\mathbb{Q}}}\cdot\sharp X(\mathbb{F}_{q}[t]/(t^{n})),\ n\geq1$
converges for any finite field $\mathbb{F}_{q}$ of characteristic
$p$.
\end{enumerate}
If 1. and 2. hold, then there is a finite set of primes $S$ such
that for $p\notin S$, $q$ a \textit{large enough} prime power of
$p$ and $F$ any finite extension of $\mathbb{Q}_{p}$ with residue
field $\mathbb{F}_{q}$, $Z_{f}(s)$ has its poles of largest real
part at $\Rea(s)=-c$ and:\[
\underset{n\rightarrow +\infty}{\lim}\frac{\sharp X(\mathbb{F}_{q}[t]/(t^{n}))}{q^{n\dim X_{\mathbb{Q}}}}=\frac{-1}{q^c-1}\cdot\Res_{T=q^c}Z_f.
\]

\end{prop}

\begin{proof}

The implication 2. $\Rightarrow$ 1. is a consequence of the implication
(v) $\Rightarrow$ (iii) from \cite[Thm. 4.1.]{Gla19}. Glazer's result
involves $\sharp X(\mathcal{O}/\mathfrak{m}^{n})$ instead of $\sharp X(\mathbb{F}_{q}[t]/(t^{n}))$,
where $\mathcal{O}$ is the valuation ring of the unramified extension
of $\mathbb{Q}_{p}$ with residue field $\mathcal{O}/\mathfrak{m}\simeq\mathbb{F}_{q}$.
However, for a given $X$, these counts are equal for almost all characteristics
by Proposition \ref{Prop/TsfPrinc}.

Conversely, suppose 1. Then by implication (iii) $\Rightarrow$ (iv')
from \cite[Thm. 4.1.]{Gla19}, for almost all primes $p$, the sequence
$q^{-n\dim X_{\mathbb{Q}}}\cdot\sharp X(\mathbb{F}_{q}[t]/(t^{n})),\ n\geq1$
is bounded for any finite field $\mathbb{F}_{q}$ of characteristic
$p$. Note that we used again Proposition \ref{Prop/TsfPrinc} in
order to replace $\sharp X(\mathcal{O}/\mathfrak{m}^{n})$ by $\sharp X(\mathbb{F}_{q}[t]/(t^{n}))$
in the conclusion of Glazer's theorem.

Moreover, since $X_{\overline{\mathbb{Q}}}$ has rational singularities,
$Z_{f}$ has its poles of largest real part at $\text{Re}(s)=-\mathrm{lct}(\mathbb{A}_{\mathbb{Q}}^{r},X_{\mathbb{Q}})=-c$.
Here, the second equality follows from \cite[Prop. 1.4.]{Mus01} and
\cite[Cor. 0.2.]{Mus02}. Then we can apply \cite[Lem. 4.7.]{Wys17b}
and conclude that the sequence $q^{-n\dim X_{\mathbb{Q}}}\cdot\sharp X(\mathbb{F}_{q}[t]/(t^{n})),\ n\geq1$
converges and that its limit is given by the residue formula above.
\end{proof}

\begin{rmk}

The proof by Aizenbud and Avni \cite[Thm. 3.0.3.]{AA18} of the implication
1. $\Rightarrow$ 2. in Proposition \ref{Prop/JetCountLim} relies
on the Lang-Weil estimates, Denef's explicit formula for $Z_{f}(s)$
and a characterisation by Musta\c{t}\u{a} of top-dimensional irreducible
components of jet schemes of $X_{\bar{\QQ}}$ in terms of the quantities
$(N_{i},\nu_{i})$ \cite[Thm. 3.2.]{Mus01} (when $X_{\bar{\QQ}}$
is irreducible). None of these require that $X_{\bar{\QQ}}$ be locally
complete intersection. For l.c.i. varieties, irreducibility of jet
schemes is implied by the fact that $X$ has rational singularities
\cite[Thm. 3.3.]{Mus01}. But irreducibility of jet schemes also follows
directly from the assumption of Proposition \ref{Prop/MustCrit} on
the dimension of jet schemes over the singular locus of $X_{\bar{\QQ}}$.
Hence, even if $X_{\bar{\QQ}}$ is not l.c.i., if $\dim\pi_{m}^{-1}(X_{\sg})<(m+1)\cdot\dim(X)$
for all $m\geq1$, then the sequence $q^{-n\dim X_{\mathbb{Q}}}\cdot\sharp X(\mathbb{F}_{q}[t]/(t^{n})),\ n\geq1$
converges.

Instead, our motivation for working with l.c.i. varieties has to do
with the fact that the above properties of jet schemes become étale-local
under this assumption (by Proposition \ref{Lem/RatSgDesc}, since
they correspond to rational singularities). Moreover, l.c.i. varieties
are the appropriate setting to compute counts of jets as p-adic integrals
(see the proof of Theorem \ref{Thm/Ch3JetCountCanMeas} below).

\end{rmk}

\subsection{Counts of jets on local complete intersections and p-adic volumes
\label{Subsect/JetCount=000026p-adicVol}}

In this section, we provide a geometric interpretation for the limit
of jet-counts over finite fields. Our result holds for $\ZZ$-schemes
$X$ such that $X_{\overline{\mathbb{Q}}}$ has l.c.i. and rational
singularities. Using a result of Aizenbud and Avni \cite{AA16}, we
show that counts of jets converge to a certain p-adic integral over
the analytification of $X^{\mathrm{sm}}$. We then identify this p-adic
integral to the canonical volume of $X$ built in Section \ref{Subsect/p-adic}.

Let $F$ be a local field of characteristic zero, with valuation ring
$\mathcal{O}$ and residue field $\mathbb{F}_{q}$. Consider an $\mathcal{O}$-scheme
$X$ and $X^{\natural}:=X^{\mathrm{sm}}(F)\cap X(\mathcal{O})$. Suppose
that the structure morphism $X\rightarrow\Spec(\mathcal{O})$ is Gorenstein,
of pure dimension $d$ (see \cite[\href{https://stacks.math.columbia.edu/tag/0C02}{Tag 0C02}]{SP}).
Then $X$ admits a canonical invertible sheaf $\Omega_{X/\mathcal{O}}$,
which restricts to $\Omega_{X^{\mathrm{sm}}/\mathcal{O}}^{d}$ on
the smooth locus $X^{\mathrm{sm}}$ (see \cite[\S 3.5.]{Con00}).
Thus we can build a canonical measure $\nu_{\mathrm{can}}$ as in
Example \ref{Exmp/CanMeas}. For $\ZZ$-schemes $X$ such that $X_{\overline{\mathbb{Q}}}$
has l.c.i. and rational singularities, we show that the jet-counts
we consider converge to $\nu_{\mathrm{can}}(X^{\natural})$, when
the residual characteristic of $F$ is large enough. Let us first
state an intermediate result for appropriate $\mathcal{O}$-schemes:

\begin{lem} \label{Lem/JetCountCanMeas}

Let $X$ be a finite type $\mathcal{O}$-scheme. Assume that $X$
is flat over $\mathcal{O}$, with l.c.i. geometric fibers of pure
dimension $d$. Assume also that $X_{\bar{F}}$ has rational singularities.

Then there is a well-defined canonical measure $\nu_{\mathrm{can}}$
on $X^{\natural}=X^{\mathrm{sm}}(F)\cap X(\mathcal{O})$. Moreover,
the sequence $q^{-nd}\cdot\sharp X(\mathcal{O}/(\varpi^{n})),\ n\geq1$
converges and its limit is given by:\[
\underset{n\rightarrow +\infty}{\lim}\frac{\sharp X(\mathcal{O}/(\varpi^{n}))}{q^{nd}}
=
\nu_{\mathrm{can}}(X^{\natural}).
\]\end{lem}

\begin{proof}

Since $X\rightarrow\Spec(\mathcal{O})$ is flat, with l.c.i. geometric
fibers, it is a Gorenstein morphism (of pure dimension $d$, by assumption
- see \cite[\href{https://stacks.math.columbia.edu/tag/0C02}{Tag 0C02}]{SP}).
Then the construction above applies and $X^{\natural}$ can be endowed
with a canonical measure $\nu_{\mathrm{can}}$.

We now describe $X$ locally as a fiber of an appropriate morphism
$\varphi:\mathbb{A}_{\mathcal{O}}^{r}\rightarrow\mathbb{A}_{\mathcal{O}}^{m}$.
On the one hand, this allows us to rewrite counts of $\mathcal{O}/(\varpi^{n})$-points
as measures of $p$-adic balls with respect to the pushforward by
$\varphi$ of the Haar measure on $\mathcal{O}^{r}$, following \cite[\S 4.2.]{Wys17b}.
On the other hand, if $\varphi$ is a FRS morphism, we can apply a
Fubini theorem by Aizenbud and Avni \cite[Thm. 3.16.]{AA16} to compute
this pushforward measure and relate it to the canonical measure on
$X$.

As $X\rightarrow\Spec(\mathcal{O})$ is flat, with l.c.i. geometric
fibers, we can cover $X$ with affine open subsets of the form: \[
U_i=\Spec(\mathcal{O}[T_1,\ldots,T_{r_i}]/(f_{1,i},\ldots,f_{m_i,i})),
\] whose nonempty fibers over $\mathcal{O}$ have dimension $d=r_{i}-m_{i}$
(see \cite[\href{https://stacks.math.columbia.edu/tag/01UB}{Tag 01UB}]{SP}).
Observe that $U_{i}$ is of pure dimension $d$ over $\mathcal{O}$
if, and only if, $(U_{i})_{\mathbb{F}_{q}}\ne\emptyset$, as $U_{i}$
is flat over $\mathcal{O}$. Moreover, $X^{\natural}=\bigcup_{i}U_{i}^{\natural}$
as noted above, and $\sharp X(\mathcal{O}/(\varpi^{n}))$ can be obtained
from $\sharp U_{i}(\mathcal{O}/(\varpi^{n}))$ by inclusion-exclusion
(see \cite[Lem. 3.1.1.]{AA18}). Therefore, the $U_{i}$ whose contribution
to $\nu_{\mathrm{can}}(X^{\natural})$ and $\sharp X(\mathcal{O}/(\varpi^{n}))$
is non-zero are of pure dimension $d$ over $\mathcal{O}$ and we
may restrict to those for the rest of the proof.

Let $U_{i}$ be such an open subset. We temporarily drop the subscript
$i$ to ease notations. Consider $\varphi:=(f_{1},\ldots,f_{m}):\mathbb{A}_{\mathcal{O}}^{r}\rightarrow\mathbb{A}_{\mathcal{O}}^{m}$.
The locus $M=\{x\in\mathbb{A}_{\mathcal{O}}^{r}\ \vert\ \dim_{x}(\varphi^{-1}(\varphi(x)))\leq d\}$
is open, by Chevalley's semicontinuity theorem \cite[Thm. 13.1.3.]{Gro66},
and contains $U=\varphi^{-1}(0)$. Moreover, $\dim_{x}(\varphi^{-1}(\varphi(x)))\geq r-m=d$
by construction, so $\dim_{x}(\varphi^{-1}(\varphi(x)))=d$ for all
$x\in M$. Therefore, $\varphi\vert_{M}$ is flat (see \cite[Prop. 6.1.5.]{Gro65})
and $N:=\varphi(M)\subseteq\mathbb{A}_{\mathcal{O}}^{m}$ is open.
We have thus obtained a morphism $\varphi:M\rightarrow N$ which is
flat, with l.c.i. geometric fibers of pure dimension $d$, and such
that $U=\varphi^{-1}(0)$.

Let us now relate counts of $\mathcal{O}/(\varpi^{n})$-points on
$U$ to the pushforward by $\varphi$ of the Haar measure on $\mathcal{O}^{r}$.
Let $\omega_{M}\in\Gamma(M,\Omega_{M/\mathcal{O}}^{r})$ (resp. $\omega_{N}\in\Gamma(N,\Omega_{N/\mathcal{O}}^{m})$)
be a gauge form and $\nu_{M}$ (resp. $\nu_{N}$) the associated measure
on $M^{\natural}$ (resp. $N^{\natural}$). We also call $\varphi:M^{\natural}\rightarrow N^{\natural}$
the map of $F$-analytic manifolds induced by $\varphi$. Then $\varpi^{n}\mathcal{O}^{m}\subseteq N^{\natural}$
and\[
\frac{\sharp U(\mathcal{O}/(\varpi^n))}{q^{nd}}
=
\sharp\{x\in M(\mathcal{O}/(\varpi^n))\ \vert\ \varphi(x)=0\in(\mathcal{O}/\varpi^n)^m\}\cdot\frac{q^{-nr}}{q^{-nm}}
=
\frac{\varphi_*\nu_M(\varpi^n\mathcal{O}^m)}{\nu_N(\varpi^n\mathcal{O}^m)}.
\]As shown in \cite[\S 3.3.]{AA16}, the measure $\varphi_{*}(\nu_{M})$
is absolutely continuous with respect to $\nu_{N}$ and its density
at $y\in N^{\natural}$ is given by a $p$-adic integral on $\varphi^{-1}(y)^{\natural}$.
We now build explicit gauge forms inducing the corresponding measures
on $\varphi^{-1}(y)^{\natural}$. As before, $\varphi$ is a Gorenstein
morphism of pure dimension $d$, so there exists a canonical invertible
sheaf $\Omega_{M/N}$, which restricts to $\Omega_{M^{\mathrm{sm}}/N}^{d}$
on the smooth locus $M^{\mathrm{sm}}$ of $\varphi$. Moreover, there
is an isomorphism of invertible sheaves $\Omega_{M/\mathcal{O}}^{r}\simeq\Omega_{M/N}\otimes\varphi^{*}\Omega_{N/\mathcal{O}}^{m}$,
by \cite[Thm. 4.3.3.]{Con00}. From this we obtain a nowhere-vanishing
section $\eta\in\Gamma(M,\Omega_{M/N})$ such that $\omega_{M}=\eta\otimes\varphi^{*}\omega_{N}$
and which restricts to a nowhere-vanishing section $\eta_{y}\in\Gamma(\varphi^{-1}(y),\Omega_{\varphi^{-1}(y)/\mathcal{O}})$
for any $y\in N^{\natural}$, by \cite[Thm. 3.6.1.]{Con00}. Let us
call $\nu_{y}$ the measure induced by the gauge form $\eta_{y}\vert_{\varphi^{-1}(y)^{\mathrm{sm}}}$
on $\varphi^{-1}(y)^{\natural}$. When $y=0$, we simply write $\eta:=\eta_{0}$
and $\nu:=\nu_{0}$ for the induced measure on $U^{\natural}=\varphi^{-1}(0)^{\natural}$.
By construction, $\nu$ is the restriction of $\nu_{\mathrm{can}}$
to $U^{\natural}$.

The next step is to apply \cite[Thm. 3.16.]{AA16}. We first need
to restrict $\nu_{M}$ to a locus where $\varphi_{F}$ is FRS. Since
$\varphi_{F}$ is flat, the locus where geometric fibers of $\varphi_{F}$
have rational singularities is open - see  \cite[Thm. 4]{Elk78} -
and contains $U_{F}$. Denote by $Z_{F}\subseteq M_{F}$ the complement
of this locus. Then $Z_{F}(F)\cap U(F)=\emptyset$. Since $M^{\natural}$
and $N^{\natural}$ are compact, $\varphi$ is closed and there exists
some $n_{0}\geq0$ such that $Z_{F}(F)\cap\varphi^{-1}(\varpi^{n_{0}}\mathcal{O}^{m})=\emptyset$.
Consequently, $\varphi_{F}$ is FRS on $M_{F}\setminus Z_{F}$ and
$A:=\varphi^{-1}(\varpi^{n_{0}}\mathcal{O}^{m})$ is a compact subset
of $(M_{F}\setminus Z_{F})(F)$. By \cite[Thm. 3.16.]{AA16}, the
density of $\varphi_{*}\left(\nu_{M}\vert_{A}\right)$ with respect
to $\nu_{N}$ is given by the function $y\mapsto\mathbf{1}_{\varpi^{n_{0}}\mathcal{O}^{m}}(y)\cdot\nu_{y}(\varphi^{-1}(y)^{\natural})$,
which is continuous. Therefore, for $n\geq n_{0}$:\begin{align*}
\frac{\sharp U(\mathcal{O}/(\varpi^n))}{q^{nd}}
&=
\frac{\varphi_*\nu_M(\varpi^n\mathcal{O}^m)}{\nu_N(\varpi^n\mathcal{O}^m)}
=
\frac{\varphi_*\left(\nu_M\vert_{A}\right)(\varpi^n\mathcal{O}^m)}{\nu_N(\varpi^n\mathcal{O}^m)} \\
&=
\frac{\int_{\varpi^n\mathcal{O}^m}\nu_{y}(\varphi^{-1}(y)^{\natural})d\nu_N}{\int_{\varpi^n\mathcal{O}^m}d\nu_N}
\underset{n\rightarrow+\infty}{\longrightarrow}
\nu_0(\varphi^{-1}(0)^{\natural})=\nu(U^{\natural}).
\end{align*}Finally, we compute $\nu_{\mathrm{can}}(X^{\natural})$ from $\nu_{i}(U_{i}^{\natural})$
by inclusion-exclusion. Counts of $\mathcal{O}/(\varpi^{n})$-points
on $X$ can also be computed by inclusion-exclusion, as mentioned
above. Adding up the contributions of all $U_{i}$, we obtain:\[
\underset{n\rightarrow +\infty}{\lim}\frac{\sharp X(\mathcal{O}/(\varpi^{n}))}{q^{nd}}
=
\nu_{\mathrm{can}}(X^{\natural}).
\]\end{proof}

We are now in a position to prove our main technical result for a
large class of $\mathbb{Z}$-schemes of finite type. We show how to
relate the counts $\sharp X(\mathbb{F}_{q}[t]/(t^{n}))$ to the canonical
measure on $X^{\natural}=X^{\mathrm{sm}}(F)\cap X(\mathcal{O})$ when
the residual characteristic is large enough.

\begin{thm} \label{Thm/JetCountCanMeas}

Let $X$ be a $\mathbb{Z}$-scheme of finite type and assume that
$X_{\overline{\mathbb{Q}}}$ is locally complete intersection, of
pure dimension $d$ and has rational singularities. Then for $p$
large enough, the $\mathcal{O}$-scheme $X_{\mathcal{O}}$ satisfies
the assumptions of Lemma \ref{Lem/JetCountCanMeas}. Moreover, the
sequence $q^{-nd}\cdot\sharp X(\mathbb{F}_{q}[t]/(t^{n})),\ n\geq1$
converges and its limit is given by:\[
\underset{n\rightarrow +\infty}{\lim}\frac{\sharp X(\mathbb{F}_{q}[t]/(t^{n}))}{q^{nd}}
=
\nu_{\mathrm{can}}(X^{\natural}).
\]\end{thm}

\begin{proof}

Assume that $X_{\mathcal{O}}$ satisfies the hypotheses of Lemma \ref{Lem/JetCountCanMeas}.
From \cite[Prop. 3.0.2.]{AA18}, we obtain that, for $p$ large enough:\[
\underset{n\rightarrow +\infty}{\lim}\frac{\sharp X(\mathbb{F}_{q}[t]/(t^{n}))}{q^{nd}}
=
\underset{n\rightarrow +\infty}{\lim}\frac{\sharp X(\mathcal{O}/(\varpi^{n}))}{q^{nd}}
=
\nu_{\mathrm{can}}(X^{\natural}).
\]Let us now check that for $p$ large enough: (i) $X$ is flat over
$\mathcal{O}$, with l.c.i. geometric fibers of pure dimension $d$
and (ii) $X_{\overline{F}}$ has rational singularities. Statement
(ii) follows straightforwardly from the assumption on $X_{\overline{\mathbb{Q}}}$
by base change, so let us check (i). By generic flatness, $X_{\mathbb{Z}[\frac{1}{N}]}$
is flat over $\mathbb{Z}[\frac{1}{N}]$ for some $N\geq1$. Thus by
\cite[\href{https://stacks.math.columbia.edu/tag/01UE}{Tag 01UE}, \href{https://stacks.math.columbia.edu/tag/01UF}{Tag 01UF}]{SP}
and \cite[Thm. 13.1.3.]{Gro66}, the locus $U\subseteq X_{\mathbb{Z}[\frac{1}{N}]}$
where the structure morphism $X_{\mathbb{Z}[\frac{1}{N}]}\rightarrow\Spec(\mathbb{Z}[\frac{1}{N}])$
has l.c.i. geometric fibers of pure dimension $d$ is open and contains
$X_{\mathbb{Q}}$, by assumption. By \cite[Cor. 9.5.2.]{Gro66}, we
obtain that $\{s\in\Spec(\mathbb{Z})\ \vert\ U_{s}=X_{s}\}\subseteq\Spec(\mathbb{Z})$
is open i.e. for $N$ large enough, the structure morphism is flat,
with l.c.i. geometric fibers of pure dimension $d$. So the same holds
for $X_{\mathcal{O}}$ for $p\geq N$. \end{proof}

\begin{rmk}

Note that a similar result holds for a scheme $X$ defined over the
ring of integers of some number field, with the same proof.

\end{rmk}

\subsection{Rational singularities for moment maps of totally negative quivers
\label{Subsect/RatSgTotNegQuiv}}

In this section, we prove Theorem \ref{Thm/Ch3RatSgTotNeg}. The proof
closely follows the strategy developed by Budur in \cite{Bud21}:
we apply Musta\c{t}\u{a}'s criterion (Proposition \ref{Prop/MustCrit}).
The bounds on the dimensions of jet schemes of $\mu_{Q,\dd}^{-1}(0)$
are established via an inductive argument based on étale slices of
$\mu_{Q,\dd}^{-1}(0)$ and using a dimension estimate by Crawley-Boevey
for the base step. We generalise Budur's bounds to a larger class
of quivers and fix a computational gap in his proof. Indeed, the intermediate
bounds used in \cite{Bud21} fail in certain cases (see Remark \ref{Rmk/BoundFailure}),
which are adressed using the results of \cite{Wys17b} and the arithmetic
criterion of Section \ref{Subsect/Sing}.

Before giving the proof, we recall some facts on the singularities
of quiver moment maps and certain dimension estimates, all due to
Crawley-Boevey \cite{CB01,CB03a}. We also recall the local structure
of étale slices of $\mu_{Q,\dd}^{-1}(0)$ and their relations to rational
singularities, following \cite{CB03a,Bud21}. Finally, we deduce from
Theorem \ref{Thm/Ch3RatSgTotNeg} that the jet-counts we started with
converge for the class of quivers at hand. Except for the final subsection,
we work over an algebraically closed field $\KK$ of characteristic
zero.

\paragraph*{Geometry of quiver moment maps and étale slices}

Let us first define the class of pairs $(Q,\dd)$ under consideration
in Theorem \ref{Thm/Ch3RatSgTotNeg}. Recall the symmetrised Euler
form $(\bullet,\bullet)_{Q}$ from Section \ref{Subsect/QuivRep}.

\begin{df} \label{Def/Prop(P)}

A quiver $Q$ is called totally negative if for all $\dd,\ee\in\ZZ_{\geq0}^{Q_{0}}\setminus\{0\}$,
we have $(\dd,\ee)<0$. One can check that $Q$ is totally negative
if, and only if, (i) Q has at least two loops at each vertex and (ii)
any pair of vertices of Q is joined by at least one arrow.

Let also $\dd\in\ZZ_{\geq0}^{Q_{0}}\setminus\{0\}$ be a dimension
vector. We say that $(Q,\dd)$ has property (P) if:
\begin{enumerate}
\item $Q$ is totally negative and
\item if $\supp(\dd)$ has two vertices joined by only one edge, then $\dd\vert_{\supp(\dd)}\ne(1,1)$.
\end{enumerate}
\end{df}

The main geometric inputs to our proof (singularities, dimension estimates,
étale slices) rely on the analysis of simple constitutents of $\Pi_{Q}$-modules.
We first show the existence of a simple $\Pi_{Q}$-module for all
the pairs above:

\begin{prop} \label{Prop/TotNegSimp}

Suppose that $(Q,\dd)$ has property (P). Then there exists a simple
$\Pi_{Q}$-module of dimension $\dd$.

\end{prop}

\begin{proof}

Since $Q$ is totally negative, $\dd\in F_{Q}$. By contradiction,
assume that there exist no simple $\Pi_{Q}$-modules of dimension
$\dd$. By \cite[Thm. 8.1.]{CB01}, one of the following holds:
\begin{enumerate}
\item $\supp(\dd)$ is an extended Dynkin quiver with minimal imaginary
root $\delta$ and $\dd=m\delta$ for some $m\geq2$;
\item $\supp(\dd)$ splits as a disjoint union $(Q_{0})'\sqcup(Q_{0})''$
such that there is a unique arrow joining $(Q_{0})'$ and $(Q_{0})''$,
say at vertices $i'$ and $i''$, and $d_{i'}=d_{i''}=1$;
\item $\supp(\dd)$ splits as a disjoint union $(Q_{0})'\sqcup(Q_{0})''$
such that there is a unique arrow joining $(Q_{0})'$ and $(Q_{0})''$,
say at vertices $i'$ and $i''$, and $d_{i'}=1$, the restriction
of $\supp(\dd)$ to $(Q_{0})''$ is an extended Dynkin quiver with
minimal imaginary root $\delta$ and $\dd\vert_{(Q_{0})''}=m\delta$
for some $m\geq2$.
\end{enumerate}
Cases 1. and 3. cannot happen, as $Q$ has at least two loops at each
vertex. Since the graph underlying $Q$ is complete, case 2. can only
happen if $\supp(\dd)$ has two vertices joined by a single edge and
$\dd\vert_{\supp(\dd)}=(1,1)$. This is ruled out by definition of
property (P). Thus we have reached a contradiction and there exists
a simple $\Pi_{Q}$-module of dimension $\dd$. \end{proof}

From Crawley-Boevey's study of the geometric properties of $\mu_{Q,\dd}^{-1}(0)$
,we deduce:

\begin{cor}{\cite[Thm. 1.2. - Cor. 1.4.]{CB01}} \label{Prop/GeoMomMap}

If $(Q,\dd)$ has property (P), then $\mu_{Q,\dd}^{-1}(0)$ is a reduced,
irreducible complete intersection of dimension $\dd\cdot\dd-1+2\cdot(1-\langle\dd,\dd\rangle)$.
Moreover, $M_{\Pi_{Q},\dd}=\mu_{Q,\dd}^{-1}(0)\git\GL_{\dd}$ has
dimension $2\cdot(1-\langle\dd,\dd\rangle)$.

\end{cor}

We prove our main result using dimension estimates on constructible
subsets of $\mu_{Q,\dd}^{-1}(0)$. These rely on older estimates proved
by Crawley-Boevey in \cite[\S 6.]{CB03a}. for which we need some
additional notations:

\begin{df}{\cite[\S 1.]{CB01}}

Let $M$ be a semisimple $\Pi_{Q}$-module ($Q$ arbitrary), which
decomposes into non-isomorphic direct summands as follows:\[
M\simeq\bigoplus_{i=1}^rM_i^{\oplus e_i}.
\]The semisimple type of $M$ is the multiset $\tau=(\dim(M_{i}),e_{i}\ ;\ 1\leq i\leq r)$.

\end{df}

We call $\tau$ simple (resp. strictly semisimple) if it corresponds
to a simple (resp, semisimple, but not simple) representation. Similarly,
we call $x\in\mu_{Q,\dd}^{-1}(0)$ simple (resp. strictly semisimple)
if the corresponding $\Pi_{Q}$-module is. Given a dimension vector
$\dd$, we define $\tau_{\mathrm{min},\dd}:=(\epsilon_{i},d_{i}\ ;\ i\in\supp(\dd))$.
It is the semisimple type of the $\dd$-dimensional $\Pi_{Q}$-module
corresponding to $0\in\mu_{Q,\dd}^{-1}(0)$.

\begin{df}{\cite[\S 6.]{CB03a}}

Let $N_{i},\ 1\leq i\leq r$ be a collection of non-isomorphic simple
$\Pi_{Q}$-modules. A $\Pi_{Q}$-module $M$ has top-type $(j_{s},m_{s},\ 1\leq s\leq h)$
w.r.t. $N_{i},\ 1\leq i\leq r$ if it admits a filtration $0=M_{0}\subsetneq M_{1}\subsetneq\ldots\subsetneq M_{h}=M$,
such that $M_{s}/M_{s-1}\simeq N_{j_{s}}^{\oplus m_{s}}$ and $\hom(M_{s},N_{j_{s}})=m_{s}>0$.

\end{df}

Note that, if $M$ is semisimple with top-type $(j_{s},m_{s},\ 1\leq s\leq h)$
w.r.t. $N_{i},\ 1\leq i\leq r$ and all $N_{i}$ appear as subquotients
of $M$\footnote{This can always be assumed, by possibly forgetting some $N_{i}$.},
then the semisimple type of $M$ is:\[
\tau=\left(\dim(N_i),\sum_{j_s=i}m_s\ ;\ 1\leq i\leq r\right).
\]We can now state Crawley-Boevey's dimension bound:

\begin{prop}{\cite[Lem. 6.2.]{CB03a}} \label{Prop/CBDimBound}

Let $N_{i},\ 1\leq i\leq r$ be a collection of non-isomorphic simple
$\Pi_{Q}$-modules and $(j_{s},m_{s},\ 1\leq s\leq h)$ a top-type
w.r.t. $N_{i},\ 1\leq i\leq r$. Define:\[
z_s
=
\left\{
\begin{array}{ll}
0 & \text{ if }\langle\dim(N_{j_s}),\dim(N_{j_s})\rangle=1\text{ or if there is no }t<s\text{ such that } j_t=j_s, \\
m_t & \text{ otherwise, for the largest } t<s\text{ such that } j_t=j_s.
\end{array}
\right.
\]Then the subset of $\mu_{Q,\dd}^{-1}(0)$ corresponding to $\Pi_{Q}$-modules
of top-type $(j_{s},m_{s},\ 1\leq s\leq h)$ is constructible, of
dimension at most:\[
\dd\cdot\dd-1+(1-\langle\dd,\dd\rangle)+\sum_{s=1}^hm_sz_s-\sum_{s=1}^hm_s^2\cdot(1-\langle\dim(N_{j_s}),\dim(N_{j_s})\rangle).
\]\end{prop}

We now describe étale slices of $\mu_{Q,\dd}^{-1}(0)$ in terms of
semisimple type, following Crawley-Boevey \cite{CB03a} and Budur
\cite{Bud21}. We also explain how one can transfer the rational singularities
property from the étale slices to $\mu_{Q,\dd}^{-1}(0)$. This technique
is at the heart of the inductive reasoning in \cite{Bud21}, which
we extend to a larger class of quivers below.

Let $x\in\mu_{Q,\dd}^{-1}(0)$ be a semisimple point of type $\tau=(\dd_{i},e_{i}\ ;\ 1\leq i\leq r)$.
Then its stabiliser satisfies $(\GL_{\dd})_{x}\simeq\GL_{\ee}$. Luna's
étale slice theorem \cite{Lun73} then yields a $\GL_{\ee}$-invariant,
locally closed subvariety $S\subseteq\mu_{Q,\dd}^{-1}(0)$ such that
the commutative square below is cartesian, with étale horizontal maps
(the upper horizontal map is given by the action):\[
\begin{tikzcd}[ampersand replacement = \&]
(S\times^{(\GL_{\dd})_x}\GL_{\dd},[x,\Id]) \ar[r]\ar[d] \& (\mu_{Q,\dd}^{-1}(0),x) \ar[d] \\
(S\git (\GL_{\dd})_x,x) \ar[r] \& (\mu_{Q,\dd}^{-1}(0)\git\GL_{\dd},x).
\end{tikzcd}
\]Moreover, by work of Crawley-Boevey \cite{CB03a} and Budur \cite{Bud21},
the étale slice has an étale-local description in terms of a certain
pair $(Q',\ee)$, which satisfies:
\begin{itemize}
\item $(Q')_{0}=\{1,\ldots,r\}$ and $\ee_{i}:=e_{i}$,
\item $\overline{Q'}$ has $2\cdot(1-\langle\dd_{i},\dd_{i}\rangle)$ loops
at each vertex $i$ and $-(\dd_{i},\dd_{j})$ arrows from vertex $i$
to vertex $j$.
\end{itemize}
We call such a quiver an auxiliary quiver attached to semisimple type
$\tau$. Note that $Q'$ is only determined by $\tau$ up to orientation.
In what follows, we will abuse notations and denote by $Q_{\tau}$
any choice of $Q'$, as both $\overline{Q_{\tau}}$ and property (P)
for $Q_{\tau}$ do not depend on orientation.

Recall that for a $G$-variety $X$ with quotient map $q:X\rightarrow X\git G$
($G$ is a reductive group), an open subset $U\subseteq X$ is called
$G$-saturated if $q^{-1}(q(U))=U$. Then the following holds:

\begin{prop}{\cite[\S 4.]{CB03a} \cite[Thm. 2.9.]{Bud21}} \label{Prop/MomMapEtSlice}

There exists a $(\GL_{\dd})_{x}$-saturated open neighbourhood $W\subseteq S$
of $x$ and a $\GL_{\ee}$-equivariant\footnote{Using $(\GL_{\dd})_{x}\simeq\GL_{\ee}$.}
morphism $f:(W,x)\rightarrow(\mu_{Q_{\tau},\ee}^{-1}(0),0)$ such
that the commutative diagram below has cartesian squares and étale
horizontal maps:\[
\begin{tikzcd}[ampersand replacement = \&]
(\mu_{Q_{\tau},\ee}^{-1}(0)\times^{\GL_{\ee}}\GL_{\dd},[0,\Id]) \ar[d] \& (W\times^{(\GL_{\dd})_x}\GL_{\dd},[x,\Id]) \ar[l]\ar[r]\ar[d] \& (\mu_{Q,\dd}^{-1}(0),x) \ar[d] \\
(\mu_{Q_{\tau},\ee}^{-1}(0)\git\GL_{\ee},0) \& (W\git\GL_{\ee},x) \ar[l]\ar[r] \& (\mu_{Q,\dd}^{-1}(0)\git\GL_{\dd},x).
\end{tikzcd}
\]\end{prop}

Therefore analyzing singularities of $\mu_{Q,\dd}^{-1}(0)$ at closed
orbits boils down to analyzing $0\in\mu_{Q_{\tau},\ee}^{-1}(0)$.
The following result of Le Bruyn, Procesi \cite{LBP90} tells us that
there are finitely many semisimple types $\tau$ to consider and that
they come with a partial order (see also \cite[\S2.]{Bud21} for a
detailed exposition). For simplicity, we set $M(Q,\dd):=\mu_{Q,\dd}^{-1}(0)\git\GL_{\dd}$.

\begin{prop}{\cite[Thm. 2.2.]{Bud21}} \label{Prop/QuotStrat}

Let $\tau$ be a semisimple type and $M(Q,\dd)_{\tau}$ the subset
of semisimple $\Pi_{Q}$-modules of type $\tau$. Then $M(Q,\dd)_{\tau}$
is locally closed and there are finitely many types $\tau$ such that
$M(Q,\dd)_{\tau}\ne\emptyset$. Moreover:\[
\overline{M(Q,\dd)_{\tau}}=\bigcup_{\tau'\leq\tau}M(Q,\dd)_{\tau'},
\]where $\tau'\leq\tau$ if, and only if, there exist semisimple points
$x',x\in\mu_{Q,\dd}^{-1}(0)$ of types $\tau',\tau$ such that $(\GL_{\dd})_{x}\subseteq(\GL_{\dd})_{x'}$.

\end{prop}

Let $q:\mu_{Q,\dd}^{-1}(0)\rightarrow M(Q,\dd)$ be the quotient map.
Then we define $\left(\mu_{Q,\dd}^{-1}(0)\right)_{\tau}:=q^{-1}(M(Q,\dd)_{\tau})$.
We now show how to prove that $\mu_{Q,\dd}^{-1}(0)$ has rational
singularities, also at non-closed orbits.

\begin{prop} \label{Prop/RatSgClosedOrb}

Let $\tau$ be a semisimple type arising from $(Q,\dd)$. Then $\mu_{Q_{\tau},\ee}^{-1}(0)$
has rational singularities if, and only if, $\bigcup_{\tau'\geq\tau}\left(\mu_{Q,\dd}^{-1}(0)\right)_{\tau'}$
has rational singularities.

\end{prop}

\begin{proof}

We first make the following observation: any semisimple point $x\in\left(\mu_{Q,\dd}^{-1}(0)\right)_{\tau}$
has a $\GL_{\dd}$-saturated open neighbourhood $U_{x}$ contained
in the open subset $\bigcup_{\tau'\geq\tau}\left(\mu_{Q,\dd}^{-1}(0)\right)_{\tau'}=q^{-1}\left(\bigcup_{\tau'\geq\tau}M(Q,\dd)_{\tau'}\right)$.
By possibly shrinking it, one may assume that $U_{x}$ is contained
in the image of the étale morphism $W\times^{(\GL_{\dd})_{x}}\GL_{\dd}\rightarrow\mu_{Q,\dd}^{-1}(0)$
from Proposition \ref{Prop/MomMapEtSlice}. Moreover, for all $\tau'\geq\tau$,
we have $U_{x}\cap\left(\mu_{Q,\dd}^{-1}(0)\right)_{\tau'}\ne\emptyset$.
Indeed, $q(U_{x})$ is an open neighbourhood of $q(x)\in M(Q,\dd)_{\tau}$
and $M(Q,\dd)_{\tau}\subseteq\overline{M(Q,\dd)_{\tau'}}$, so $U_{x}\cap\left(\mu_{Q,\dd}^{-1}(0)\right)_{\tau'}=q^{-1}\left(q(U_{x})\cap M(Q,\dd)_{\tau'}\right)\ne\emptyset$.

Suppose that $\mu_{Q_{\tau},\ee}^{-1}(0)$ has rational singularities.
Then by Proposition \ref{Prop/MomMapEtSlice} and Lemma \ref{Lem/RatSgDesc},
for any semisimple point $x\in\left(\mu_{Q,\dd}^{-1}(0)\right)_{\tau}$,
the neighbourhood $U_{x}$ has rational singularities. Since for $\tau'\geq\tau$,
we have $U_{x}\cap\left(\mu_{Q,\dd}^{-1}(0)\right)_{\tau'}\ne\emptyset$,
there exists some semisimple point $x'\in\left(\mu_{Q,\dd}^{-1}(0)\right)_{\tau'}$
whose neighbourhood $U_{x'}$ has rational singularities (this may
require shrinking $U_{x'}$). Since all semisimple points in $\left(\mu_{Q,\dd}^{-1}(0)\right)_{\tau'}$
are étale-locally modelled on the same auxiliary quiver $(Q_{\tau'},\ee)$,
we deduce that for \textit{all} $x'\in\left(\mu_{Q,\dd}^{-1}(0)\right)_{\tau'}$,
$U_{x'}$ has rational singularities. Finally, the open subset $\bigcup_{\tau'\geq\tau}\left(\mu_{Q,\dd}^{-1}(0)\right)_{\tau'}\subseteq\mu_{Q,\dd}^{-1}(0)$
is covered by the open neighbourhoods $U_{x},\ x\in\bigcup_{\tau'\geq\tau}\left(\mu_{Q,\dd}^{-1}(0)\right)_{\tau'}$,
so it has rational singularities. The converse follows by applying
the same reasoning to $\left(\mu_{Q_{\tau},\ee}^{-1}(0)\right)_{\tau_{\min}}$,
where $\tau_{\min}$ is the semisimple type of $0\in\mu_{Q_{\tau},\ee}^{-1}(0)$.
\end{proof}

\paragraph*{Proof of the main result}

We now prove Theorem \ref{Thm/Ch3RatSgTotNeg}. Let us recall the
statement:

\begin{thm} \label{Thm/MainRes}

Let $Q$ be a quiver and $\dd\in\ZZ_{\geq0}^{Q_{0}}\setminus\{0\}$
such that $(Q,\dd)$ has property (P) . Then $\mu_{Q,\dd}^{-1}(0)$
has rational singularities.

\end{thm}

We follow the inductive strategy developed by Budur in \cite{Bud21}.
Let us use notations from \cite{Bud21}: we write, for short, $X(Q,\dd)=\mu_{Q,\dd}^{-1}(0)\subseteq R(\overline{Q},\dd)$,
$M(Q,\dd)=X(Q,\dd)\git\GL_{\dd}$ and $Z(Q,\dd)=\left(\mu_{Q,\dd}^{-1}(0)\right)_{\tau_{\min}}$,
where $\tau_{\min}=\tau_{\min,\dd}$ is the semisimple type of $0\in X(Q,\dd)$
- see Section \ref{Subsect/MomMap}. We denote by $q:X(Q,\dd)\rightarrow M(Q,\dd)$
the quotient morphism. Recall that we defined (P) the following property
of $(Q,\dd)$: $Q$ is totally negative and if $\supp(\dd)$ has two
vertices joined by exactly one arrow, then $\dd\ne\underline{1}$.
Budur's reasoning is summed up in the following result.

\begin{thm}{\cite[Thm. 3.6.]{Bud21}} \label{Thm/BudurInduction}

Let $\mathcal{M}$ be a class of pairs $(Q,\dd)$, where $Q$ is a
quiver and $\dd\in\mathbb{N}^{Q_{0}}$ is a dimension vector. Suppose
that:
\begin{enumerate}
\item The class $\mathcal{M}$ is stable under the operation of building
pairs $(Q_{\tau},\ee)$, for $\tau$ a semisimple type occurring in
$X(Q,\dd)$ and $(Q,\dd)\in\mathcal{M}$,
\item for every $(Q,\dd)\in\mathcal{M}$, $X(Q,\dd)$ contains a simple
point and $\langle\dd,\dd\rangle<1$,
\item for every $(Q,\dd)\in\mathcal{M}$ such that $X(Q,\dd)$ contains
strictly semisimple points\footnote{In other words, $\dd\ne\epsilon_{i}$ for all $i\in Q_{0}$, as $\tau_{\min}=(\epsilon_{i},d_{i},\ i\in\supp(\dd))$.},\[
\dim q^{-1}(q(0))<2\cdot(1-\langle\dd,\dd\rangle-\sharp\{\text{loops in } Q_0\}).
\]
\end{enumerate}
Then for every $(Q,\dd)\in\mathcal{M}$, the scheme $X(Q,\dd)$ has
rational singularities.

\end{thm}

Let us consider $\mathcal{M}$ the class of pairs $(Q,\dd)$ satisfying
property (P) and $\supp(\dd)=Q$. Then $\mathcal{M}$ is preserved
under taking auxiliary quivers (this is an easy consequence of \cite[Prop. 2.11.]{Bud21}),
so in the above theorem, Assumption 1 is verified. Assumption 2 follows
from total negativity and Proposition \ref{Prop/TotNegSimp}. However,
Assumption 3 is satisfied for most, but not all dimension vectors.
This is the content of the following Lemma and Remark.

\begin{lem} \label{Lem/DimBoundTotNeg}

Let $\tau=\tau_{\min}$. Suppose that $(Q,\dd)$ has property (P).
Then either $\dd=\underline{1}$ or:\[
\dim q^{-1}(q(0))<2\cdot(1-\langle\dd,\dd\rangle-\sharp\{\text{loops in } Q_0\}).
\]\end{lem}

\begin{rmk} \label{Rmk/BoundFailure}

When $\dd=\underline{1}$, we may have:\[
\dim q^{-1}(q(0))=2\cdot(1-\langle\dd,\dd\rangle-\sharp\{\text{loops in } Q_0\}).
\]This is the case, for instance, when $Q$ has two vertices with two
loops each and joined by two arrows (regardless of their orientation).
This quiver arises as an auxiliary quiver for the quiver with one
vertex and two loops. Actually, within the class of quivers with dimension
vectors considered in \cite[Prop. 2.26.]{Bud21}, Assumption 3 fails
exactly for that pair. This is due to a computational gap in the proof
of \cite[Prop. 2.23.]{Bud21}. More precisely, on the first line of
the proof of \cite[Prop. 2.23.]{Bud21}, the right-hand side should
be $2(g-1)(n^{2}-\sum_{i}\beta_{i})-2r+2$ instead of $2(g-1)(n^{2}-\sum_{i}\beta_{i})$.

\end{rmk}

Therefore, we show by other means that $X(Q,\dd)$ has rational singularities
when $\dd=\underline{1}$. In order to incorporate those cases into
Budur's inductive argument, we prove the following modified version
of Theorem \ref{Thm/BudurInduction}:

\begin{thm} \label{Thm/ModifInduction}

Let $\mathcal{M}$ be a class of pairs $(Q,\dd)$, where $Q$ is a
quiver and $\dd\in\mathbb{N}^{Q_{0}}$ is a dimension vector. Let
us make the following assumptions:
\begin{enumerate}
\item The class $\mathcal{M}$ is stable under the operation of building
pairs $(Q_{\tau},\ee)$, for $\tau$ a semisimple type occurring in
$X(Q,\dd)$ and $(Q,\dd)\in\mathcal{M}$,
\item For every $(Q,\dd)\in\mathcal{M}$, $X(Q,\dd)$ contains a simple
point,
\item For every $(Q,\dd)\in\mathcal{M}$, suppose either that $X(Q,\dd)$
has rational singularities or that the following inequality holds:\[
\dim q^{-1}(q(0))<2\cdot(1-\langle\dd,\dd\rangle-\sharp\{\text{loops in } Q_0\}).
\]
\end{enumerate}
Then for every $(Q,\dd)\in\mathcal{M}$, $X(Q,\dd)$ has rational
singularities.

\end{thm}

Note that we removed the assumption $\langle\dd,\dd\rangle<1$ from
Assumption 2. This is harmless, for when $X(Q,\dd)$ contains a simple
point, $\langle\dd,\dd\rangle\leq1$ and equality only occurs when
$\dd=\epsilon_{i}$, for $i\in Q_{0}$ a vertex without loops, in
which case $X(Q,\dd)$ is smooth. Indeed, if $X(Q,\dd)$ contains
a simple point and $\langle\dd,\dd\rangle=1$, then $M(Q,\dd)$ is
reduced to a point by Proposition \ref{Prop/GeoMomMap}. Thus $X(Q,\dd)$
does not contain strictly semisimple points, which can only happen
if $\dd=\epsilon_{i}$ for some $i\in Q_{0}$. The number of loops
at $i$ is then $1-\langle\epsilon_{i},\epsilon_{i}\rangle=1-\langle\dd,\dd\rangle=0$.
At any rate, since Theorem \ref{Thm/MainRes} only concerns totally
negative quivers, the assumption $\langle\dd,\dd\rangle<1$ is automatically
satisfied.

Before proving Theorem \ref{Thm/ModifInduction}, we recall the following
two results from \cite{Bud21}:

\begin{lem}{\cite[Lem. 2.16.]{Bud21}} \label{Lem/BoundDimStratZ}

Let $Q$ be a quiver and $\dd\in\mathbb{N}^{Q_{0}}$ a dimension vector.
Assume that:\[
\dim q^{-1}(q(0))<2\cdot(1-\langle\dd,\dd\rangle-\sharp\{\text{loops in } Q_0\}).
\] Then $\dim Z(Q,\dd)<2(1-\langle\dd,\dd\rangle)$.

\end{lem}

\begin{lem}{\cite[Lem. 3.3. - Proof of Lem. 3.4.]{Bud21}} \label{Lem/JetsMomMap}

Let $\pi_{m}:X(Q,\dd)_{m}\rightarrow X(Q,\dd)$ be the truncation
of $m$-jets. Then:\[
\pi_m^{-1}(0)\simeq
\left\{
\begin{array}{ll}
R(\overline{Q},\dd)\times X(Q,\dd)_{m-2} & ,\ m\geq2, \\
R(\overline{Q},\dd) & ,\ m=1.
\end{array}
\right.
\]Moreover, $\dim\pi_{m}^{-1}(Z(Q,\dd)\cap X(Q,\dd)_{\sg})\leq\dim Z(Q,\dd)+\dim\pi_{m}^{-1}(0)$
(see \cite[Proof of Lem. 3.4. - Eqn. (10)]{Bud21}).

\end{lem}

\begin{proof}[Proof of Theorem \ref{Thm/ModifInduction}]

Consider a pair $(Q,\dd)\in\mathcal{M}$ for which it is not already
assumed that $X(Q,\dd)$ has rational singularities. We proceed by
descending induction on semisimple types occurring in $X(Q,\dd)$
i.e. we show that $X(Q_{\tau},\ee)$ has rational singularities for
all semisimple types $\tau$ occurring in $X(Q,\dd)$. When $\tau$
is maximal (i.e. simple, by Assumption 2), $Q_{\tau}$ has one vertex
and $\ee=\underline{1}$, so $X(Q_{\tau},\ee)$ is smooth.

Let us prove the induction step. We apply Proposition \ref{Prop/MustCrit}
to $X(Q,\dd)$, which is a complete intersection by Assumption 2 and
Proposition \ref{Prop/GeoMomMap}. We show that:\[
\dim\pi_m^{-1}(X(Q,\dd)_{\sg})<(m+1)\cdot\dim X(Q,\dd).
\]By induction, we assume that $X(Q_{\tau},\ee)$ has rational singularities
for $\tau>\tau_{\min}$. By Proposition \ref{Prop/RatSgClosedOrb},
we get that the open subset $X(Q,\dd)\setminus Z(Q,\dd)$ has l.c.i.
rational singularities and we obtain by Proposition \ref{Prop/MustCrit}:\[
\dim\pi_m^{-1}(X(Q,\dd)_{\sg}\setminus Z(Q,\dd))<(m+1)\cdot\dim X(Q,\dd).
\]Now, consider $\tau=\tau_{\min}$ i.e. $(Q_{\tau},\ee)=(Q,\dd)$.
From Proposition \ref{Prop/GeoMomMap} and Lemmas \ref{Lem/DimBoundTotNeg},
\ref{Lem/BoundDimStratZ}, \ref{Lem/JetsMomMap}, we obtain:\[
\dim\pi_m^{-1}(Z(Q,\dd)\cap X(Q,\dd)_{\sg})\leq\dim\left(Z(Q,\dd)\right)+\dim\pi_{m}^{-1}(0)<\dim M(Q,\dd)+\dim \pi_{m}^{-1}(0).
\]We prove that $\dim\pi_{m}^{-1}(Z(Q,\dd)\cap X(Q,\dd)_{\sg})<(m+1)\cdot\dim X(Q,\dd)$
by induction on $m$. For $m=1$, Lemma \ref{Lem/JetsMomMap} gives:\[
\dim M(Q,\dd)+\dim \pi_{m}^{-1}(0)=\dim M(Q,\dd)+\dim R(\overline{Q},\dd)=2\cdot\left(1-\langle\dd,\dd\rangle+\dd\cdot\dd-\langle\dd,\dd\rangle\right)=2\dim X(Q,\dd).
\]For $m\geq2$, we obtain:\[
\begin{split}
\dim M(Q,\dd)+\dim \pi_{m}^{-1}(0) & =\dim M(Q,\dd)+\dim R(\overline{Q},\dd)+\dim X(Q,\dd)_{m-2}, \\
& \leq2\dim X(Q,\dd)+(m-1)\cdot\dim X(Q,\dd)=(m+1)\cdot\dim X(Q,\dd).
\end{split}
\]where the second inequality holds by induction on $m$ and using the
following inequalities at step $m-2$:\[
\left\{
\begin{array}{rcl}
\dim\pi_{m-2}^{-1}(X(Q,\dd)_{\sm}) & = & (m-1)\cdot\dim X(Q,\dd), \\
\dim\pi_{m-2}^{-1}(Z(Q,\dd)\cap X(Q,\dd)_{\sg}) & < & (m-1)\cdot\dim X(Q,\dd), \\
\dim\pi_{m-2}^{-1}(X(Q,\dd)_{\sg}\setminus Z(Q,\dd)) & < & (m-1)\cdot\dim X(Q,\dd).
\end{array}
\right.
\] Therefore, at step $m$, we obtain the following inequalities:\[
\left\{
\begin{array}{rcl}
\dim\pi_{m}^{-1}(Z(Q,\dd)\cap X(Q,\dd)_{\sg}) & < & (m+1)\cdot\dim X(Q,\dd), \\
\dim\pi_{m}^{-1}(X(Q,\dd)_{\sg}\setminus Z(Q,\dd)) & < & (m+1)\cdot\dim X(Q,\dd).
\end{array}
\right.
\]So we obtain $\dim\pi_{m}^{-1}(X(Q,\dd)_{\sg})<(m+1)\cdot\dim X(Q,\dd)$
for all $m\geq1$, which proves, by Musta\c{t}\v{a}'s criterion,
that $X(Q,\dd)$ has rational singularities. \end{proof}

To complete the proof, we need to prove Lemma \ref{Lem/DimBoundTotNeg}
and show that $X(Q,\dd)$ has rational singularities when $\dd=\underline{1}$.
The latter fact is a consequence of the jet-counts carried out in
\cite{Wys17b} and Proposition \ref{Prop/JetCountLim}. When $\dd=\underline{1}$,
Wyss gave a convergence criterion for the sequence $q^{-n\dim X(Q,\dd)}\cdot\sharp X(Q,\dd)(\mathbb{F}_{q}[t]/t^{n}),\ n\geq1$
and an explicit formula for its limit in terms of the graphical hyperplane
arrangement associated to $Q$.

\begin{prop}{\cite[Cor. 4.27.]{Wys17b}} \label{Prop/ToricFormulaB}

Let $Q$ be a quiver and $\dd=\underline{1}$. Let $p$ be a large
enough prime and $\mathbb{F}_{q}$ be any finite field of characteristic
$p$. Then $q^{-n\dim X(Q,\dd)}\cdot\sharp X(Q,\dd)(\mathbb{F}_{q}[t]/(t^{n})),\ n\geq1$
converges if, and only if, the underlying graph of $Q$ is 2-connected
i.e. removing one edge does not disconnect the graph. In that case,
there exists an explicit rational fraction $W\in\mathbb{Q}(T)$ depending
only on the underlying graph of $Q$ such that:\[
\underset{n\rightarrow+\infty}{\lim}\frac{\sharp X(Q,\dd)(\mathbb{F}_{q}[t]/(t^{n}))}{q^{n\dim X(Q,\dd)}}=W(q).
\]\end{prop}

From this we deduce:

\begin{cor} \label{Cor/RatSgToric}

If $Q$ is 2-connected and $\dd=\underline{1}$, then $X(Q,\dd)$
has rational singularities.

\end{cor}

\begin{proof}

We deduce the claim from Proposition \ref{Prop/ToricFormulaB} by
applying Proposition \ref{Prop/JetCountLim}. We just need to check
that $X(Q,\dd)$ has l.c.i. singularities. We apply \cite[Thm. 1.2.]{CB01}:
if there exists a simple $\Pi_{Q}$-module with dimension vector $\dd$,
then $X(Q,\dd)$ is a complete intersection.

By \cite[Thm. 1.2.]{CB01} again, there exists a simple $\Pi_{Q}$-module
with dimension vector $\dd$ if, and only if, for any decomposition
$\dd=\dd_{1}+\ldots+\dd_{r}$, \[
1-\langle\dd,\dd\rangle>1-\langle\dd_1,\dd_1\rangle+\ldots+1-\langle\dd_r,\dd_r\rangle.
\]Moreover, for a connected quiver $Q$, we have $b(Q):=1-\sharp Q_{0}+\sharp Q_{1}=1-\langle\dd,\dd\rangle$
and $Q$ is $2$-connected if, and only if, for any decomposition
$\dd=\dd_{1}+\ldots+\dd_{r}$, the inequality $b(Q)>b(\supp(\dd_{1}))+\ldots+b(\supp(\dd_{r}))$
holds (see Section \ref{Subsect/WyssConj}).

Thus $X(Q,\dd)$ has l.c.i. singularities and we conclude from Propositions
\ref{Prop/ToricFormulaB} and \ref{Prop/JetCountLim} that $X(Q,\dd)$
also has rational singularities. Note that we could also use \cite[Thm. 2.11.]{HSS21}.
\end{proof}

We now turn to the proof of Lemma \ref{Lem/DimBoundTotNeg}. We first
prove the following intermediate inequality:

\begin{lem} \label{Lem/DimBoundLoops}

Let $Q=S_{g}$ be the quiver with one vertex and $g\geq2$ loops and
$d\geq2$. Let $(j_{s},m_{s},\ 1\leq s\leq h)$ be a top type compatible
with $\tau_{\min}$. Then:\[
2\cdot(1-\langle d,d\rangle-g)-\left(d^2-1+1-\langle d,d\rangle+m_{2}m_{1}+\ldots m_{h}m_{h-1}-g\cdot(m_{1}^2+\ldots m_{h}^2)\right)\geq d-1.
\]\end{lem}

\begin{proof}

Rearranging terms, the left-hand side reads:\[
\begin{split}
& g\cdot\left(d^2+\sum_sm_s^2-2\right)-2(d^2-1)-(m_{2}m_{1}+\ldots m_{h}m_{h-1}) \\
\geq\ & 2\cdot\left(d^2+\sum_sm_s^2-2\right)-2(d^2-1)-(m_{2}m_{1}+\ldots m_{h}m_{h-1}) \\
\geq\ & 2\left(\sum_sm_s^2-1\right)-(m_{2}m_{1}+\ldots m_{h}m_{h-1}) \\
\geq\ &\sum_sm_s^2-1+\frac{1}{2}\cdot\left(m_1^2+m_h^2+(m_1-m_2)^2+\ldots+(m_{h-1}-m_h)^2-2\right) \\
\geq\ & d-1.
\end{split}
\]Note that both sides are equal when $g=2$ and $m_{1}=\ldots=m_{s}=1$.
\end{proof}

\begin{proof}[Proof of Lemma \ref{Lem/DimBoundTotNeg}]

Without loss of generality, we assume that $\supp(\dd)=Q$; otherwise
we restrict to $\text{supp}(\dd)$, which also has property (P). Let
$g_{i}$ be the number of loops of $Q$ at vertex $i$, $r_{ij}$
be the number of arrows between vertices $i\ne j$ and $(j_{s},m_{s},\ 1\leq s\leq h)$
be a top type compatible with $\tau$. Then, by Proposition \ref{Prop/CBDimBound},
the inequality above holds if the following holds for arbitrary top-type:\[
\sum_id_i^2-1+1-\langle\dd,\dd\rangle+\sum_sm_sz_s-\sum_sm_s^2g_{j_s}<2\cdot(1-\langle\dd,\dd\rangle-\sum_ig_i).
\]We now want to split this inequality along vertices of $Q$. Let us
first rearrange the indices $1\leq s\leq h$ so that $j_{1}=\ldots=j_{s_{1}}=1$,
and so on until $j_{s_{r-1}+1}=\ldots=j_{s_{r}}=r$. Then $z_{s_{i-1}+1}=0$
and $z_{s_{i-1}+j+1}=m_{s_{i-1}+j}$ for $j>0$, $1\leq i\leq r$,
with the convention that $s_{0}=0$. We obtain:\[
\sum_id_i^2-1+1-\langle\dd,\dd\rangle+\sum_i(m_{s_{i-1}+2}m_{s_{i-1}+1}+\ldots m_{s_i}m_{s_i-1})-\sum_ig_i\cdot(m_{s_{i-1}+1}^2+\ldots m_{s_i}^2)<2\cdot(1-\langle\dd,\dd\rangle-\sum_ig_i).
\]We know from Lemma \ref{Lem/DimBoundLoops} that for all $1\leq i\leq r$:\[
\begin{split}
& 2\cdot(1-\langle d_i\epsilon_i,d_i\epsilon_i\rangle-g_i) \\
& -\left(d_i^2-1+1-\langle d_i\epsilon_i,d_i\epsilon_i\rangle+m_{s_{i-1}+2}m_{s_{i-1}+1}+\ldots+ m_{s_i}m_{s_i-1}-g_i\cdot(m_{s_{i-1}+1}^2+\ldots m_{s_i}^2)\right) \\
\geq\ & d_i-1,
\end{split}
\]where equality holds for top type $(m_{s}=1,\ 1\leq s\leq h)$. Taking
the remaining terms in the inequality (right-hand side minus left-hand
side) gives:\[
1-\langle\dd,\dd\rangle-\sum_i(1-\langle d_i\epsilon_i,d_i\epsilon_i\rangle)-(r-1)=\sum_{i\ne j}r_{ij}d_id_j-2(r-1).
\]Set $r_{1}$ (resp. $r_{2}$) the number of vertices $i\in Q_{0}$
such that $d_{i}=1$ (resp. $d_{i}\geq2$). Then $r=r_{1}+r_{2}$
(recall that we assume that $\supp(\dd)=Q$). Then:\[
\begin{split}
\sum_{i\ne j}r_{ij}d_id_j-2(r-1) & \geq 4 \cdot\frac{r_2(r_2-1)}{2}+2r_1r_2+\frac{r_1(r_1-1)}{2}-2(r-1) \\
& \geq 2(r_2-1)(r-1)+\frac{r_1(r_1-1)}{2}.
\end{split}
\]Summing everything, we obtain:\[
\begin{split}
& 2\cdot(1-\langle\dd,\dd\rangle-\sum_ig_i) \\
& -\left(\sum_id_i^2-1+1-\langle\dd,\dd\rangle+\sum_i(m_{s_{i-1}+2}m_{s_{i-1}+1}+\ldots m_{s_i}m_{s_i-1})-\sum_ig_i\cdot(m_{s_{i-1}+1}^2+\ldots m_{s_i}^2)\right) \\
\geq & \sum_i(d_i-1)+2(r_2-1)(r-1)+\frac{r_1(r_1-1)}{2}.
\end{split}
\]If $\dd>\underline{1}$, the right-hand side is positive, so we get
the desired inequality. \end{proof}

Theorem \ref{Thm/MainRes} now follows from Lemma \ref{Lem/DimBoundTotNeg},
Corollary \ref{Cor/RatSgToric} and Theorem \ref{Thm/ModifInduction}.

\begin{rmk} 

Unfortunately, it may happen that $X(Q,\dd)$ has rational singularities,
while the dimension bound from \cite[Lem. 2.16.]{Bud21} fails. This
is the case, for instance, with the following quivers: $\bullet\rightrightarrows\bullet$
and $\bullet\rightrightarrows\bullet\rightrightarrows\bullet$, when
$\dd=\underline{1}$ (both are 2-connected).

\end{rmk}

\paragraph*{Convergence for jet-counts}

Circling back to our initial question, we deduce from Theorem \ref{Thm/MainRes}
that the jet-counts at hand converge when $n$ goes to infinity. Let
$Q$ be a quiver and $\dd\in\ZZ_{\geq0}^{Q_{0}}$. If $(Q,\dd)$ has
property (P), we already know from Proposition \ref{Prop/GeoMomMap}
that $\mu_{Q,\dd}^{-1}(0)$ is a complete intersection. Thus we obtain
from Proposition \ref{Prop/JetCountLim}:

\begin{cor} \label{Cor/TotNegJetLim}

Let $Q$ be a quiver and $\dd\in\ZZ_{\geq0}^{Q_{0}}\setminus\{0\}$
such that $(Q,\dd)$ has property (P). Then for almost all primes
$p$ and for all finite fields $\mathbb{F}_{q}$ of characteristic
$p$, $q^{-n\dim\mu_{Q,\dd}^{-1}(0)}\cdot\sharp\mu_{Q,\dd}^{-1}(0)(\mathbb{F}_{q}[t]/(t^{n}))$
converges when $n$ goes to infinity.

\end{cor}

Now that we have established the existence of $\underset{n\rightarrow+\infty}{\lim}q^{-n\dim\mu_{Q,\dd}^{-1}(0)}\cdot\sharp\mu_{Q,\dd}^{-1}(0)(\mathbb{F}_{q}[t]/(t^{n}))$
for all totally negative quivers, it is natural to ask whether this
limit is uniform over all finite fields $\mathbb{F}_{q}$, as shown
by Wyss in the case $\dd=\underline{1}$.

\begin{cj}

Let $Q$ be a quiver and $\dd\in\mathbb{Z}_{\geq0}^{Q_{0}}\setminus\{0\}$
such that $(Q,\dd)$ has property (P). There exists a rational fraction
$W\in\mathbb{Q}(T)$ such that, for almost all primes $p$ and for
all finite fields $\mathbb{F}_{q}$ of characteristic $p$:\[
W(q)=\underset{n\rightarrow+\infty}{\lim}q^{-n\dim \mu_{Q,\dd}^{-1}(0)}\cdot\sharp\mu_{Q,\dd}^{-1}(0)(\mathbb{F}_{q}[t]/(t^{n})).
\]\end{cj}

\subsection{Generalisation to other moduli \label{Subsect/Mod2CY}}

In this section, we generalise the results of Section \ref{Subsect/RatSgTotNegQuiv}
to a larger class of moduli. These are moduli stacks of objects in
2-Calabi-Yau categories, as formalised by Davison in \cite{Dav21a}.
We study singularities of these moduli using a local description in
terms of quiver moment maps, also due to Davison. We first collect
examples of such moduli stacks and recall their unified local description
by Davison. Then we prove Theorems \ref{Thm/Ch3RatSg2CYMod} and \ref{Thm/Ch3p-adicVol2CYMod}
and give some examples where they apply. Throughout, we work over
an algebraically closed field $\KK$ of characteristic zero. We will
denote by $\hom(M,N)$ (resp, $\ext^{i}(M,N)$) the dimensions of
Hom (resp. Ext) spaces in abelian categories.

\paragraph*{Moduli of totally negative 2CY categories}

We first gather local models of certain moduli stacks of objects of
2-Calabi-Yau categories, following \cite{Dav21a}. Such local models
were obtained separately in \cite{CB03a,Bud21}, \cite{BGV16,KS23a}
and \cite{AS18,BZ19} and united in the more general framework of
\cite{Dav21a}. In each case, local models are constructed from so-called
Ext-quivers of semisimple (or polystable) objects. The moduli stacks
we consider are also disjoint unions of global quotient stacks $[X_{\alpha}/G_{\alpha}]$,
where $X_{\alpha}$ is a scheme and $G_{\alpha}$ is a reductive group.
This gives us local models for the schemes $X_{\alpha}$, which generalise
the étale slices described above.

Let us first give a definition of 2-Calabi-Yau category which covers
our examples below. For simplicity, we use the definition for triangulated
categories given in \cite[\S2.6.]{Kel08}, although Davison works
with a refined notion of Calabi-Yau structures suited to differential-graded
enhancements of the triangulated categories we consider. As we will
see below, the coarser definition is sufficient for the computations
of Ext-quivers covered here. In what follows, $\KK$ denotes a base
field.

\begin{df} \label{Def/CYCat}

Let $d\in\mathbb{Z}$ and $\mathcal{T}$ a $\KK$-linear, Hom-finite,
triangulated category admitting a Serre functor as defined in \cite[\S2.6.]{Kel08}.
We say that $\mathcal{T}$ is $d$-Calabi-Yau if there exists a family
of $\KK$-linear forms $t_{X}:\Hom(X,X[d])\rightarrow\KK,\ X\in\mathcal{T}$
such that, for all $p,q\in\mathbb{Z}$ satisfying $p+q=d$ and for
all $f\in\Hom(X,Y[p])$ and $g\in\Hom(Y,X[q])$, the pairing:\[
\begin{array}{cll}
\Hom(X,Y[p])\times\Hom(Y,X[q]) & \rightarrow & \KK \\
(f,g) & \mapsto & t_X(g[p]\circ f)
\end{array}
\]is non-degenerate and $t_{X}(g[p]\circ f)=(-1)^{pq}\cdot t_{Y}(f[q]\circ g)$.

\end{df}

The triangulated categories considered below are subcategories of
$D^{b}(\mathcal{C})$, where $\mathcal{C}$ is an abelian category
e.g. modules over an algebra or quasi-coherent sheaves over an algebraic
variety. In particular, if $D^{b}(\mathcal{C})$ is $2$-Calabi-Yau,
then for all $X\in\mathcal{C}$, $\Ext_{\mathcal{C}}^{1}(X,X)$ inherits
a symplectic form. Thus $\Ext_{\mathcal{C}}^{1}(X,X)$ must be even-dimensional.

One can also define a notion of Calabi-Yau algebras as in \cite[\S 3.2.]{Gin06}
and \cite[\S 7]{VdB15}. Set $A$ an algebra and consider $\mathcal{C}$
the category of right $A$-modules, $D^{b}(A)$ the triangulated subcategory
of $D^{b}(\mathcal{C})$ formed by complexes whose total cohomology
is finite-dimensional. Then by \cite[\S4.1-2.]{Kel08}, if $A$ is
a $d$-Calabi-Yau algebra, $D^{b}(A)$ is $d$-Calabi-Yau as a triangulated
category. Here is the definition:

\begin{df} \label{Def/CYAlg}

Let $A$ be a $\KK$-algebra and $d\in\mathbb{Z}$. The algebra $A$
is called $d$-Calabi-Yau if:
\begin{enumerate}
\item As an $A-A$-bimodule, $A$ admits a projective resolution of finite
length by finite-dimensional projective $A-A$-bimodules,
\item There is a quasi-isomorphism $\mathrm{RHom}(A,A^{\mathrm{op}}\otimes A)\simeq A[-d]$
of complexes of $A-A$-bimodules.
\end{enumerate}
\end{df}

We now concretely describe the moduli stacks that we deal with in
this section.

\subparagraph*{$\Pi_{Q}$-modules}

Given a quiver $Q$, we consider the moduli stack $\mathfrak{M}_{\Pi_{Q}}$
of representations of the preprojective algebra $\Pi_{Q}$. It is
the union of the following quotient stacks:\[
\mathfrak{M}_{\Pi_Q}
=\bigsqcup_{\dd\in\mathbb{N}^{Q_0}}\mathfrak{M}_{\Pi_Q,\dd}
=\bigsqcup_{\dd\in\mathbb{N}^{Q_0}}\left[\mu_{Q,\dd}^{-1}(0)/\GL_{\dd}\right].
\]When the underlying graph of $Q$ has no Dynkin diagram of type A,
D, or E among its connected components (in particular, when $Q$ is
totally negative), $\Pi_{Q}$ is a 2-Calabi-Yau algebra \cite[\S 4.2.]{Kel08}.
Moreover:\[
\hom_{\Pi_Q}(M,N)-\ext_{\Pi_Q}^1(M,N)+\ext_{\Pi_Q}^2(M,N)=(\dim(M),\dim(N)).
\]We already saw étale-local models of $\mu_{Q,\dd}^{-1}(0)$ and GIT
quotients $M_{\Pi_{Q},\dd}:=\mu_{Q,\dd}^{-1}(0)\git\GL_{\dd}$ in
Sections \ref{Subsect/MomMap} and \ref{Subsect/RatSgTotNegQuiv}.
Let us just mention how the auxiliary quivers $(Q_{\tau},\ee)$ are
related to Ext-groups when $\Pi_{Q}$ is $2$-Calabi-Yau. Given a
semisimple $\Pi_{Q}$-module $M=\bigoplus_{i=1}^{r}M_{i}^{\oplus e_{i}}$
of type $\tau=(\dd_{i},e_{i}\ ;\ 1\leq i\leq r)$, the Ext-quiver
of $M$ is the quiver with set of vertices $\{1,\ldots,r\}$ and $\ext_{\Pi_{Q}}^{1}(M_{i},M_{j})$
arrows from vertex $i$ to vertex $j$. The identity above shows that
the Ext-quiver of $M$ is $\overline{Q_{\tau}}$.

\subparagraph*{$\Lambda^{q}(Q)$-modules}

Let $Q$ be a quiver and $q\in(\KK^{\times})^{Q_{0}}$. The multiplicative
preprojective algebra $\Lambda^{q}(Q)$ was introduced by Crawley-Boevey
and Shaw in \cite{CBS06}. We refer to their article for a precise
definition and directly describe moduli of $\Lambda^{q}(Q)$-modules.
Fix a total ordering $<$ of $Q_{1}$. A $\Lambda^{q}(Q)$-module
of dimension $\dd$ is given by a collection of matrices $M_{a},M_{a}^{*},\ a\in Q_{1}$
of size $d_{t(a)}\times d_{s(a)}$ (resp. $d_{s(a)}\times d_{t(a)}$)
such that, for all $a\in Q_{1}$, the matrices $I_{d_{t(a)}}+M_{a}M_{a}^{*}$
and $I_{d_{s(a)}}+M_{a}^{*}M_{a}$ are invertible and:\[
\prod_{\substack{a\in (Q_1,<) \\ t(a)=i}}(I_{d_{t(a)}}+M_aM_a^*)\times\prod_{\substack{a\in (Q_1,<) \\ s(a)=i}}(I_{d_{s(a)}}+M_a^*M_a)^{-1}=q_i\cdot I_{d_i}.
\]As in the case of additive preprojective algebras, two collections
of matrices correspond to isomorphic modules if they are conjugated
by an element of $\GL_{\dd}$. Therefore, we obtain the moduli stack:\[
\mathfrak{M}_{\Lambda^{q}(Q)}
=\bigsqcup_{\dd\in\mathbb{N}^{Q_0}}\mathfrak{M}_{\Lambda^{q}(Q),\dd}
=\bigsqcup_{\dd\in\mathbb{N}^{Q_0}}[R(\Lambda^{q}(Q),\dd)/\GL_{\dd}].
\]where $R(\Lambda^{q}(Q),\dd)\subseteq R(\overline{Q},\dd)$ is the
affine, locally closed subvariety defined by the above conditions.
Note that $\Lambda^{q}(Q)$ and $R(\Lambda^{q}(Q),\dd)$ do not depend
(up to isomorphism) on the orientation of $Q$, nor on the choice
of the ordering on $Q_{1}$ - see \cite[Thm. 1.4.]{CBS06}. Moreover,
since $R(\Lambda^{q}(Q),\dd)$ is affine, we obtain a GIT quotient
$M_{\Lambda^{q}(Q),\dd}=R(\Lambda^{q}(Q),\dd)\git\GL(\dd)$, which
parametrises semisimple $\Lambda^{q}(Q)$-modules of dimension $\dd$
similarly to $M_{\Pi_{Q},\dd}$.

When $Q$ is connected and contains at least one (not necessarily
oriented) cycle, Kaplan and Schedler proved that $\Lambda^{q}(Q)$
is a $2$-Calabi-Yau algebra (\cite[Thm. 1.2.]{KS23a}). This is the
case, for instance, when $Q$ is totally negative. Moreover, for any
pair of $\Lambda^{q}(Q)$-modules $(M,N)$, the following identity
holds as in the additive case (see \cite[Thm. 1.6.]{CBS06}):\[
\hom_{\Lambda^{q}(Q)}(M,N)-\ext_{\Lambda^{q}(Q)}^1(M,N)+\ext_{\Lambda^{q}(Q)}^2(M,N)=(\dim(M),\dim(N)).
\]Given a semisimple $\Lambda^{q}(Q)$-module $M=\bigoplus_{i=1}^{r}M_{i}^{\oplus e_{i}}$
of type $\tau=(\dd_{i},e_{i}\ ;\ 1\leq i\leq r)$, one constructs
an auxiliary quiver and a dimension vector $(Q_{\tau},\ee)$ as in
the additive case (see Section \ref{Subsect/RatSgTotNegQuiv}). Likewise,
the identity above shows that the Ext-quiver of $M$ is $\overline{Q_{\tau}}$.

\subparagraph*{Semistable coherent sheaves on K3 surfaces}

Let $(S,H)$ be a complex, projective, polarised K3 surface. We consider
the moduli stack $\mathfrak{M}_{S,H}$ of Gieseker $H$-semistable,
coherent sheaves on $S$, as described in \cite{HL10}. It is the
union of the following quotient stacks:\[
\mathfrak{M}_{S,H}
=\bigsqcup_{\textbf{v}\in \mathrm{H}^*(S,\mathbb{Z})}\mathfrak{M}_{S,H,\textbf{v}}
=\bigsqcup_{\textbf{v}\in \mathrm{H}^*(S,\mathbb{Z})}[U_{\textbf{v},m_{\textbf{v}}}^{\mathrm{ss}}/\GL_{P(m_{\textbf{v}})}].
\]where $\mathfrak{M}_{S,H}(\textbf{v})$ is the substack of semistable
sheaves with Mukai vector $\textbf{v}$. Given a coherent sheaf $\mathcal{F}$
on $S$, its Mukai vector is by definition $\textbf{v}(\mathcal{F}):=(\rk(\mathcal{F}),c_{1}(\mathcal{F}),\frac{1}{2}c_{1}(\mathcal{F})^{2}-c_{2}(\mathcal{F})+\rk(\mathcal{F}))$
and determines the Hilbert polynomial of $\mathcal{F}$. For a given
Mukai vector $\textbf{v}$ and associated Hilbert polynomial $P\in\mathbb{Q}[t]$,
there exists an integer $m_{\textbf{v}}>0$ such that, for any Gieseker
$H$-semistable, coherent sheaf $\mathcal{F}$ on $S$ with Hilbert
polynomial $P$, the sheaf $\mathcal{F}(m_{\textbf{v}})$ is generated
by global sections (see \cite[\S4.3.]{HL10} for details). Such a
sheaf $\mathcal{F}$, together with a choice of a basis of $\Gamma(S,\mathcal{F}(m_{\textbf{v}}))$,
corresponds to a point in $\text{Quot}(\mathcal{O}_{S}(-m_{\textbf{v}})^{\oplus P(m_{\textbf{v}})},P)$.
The locus $U_{\textbf{v},m_{\textbf{v}}}\subseteq\text{Quot}(\mathcal{O}_{S}(-m_{\textbf{v}})^{\oplus P(m_{\textbf{v}})},P)$
of quotient sheaves with Mukai vector $\textbf{v}$ is an open and
closed subset (see \cite[Ch. 10.2.]{Huy16}), which is preserved under
the action of $\GL_{P(m_{\textbf{v}})}$ on $\text{Quot}(\mathcal{O}_{S}(-m_{\textbf{v}})^{\oplus P(m_{\textbf{v}})},P)$.
This action admits a linearisation such that the semistable locus
$U_{\textbf{v},m_{\textbf{v}}}^{\mathrm{ss}}\subseteq\text{Quot}(\mathcal{O}_{S}(-m_{\textbf{v}})^{\oplus P(m_{\textbf{v}})},P)$
parametrises globally generated quotients $\mathcal{O}_{S}(-m_{\textbf{v}})^{\oplus P(m_{\textbf{v}})}\twoheadrightarrow\mathcal{F}$
inducing an isomorphism on global sections. Thus $\mathfrak{M}_{S,H,\textbf{v}}\simeq[U_{\textbf{v},m_{\textbf{v}}}^{\mathrm{ss}}/\GL_{P(m_{\textbf{v}})}]$.
Taking the associated GIT quotient, we obtain the moduli space $M_{\textbf{v}}=M_{S,H,\textbf{v}}$,
which parametrises $H$-polystable sheaves on $S$. We denote by $M_{\textbf{v}}^{s}\subseteq M_{\textbf{v}}$
the open locus of $H$-stable sheaves.

Let us denote Mukai vectors by $\textbf{v}=(r,\textbf{c},a)\in\mathbb{Z}_{\geq0}\oplus\mathrm{NS}(S)\oplus\mathrm{H}^{4}(S,\mathbb{Z})$.
Recall that the Mukai pairing is defined by: $\textbf{v}_{1}\cdot\textbf{v}_{2}=\textbf{c}_{1}\cdot\textbf{c}_{2}-r_{1}a_{2}-r_{2}a_{1}$.
Note that $\mathrm{NS}(S)$ is a lattice (see \cite[Ch. 1.]{Huy16}).
Consider $D^{b}(S)$ the derived category formed by complexes of quasi-coherent
sheaves on $S$ with bounded coherent cohomology. Since $S$ is a
K3 surface, $D^{b}(S)$ is 2-Calabi-Yau, by Serre duality. Moreover,
for any pair $\mathcal{F}_{1},\mathcal{F}_{2}$ of semistable coherent
sheaves, the following identity holds \cite[Cor. 6.1.5.]{HL10}:\[
\hom(\mathcal{F}_{1},\mathcal{F}_{2})-\ext^1(\mathcal{F}_{1},\mathcal{F}_{2})+\ext^2(\mathcal{F}_{1},\mathcal{F}_{2})=-\textbf{v}_{1}\cdot\textbf{v}_{2}=-(\textbf{c}_{1}\cdot\textbf{c}_{2}-r_{1}a_{2}-r_{2}a_{1}).
\]Given a polystable sheaf $\mathcal{F}=\bigoplus_{i=1}^{r}\mathcal{F}_{i}^{\oplus e_{i}}$,
with distinct stable summands $\mathcal{F}_{i}$ of Mukai vectors
$\textbf{v}_{i}$, we construct an auxiliary quiver $Q'$ satisfying
the following conditions: $Q'_{0}=\{1,\ldots,r\}$ and the number
of arrows from $i$ to $j$ in $\overline{Q'}$ is:\[
\ext^{1}(\mathcal{F}_{i},\mathcal{F}_{j})
=
\left\{
\begin{array}{ll}
2+\textbf{v}_{i}\cdot\textbf{v}_{i}, & \text{ if } i=j, \\
\textbf{v}_{i}\cdot\textbf{v}_{j}, & \text{ if } i\ne j.
\end{array}
\right.
\]Thus, $\overline{Q'}$ is the Ext-quiver associated to $\mathcal{F}$.
Note that $\ext^{1}(\mathcal{F}_{i},\mathcal{F}_{i})$ is indeed even,
thanks to the $2$-Calabi-Yau property.

\subparagraph*{Local models}

We now describe local models of the above moduli stacks in a unified
manner, following \cite{Dav21a}. Let us first stress common features
of the aforementioned stacks. In what follows, $\mathfrak{M}$ denotes
either $\mathfrak{M}_{\Pi_{Q}}$, $\mathfrak{M}_{\Lambda^{q}(Q)}$
$\mathfrak{M}_{S,H}$ and $\mathcal{A}$ the corresponding category
of objects (the groupoid associated to $\mathcal{A}$ is isomorphic
to $\mathfrak{M}(\KK)$). Then $\mathfrak{M}$ is a union of quotient
stacks:\[
\mathfrak{M}=\bigsqcup_{\alpha}\mathfrak{M}_{\alpha}=\bigsqcup_{\alpha}[X_{\alpha}/G_{\alpha}],
\]where $X_{\alpha}$ is a finite-type $\KK$-scheme and $G_{\alpha}$
is a reductive linear algebraic group. An object $F\in\mathcal{A}$
corresponds to a point $x\in\mathfrak{M}$ and we call $\alpha(F)$
the dimension vector (resp. the Mukai vector) of $F$. If $x\in\mathfrak{M}$
is closed (i.e. the orbit of $x\in X_{\alpha}$ is closed), then $x$
corresponds to a semisimple (or polystable) object $F=\bigoplus_{i=1}^{r}F_{i}^{\oplus e_{i}}\in\mathcal{A}$.
We call $\tau(F)=(\alpha(F_{i}),e_{i}\ ;\ 1\leq i\leq r)$ the type
of $F$ (or $x$). Note that $\Aut(F)\simeq\GL_{\ee}$. As explained
above, one can build from $\tau(F)$ an auxiliary quiver $Q_{\tau}$
such that $\overline{Q_{\tau}}$ is the Ext-quiver of $F$.

Another common feature of quotient stacks $\mathfrak{M}_{\alpha}=[X_{\alpha}/G_{\alpha}]\subseteq\mathfrak{M}$
is the existence of good categorical quotients $M_{\alpha}=X_{\alpha}\git G_{\alpha}$
in the sense of Geometric Invariant Theory \cite[Ch. 6.]{Dol03}.
We can stratify $M_{\alpha}$ according to the conjugacy class of
stabilisers $G_{x}\subseteq G_{\alpha},\ x\in X_{\alpha}$ (where
$x$ has a closed orbit). For a reductive subgroup $H\subseteq G_{\alpha}$,
the stratum $(M_{\alpha})_{(H)}$ is locally closed and corresponds
to closed orbits in $X_{\alpha}$ with stabiliser subgroups in the
conjugacy class of $H$. Moreover:\[
\overline{(M_{\alpha})_{(H)}}=\bigcup_{(H')\leq (H)}(M_{\alpha})_{(H')},
\]where $(H')\leq(H)$ if $H$ is conjugated to a subgroup of $H'$.
Note that the type of $x$ determines the conjugacy class of $G_{x}$.

When $X_{\alpha}$ is an affine space with a linear $G_{\alpha}$-action,
properties of this stratification were proved in \cite[Lem. 5.5.]{Sch80}.
If $X_{\alpha}$ is an affine, finite-type $G_{\alpha}$-scheme, then
there exists a closed $G_{\alpha}$-equivariant embedding of $X_{\alpha}$
into a $G_{\alpha}$-representation (see \cite[Prop. 2.3.5.]{Bri17a}),
so we obtain the stratification by restriction. In general, $X_{\alpha}\git G_{\alpha}$
is locally isomorphic to an affine GIT quotient, so we obtain the
stratification by gluing.

\begin{rmk}

In \cite{LBP90}, Le Bruyn and Procesi showed that, for a quiver $Q$
and a dimension vector $\dd$, conjugacy classes of stabilisers of
closed orbits in $R(Q,\dd)$ are in bijection with semisimple types
appearing in $R(Q,\dd)$. However, given a polystable sheaf $\mathcal{F}=\bigoplus_{i=1}^{r}\mathcal{F}_{i}^{\oplus e_{i}}$
on a K3 surface, one can only recover the Hilbert polynomials $P_{\mathcal{F}_{i}},\ 1\leq i\leq r$
from the conjugacy class of its stabiliser and not the Mukai vectors
$\mathbf{v}(\mathcal{F}_{i}),\ 1\leq i\leq r$.

\end{rmk}

The above moduli stacks all arise from a 2-Calabi-Yau category, as
shown in \cite[\S7.]{Dav21a}. Davison works with a dg-enhancement
$\mathcal{T}{}^{\mathrm{dg}}$ of a triangulated category $\mathcal{T}$
containing $\mathcal{A}$. Then $\mathcal{T}{}^{\mathrm{dg}}$ carries
a left 2-Calabi-Yau structure, as defined by Brav and Dyckerhoff \cite{BD19}.
In the examples above, $\mathcal{T}$ is a category of complexes (dg-modules
over the dg-version of $\Pi_{Q}$ or $\Lambda^{q}(Q)$, complexes
of sheaves on a K3 surface) and $\mathcal{A\subset\mathcal{T}}$ is
an abelian subcategory of the category $\mathcal{B\subset\mathcal{T}}$
of complexes concentrated in degree zero. Then $\mathfrak{M}$ sits
as an open substack in the truncation of the derived moduli stack
of objects of $\mathcal{T}{}^{\mathrm{dg}}$, defined by Toën and
Vaquié \cite{TV07}.

Although under certain assumptions, the left 2-Calabi-Yau structure
on $\mathcal{T}{}^{\mathrm{dg}}$ does make $\mathcal{T}$ a 2-Calabi-Yau
category in the sense of Definition \ref{Def/CYCat} (see for instance
\cite[\S 10.1.]{KW21}), in the examples involving quivers, $\mathcal{T}$
may differ from the derived categories of $\Pi_{Q}$ (resp. $\Lambda^{q}(Q)$).
Indeed, in thoses cases, $\mathcal{T}$ is built from the dg-versions
of $\Pi_{Q}$ (resp. $\Lambda^{q}(Q)$). Consequently, in the theorem
below, the Ext-quivers associated to semisimple $\Pi_{Q}$-modules
(resp. $\Lambda^{q}(Q)$-modules) should a priori be computed using
Hom-spaces of $\mathcal{T}$.

However, when $Q$ is totally negative, the dg-version of $\Pi_{Q}$
(resp. $\Lambda^{q}(Q)$) is quasi-isomorphic to $\Pi_{Q}$ itself
(resp. $\Lambda^{q}(Q)$) - see \cite[\S 4.2.]{Kel08} and \cite[Prop. 4.4.]{KS23a}
- so $\mathcal{T}$ is equivalent to the derived category of $\Pi_{Q}$
(resp. $\Lambda^{q}(Q)$). In the case of coherent sheaves on a K3
surface $S$, the dg-category $\mathcal{T}{}^{\mathrm{dg}}$ under
consideration is a dg-enhancement of $D^{b}(S)$ - see \cite[\S 5.2.]{BD19},
\cite[\S 7.2.5.]{Dav21a}. Therefore, in our examples, Ext-quivers
can be computed from Ext-groups in $\mathcal{A}$. For this reason,
we do not give further details and refer to \cite{Dav21a} for a more
thorough discussion of the 2-Calabi-Yau structures at play. We collect
the common features described above in the following definition:

\begin{df} \label{Def/QuotStack2CY}

We call $[X/G]$ a quotient stack coming from a 2-Calabi-Yau category
if:
\begin{enumerate}
\item the stack $[X/G]$ is an open substack of the truncation of the derived
moduli stack of objects of a dg-category $\mathcal{T}{}^{\mathrm{dg}}$
endowed with a left 2-Calabi-Yau structure,
\item closed points of $[X/G]$ correspond to semisimple objects of an abelian,
finite-length, $\KK$-linear subcategory $\mathcal{A}\subset\mathcal{T}$;
moreover, if $x\in[X/G](\KK)$ corresponds to the semisimple object
$F=\bigoplus_{i=1}^{r}F_{i}^{\oplus e_{i}}\in\mathcal{A}$, then $(F_{1},\ldots,F_{r})$
is a $\Sigma$-collection in $\mathcal{T}$ (as defined in \cite[\S 1.3.]{Dav21a}),
\item the group $G$ is reductive and $X$ has a good categorical quotient
$X\rightarrow M:=X\git G$.
\end{enumerate}
\end{df}

It follows from the assumptions above that, if $x\in[X/G](\KK)$ corresponds
to the semisimple object $F=\bigoplus_{i=1}^{r}F_{i}^{\oplus e_{i}}\in\mathcal{A}$,
then $\Aut(x)\simeq\GL_{\ee}$.

We now describe local models of $X$ and $M$. In the case of $\mathfrak{M}_{S,H,\textbf{v}}$,
the theorem below follows from the formality result in \cite{BZ19}
(see also \cite[\S3-4.]{AS18} and the references therein). In the
case of $\mathfrak{M}_{\Lambda_{Q}^{q}}$, when $\Lambda^{q}(Q)$
is $2$-Calabi-Yau, a local description can be obtained from \cite[Thm. 5.12 - Thm. 5.16.]{KS23a},
using a $G$-equivariant version of Artin's approximation theorem
(see, for instance, \cite{AHR20}).

\begin{thm}{\cite[Thm. 5.11.]{Dav21a}} \label{Thm/LocMod2CY}

Let $[X/G]$ be a quotient stack coming from a 2-Calabi-Yau category
and $x\in[X/G](\KK)$ a closed point corresponding to $F\simeq\bigoplus_{i=1}^{r}F_{i}^{\oplus e_{i}}$.
Let $Q'$ be an auxiliary quiver associated to $F$ (i.e. $\overline{Q'}$
is the Ext-quiver of $F$). Then there exists an affine $\GL_{\ee}$-variety
$W$, with a fixed point $w$ and a commutative diagram:\[
\begin{tikzcd}[ampersand replacement = \&]
([\mu_{Q',\ee}^{-1}(0)/\GL_{\ee}],0) \ar[d] \& ([W/\GL_{\ee}],w) \ar[l]\ar[r]\ar[d] \& ([X/G],x) \ar[d] \\
(M_{\Pi_{Q'},\ee},0) \& (W\git\GL_{\ee},w) \ar[l]\ar[r] \& (M,x).
\end{tikzcd}
\]such that the horizontal maps are étale and the squares are cartesian.

\end{thm}

Similarly to section \ref{Subsect/RatSgTotNegQuiv}, this also gives
us $\GL_{\ee}$-equivariant étale morphisms $(W,w)\rightarrow(\mu_{Q',\ee}^{-1}(0),0)$
and $(W,w)\rightarrow(X,x)$ which induce:\[
\begin{tikzcd}[ampersand replacement = \&]
(\mu_{Q',\ee}^{-1}(0)\times^{\GL_{\ee}}G,[0,\Id]) \ar[d] \& (W\times^{\GL_{\ee}}G,[w,\Id]) \ar[l]\ar[r]\ar[d] \& (X,x) \ar[d] \\
(M_{\Pi_{Q'},\ee},0) \& (W\git\GL(\ee),w) \ar[l]\ar[r] \& (M,x).
\end{tikzcd}
\]

\paragraph*{Applications}

We now apply the results of Section \ref{Subsect/RatSgTotNegQuiv}
to moduli of objects in 2-Calabi-Yau categories, using Theorem \ref{Thm/LocMod2CY}.
Let $\mathfrak{M}=[X/G]$ be a quotient stack coming from a 2-Calabi-Yau
category. As a consequence of Theorem \ref{Thm/MainRes}, if all auxiliary
quivers arising from $\mathfrak{M}$ have property (P), then $X$
has l.c.i. and rational singularities. This leads us to the notion
of a totally negative category, as introduced in \cite{DHSM22}. We
define it using Hom-spaces in triangulated categories, similarly to
Definition \ref{Def/CYCat}, since in the example we consider, Ext-quivers
can be computed from such Hom-spaces.

\begin{df} \label{Def/TotNegCat}

Let $\mathcal{T}$ be a $\KK$-linear, Hom-finite, triangulated category.
Suppose that for all $F_{1},F_{2}\in\mathcal{T}$, we have $\Hom(F_{1},F_{2}[i])=0$
for all but finitely many $i\in\mathbb{Z}$. The category $\mathcal{T}$
is called totally negative if, for all $F_{1},F_{2}\in\mathcal{T}$:\[
\sum_{i\geq0}(-1)^i\cdot\hom(F_1,F_2[i])<0.
\]\end{df}

For example, when $\Pi_{Q}$ or $\Lambda^{q}(Q)$ are 2-Calabi-Yau,
their derived categories are totally negative if, and only if, $Q$
is totally negative (see Section \ref{Subsect/Mod2CY}).

Actually, to obtain l.c.i., rational singularities for $X$, we only
need weaker geometric assumptions on $\mathfrak{M}$, as we will only
be considering auxiliary quivers arising from $X$, not from a whole
totally negative category. We also require an additional property
on the subset of simple objects so that auxiliary quivers also satisfy
property (P). Nevertheless, the total negativity assumption below
is met as soon as $\mathfrak{M}$ parametrises objects living in a
totally negative category.

\begin{thm} \label{Thm/RatSgTotNeg2CY}

Let $\mathfrak{M}=[X/G]$ be a quotient stack coming from a 2-Calabi-Yau
category and $M:=X\git G$. Assume that:
\begin{enumerate}
\item for any closed point $x\in\mathfrak{M}(\KK)$, the Ext-quiver associated
to $x$ is the double of a totally negative quiver and
\item simple objects form a dense subset of $M$.
\end{enumerate}
Then $X$ is locally complete intersection and has rational singularities.

\end{thm}

\begin{proof}

Let $x\in X(\KK)$ be a point with closed orbit and $(\overline{Q'},\ee)$
the associated Ext-quiver. Then by Theorem \ref{Thm/LocMod2CY}, there
exists an affine $\GL(\ee)$-variety $W$, with a fixed point $w$
and a commutative diagram:\[
\begin{tikzcd}[ampersand replacement = \&]
(\mu_{Q',\ee}^{-1}(0)\times^{\GL_{\ee}}G,[0,\Id]) \ar[d] \& (W\times^{\GL_{\ee}}G,[w,\Id]) \ar[l]\ar[r]\ar[d] \& (X,x) \ar[d] \\
(M_{\Pi_{Q'},\ee},0) \& (W\git\GL(\ee),w) \ar[l]\ar[r] \& (M,x)
\end{tikzcd}
\]such that the horizontal maps are étale and the squares are cartesian.
Since the set of simple objects is dense in $M$, one can find a simple
object in the image of $(W\git\GL_{\ee},w)\rightarrow(M,x)$. Moreover,
strongly étale morphisms preserve stabilisers (see \cite[Lem. 2.6.]{Bud21}),
so we can find a semisimple point in $\mu_{Q',\ee}^{-1}(0)$ whose
stabiliser is $\KK^{\times}$ i.e. there is a simple $\Pi_{Q'}$-module
of dimension $\ee$. From this and assumption 1., we deduce that $Q'$
satisfies property (P).

Therefore, by Theorem \ref{Thm/MainRes} and Lemmas \ref{Lem/LciSgDesc},
\ref{Lem/RatSgDesc}, any semisimple point of $X$ has a $G$-saturated
neighbourhood which is locally complete intersection and has rational
singularities. Since $X$ is covered by such neighbourhoods, the conclusion
follows. \end{proof}

We now give a few examples and apply our result to moduli of coherent
sheaves on K3 surfaces and representations of multiplicative preprojective
algebras.

\subparagraph*{$\Lambda^{q}(Q)$-modules}

In the case of the multiplicative preprojective algebra $\Lambda^{q}(Q)$,
the auxiliary quivers are computed as in the additive case from $Q$
itself. Therefore, if $Q$ is totally negative, then all auxiliary
quivers are. The scheme $R(\Lambda^{q}(Q),\dd)$ is also known to
be an affine complete intersection, by \cite[Thm. 1.11.]{CBS06}.
Thus, we obtain:

\begin{cor} \label{Cor/RatSgMultPreproj}

Let $Q$ be a quiver and $\dd\in\mathbb{Z}_{\geq0}^{Q_{0}}\setminus\{0\}$
such that $(Q,\dd)$ has property (P). Then $R(\Lambda^{q}(Q),\dd)$
is a complete intersection and has rational singularities.

\end{cor}

\begin{proof}

We apply Theorem \ref{Thm/RatSgTotNeg2CY}. It remains to check assumption
2. Since $(Q,\dd)$ has property (P), there exists a simple $\Pi_{Q}$-module
of dimension $\dd$. Thus by \cite[Thm. 1.2.]{CB01} and \cite[Thm. 1.11.]{CBS06},
simple points form an open dense subset of $R(\Lambda^{q}(Q),\dd)$,
hence a dense subset of $M_{\Lambda^{q}(Q),\dd}$. \end{proof}

\subparagraph*{Semistable coherent sheaves on a K3 surface}

In the case of sheaves on a K3 surface, controlling Ext-quivers of
all polystable sheaves with given Mukai vector is more delicate. In
\cite{BZ19}, it was already mentioned that for many polystable sheaves,
one obtains the auxiliary quivers of the $g$-loop quiver. We spell
out conditions for this to be true. However, one can also obtain étale-local
models which are much more difficult to study. We illustrate this
with the case of one-dimensional sheaves on an elliptic K3 surface.

Let $(S,H)$ be a complex, projective, polarised K3 surface. Given
a polystable sheaf $\mathcal{F}=\bigoplus_{i=1}^{r}\mathcal{F}_{i}^{\oplus e_{i}}$,
with distinct stable summands $\mathcal{F}_{i}$ of Mukai vectors
$\textbf{v}_{i}$, recall that the Ext-quiver $\overline{Q'}$ has
the following number of arrows from $i$ to $j$:\[
\ext^1(\mathcal{F}_{i},\mathcal{F}_{j})=
\left\{
\begin{array}{ll}
2+\textbf{v}_{i}\cdot\textbf{v}_{i}, & \text{ if } i=j, \\
\textbf{v}_{i}\cdot\textbf{v}_{j}, & \text{ if } i\ne j.
\end{array}
\right.
\]The quiver $Q'$ is totally negative if, and only if, $\textbf{v}_{i}\cdot\textbf{v}_{j}>0$
for all $i,j$. To see this, one can simply check when $Q'$ has at
least two loops at each vertex and one arrow joining any pair of vertices
(recall from Section \ref{Subsect/Mod2CY} that $\textbf{v}_{i}\cdot\textbf{v}_{i}$
is even). Otherwise, one can also notice that:\[
(\dd,\ee)_{Q'}=-\left(\sum_i d_i\textbf{v}_{i}\right)\cdot\left(\sum_i e_i\textbf{v}_{i}\right).
\]We now give examples of auxiliary quivers arising from $\mathfrak{M}_{S,H}(\textbf{v})$.
We rely on criteria for the existence of (semi)stable sheaves of a
given Mukai vector, due to Yoshioka \cite[Thm. 8.1.]{Yos01} (see
also \cite[\S2.4.]{KLS06}).

\begin{df} \label{Def/PrimPosMukVec}

A Mukai vector $\textbf{v}$ is called primitive if it cannot be written
$\textbf{v}=m\textbf{v}_{0}$, for some other Mukai vector $\textbf{v}_{0}$
and $m\geq2$.

A primitive vector $\textbf{v}=(r,\textbf{c},a)$ is called positive
if $\textbf{v}\cdot\textbf{v}\geq-2$ and one of the following holds:
\begin{itemize}
\item $r>0$ and $\textbf{c}\in\mathrm{NS}(S)$,
\item $r=0$, $\textbf{c}\in\mathrm{NS}(S)$ is effective and $a\ne0$,
\item $r=0$, $\textbf{c}=0$ and $a>0$.
\end{itemize}
\end{df}

Let $\text{Amp}(S)$ be the space of polarizations of $S$ and $\textbf{v}$
a primitive, positive Mukai vector. One can associate walls in $\text{Amp}(S)\otimes_{\mathbb{Z}}\RR$
to $\textbf{v}$, i.e. subspaces of real codimension one, such that,
for $H$ lying outside these walls, there are no strictly semistable
sheaves with Mukai vector $\textbf{v}$ (see \cite[\S1.4.]{Yos01}
for $\rk(\textbf{v})=0$ and \cite[Ch. 4.C.]{HL10} for $\rk(\textbf{v})>0$).
Those polarizations are called $\textbf{v}$-generic. The existence
results for semistable sheaves are summed up in the following:

\begin{prop}{\cite[Thm. 2.2-4.]{AS18}} \label{Prop/ExisStabShK3}

If $\textbf{v}$ is a primitive, positive Mukai vector, then $M_{\textbf{v}}\ne\emptyset$.
If moreover $H$ is $\textbf{v}$-generic, then $M_{\textbf{v}}=M_{\textbf{v}}^{s}$.

\end{prop}

In general, there is no reason for auxiliary quivers of all polystable
sheaves in $M_{\textbf{v}}$ to be totally negative. However, if we
choose a moduli space $M_{\textbf{v}}$ of high enough dimension and
a generic polarization, then the auxiliary quiver of any polystable
sheaf in $M_{\textbf{v}}$ is totally negative (and it arises from
a $g$-loop quiver). This was already observed in \cite{BZ19}.

\begin{prop} \label{Prop/ExtQuivK3}

Let $\textbf{v}=m\textbf{v}_{0}$, where $\textbf{v}_{0}$ is a primitive,
positive Mukai vector and assume that $H$ is $\textbf{v}_{0}$-generic.
Suppose also that $\textbf{v}\cdot\textbf{v}>0$. Then for any polystable
sheaf $\mathcal{F}=\bigoplus_{i}\mathcal{F}_{i}^{\oplus e_{i}}$ with
Mukai vector $\textbf{v}$, the auxiliary quiver of $\mathcal{F}$
is an auxiliary quiver of some $g$-loop quiver, where $g\geq2$.

\end{prop}

\begin{proof}

Call $\mathbf{v}_{i}$ the Mukai vector of $\mathcal{F}_{i}$. We
leave out the obvious case where $\mathcal{F}$ is stable. Since $H$
is $\textbf{v}_{0}$-generic and by construction of walls, there exist
$r_{i}\in\mathbb{Q}$ such that $\mathbf{v}_{i}=r_{i}\mathbf{v}$
(see \cite[\S2.4.]{KLS06}). Since $\mathbf{v}_{0}$ is primitive,
$\mathbf{v}_{i}=m_{i}\mathbf{v}_{0}$ for some $m_{i}\in\mathbb{Z}$.
Moreover, $\mathbf{v}_{0}\cdot\mathbf{v}_{0}$ is a positive even
integer, hence $\mathbf{v}_{0}\cdot\mathbf{v}_{0}=2g-2$, with $g\geq2$.
Consequently, the auxiliary quiver of $\mathcal{F}$ has $1+m_{i}^{2}(g-1)$
loops at each vertex and $2m_{i}m_{j}(g-1)$ arrows joining vertices
$i\ne j$. One recognises the auxiliary quiver of the $g$-loop quiver
for the semisimple type $(m_{i},e_{i}\ ;\ 1\leq i\leq r)$ (see \cite[Prop. 2.26.]{Bud21}).
\end{proof}

In that case, \cite[Thm. 1.1.]{Bud21} yields the following corollary,
already observed in \cite{BZ19}:

\begin{cor}

Under the hypotheses of Proposition \ref{Prop/ExtQuivK3}, $U_{\textbf{v},m_{\textbf{v}}}^{\mathrm{ss}}$
is locally complete intersection and has rational singularities.

\end{cor}

In \cite[\S6.]{AS18}, Arbarello and Saccà analysed Ext-quivers of
a pure, one-dimensional torsion sheaf $\mathcal{F}$ with primitive,
positive Mukai vector $\textbf{v}=(0,\textbf{[D]},\chi)$ ($D$ effective,
$\chi\ne0$). Write $D=\sum_{i\in I}n_{i}D_{i}$ where $n_{i}\geq0$
and $D_{i}$ is integral. Then the auxiliary quiver of $\mathcal{F}$
has vertex set $I$, has $g(D_{i})$ loops at vertex $i$ and $D_{i}\cdot D_{j}$
arrows joining vertices $i$ and $j$. The following example illustrates
how difficult the geometry of étale-local models of $\mathfrak{M}_{S,H,\textbf{v}}$
gets.

\begin{exmp}

Let $\textbf{v}=(0,\textbf{[D]},1)$, where $D$ is an elliptic curve
in $S$ (for instance, suppose $S$ is an elliptic K3 surface and
$D$ is a generic fibre). Then $\textbf{v}$ is primitive, positive,
$\textbf{v}\cdot\textbf{v}=2g(D)-2=0$ and taking a $\textbf{v}$-generic
polarization, we obtain by Proposition \ref{Prop/ExisStabShK3} a
stable sheaf $\mathcal{F}$ with Mukai vector $\textbf{v}$. For any
$n\geq1$, the auxiliary quiver of $\mathcal{F}^{\oplus n}$ is the
Jordan quiver (one vertex with one loop), with dimension vector $n$.
The associated local model is the commuting scheme:\[
C_2:=V([x,y]_{i,j},\ 1\leq i,j\leq n)\subset\Mat(n,\KK)_x\times\Mat(n,\KK)_y
\]Commuting schemes have been studied by many authors and were only
recently proved to be normal and Cohen-Macaulay (see \cite{Cha20}
and the references therein). Whether commuting schemes have rational
singularities is a more difficult question and, as far as we know,
this problem is still open.

\end{exmp}

\subparagraph*{Jet-counts}

Finally, we draw consequences of \ref{Thm/Ch3RatSgTotNeg} for counts
of jets on moduli stacks of objects in 2-Calabi-Yau categories. Consider
a quotient stack $[X/G]$ coming from a 2-Calabi-Yau category. If
$[X/G]$ is locally modelled on moment maps of totally negative quivers,
then Theorem \ref{Thm/JetCountCanMeas} applies and we obtain the
desired interpretation of limits of jet-counts for a large class of
moduli spaces.

\begin{thm} \label{Thm/CanMeas2CYMod}

Suppose that $X$ is of finite type, of pure dimension $d$ and defined
over $\mathbb{Q}$. Assume also that $[X/G]$ satisfies the hypotheses
of Theorem \ref{Thm/RatSgTotNeg2CY}. 

Then for $p$ large enough, $X$ is defined over $\mathcal{O}$ and
the canonical measure $\nu_{\mathrm{can}}$ on $X^{\natural}=X^{\mathrm{sm}}(F)\cap X(\mathcal{O})$
is well-defined. Moreover, the sequence $q^{-nd}\cdot\sharp X(\mathbb{F}_{q}[t]/t^{n}),\ n\geq1$
converges and its limit is given by:\[
\underset{n\rightarrow +\infty}{\lim}\frac{\sharp X(\mathbb{F}_{q}[t]/(t^{n}))}{q^{nd}}
=
\nu_{\mathrm{can}}(X^{\natural}).
\]\end{thm}

Note that the assumption that $X$ is of pure dimension $d$ may be
dropped. Since $X$ is locally complete intersection, its connected
components are equidimensional and the jet-counts may be carried out
separately on each connected component with the appropriate value
of $d$.

\begin{rmk} \label{Rmk/CountJetsStacks}

When $G$ is connected, jet-counts on $X$ are equivalent to jet-counts
on the moduli stack $[X/G]$, weighted by their number of automorphisms:\[
\vol([X/G](\mathbb{F}_q[t]/(t^n))=\frac{\sharp X(\mathbb{F}_q[t]/(t^n))}{\sharp G(\mathbb{F}_q[t]/(t^n))}.
\]Indeed, the datum of a principal $G$-bundle $P\rightarrow\Spec(\mathbb{F}_{q}[t]/(t^{n}))$
and a $G$-equivariant morphism $P\rightarrow X$ is equivalent to
the datum of an orbit of $X(\mathbb{F}_{q}[t]/(t^{n}))$ under the
action of $G(\mathbb{F}_{q}[t]/(t^{n}))$. This is due to the fact
that the restricted morphism $P_{\mathbb{F}_{q}}\rightarrow\Spec(\mathbb{F}_{q})$
has a section, by Lang's theorem. By Hensel's lemma, this section
extends to $\Spec(\mathbb{F}_{q}[t]/(t^{n}))$ and so $P\simeq G\times\Spec(\mathbb{F}_{q}[t]/(t^{n}))$.
As a consequence, when $[X/G]$ satisfies the assumptions of Theorem
\ref{Thm/RatSgTotNeg2CY}, the weighted jet-count on $[X/G]$, once
normalised, converges to $\frac{\nu_{\mathrm{can}}(X^{\natural})}{\nu_{\mathrm{can}}(G^{\natural})}$
as $n$ goes to infinity.

\end{rmk}

\pagebreak{}

\section{Categorification of toric Kac polynomials, positivity and purity
\label{Chap/Categorification}}

In this chapter, we prove that Kac polynomials of certain quivers
with multiplicities have non-negative coefficients, in rank $\rr=\underline{1}$.
Our proof relies on an inductive procedure, which contracts or deletes
edges of the quiver. This is inspired from the study of toric Kac
polynomials by Abdelgadir, Mellit and Rodriguez-Villegas \cite{AMRV22}.
Moreover, we provide a categorification of the Kac polynomials under
study. This is done by computing the cohomology of certain preprojective
stacks $\mathfrak{M}_{\Pi_{Q,\nn},\rr}$, using the above contraction-deletion
procedure to stratify $\mathfrak{M}_{\Pi_{Q,\nn},\rr}$ (when $\rr=\underline{1}$).
The results in this Chapter are all exposed in the preprint \cite{Ver23}.

Fix $\KK$ a finite field. Let $Q$ be a quiver and $\alpha\geq1$.
To simplify notations, we call $(Q,\alpha)$ the quiver with multiplicities
$(Q,\nn)$, where $\nn=(\alpha)_{i\in Q_{0}}$, and $\langle\bullet,\bullet\rangle$
the Euler form of $Q$ (without multiplicities). When $\rr=\underline{1}$,
it was already shown in \cite{Wys17b,HLRV24} that $A_{(Q,\alpha),\rr}$
is a polynomial. It was also conjectured that $A_{(Q,\alpha),\rr}$
has non-negative coefficients, generalising Kac's conjecture in the
case $\alpha=1$ \cite[Rmk. 7.7.i.]{HLRV24}. We prove this conjecture.

\begin{thm} \label{Thm/Ch4PosToricKacPol}

Let $\rr=\underline{1}$. Then $A_{(Q,\alpha),\rr}$ has non-negative
coefficients.

\end{thm}

As was already observed in Chapter \ref{Chap/KacPolynomials}, the
counts $A_{(Q,\alpha),\rr}$ are strongly related to point-counts
of $\mu_{(Q,\alpha),\rr}^{-1}(0)$ over finite fields. Here, we recover
$A_{(Q,\alpha),\rr}$ from a generic fibre of $\mu_{(Q,\alpha),\rr}$,
generalising results of Crawley-Boevey, Van den Bergh \cite{CBVB04}
and Davison \cite{Dav23a}.

\begin{thm} \label{Thm/Ch4CountDefMomMapFibre}

Let $\rr\in\ZZ_{\geq0}^{Q_{0}}$ be an indivisible rank vector and
$\lambda\in\ZZ^{Q_{0}}$ be generic with respect to $\rr$. Then:
\[
\frac{\sharp\mu_{(Q,\alpha),\rr}^{-1}(t^{\alpha-1}\cdot\lambda)}{\sharp\GL_{\alpha,\rr}}=q^{-\alpha\langle\rr,\rr\rangle}\cdot\frac{A_{(Q,\alpha),\rr}}{1-q^{-1}}.
\]

\end{thm}

Finally, we upgrade the relation between toric Kac polynomials of
$(Q,\alpha)$ and fibres of $\mu_{(Q,\alpha),\rr}$ to a cohomological
one. Indeed, we show that the compactly supported cohomology of the
complex stacks $\left[\mu_{(Q,\alpha),\rr}^{-1}(t^{\alpha-1}\cdot\lambda)/\GL_{\alpha,\rr}\right]$
and $\left[\mu_{(Q,\alpha),\rr}^{-1}(0)/\GL_{\alpha,\rr}\right]$
is pure and contains a pure Hodge structure with E-polynomial $A_{(Q,\alpha),\rr}$.
This is reminiscent of the cohomological integrality established in
\cite{Dav23c} for preprojective stacks of quivers without multiplicities.

\begin{thm} \label{Thm/Ch4PurityDefToricMomMapFibre}

Let $\rr=\underline{1}$. then:\[
\HH_{\mathrm{c}}^{\bullet}\left(\left[\mu_{(Q,\alpha),\rr}^{-1}(t^{\alpha-1}\cdot\lambda)/\GL_{\alpha,\rr}\right]\right)
\simeq
A_{(Q,\alpha),\rr}(\mathbb{L})\otimes\mathbb{L}^{1-\alpha\langle\rr,\rr\rangle}\otimes\HH_{\mathrm{c}}^{\bullet}\left(\mathrm{B}\mathbb{G}_{\mathrm{m}}\right)
\]In particular, $\HH_{\mathrm{c}}^{\bullet}\left(\left[\mu_{(Q,\alpha),\rr}^{-1}(t^{\alpha-1}\cdot\lambda)/\GL_{\alpha,\rr}\right]\right)$
carries a pure Hodge structure.

\end{thm}

\begin{thm} \label{Thm/Ch4PurityToricPreprojStack}

Let $\rr=\underline{1}$. Then:

\[
\HH_{\mathrm{c}}^{\bullet}\left(\left[\mu_{(Q,\alpha),\rr}^{-1}(0)/\GL_{\alpha,\rr}\right]\right)
\otimes\mathbb{L}^{\otimes\alpha\langle\rr,\rr\rangle}
\simeq
\bigoplus_{Q_0=I_1\sqcup\ldots\sqcup I_s}
\bigotimes_{j=1}^s
\left(
A_{(Q\vert_{I_j},\alpha),\rr\vert_{I_j}}(\mathbb{L})\otimes\mathbb{L}\otimes \HH_{\mathrm{c}}^{\bullet}(\mathrm{B}\mathbb{G}_m)
\right).
\]

In particular, $\HH_{\mathrm{c}}^{\bullet}\left(\left[\mu_{(Q,\alpha),\rr}^{-1}(0)/\GL_{\alpha,\rr}\right]\right)$
carries a pure Hodge structure.

\end{thm}

Throughout this chapter, we will use the following simplified notations.
Given a base field $\KK$ and $\alpha\geq1$, we define $\mathcal{O}_{\alpha}:=\KK[t]/(t^{\alpha})$,
as in \cite{Wys17b}. We denote as above $(Q,\alpha)$ the quiver
with multiplicities $(Q,\nn)$, where $\nn=(\alpha)_{i\in Q_{0}}$.
We also write $\GL_{\alpha,\rr}:=\GL_{\nn,\rr}$, $\mathbb{G}_{\mathrm{m},\alpha}:=\GL_{\alpha,1}$,
$R(Q,\alpha;\rr):=R(Q,\nn;\rr)$ and $\mu_{(Q,\alpha),\rr}:=\mu_{(Q,\nn),\rr}$.
We call $\langle\bullet,\bullet\rangle$ the Euler form of $Q$ (without
multiplicities). Note that $\langle\bullet,\bullet\rangle_{H}=\alpha\cdot\langle\bullet,\bullet\rangle$,
where $H=H(C,D,\Omega)$ is the path algebra of the Cartan matrix
(with symmetriser and orientation) associated to $(Q,\alpha)$. Unless
we specify otherwise, all counts are carried out in the ring of volumes
$\mathcal{V}$ (see Section \ref{Subsect/KrullSchmidt}).

We first set up some graph-theoretic conventions in Section \ref{Subsect/Graph}.
We then describe the contraction-deletion procedure in Section \ref{Subsect/Algo}
and prove Theorem \ref{Thm/Ch4PosToricKacPol}. Finally, we prove
Theorem \ref{Thm/Ch4CountDefMomMapFibre} and carry out the categorification
of toric Kac polynomials (Theorems \ref{Thm/Ch4PurityDefToricMomMapFibre}
and \ref{Thm/Ch4PurityToricPreprojStack}) in Section \ref{Subsect/Purity}.

\subsection{Graph-theoretic conventions \label{Subsect/Graph}}

In this section, we briefly set notations for the operations on quivers
used in the proof of Theorems \ref{Thm/Ch4PosToricKacPol}, \ref{Thm/Ch4PurityDefToricMomMapFibre}
and \ref{Thm/Ch4PurityToricPreprojStack}. These include restricting
to subquivers, contracting and deleting arrows. We also introduce
certain labeled spanning trees of quivers, which we will use to index
strata of the stack $\left[\mu_{(Q,\alpha),\rr}^{-1}(0)/\GL_{\alpha,\rr}\right]$
(when $\rr=\underline{1}$), in order to compute its cohomology.

Let $Q$ be a quiver. We first recall notations from Section \ref{Subsect/WyssConj}
for subquivers obtained from $Q$ by restriction.

\begin{df}

Let $I\subseteq Q_{0}$ be a subset of vertices and $J\subseteq Q_{1}$
be a subset of arrows of $Q$.
\begin{enumerate}
\item $Q\vert_{I}$ is the quiver with set of vertices $I$ and set of arrows
$Q_{1,I}:=\{a\in Q_{1}\ \vert\ s(a),t(a)\in I\}$. The source and
target maps of $Q\vert_{I}$ are obtained from those of $Q$ by restriction.
\item $Q\vert_{J}$ is the quiver with set of vertices $Q_{0}$ and set
of arrows $J$. The source and target maps of $Q\vert_{J}$ are obtained
from those of $Q$ by restriction.
\end{enumerate}
\end{df}

We also define quivers obtained from $Q$ by contracting or deleting
an arrow.

\begin{df}

Let $a\in Q_{1}$ be an arrow of $Q$. Consider the equivalence relation
$\sim_{a}$ on $Q_{0}$, whose equivalence classes are $\{s(a),t(a)\}$
and $\{i\},\ i\in Q_{0}\setminus\{s(a),t(a)\}$. Given a vertex $i\in Q_{0}$,
we denote by $[i]$ the equivalence class of $i$ under $\sim_{a}$.
\begin{enumerate}
\item $Q/a$ is the quiver obtained from $Q$ by contracting $a$ i.e. its
set of vertices is $Q_{0}/\sim_{a}$ and its set of arrows is $(Q/a)_{1}:=\{[s(b)]\rightarrow[t(b)],\ b\in Q_{1}\setminus\{a\}\}$.
The source and target maps are obtained from those of $Q$ by composing
with $Q_{0}\rightarrow Q_{0}/\sim_{a}$.
\item $Q\setminus a$ is the quiver obtained from $Q$ by deleting $a$
i.e. its set of vertices is $Q_{0}$ and its set of arrows is $(Q\setminus a)_{1}:=Q_{1}\setminus\{a\}$.
The source and target maps are obtained from those of $Q$ by restriction.
\end{enumerate}
Likewise, given $\lambda\in\ZZ^{Q_{0}}$, we define:
\begin{enumerate}
\item $\lambda/a\in\ZZ^{(Q/a)_{0}}$, given by \[
(\lambda/a)_i=
\left\{
\begin{array}{ll}
\lambda_{s(a)}+\lambda_{t(a)} & ,\ i=[s(a)]=[t(a)] \\
\lambda_i & ,\ i\not\sim_a s(a),t(a)
\end{array}
\right.
.
\]
\item $\lambda\setminus a\in\ZZ^{(Q\setminus a)_{0}}=\ZZ^{Q_{0}}$, which
is simply given by $\lambda$.
\end{enumerate}
\end{df}

Finally, we introduce a certain labeling for spanning trees of quivers.
Labels are given by valuations of elements in $\mathcal{O}_{\alpha}$.
This will be used to keep track of the different outcomes of our contraction-deletion
algorithm. 

\begin{df}[Valued spanning tree]

A spanning tree of a connected quiver $Q$ is a connected subquiver
$Q\vert_{J}$, where $J\subseteq Q_{1}$ has cardinality $\sharp Q_{0}-1$.
In other words, $b(Q\vert_{J})=0$, which means that $Q\vert_{J}$
has no cycles i.e. it is a tree. A valued spanning tree $T$ of $Q$
is the datum of a spanning tree $Q\vert_{J}$ ($J\subseteq Q_{1}$),
along with a labeling $v:J\rightarrow\mathbb{Z}$.

\end{df}

\subsection{A contraction-deletion algorithm \label{Subsect/Algo}}

We now turn to the proof of Theorem \ref{Thm/Ch4PosToricKacPol}.
Recall from \cite[Prop. 4.34.]{Wys17b} that $A_{(Q,\alpha),\underline{1}}$
enjoys the following explicit formula:\[
A_{(Q,\alpha),\underline{1}}=
\sum_{\substack{E_1\subseteq\ldots\subseteq E_{\alpha}\subseteq Q_1 \\ c(Q\vert_{E_{\alpha}})=1}}
(q-1)^{b(Q\vert_{E_{\alpha}})}q^{\sum_{k=1}^{\alpha-1}b(Q\vert_{E_{k}})}.
\] Set $\rr=\underline{1}$, $\KK=\FF_{q}$ and call $R(Q,\alpha):=R(Q,\alpha;\underline{1})$
for short. Let $R(Q,\alpha)_{\mathrm{ind.}}\subseteq R(Q,\alpha)$
be the constructible subset of (absolutely) indecomposable representations\footnote{By Proposition \ref{Prop/EndRings}, a locally free representation
of rank $\underline{1}$ is indecomposable if, and only if, it is
absolutely indecomposable.}. Suppose also that $Q$ is connected, so that $R(Q,\alpha)_{\mathrm{ind.}}\ne\emptyset$.
The above formula is obtained by counting $\left(\mathcal{O}_{\alpha}^{\times}\right)^{Q_{0}}$-orbits
(of $\FF_{q}$-points) along a stratification of $R(Q,\alpha)$, which
is defined by prescribing valuations $\val(x_{a}),\ a\in Q_{1}$.
The datum of valuations $\val(x_{a})$ is encoded in the subsets $E_{1}\subseteq\ldots\subseteq E_{\alpha}\subseteq Q_{1}$
and there are $(q-1)^{b(Q\vert_{E_{\alpha}})}q^{\sum_{k=1}^{\alpha-1}b(Q\vert_{E_{k}})}$
orbits in the stratum associated to $E_{1}\subseteq\ldots\subseteq E_{\alpha}\subseteq Q_{1}$.
As the formula shows, the polynomial counting orbits in a given stratum
does not necessarily have non-negative coefficients. To fix this,
we consider a coarser stratification of $R(Q,\alpha)$ based on valued
spanning trees and inspired from \cite{AMRV22}. Let us describe the
algorithm we use to assign a valued spanning tree to $x\in R(Q,\alpha)$.

\paragraph*{Contraction-deletion algorithm}

Fix a total ordering of $Q_{1}$. Let $x\in R(Q,\alpha)_{\mathrm{ind}}$.
We build a valued spanning tree of $Q$, called $T_{x}$, by following
the algorithm below:
\begin{enumerate}
\item If $Q_{1}$ contains at least one non-loop arrow $a$ such that $x_{a}\in\mathcal{O}_{\alpha}^{\times}$,
call $a_{0}\in Q_{1}$ the largest such arrow and apply step 1 to
the induced representation $x'\in R(Q/a_{0},\alpha)$. Otherwise,
apply step 2 to $x$.
\item If $Q_{1}$ contains at least one loop, call $a_{0}\in Q_{1}$ the
largest loop and apply step 2 to the induced representation $x'\in R(Q\setminus a_{0},\alpha)$.
Otherwise, apply step 3 to $x$.
\item If $Q_{1}\ne\emptyset$ and $\val(x_{a})>0$ for all $a\in Q_{1}$,
apply step 1 to the representation $x'\in R(Q,\alpha-1)$ induced
by the isomorphism of $\mathcal{O}_{\alpha}$-modules $\mathcal{O}_{\alpha-1}\simeq t\mathcal{O}_{\alpha}$.
\end{enumerate}
Note that the algorithm terminates in step 3, when $Q$ is a one-vertex
quiver with no loops. Indeed, since $x\in R(Q,\alpha)_{\mathrm{ind}}$,
the restriction of $Q$ to $\{a\in Q_{1}\ \vert\ x_{a}\ne0\}\subseteq Q_{1}$
is connected and the algorithm ends up contracting exactly the edges
of a spanning tree of $Q$, which we call $T_{x}$. Note that the
assumption $\val(x_{a})>0$ in step 3 is necessarily satisfied for
all $a\in Q_{1}$, as non-loop arrows (resp. loops) such that $x_{a}\in\mathcal{O}_{\alpha}^{\times}$
are contracted (resp. deleted) in step 1 (resp. step 2). We define
the valued spanning tree associated to $x$ as $T_{x}$, with labeling
$v:a\mapsto\val(x_{a})$.

\begin{exmp}

Let us illustrate the above algorithm on the following quiver:
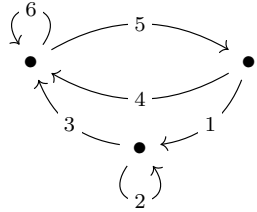
\[
\begin{tikzcd}[ampersand replacement=\&]
\bullet \arrow[loop, distance=2em, in=125, out=55, "6" description] \arrow[rr, bend left, "5" description] \& \& \bullet \arrow[ll, bend left, "4" description] \arrow[ld, bend left, "1" description] \\
 \& \bullet \arrow[lu, bend left, "3" description] \arrow[loop, distance=2em, in=305, out=235, "2" description] \&
\end{tikzcd}
\] and representation $x=(x_{1},x_{2},x_{3},x_{4},x_{5},x_{6})=(t,t^{2},t^{2},1,t,1)$.
Here, the order on $Q_{1}$ is the natural order on the integers.
Then $T_{x}$ is the tree with edges $\{1,4\}$.\[
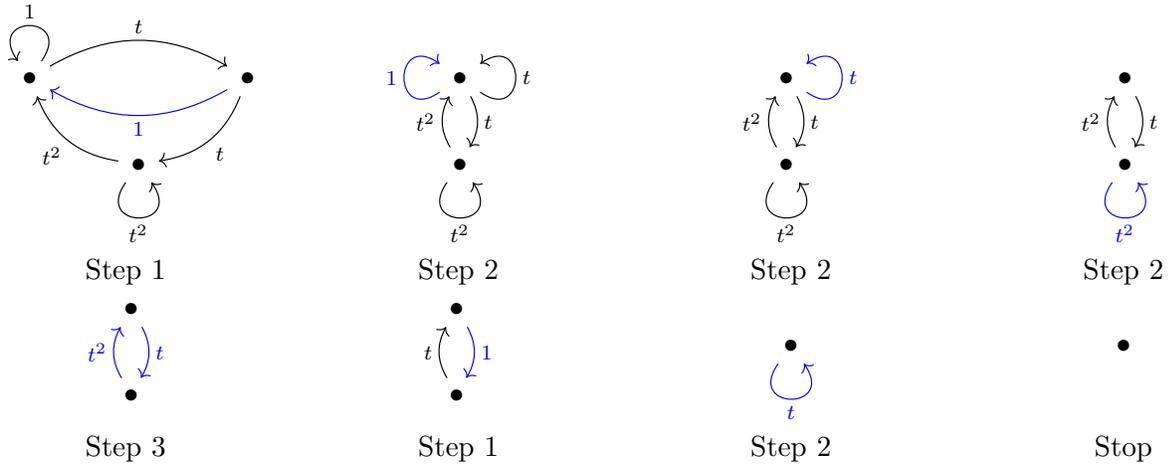
\begin{tabular}{>{\centering\arraybackslash}p{4cm} >{\centering\arraybackslash}p{4cm} >{\centering\arraybackslash}p{4cm} >{\centering\arraybackslash}p{4cm} >{\centering\arraybackslash}p{4cm}}
\begin{tikzcd}[ampersand replacement=\&]
\bullet \arrow[loop, distance=2em, in=125, out=55, "1", swap] \arrow[rr, bend left, "t"] \& \& \bullet \arrow[ll, bend left, "1", blue] \arrow[ld, bend left, "t"] \\
 \& \bullet \arrow[lu, bend left, "t^2"] \arrow[loop, distance=2em, in=305, out=235, "t^2", swap] \&
\end{tikzcd} &
\begin{tikzcd}[ampersand replacement=\&]
\bullet \arrow[loop, distance=2em, in=35, out=325, "t",  swap] \arrow[loop, distance=2em, in=145, out=215, "1", blue] \arrow[d, bend left, "t"] \\
\bullet \arrow[u, bend left, "t^2"] \arrow[loop, distance=2em, in=305, out=235, "t^2", swap]
\end{tikzcd} &
\begin{tikzcd}[ampersand replacement=\&]
\bullet \arrow[loop, distance=2em, in=35, out=325, "t",  swap, blue] \arrow[loop, distance=2em, in=145, out=215, phantom] \arrow[d, bend left, "t"] \\
\bullet \arrow[u, bend left, "t^2"] \arrow[loop, distance=2em, in=305, out=235, "t^2", swap]
\end{tikzcd} &
\begin{tikzcd}[ampersand replacement=\&]
\bullet \arrow[d, bend left, "t"]  \\
\bullet \arrow[u, bend left, "t^2"] \arrow[loop, distance=2em, in=305, out=235, "t^2", swap, blue]
\end{tikzcd} \\
\text{Step 1} & \text{Step 2} & \text{Step 2} & \text{Step 2} \\
\begin{tikzcd}[ampersand replacement=\&]
\bullet \arrow[d, bend left, "t", blue]  \\
\bullet \arrow[u, bend left, "t^2", blue]
\end{tikzcd} &
\begin{tikzcd}[ampersand replacement=\&]
\bullet \arrow[d, bend left, "1", blue]  \\
\bullet \arrow[u, bend left, "t"]
\end{tikzcd} &
\begin{tikzcd}[ampersand replacement=\&]
\bullet \arrow[loop, distance=2em, in=305, out=235, "t", swap, blue]
\end{tikzcd} &
\begin{tikzcd}[ampersand replacement=\&]
\bullet
\end{tikzcd} \\
\text{Step 3} & \text{Step 1} & \text{Step 2} & \text{Stop}
\end{tabular}
\]

\end{exmp}

Given a valued spanning tree $T$, we define $R(Q,\alpha)_{T}:=\{x\in R(Q,\alpha)\ \vert\ T_{x}=T\}\subseteq R(Q,\alpha)_{\mathrm{ind.}}$.
Given that $T_{x}$ only depends on $\val(x_{a}),\ a\in Q_{1}$, the
variety $R(Q,\alpha)_{T}$ is a $\mathbb{G}_{\mathrm{m},\alpha}^{Q_{0}}$-invariant
constructible subset of $R(Q,\alpha)_{\mathrm{ind.}}$. Let us describe
the strata $R(Q,\alpha)_{T}$ in more details.

Set $T$ a valued spanning tree with valuation $v_{T}$, a point $x\in R(Q,\alpha)_{T}$
and $a\in Q_{1}$ a non-loop edge. If $a\in T$, then $\val(x_{a})=v_{T}(a)$
by construction. Suppose now that $a\notin T$. Then there exists
a unique (unoriented) path in $T$ joining the end vertices of $a$.
Let us denote by $T_{a}\subseteq Q_{1}$ the corresponding set of
edges and $v_{T_{a}}$ the largest valuation of $T_{a}$. When running
the contraction-deletion algorithm, $a$ is contracted into a loop
exactly when the smallest edge of $T_{a}$ with valuation $v_{T_{a}}$
is contracted, i.e. when the last edge of $T_{a}$ is contracted.
Let us call $e_{T_{a}}$ this edge. This means that $\val(x_{a})\geq v_{T_{a}}$,
otherwise $a$ would have been contracted before $e_{T_{a}}$. Moreover,
if $a>e_{T_{a}}$, then $\val(x_{a})>v_{T_{a}}$, or again, $a$ would
have been contracted before $e_{T_{a}}$. Therefore $\val(x_{a})\geq v_{T_{a}}+\delta_{a>e_{T_{a}}}$.
If $a\in Q_{1}$ is a loop edge, we consider this condition to be
empty i.e. $v_{T_{a}}+\delta_{a>e_{T_{a}}}\geq0$. One can check that
these inequalities imply $x\in R(Q,\alpha)_{T}$, so:\[
R(Q,\alpha)_{T}
=
\left\{
x\in R(Q,\alpha)\
\left\vert 
\begin{array}{ll}
\val(x_{a})=v_{T}(a), & a\in T \\
\val(x_{a})\geq v_{T_{a}}+\delta_{a>e_{T_{a}}}, & a\not\in T
\end{array}
\right.
\right\}
.
\]

Let us call $A_{(Q,\alpha),T}$ the number of orbits (of $\FF_{q}$-points
in $R(Q,\alpha)_{\mathrm{ind.}}$) lying in $R(Q,\alpha)_{T}$. Then
summing over all valued spanning tree of $Q$, we obtain:\[
A_{(Q,\alpha),\underline{1}}=\sum_TA_{(Q,\alpha),T}.
\] The advantage of considering coarser strata is that $A_{(Q,\alpha),T}$
can be computed recursively as in \cite{AMRV22}, following the contraction-deletion
algorithm:

\begin{prop} \label{Prop/KacPolContDel}

Let $T$ be a valued spanning tree of $Q$. Denote by $Q',T'$ the
quiver and valued spanning tree obtained from $Q,T$ by running the
contraction-deletion algorithm for an arbitrary $x\in R(Q,\alpha)_{T}$
and stopping after the first occurrence of step 3. Consider: 
\begin{itemize}
\item $n_{1}$ the number of loops of $Q$; 
\item $n_{2}$ the number of non-loop arrows $a\in Q_{1}$ which get contracted
into loops during step 1 and satisfying $a<e_{T_{a}}$;
\item $n_{3}$ the number of non-loop arrows $a\in Q_{1}$ which get contracted
into loops during step 1 and satisfying $a>e_{T_{a}}$.
\end{itemize}
Then:\[
A_{(Q,\alpha),T}=q^{\alpha n_1 + \alpha n_2 + (\alpha-1)n_3}\cdot A_{(Q',\alpha-1),T'}.
\]

\end{prop}

\begin{proof}

Throughout the proof, we will be considering orbits of $\FF_{q}$-rational
points. Let $a\in Q_{1}$ be the first arrow of $T$ which gets contracted
during step 1. Then the $\left(\mathcal{O}_{\alpha}^{\times}\right)^{Q_{0}}$-orbits
of $R(Q,\alpha)_{T}$ are in one-to-one correspondence with $\left(\mathcal{O}_{\alpha}^{\times}\right)^{(Q/a)_{0}}$-orbits
of $\{x\in R(Q,\alpha)_{T}\ \vert\ x_{a}=1\}$, where the factor $\mathcal{O}_{\alpha}^{\times}\subseteq\left(\mathcal{O}_{\alpha}^{\times}\right)^{Q_{0}}$
corresponding to $[s(a)]=[t(a)]\in(Q/a)_{0}$ acts diagonally on the
modules $\mathcal{O}_{\alpha}$ located at vertices $s(a)$ and $t(a)$.

Repeating this reasoning for all arrows $a_{1},\ldots,a_{s}$ of $T$
which get contracted (in step 1) during the construction of $Q'$,
we obtain that the $\left(\mathcal{O}_{\alpha}^{\times}\right)^{Q_{0}}$-orbits
of $R(Q,\alpha)_{T}$ are in one-to-one correspondence with $\left(\mathcal{O}_{\alpha}^{\times}\right)^{Q'_{0}}$-orbits
of:\[
\{x\in R(Q,\alpha)_{T}\ \vert\ x_{a_1}=\ldots=x_{a_s}=1\}
\simeq
R(Q',\alpha)_{T'}\times\mathcal{O}_{\alpha}^{\oplus n_1}\times\mathcal{O}_{\alpha}^{\oplus n_2}\times\left(t\mathcal{O}_{\alpha}\right)^{\oplus n_3}
.
\] The above isomorphism is $\left(\mathcal{O}_{\alpha}^{\times}\right)^{Q'_{0}}$-equivariant,
with trivial action on the three last factors of the right-hand side.
Here $T'$ is obtained from $T$ by contracting $a_{1},\ldots,a_{s}$
and leaving valuations unchanged. By abuse of notation, let us also
call $T'$ the valued spanning tree obtained by contracting $a_{1},\ldots,a_{s}$
and decreasing valuations by 1. Then we obtain as claimed:\[
A_{(Q,\alpha),T}=q^{\alpha n_1 + \alpha n_2 + (\alpha-1)n_2}\cdot A_{(Q',\alpha-1),T'}.
\]\end{proof}

The contraction-deletion algorithm terminates in step 3, when $Q'$
is a one-vertex quiver without arrows. Then $A_{(Q',\alpha'),\bullet}=1$
and it follows from Proposition \ref{Prop/KacPolContDel} that\[
A_{(Q,\alpha),T}=q^{n_T}
\ ; \ 
n_T:=\sum_{a\not\in T}(\alpha-v_{T_a}-\delta_{a>e_{T_a}}).
\]This yields the following:

\begin{thm} \label{Thm/PositivityToricKacPol}

Let $\rr=\underline{1}$, then $A_{(Q,\alpha),\rr}$ has non-negative
coefficients.

\end{thm}

\begin{rmk}

In \cite{AMRV22}, toric Kac polynomials are shown to be specialisations
of Tutte polynomials. This can be deduced from a recursive relation
satisfied by $A_{(Q,1),\underline{1}}$ under contraction-deletion
of edges, for which Tutte polynomials are universal. This relation
relies on the fact that, for $x\in R(Q,1)$ and $a\in Q_{1}$, either
$x_{a}$ is invertible or $x_{a}=0$. This is no longer the case for
$\alpha>1$ and $A_{(Q,\alpha),\underline{1}}$ cannot be obtained
as a specialisation of the Tutte polynomial of $Q$ (see also \cite[Prop. 7.16.]{HLRV24}).
\end{rmk}

\subsection{Purity of toric preprojective stacks in higher depth \label{Subsect/Purity}}

In this section, we categorify our previous results, that is, we show
that $A_{(Q,\alpha),\underline{1}}$ yields the E-series of a certain
preprojective stack $\left[\mu_{(Q,\alpha),\underline{1}}(t^{\alpha-1}\cdot\lambda)/\mathbb{G}_{\mathrm{m},\alpha}^{Q_{0}}\right]$.
We first prove that this stack has the appropriate point-count over
finite fields (Theorem \ref{Thm/Ch4CountDefMomMapFibre}); then we
exploit the stratification by valued spanning trees to show that it
has pure cohomology (Theorem \ref{Thm/Ch4PurityDefToricMomMapFibre}).
Finally, we generalise this to the preprojective stack $\left[\mu_{(Q,\alpha),\underline{1}}(0)/\mathbb{G}_{\mathrm{m},\alpha}^{Q_{0}}\right]$
(Theorem \ref{Thm/Ch4PurityToricPreprojStack}), thereby categorifying
part of the plethystic formula established in Chapter \ref{Chap/KacPolynomials}. 

Let us first prove Theorem \ref{Thm/Ch4CountDefMomMapFibre}. This
generalises Crawley-Boevey and Van den Bergh's result relating Kac
polynomials and counts of $\FF_{q}$-points on the representation
variety of a \textit{deformed} preprojective algebra. Set $\KK=\FF_{q}$.
Recall that $\lambda\in\ZZ^{Q_{0}}$ is called generic with respect
to a rank vector $\rr\in\ZZ_{\geq0}^{Q_{0}}$ if $\lambda\cdot\rr=0$
and $\lambda\cdot\rr'\ne0$ for all $0<\rr'<\rr$.

\begin{thm} \label{Thm/CountMomMapGenFib}

Let $\rr\in\ZZ_{\geq0}^{Q_{0}}$ be an indivisible rank vector and
$\lambda\in\ZZ^{Q_{0}}$ be generic with respect to $\rr$. Suppose
that $\FF_{q}$ has characteristic larger than $\sum_{i}\vert\lambda_{i}\vert\cdot r_{i}$.
Then: \[
\frac{\sharp\mu_{(Q,\alpha),\rr}^{-1}(t^{\alpha-1}\cdot\lambda)}{\sharp\GL_{\alpha,\rr}}
=q^{-\alpha\langle\rr,\rr\rangle}\cdot\frac{A_{(Q,\alpha),\rr}}{1-q^{-1}}.
\]

\end{thm}

\begin{proof}

Consider the projection map $\pi:\mu_{(Q,\alpha),\rr}^{-1}(t^{\alpha-1}\cdot\lambda)\subseteq R(\overline{Q},\alpha;\rr)\rightarrow R(Q,\alpha;\rr)$.
We claim that, when $\lambda$ is generic, the image of $\pi$ coincides
with the constructible subset of absolutely indecomposable representations.
Moreover, given $x$ an $\FF_{q}$-point of $R(Q,\alpha;\rr)$ corresponding
to a representation $M$, then $\pi^{-1}(x)$ has cardinality $\sharp\Ext^{1}(M,M)$
by Proposition \ref{Prop/MomMapExSeq}. Therefore, we can compute
$\sharp\mu_{(Q,\alpha),\rr}^{-1}(t^{\alpha-1}\cdot\lambda)$ by summing
over isomorphism classes $[M]$ of absolutely indecomposable representations
in rank $\rr$:\[
\frac{\sharp\mu_{(Q,\alpha),\rr}^{-1}(t^{\alpha-1}\cdot\lambda)}{\sharp\GL_{\alpha,\rr}}
=
\sum_{[M]}\frac{\sharp\Ext^1(M,M)}{\sharp\Aut(M)}.
\] Moreover, for $M$ absolutely indecomposable, $\sharp\Aut(M)=\frac{q-1}{q}\cdot\sharp\End(M)=q^{\alpha\langle\rr,\rr\rangle}\cdot\frac{q-1}{q}\cdot\sharp\Ext^{1}(M,M)$
by Propositions \ref{Prop/MomMapExSeq} and \ref{Prop/EndRings}.
This yields:\[
\frac{\sharp\mu_{(Q,\alpha),\rr}^{-1}(t^{\alpha-1}\cdot\lambda)}{\sharp\GL_{\alpha,\rr}}
=q^{-\alpha\langle\rr,\rr\rangle}\cdot\frac{A_{(Q,\alpha),\rr}}{1-q^{-1}}.
\]

Let us now prove the claim. We actually prove the following stronger
fact, by analogy with \cite[Thm. 3.3.]{CB01}: let $x$ be an $\FF_{q}$-point
of $R(Q,\alpha;\rr)$ corresponding to a representation $M$. Then
$x$ admits a lift in $\mu_{(Q,\alpha),\rr}^{-1}(t^{\alpha-1}\cdot\lambda)$
if, and only if, any direct summand of $M$ of rank $\rr'\leq\rr$
satisfies $\lambda\cdot\rr'=0$.

Indeed, if $(x,y)$ is an $\FF_{q}$-point of $\mu_{(Q,\alpha),\rr}^{-1}(t^{\alpha-1}\cdot\lambda)$
such that $\pi(x,y)=x$, then by Proposition \ref{Prop/MomMapExSeq},
$\mu_{(Q,\alpha),\rr}(x,y)=t^{\alpha-1}\cdot\lambda$ lies in the
kernel of $\mathfrak{gl}_{\alpha,\rr}\rightarrow\End(M)^{\vee}$.
Therefore, pairing $\mu_{(Q,\alpha),\rr}(x,y)=t^{\alpha-1}\cdot\lambda$
with the projection onto an $\rr'$-dimensional direct summand of
$M$ yields $\FF_{q}\ni\lambda\cdot\rr'=0$. This gives $\ZZ\ni\lambda\cdot\rr'=0$,
due to the bound on the characteristic of $\FF_{q}$.

Conversely, suppose that any direct summand of $M$ (call its rank
$\rr'$) satisfies $\lambda\cdot\rr'=0$. It is enough to show that,
for any indecomposable direct summand of $M$ (corresponding to an
$\FF_{q}$-point $x'$ of $R(Q,\alpha;\rr')$), $x'$ admits a lift
in $\mu_{(Q,\alpha),\rr'}^{-1}(t^{\alpha-1}\cdot\lambda)$. Thus we
assume that $M$ is indecomposable. Switching to $\bar{\FF_{q}}$,
Proposition \ref{Prop/EndRings} implies that any $\xi\in\End(M)$
is the sum of a scalar $\xi_{0}\cdot\Id\in\bar{\FF_{q}}$ and a nilpotent
endomorphism, i.e. $\xi$ is the sum of $\xi_{0}\cdot\Id$ and a tuple
of nilpotent matrices modulo $t$ (note that, over $\mathcal{O}_{\alpha}$,
not all nilpotent elements are traceless). Therefore, for all $\xi\in\End(M)$,
we get $r(\Tr(t^{\alpha-1}\lambda\cdot\xi))=\xi_{0}t^{\alpha-1}\lambda\cdot\rr=0$,
which implies, by Proposition \ref{Prop/MomMapExSeq}, that $x$ lifts
to $\mu_{(Q,\alpha),\rr'}^{-1}(t^{\alpha-1}\cdot\lambda)$. \end{proof}

Let us now prove a cohomological upgrade of Theorems \ref{Thm/CountMomMapGenFib}
and \ref{Thm/PositivityToricKacPol}. Set $\KK=\CC$. We show by direct
computation that the compactly supported cohomology of $\left[\mu_{Q,\rr,\alpha}^{-1}(t^{\alpha-1}\cdot\lambda)/\GL(\rr,\mathcal{O}_{\alpha})\right]$
($\rr=\underline{1}$) is pure and contains a pure Hodge structure
with E-poynomial $A_{(Q,\alpha),\rr}$. This leads us to compute the
compactly supported cohomology of $\left[\mu_{Q,\rr,\alpha}^{-1}(0)/\GL(\rr,\mathcal{O}_{\alpha})\right]$,
which is also pure and satisfies a formula analogous to cohomological
integrality in \cite{Dav23c}.

We assume throughout that $Q$ is connected (the non-connected case
is analogous, by treating connected components separately). Set $\rr=1$
and call $\mu_{(Q,\alpha),\rr}:=\mu_{Q,\alpha}$ for short. We compute
the compactly supported cohomology of $\left[\mu_{Q,\alpha}^{-1}(t^{\alpha-1}\cdot\lambda)/\GL_{\alpha,\rr}\right]$
using a $\GL_{\alpha,\rr}$-equivariant stratification, where strata
are labeled by valued spanning trees. Let $\pi:\mu_{Q,\alpha}^{-1}(t^{\alpha-1}\cdot\lambda)\subseteq R(\overline{Q},\alpha)\rightarrow R(Q,\alpha)$
be the projection $(x,y)\mapsto x$. For $T$ a valued spanning tree
of $Q$, define the stratum $\mu_{Q,\alpha}^{-1}(t^{\alpha-1}\cdot\lambda)_{T}:=\pi^{-1}\left(R(Q,\alpha)_{T}\right)$.
We use the term ``stratification'' in a loose sense: the set of
spanning trees of $Q$ can be endowed with a partial order such that,
for any valued spanning tree $T$,\[
\bigsqcup_{T'\leq T}\mu_{Q,\alpha}^{-1}(t^{\alpha-1}\cdot\lambda)_{T'}\subseteq\mu_{Q,\alpha}^{-1}(t^{\alpha-1}\cdot\lambda)
\] is closed. One can then compute the (compactly supported) cohomology
of $\left[\mu_{Q,\rr,\alpha}^{-1}(t^{\alpha-1}\cdot\lambda)/\GL_{\alpha,\rr}\right]$
from the cohomology of the strata using successive open-closed decompositions
and Lemma \ref{Lem/Strat}. The following proposition gives us the
partial order on strata that we are looking for:

\begin{prop} \label{Prop/Strata}

The following closed subset of $R(Q,\alpha)_{\mathrm{ind}}$ is a
union of strata $R(Q,\alpha)_{T'}$:

\[
R(Q,\alpha)_{\leq T}
:=
\left\{
x\in R(Q,\alpha)\
\left\vert 
\begin{array}{ll}
\val(x_{a})\geq v_{T}(a), & a\in T \\
\val(x_{a})\geq v_{T_{a}}+\delta_{a>e_{T_{a}}}, & a\not\in T
\end{array}
\right.
\right\}
.
\]

\end{prop}

\begin{proof}

Let $T'$ be a valued spanning tree of $Q$ and suppose that $R(Q,\alpha)_{\leq T}\cap R(Q,\alpha)_{T'}\ne\emptyset$.
We show that for any $a\in Q_{1}$, we have $v_{T'_{a}}+\delta_{a>e_{T'_{a}}}\geq v_{T_{a}}+\delta_{a>e_{T_{a}}}$,
hence $R(Q,\alpha)_{T'}\subseteq R(Q,\alpha)_{\leq T}$. Note that
when $a$ belongs to $T$, we have: $T_{a}=Q\vert_{\{a\}}$ and $v_{T_{a}}=v_{T}(a)$
and $e_{T_{a}}=a$, so the right-hand side simplifies to $v_{T}(a)$.
The same goes for the left-hand side when $a$ belongs to $T'$. We
write $a\in T$ for short when $a$ is an arrow of $T$.

Let us prove the claim. Take $x\in R(Q,\alpha)_{\leq T}\cap R(Q,\alpha)_{T'}$;
since $T_{a}\subseteq\cup_{b\in T'_{a}}T_{b}$, we obtain:\[
v_{T'_{a}}=\max_{b\in T'_{a}}\{v_{T'}(b)\}=\max_{b\in T'_{a}}\{\val(x_{b})\}\geq\max_{b\in T'_{a}}\{v_{T_{b}}\}\geq v_{T_{a}}.
\]The second equality holds because $x\in R(Q,\alpha)_{T'}$, whereas
$\val(x_{b})\geq v_{T_{b}}$ because $x\in R(Q,\alpha)_{\leq T}$.
We need to further show that, if $e_{T_{a}}<a\leq e_{T'_{a}}$, then
$v_{T'_{a}}>v_{T_{a}}$. Suppose $e_{T_{a}}<a\leq e_{T'_{a}}$ and
consider some $b\in T'_{a}$ such that $e_{T_{a}}\in T_{b}$. Then:\[
v_{T'_a}\geq v_{T'}(b)=\val(x_b)\geq v_{T_b}+\delta_{b>e_{T_b}}\geq v_T(e_{T_a})=v_{T_a}.
\]We argue that either the leftmost or the rightmost inequality is strict.
Indeed, if $v_{T'_{a}}=v_{T'}(b)$ and $v_{T_{b}}=v_{T_{a}}$, then
$b\geq e_{T'_{a}}$ (as $v_{T'}(b)=v_{T'}(e_{T'_{a}})$) and $e_{T_{a}}\geq e_{T_{b}}$
(as $v_{T}(e_{T_{a}})=v_{T}(e_{T_{b}})$). Since we assumed that $e_{T_{a}}<a\leq e_{T'_{a}}$,
we obtain $b>e_{T_{b}}$ and the rightmost inequality is strict. This
concludes the proof. \end{proof}

The partial order on valued spanning trees is then given by $T'\leq T$
if, and only if, $R(Q,\alpha)_{T'}\subseteq R(Q,\alpha)_{\leq T}$.
Let us now compute the compactly supported cohomology of $\left[\mu_{Q,\alpha}^{-1}(t^{\alpha-1}\cdot\lambda)_{T}/\GL_{\alpha,\rr}\right]$.

\begin{prop} \label{Prop/CohStrat}

Let $T$ be a valued spanning tree of $Q$. Denote by $Q'$ (resp.
$T'$) the quiver (resp. the valued spanning tree) obtained from $Q$
and $T$ by running the contraction-deletion algorithm for an arbitrary
$x\in R(Q,\alpha)_{T}$ and stopping after the first occurrence of
step 3. Let $\lambda'\in\ZZ^{Q_{0}}$ be the vector obtained from
$\lambda$ by applying the corresponding contractions. Consider: 
\begin{itemize}
\item $n_{1}$ the number of loops of $Q$; 
\item $n_{2}$ the number of non-loop arrows $a\in Q_{1}$ which get contracted
into loops during step 1 and satisfying $a<e_{T_{a}}$;
\item $n_{3}$ the number of non-loop arrows $a\in Q_{1}$ which get contracted
into loops during step 1 and satisfying $a>e_{T_{a}}$.
\end{itemize}
Then:\[
\HH_{\mathrm{c}}^{\bullet}\left(\left[\mu_{Q,\alpha}^{-1}(t^{\alpha-1}\cdot\lambda)_T/\mathbb{G}_{\mathrm{m},\alpha}^{Q_0}\right]\right)
\simeq
\HH_{\mathrm{c}}^{\bullet}\left(\left[\mu_{Q',\alpha-1}^{-1}(t^{\alpha-2}\cdot\lambda')_{T'}/\mathbb{G}_{\mathrm{m},\alpha-1}^{Q'_0}\right]\right)
\otimes\mathbb{L}^{\otimes\left(2\alpha n_1+2\alpha n_2+(2\alpha-1)n_3+\sharp Q'_1-\sharp Q'_0\right)}.
\]

\end{prop}

\begin{proof}

We show that there is a $\mathbb{G}_{\mathrm{m},\alpha}^{Q_{0}}$-equivariant
isomorphism:\[
\Psi:
\mu_{Q,\alpha}^{-1}(t^{\alpha-1}\cdot\lambda)_T
\simeq
\left(
\mu_{Q',\alpha-1}^{-1}(t^{\alpha-2}\cdot\lambda')_{T'}
\times^{\mathbb{G}_{\mathrm{m},\alpha}^{Q'_0}}\mathbb{G}_{\mathrm{m},\alpha}^{Q_0}
\right)
\times\mathbb{A}^{2\alpha n_1}\times\mathbb{A}^{2\alpha n_2+(2\alpha-1)n_3}\times\mathbb{A}^{\sharp Q'_1}.
\] Let us call:
\begin{itemize}
\item $a_{1},\ldots,a_{s}\in Q_{1}$ the arrows of $T$ which get contracted
(in step 1) during the construction of $Q'$;
\item $a'_{1},\ldots,a'_{n_{1}}\in Q_{1}$ the loops of $Q$;
\item $a''_{1},\ldots,a''_{n_{2}}\in Q_{1}$ the non-loop arrows of $Q$
which get contracted into loops (in step 1) and deleted (in step 2)
during the construction of $Q'$ and which satisfy $a<e_{T_{a}}$;
\item $b''_{1},\ldots,b''_{n_{3}}\in Q_{1}$ the non-loop arrows of $Q$
which get contracted into loops (in step 1) and deleted (in step 2)
during the construction of $Q'$ and which satisfy $a>e_{T_{a}}$.
\end{itemize}
Let $(x,y)\in\mu_{Q,\alpha}^{-1}(t^{\alpha-1}\cdot\lambda)_{T}$.
One can check that $(x,y)$ determines a point $(x',y')\in\mu_{Q',\alpha}^{-1}(t^{\alpha-1}\cdot\lambda')_{T'}$
such that, for all $a\in Q'_{1}$, we have $x'_{a}\in t\mathcal{O}_{\alpha}$.
We describe the components of $\Psi(x,y)$:
\begin{itemize}
\item Note that $\mu_{Q',\alpha-1}^{-1}(t^{\alpha-2}\cdot\lambda')_{T'}\times^{\mathbb{G}_{\mathrm{m},\alpha}^{Q'_{0}}}\mathbb{G}_{\mathrm{m},\alpha}^{Q_{0}}\simeq\mu_{Q',\alpha-1}^{-1}(t^{\alpha-2}\cdot\lambda')_{T'}\times\mathbb{G}_{\mathrm{m},\alpha}^{s}$.
The component of $\Psi(x,y)$ along $\mu_{Q',\alpha-1}^{-1}(t^{\alpha-2}\cdot\lambda')_{T'}$
is induced by $(x',y')$ via the morphisms of $\mathcal{O}_{\alpha}$-modules
$t\mathcal{O}_{\alpha}\simeq\mathcal{O}_{\alpha-1}$ (for the $x$-coordinate)
and $\mathcal{O}_{\alpha}\twoheadrightarrow\mathcal{O}_{\alpha-1}$
(for the $y$-coordinate). The component along $\mathbb{G}_{\mathrm{m},\alpha}^{s}$
is $(x_{a_{t}})_{1\leq t\leq s}$;
\item $\Psi(x,y)_{\mathbb{A}^{2\alpha n_{1}}}=(x_{a'_{n}},y_{a'_{n}})_{1\leq n\leq n_{1}}$
- note that the moment map equation imposes no conditions on $(x_{a'_{n}},y_{a'_{n}})$,
as $\rr=\underline{1}$;
\item $\Psi(x,y)_{\mathbb{A}^{2\alpha n_{2}+(2\alpha-1)n_{3}}}=\left((x_{a''_{n}},y_{a''_{n}})_{1\leq n\leq n_{2}},(x_{b''_{n}},y_{b''_{n}})_{1\leq n\leq n_{3}}\right)$
- note that $x_{a''_{n}}\in t\mathcal{O}_{\alpha}$ by assumption;
moreover the coordinates $(x_{a''_{n}},y_{a''_{n}})_{1\leq n\leq n_{2}},(x_{b''_{n}},y_{b''_{n}})_{1\leq n\leq n_{3}}$
may be chosen freely and determine $x_{a_{1}}y_{a_{1}},\ldots,x_{a_{s}}y_{a_{s}}$
through the moment map equation;
\item Write $y'_{a}=\sum_{k}y'_{a,k}\cdot t^{k}$ for $a\in Q'_{1}$; then
$\Psi(x,y)_{\mathbb{A}^{\sharp Q'_{1}}}=(y'_{a,\alpha-1})_{a\in Q'_{1}}$.
\end{itemize}
Consequently, $\Psi(x,y)$ contains all the coordinates of $(x,y)$
except for $y_{a_{1}},\ldots,y_{a_{s}}$. These are determined from
$\Psi(x,y)$ using the moment map equation. The action of $\mathbb{G}_{\mathrm{m},\alpha}^{Q_{0}}$
on $\mu_{Q',\alpha-1}^{-1}(t^{\alpha-2}\cdot\lambda')_{T'}\times\mathbb{G}_{\mathrm{m},\alpha}^{s}$
is induced by a choice of splitting $\mathbb{G}_{\mathrm{m},\alpha}^{Q_{0}}\simeq\mathbb{G}_{\mathrm{m},\alpha}^{Q'_{0}}\times\mathbb{G}_{\mathrm{m},\alpha}^{s}$
and makes the morphism $\mu_{Q,\alpha}^{-1}(t^{\alpha-1}\cdot\lambda)_{T}\rightarrow\mu_{Q',\alpha-1}^{-1}(t^{\alpha-2}\cdot\lambda')_{T'}\times^{\mathbb{G}_{\mathrm{m},\alpha}^{Q'_{0}}}\mathbb{G}_{\mathrm{m},\alpha}^{Q_{0}}$
$\mathbb{G}_{\mathrm{m},\alpha}^{Q_{0}}$-equivariant. The $\mathbb{G}_{\mathrm{m},\alpha}^{Q_{0}}$-action
on the remaining components of the right-hand side is transferred
from the action on the left-hand side.

Using the isomorphism $\Psi$, Lemmas \ref{Lem/AffFib} and \ref{Lem/GrpChg}
yield:\[
\HH_{\mathrm{c}}^{\bullet}\left(\left[\mu_{Q,\alpha}^{-1}(t^{\alpha-1}\cdot\lambda)_T/\mathbb{G}_{\mathrm{m},\alpha}^{Q_{0}}\right]\right)
\simeq
\HH_{\mathrm{c}}^{\bullet}\left(\left[\mu_{Q',\alpha-1}^{-1}(t^{\alpha-2}\cdot\lambda')_{T'}/\mathbb{G}_{\mathrm{m},\alpha}^{Q'_{0}}\right]\right)
\otimes\mathbb{L}^{\otimes\left(2\alpha n_1+2\alpha n_2+(2\alpha-1)n_3+\sharp Q'_1\right)}.
\] Since the action of $\mathbb{G}_{\mathrm{m},\alpha}^{Q'_{0}}$ on
$\mu_{Q',\alpha-1}^{-1}(t^{\alpha-2}\cdot\lambda')_{T'}$ factors
through the quotient group $\mathbb{G}_{\mathrm{m},\alpha-1}^{Q'_{0}}$,
Lemma \ref{Lem/DepthChg} gives the desired formula. \end{proof}

When the contraction-deletion algorithm terminates, the computation
finishes with $\HH_{\mathrm{c}}^{\bullet}\left(\left[\mathrm{pt}/\mathbb{G}_{\mathrm{m},\alpha'}\right]\right)\simeq\mathbb{L}^{\otimes-\alpha'}\otimes\HH_{\mathrm{c}}^{\bullet}\left(\mathrm{B}\mathbb{G}_{\mathrm{m}}\right)$
, which is pure. This shows that $\HH_{\mathrm{c}}^{\bullet}\left(\left[\mu_{Q,\alpha}^{-1}(t^{\alpha-1}\cdot\lambda)_{T}/\mathbb{G}_{\mathrm{m},\alpha}^{Q_{0}}\right]\right)$
is pure, of Tate type, which allows us to conclude using Lemma \ref{Lem/Strat}
and Theorem \ref{Thm/CountMomMapGenFib}:

\begin{thm} \label{Thm/CohMomMapGenFib}

Let $\rr=\underline{1}$. then:\[
\HH_{\mathrm{c}}^{\bullet}\left(\left[\mu_{(Q,\alpha),\rr}^{-1}(t^{\alpha-1}\cdot\lambda)/\GL_{\alpha,\rr}\right]\right)
\simeq
A_{(Q,\alpha),\rr}(\mathbb{L})\otimes\mathbb{L}^{1-\alpha\langle\rr,\rr\rangle}\otimes\HH_{\mathrm{c}}^{\bullet}\left(\mathrm{B}\mathbb{G}_{\mathrm{m}}\right)
\]In particular, $\HH_{\mathrm{c}}^{\bullet}\left(\left[\mu_{(Q,\alpha),\rr}^{-1}(t^{\alpha-1}\cdot\lambda)/\GL_{\alpha,\rr}\right]\right)$
carries a pure Hodge structure.

\end{thm}

There is more: as can be seen from the proof of Proposition \ref{Prop/Strata},
the parameter $\lambda$ plays no role in computing the cohomology
of the strata. If we replace $\lambda$ with 0, we may compute $\HH_{\mathrm{c}}^{\bullet}\left(\left[\mu_{(Q,\alpha),\rr}^{-1}(0)/\GL_{\alpha,\rr}\right]\right)$
by pulling back a stratification of $R(Q,\alpha)$ instead of $R(Q,\alpha)_{\mathrm{ind.}}$.
For a general $x\in R(Q,\alpha)$ the restriction of $Q$ to $\{a\in Q_{1}\ \vert\ x_{a}\ne0\}$
may not be connected, so we should index strata of $R(Q,\alpha)$
by (i) partitions $Q_{0}=I_{1}\sqcup\ldots\sqcup I_{s}$ and (ii)
valued spanning trees for $Q\vert_{I_{1}},\ldots,Q\vert_{I_{s}}$.
Therefore, given such a collection $(I=(I_{1},\ldots,I_{s}),T=(T_{1},\ldots,T_{s}))$,
we define:\[
R(Q,\alpha)_{(I,T)}:=
\left\{
x\in R(Q,\alpha)\
\left\vert
\begin{array}{l}
\forall 1\leq t\leq s,\ (x_a)_{a\in Q_{1,I_s}}\in R(Q\vert_{I_s},\alpha)_{T_s} \\
\forall a\not\in\bigcup_{t=1}^sQ_{1,I_t},\ x_a=0
\end{array}
\right.
\right\}
.
\]The same proof as Proposition \ref{Prop/Strata} shows that $\HH_{\mathrm{c}}^{\bullet}\left(\left[\mu_{Q,\alpha}^{-1}(0)_{(I,T)}/\mathbb{G}_{\mathrm{m},\alpha}^{Q_{0}}\right]\right)$
is pure, of Tate type for all $(I,T)$. Combined with Theorem \ref{Thm/Ch2ExpFmlKacPol},
this shows:

\begin{thm} \label{Thm/CohIntgr}

Let $\rr=\underline{1}$. Then:

\[
\HH_{\mathrm{c}}^{\bullet}\left(\left[\mu_{(Q,\alpha),\rr}^{-1}(0)/\GL_{\alpha,\rr}\right]\right)
\otimes\mathbb{L}^{\otimes\alpha\langle\rr,\rr\rangle}
\simeq
\bigoplus_{Q_0=I_1\sqcup\ldots\sqcup I_s}
\bigotimes_{j=1}^s
\left(
A_{(Q\vert_{I_j},\alpha),\rr\vert_{I_j}}(\mathbb{L})\otimes\mathbb{L}\otimes \HH_{\mathrm{c}}^{\bullet}(\mathrm{B}\mathbb{G}_m)
\right).
\]

In particular, $\HH_{\mathrm{c}}^{\bullet}\left(\left[\mu_{(Q,\alpha),\rr}^{-1}(0)/\GL_{\alpha,\rr}\right]\right)$
carries a pure Hodge structure.

\end{thm}

\begin{rmk}

Theorem \ref{Thm/CohIntgr} may be interpreted as a higher depth analogue
of a PBW theorem for preprojective cohomological Hall algebras \cite{Dav23c},
restricted to $\rr\leq\underline{1}$. Let us call $\mathrm{Sym}$
the operator on $\ZZ\times\ZZ_{\geq0}^{Q_{0}}$-graded mixed Hodge
structures which categorifies the plethystic exponential - see \cite[\S 3.2]{DM20}.
Then Theorem \ref{Thm/CohIntgr} can be stated as follows:\[
\bigoplus_{\rr\leq\underline{1}}\HH_{\mathrm{c}}^{\bullet}\left(\left[\mu_{(Q,\alpha),\rr}^{-1}(0)/\GL_{\alpha,\rr}\right]\right)\otimes\mathbb{L}^{\otimes\alpha\langle\rr,\rr\rangle}
\simeq
\mathrm{Sym}
\left.
\left(
\bigoplus_{\rr>0}A_{(Q,\alpha),\rr}(\mathbb{L})\otimes\mathbb{L}\otimes \HH_{\mathrm{c}}^{\bullet}(\mathrm{B}\mathbb{G}_m)
\right)
\right\vert_{\rr\leq\underline{1}}.
\]

\end{rmk}

\begin{rmk} \label{Rmk/CohUpgrade=000026Positivity}

Note that, while in \cite{Dav18,Dav23c}, positivity of Kac polynomials
is deduced from a cohomological PBW isomorphism, here the proof of
Theorem \ref{Thm/CohIntgr} \textit{relies} on positivity for $A_{(Q,\alpha),\rr}$
- in order to make sense of the graded pure Hodge structure $A_{(Q,\alpha),\rr}(\mathbb{L})$.

In \cite{Dav23c}, the existence of graded mixed Hodge structures
$\mathrm{BPS}_{Q,\dd}^{\vee},\ \dd\in\ZZ_{\geq0}^{Q_{0}}$ satisfying\[
\bigoplus_{\dd\geq0}
\HH_{\mathrm{c}}^{\bullet}\left(\left[\mu_{Q,\dd}^{-1}(0)/\GL_{\dd}\right]\right)\otimes\mathbb{L}^{\otimes\langle\dd,\dd\rangle}
\simeq
\mathrm{Sym}
\left(
\bigoplus_{\dd\geq0}\mathrm{BPS}_{Q,\dd}^{\vee}\otimes\mathbb{L}\otimes \HH_{\mathrm{c}}^{\bullet}(\mathrm{B}\mathbb{G}_m)
\right)
\] is proved using cohomological Donaldson-Thomas theory. Purity of
$\HH_{\mathrm{c}}^{\bullet}\left(\left[\mu_{Q,\dd}^{-1}(0)/\GL_{\dd}\right]\right)$
then implies purity of $\mathrm{BPS}_{Q,\dd}^{\vee}$, which in turn
shows that $A_{Q,\dd}$ has non-negative coefficients.

In our setting however, the isomorphism of Theorem \ref{Thm/CohIntgr}
is deduced from the fact that both sides are pure, of Tate type and
have the same E-polynomial by Theorem \ref{Thm/Ch2ExpFmlKacPol}.
The well-definedness of $A_{(Q,\alpha),\rr}(\mathbb{L})$ relies on
Theorem \ref{Thm/PositivityToricKacPol}, since we do not know by
other means that there exists a graded pure Hodge structure analogous
to $\mathrm{BPS}_{Q,\dd}^{\vee}$.

\end{rmk}

\pagebreak{}

\section{Towards Hall algebras for quivers with multiplicities \label{Chap/HallAlg}}

In this Chapter, we collect some computations that we made with the
aim of building a cohomological Hall algebra for the quiver with equal
multiplicities $(Q,\nn)=(A_{2},\alpha)$. This quiver is obtained
by choosing an orientation of the Dynkin diagram $A_{2}$:\[
\begin{tikzcd}[ampersand replacement=\&]
\bullet \ar[r] \& \bullet
\end{tikzcd}
.
\]We first consider the analogues of Lusztig sheaves $\mathcal{L}_{\blacksquare}$
in that setting. These are l-adic complexes which recover the functions
$[E_{i_{1}}]*\ldots*[E_{i_{r}}]$ in the Ringel-Hall algebra under
the ``Faisceaux-Fonctions'' correspondence. Here $[E_{i}]$ denotes
the indicator function associated to the one-dimensional zero representation
concentrated at vertex $i\in Q_{0}$. For quivers without multiplicities,
Lusztig sheaves are semisimple, by virtue of the decomposition theorem.
Their (shifted) perverse summands form the category of complexes used
by Lusztig to categorify the Ringel-Hall algebra \cite{Lus91}.

More recently, Fan studied the analogues of Lusztig sheaves for quivers
with multiplicities of the form $(Q,\alpha)$ \cite{Fan14}. Fan claims
that Lusztig sheaves are also semisimple in that context, using a
purity argument generalising the decomposition theorem \cite[Prop. 2.]{Fan14}.
However, there is a gap in the proof and we provide a counterexample
of a non-semisimple Lusztig sheaf.

\begin{prop} \label{Prop/Ch5SemisimpFail}

Consider the quiver with multiplicities $(Q,n)=(A_{2},2)$ and rank
vector $\rr=(2,1)$. Set $\blacksquare=(e_{1},e_{2},e_{1})$. Then
the Lusztig sheaf $\mathcal{L}_{\blacksquare}$ is not semisimple.

\end{prop}

Then we cannot guarantee that the perverse constitutents of $\mathcal{L}_{\blacksquare}$
are preserved by Lusztig's induction and restriction functors, which
is necessary in order to categorify the Ringel-Hall algebra. However,
the Ringel-Hall algebra built by Geiss, Leclerc and Schröer in \cite{GLS16}
is still available. Since characteristic cycles of complexes can be
computed from constructible functions, we exploit the Ringel-Hall
algebra multiplication to build a product on lagrangian cycles in
$\HH_{\bullet}^{\mathrm{BM}}(\mathfrak{M}_{\Pi_{(Q,\nn)}})$.

\begin{prop} \label{Prop/Ch5TopCoHA}

Consider the quiver with multiplicities $(Q,\nn)=(A_{2},\alpha)$,
where $\alpha\geq1$. For $\rr\in\ZZ_{\geq0}^{Q_{0}}$, consider the
$\QQ$-vector space $\mathbf{H}_{(A_{2},\alpha),\rr}$ of constructible
functions on $R(A_{2},\alpha;\rr)$ which are constant along $\GL_{\alpha,\rr}$-orbits.
Then the characteristic cycle map induces an isomorphism of $\ZZ_{\geq0}^{Q_{0}}$-graded
$\QQ$-vector spaces:\[
\mathrm{CC}:
\mathbf{H}_{(A_2,\alpha)}:=\bigoplus_{\rr\in\ZZ_{\geq0}^{Q_0}}\mathbf{H}_{(A_2,\alpha),\rr}
\overset{\sim}{\longrightarrow}
\bigoplus_{\rr\in\ZZ_{\geq0}^{Q_0}}\HH_{\mathrm{top}}^{\mathrm{BM}}\left(\mathfrak{M}_{\Pi_{(Q,\alpha)},\rr}\right)
,
\]where $\HH_{\mathrm{top}}^{\mathrm{BM}}\left(\mathfrak{M}_{\Pi_{(Q,\nn)},\rr}\right)$
is the equivariant Borel-Moore homology group of degree $2\dim\left(\mathfrak{M}_{\Pi_{(Q,\nn)},\rr}\right)$.

Moreover, when endowed with the Ringel-Hall algebra product, $\mathbf{H}_{(A_{2},\alpha)}$
is isomorphic to $U(\tilde{\mathfrak{n}})$, where $\tilde{\mathfrak{n}}$
is the trivial extension of the upper half of $\mathfrak{sl}_{3}$
by an abelian Lie algebra of dimension $\alpha-1$.

\end{prop}

\subsection{Failure of semisimplicity for Lusztig sheaves}

In this section, we work over a finite base field $\KK$. Fix a quiver
with equal multiplicities $(Q,n)$. Let us first recall the construction
of Lusztig sheaves, following \cite{Lus91,Fan14}. Consider a sequence
of rank vectors of the form $\blacksquare=(\epsilon_{i_{1}},\ldots,\epsilon_{i_{r}})$,
where $i_{1},\ldots,i_{r}\in Q_{0}$, and $\rr=\sum_{k=1}^{r}\epsilon_{i_{k}}$.
The Lusztig sheaf $\mathcal{L}_{\blacksquare}$ is a mixed ($\GL_{n,\rr}$-equivariant)
l-adic complex in $D_{\mathrm{c}}^{b}(R(Q,n;\rr),\bar{\QQ_{l}})$,
such that $\chi_{\mathcal{L}_{\blacksquare}}$ is the function $[E_{i_{1}}]*\ldots*[E_{i_{r}}]$
in the Ringel-Hall algebra of $(Q,n)$. Given $x\in R(Q,n;\rr)(\KK)$,
the value of $[E_{i_{1}}]*\ldots*[E_{i_{r}}]$ at $x$ is the number
of flags of locally-free submodules:\[
M_{\blacksquare}:\ M_{r+1}=\{0\}\subseteq M_r\subseteq\ldots\subseteq M_1=\bigoplus_{i\in Q_0}\left( \KK[t]/(t^n)\right)^{\oplus r_i},
\]which are stable under linear maps $x$ and such that $\rk(M_{k})-\rk(M_{k+1})=\epsilon_{i_{k}}$
for $1\leq k\leq r$. This motivates the following definition.

\begin{df}[Lusztig sheaves]

Let $(Q,n)$ be a quiver with equal multiplicities. Consider a sequence
$\blacksquare=(\epsilon_{i_{1}},\ldots,\epsilon_{i_{r}})$, where
$i_{1},\ldots,i_{r}\in Q_{0}$, and $\rr=\sum_{k=1}^{r}\epsilon_{i_{k}}$.

Define $\mathcal{F}_{\blacksquare,n}$ to be the variety of flags
of locally free submodules:\[
M_{\blacksquare}:\ M_{r+1}=\{0\}\subseteq M_r\subseteq\ldots\subseteq M_1=\bigoplus_{i\in Q_0}\left( \KK[t]/(t^n)\right)^{\oplus r_i},
\]such that $\rk(M_{k})-\rk(M_{k+1})=\epsilon_{i_{k}}$ for $1\leq k\leq r$.
Define also $\tilde{\mathcal{F}}_{\blacksquare,n}$ as the following
variety of pairs:\[
\tilde{\mathcal{F}}_{\blacksquare,n}:=
\left\{
(x,M_{\blacksquare})\in R(Q,n;\rr)\times \mathcal{F}_{\blacksquare,n}\ \vert\ x(M_{\blacksquare})\subseteq M_{\blacksquare}
\right\}
\]and call $\pi_{\blacksquare}:\tilde{\mathcal{F}}_{\blacksquare,n}\rightarrow R(Q,n;\rr)$
the first projection.

The Lusztig sheaf $\mathcal{L}_{\blacksquare}\in D_{\mathrm{c}}^{b}(R(Q,n;\rr),\bar{\QQ_{l}})$
is defined as $\mathcal{L}_{\blacksquare}:=(\pi_{\blacksquare})_{!}\underline{\bar{\QQ_{l}}}$.

\end{df}

As noted by Fan \cite[Lem. 1.]{Fan14}, the flag variety $\mathcal{F}_{\blacksquare,n}$
is the $(n-1)$-th jet scheme of $\mathcal{F}_{\blacksquare,1}$ and
the truncation map $\rho_{n}:\mathcal{F}_{\blacksquare,n}\rightarrow\mathcal{F}_{\blacksquare,1}$
maps $M_{\blacksquare}$ to $V_{\blacksquare}:=M_{\blacksquare}/tM_{\blacksquare}$.
This can be seen from the identification $\mathcal{F}_{\blacksquare,1}=\prod_{i\in Q_{0}}\GL_{r_{i}}/B_{i}$
, where $B_{i}\subseteq\GL_{r_{i}}$ is a Borel subgroup. Likewise,
$\mathcal{F}_{\blacksquare,n}$ may be identified as $\mathcal{F}_{\blacksquare,n}=\prod_{i\in Q_{0}}\GL_{n,r_{i}}/B_{n,i}$,
where $\GL_{n,r_{i}}$ and $B_{n,i}$ are the $(n-1)$-th jet schemes
of $\GL_{r_{i}}$ and $B_{i}$ respectively.

When $n=1$, the map $\pi_{\blacksquare}$ is proper, hence $\mathcal{L}_{\blacksquare}$
is semisimple by the decomposition theorem (see Theorem \ref{Thm/BBDGDecomp}).
Since the Lusztig sheaves are stable under induction and restriction
functors, their summands are preserved by these functors as well.
This is crucial in analysing the image of the algebra homomorphism
$U_{\nu}^{-}(\mathfrak{g}_{Q})\rightarrow\mathbf{H}_{Q}$ mentioned
in the \hyperref[Intro]{Introduction}. For $n>1$, the structure
of $\mathcal{L}_{\blacksquare}$ is studied by Fan in \cite{Fan14}.
Fan claims that $\mathcal{L}_{\blacksquare}$ is also semisimple in
that case \cite[Prop. 2.]{Fan14}. Unfortunately, there is a gap in
the proof and the claim is false. We provide a counterexample below.
Let us first give some details on Fan's strategy.

Consider the flag varieties $\mathcal{F}_{\blacksquare,n}$ and $\mathcal{F}_{\blacksquare}:=\mathcal{F}_{\blacksquare,1}$.
As mentioned above, the assignment $\rho_{n}:M_{\blacksquare}\mapsto M_{\blacksquare}/tM_{\blacksquare}$
identifies $\mathcal{F}_{\blacksquare,n}$ with the $(n-1)$-th jet
scheme of $\mathcal{F}_{\blacksquare}$. The closed immersion $\mathcal{F}_{\blacksquare}\hookrightarrow\mathcal{F}_{\blacksquare,n}$
maps $V_{\blacksquare}$ to $V_{\blacksquare}\otimes_{\KK}\KK[t]/(t^{n})$.
The variety of pairs $\tilde{\mathcal{F}}_{\blacksquare,n}$ is a
vector bundle over $\mathcal{F}_{\blacksquare,n}$. Fan also considers
the following variety of pairs:\[
\tilde{\mathcal{F}}_{\blacksquare}:=\left\{(x,V_{\blacksquare})\in R(Q,n;\rr)\times \mathcal{F}_{\blacksquare}\ \vert\ x(V_{\blacksquare}\otimes_{\KK}\KK[t]/(t^{n}))\subseteq V_{\blacksquare}\otimes_{\KK}\KK[t]/(t^{n}) \right\}.
\]In the course of his proof, Fan claims that there is a cartesian square:\[
\begin{tikzcd}[ampersand replacement=\&]
\tilde{\mathcal{F}}_{\blacksquare,n}\ar[r]\ar[d] \& \mathcal{F}_{\blacksquare,n} \ar[d, "\rho_n"] \\
\tilde{\mathcal{F}}_{\blacksquare} \ar[r,"\mathrm{pr}_2"] \& \mathcal{F}_{\blacksquare}
\end{tikzcd}
,
\]such that the composition $\tilde{\mathcal{F}}_{\blacksquare,n}\rightarrow\tilde{\mathcal{F}}_{\blacksquare}$
with the first projection $\tilde{\mathcal{F}}_{\blacksquare}\rightarrow R(Q,n;\rr)$
is $\pi_{\blacksquare}:\tilde{\mathcal{F}}_{\blacksquare,n}\rightarrow R(Q,n;\rr)$.
However, the map $\tilde{\mathcal{F}}_{\blacksquare,n}\rightarrow\tilde{\mathcal{F}}_{\blacksquare}$
is ill-defined. Indeed, if $M_{\blacksquare}\in\tilde{\mathcal{F}}_{\blacksquare,n}$
is stabilised by $x\in R(Q,n;\rr)$ and $V_{\blacksquare}:=M_{\blacksquare}/tM_{\blacksquare}$,
then $x$ does not necessarily stabilise $V_{\blacksquare}\otimes_{\KK}\KK[t]/(t^{n})$.
For instance, set $Q$ to be the Jordan quiver (one vertex and one
loop), $r=2$ and\[
x=
\left(
\begin{array}{cc}
-t & 1 \\
-t & 1
\end{array}
\right)
\ ;\ 
\mathrm{v}=
\left(
\begin{array}{c}
1 \\
t
\end{array}
\right)
\ ;\ 
\mathrm{v}_0=
\left(
\begin{array}{c}
1 \\
0
\end{array}
\right)
.
\]Then $\KK[t]/(t^{n})\cdot\mathrm{v}$ is stabilised by $x$, but not
$\KK[t]/(t^{n})\cdot\mathrm{v}_{0}.$

Fan uses the above cartesian diagram to conclude that the map $\tilde{\mathcal{F}}_{\blacksquare,n}\rightarrow\tilde{\mathcal{F}}_{\blacksquare}$
is an affine fibration. This would then imply that $\pi_{\blacksquare}:\tilde{\mathcal{F}}_{\blacksquare,n}\rightarrow R(Q,n;\rr)$
is the composition of an affine fibration and a proper map, hence
semisimplicity for $\mathcal{L}_{\blacksquare}$ (see \cite[\S 8.1.6.]{Lus10}).
However, the following example shows that the map $\pi_{\blacksquare}:\tilde{\mathcal{F}}_{\blacksquare,n}\rightarrow R(Q,n;\rr)$
is not of this form.

\begin{exmp} \label{Exmp/A2}

Let $(Q,n)=(A_{2},2)$, $\rr=(2,1)$ and $\blacksquare=(\epsilon_{1},\epsilon_{2},\epsilon_{1})$.
Then both $\tilde{\mathcal{F}}_{\blacksquare,n}$ and $\tilde{\mathcal{F}}_{\blacksquare}$
are subspaces of $R(Q,n;\rr)\times\mathcal{F}_{\blacksquare,n}$.
Here, $\mathcal{F}_{\blacksquare,n}$ is the first jet scheme of $\mathbb{P}_{\KK}^{1}$
and $R(Q,n;\rr)$ is the space of $1\times2$ matrices with coefficients
in $\KK[t]/(t^{2})$. Under the action of $\GL_{2,2}\times\mathbb{G}_{\mathrm{m},2}$,
there are three orbits in $R(Q,n;\rr)$: let us call them $O_{0}:=\GL_{n,\rr}\cdot(1,0)$,
$O_{1}:=\GL_{n,\rr}\cdot(t,0)$ and $O_{2}:=\GL_{n,\rr}\cdot(0,0)$.
Then the fibres of the projections $\tilde{\mathcal{F}}_{\blacksquare,n}\rightarrow R(Q,n;\rr)$
and $\tilde{\mathcal{F}}_{\blacksquare}\rightarrow R(Q,n;\rr)$ go
as follows.\[
\begin{tabular}{l|c|c|c}
& $x\in O_0$ & $x\in O_1$ & $x\in O_2$ \\
\hline
$\text{Fibres of }\tilde{\mathcal{F}}_{\blacksquare,n}\rightarrow R(Q,n;\rr)$ & $\mathrm{pt}$ & $\mathbb{A}_{\KK}^1$ & $\mathcal{F}_{\blacksquare,n}$ \\
\hline
$\text{Fibres of }\tilde{\mathcal{F}}_{\blacksquare}\rightarrow R(Q,n;\rr)$ & $\mathrm{pt}\text{ or }\emptyset$ & $\mathrm{pt}$ & $\mathbb{P}_{\KK}^1$
\end{tabular}
\]Note that, over $\overline{O_{1}}=O_{1}\cup O_{2}$, there is indeed
an affine fibration $\tilde{\mathcal{F}}_{\blacksquare,n}\vert_{\overline{O_{1}}}\rightarrow\tilde{\mathcal{F}}_{\blacksquare}\vert_{\overline{O_{1}}}$
as described by Fan. However, this is no longer the case over $O_{0}$,
since for a given $x\in O_{0}$, the unique free module $M\subseteq(\KK[t]/(t^{2}))^{\oplus2}$
of rank 1 such that $x(M)=0$ might not be of the form $V\otimes_{\KK}\KK[t]/(t^{2})$.

\end{exmp}

\begin{rmk}

One might want to try the following fix to obtain a nice description
of $\pi_{\blacksquare}$. Instead of $\tilde{\mathcal{F}}_{\blacksquare}$,
let us define:\[
\tilde{\mathcal{F}'}_{\blacksquare}:=\left\{(x,V_{\blacksquare})\in R(Q,n;\rr)\times \mathcal{F}_{\blacksquare}\ \vert\ \overline{x}(V_{\blacksquare})\subseteq V_{\blacksquare} \right\},
\]where $\overline{x}\in R(Q,\rr)$ is the representation induced by
$x$ (modulo $t$). This does not work either, since the fibres of
the first projection $\tilde{\mathcal{F}'}_{\blacksquare}\rightarrow R(Q,n;\rr)$
go as follows.\[
\begin{tabular}{l|c|c|c}
& $x\in O_0$ & $x\in O_1$ & $x\in O_2$ \\
\hline
$\text{Fibres of }\tilde{\mathcal{F}}_{\blacksquare,n}\rightarrow R(Q,n;\rr)$ & $\mathrm{pt}$ & $\mathbb{A}_{\KK}^1$ & $\mathcal{F}_{\blacksquare,n}$ \\
\hline
$\text{Fibres of }\tilde{\mathcal{F}'}_{\blacksquare}\rightarrow R(Q,n;\rr)$ & $\mathrm{pt}$ & $\mathbb{P}_{\KK}^1$ & $\mathbb{P}_{\KK}^1$
\end{tabular}
\]

\end{rmk}

Finally, we provide the following counterexample of a non-semisimple
Lusztig sheaf $\mathcal{L}_{\blacksquare}$.

\begin{prop} \label{Prop/SemisimpFail}

Consider the quiver with multiplicities $(Q,n)=(A_{2},2)$ and rank
vector $\rr=(2,1)$. Set $\blacksquare=(\epsilon_{1},\epsilon_{2},\epsilon_{1})$.
Then the Lusztig sheaf $\mathcal{L}_{\blacksquare}$ is not semisimple.

\end{prop}

\begin{proof}

We denote by $O_{i},\ i=0,1,2$ the orbits of $\GL_{n,\rr}\circlearrowleft R(Q,n;\rr)$
as in Example \ref{Exmp/A2}. Let us compute the perverse cohomology
sheaves of $\mathcal{L}_{\blacksquare}$. The adjunction triangle
(Proposition \ref{Prop/FormulasD_c^b}) associated to the open-closed
decomposition $j:O_{0}\hookrightarrow\overline{O_{0}}\hookleftarrow\overline{O_{1}}:i$
yields:\[
\mathrm{R}j_!\underline{\bar{\QQ_l}}_{O_0}\rightarrow\mathcal{L}_{\blacksquare}\rightarrow\mathrm{IC}(O_1)[-4](-1)\oplus\mathrm{IC}(O_2)[-4](-2).
\]Indeed, $\pi_{\blacksquare}$ induces an isomorphism $\pi_{\blacksquare}^{-1}(O_{0})\simeq O_{0}$,
so we obtain by base change (Proposition \ref{Prop/FormulasD_c^b})
that $\mathrm{R}j_{!}j^{*}\mathcal{L}_{\blacksquare}\simeq\mathrm{R}j_{!}\underline{\bar{\QQ_{l}}}_{O_{0}}$.
On the other hand, $\pi_{\blacksquare}\vert_{\pi_{\blacksquare}^{-1}(\overline{O_{1}})}$
is the composition of an affine fibration\footnote{Similarly to the proof of Lemma \ref{Lem/AffFib}, if $f:X\rightarrow Y$
is an affine fibration of dimension $d$, there is an isomorphism
$\mathrm{R}f_{!}\underline{\bar{\QQ_{l}}}_{X}\simeq\underline{\bar{\QQ_{l}}}_{Y}[-2d](-d)$.} and a proper morphism, so $\mathrm{R}i_{*}i^{*}\mathcal{L}_{\blacksquare}$
is semisimple, by Theorem \ref{Thm/BBDGDecomp}. Moreover, for $x\in O_{2}$
(resp. $x\in O_{1}$), we know that $(\mathrm{R}i_{*}i^{*}\mathcal{L}_{\blacksquare})_{x}\simeq\underline{\bar{\QQ_{l}}}[-2](-1)\oplus\underline{\bar{\QQ_{l}}}[-4](-2)$
(resp. $(\mathrm{R}i_{*}i^{*}\mathcal{L}_{\blacksquare})_{x}\simeq\underline{\bar{\QQ_{l}}}[-2](-1)$)
from the analysis of fibres done in Example \ref{Exmp/A2}. So we
obtain $\mathrm{R}i_{*}i^{*}\mathcal{L}_{\blacksquare}\simeq\mathrm{IC}(O_{1})[-4](-1)\oplus\mathrm{IC}(O_{2})[-4](-2)$\footnote{Note that there are no non-trivial, locally constant sheaves on $O_{1}$,
since the stabiliser of $x\in O_{1}$ is irreducible, hence connected.
See for instance \cite[\S III.15.]{KW01}.}. For $\mathrm{R}j_{!}\underline{\bar{\QQ_{l}}}_{O_{0}}$, we use
another adjunction triangle $\mathrm{R}j_{!}j^{*}\underline{\bar{\QQ_{l}}}_{\overline{O_{0}}}\rightarrow\underline{\bar{\QQ_{l}}}_{\overline{O_{0}}}\rightarrow\mathrm{R}i_{*}i^{*}\underline{\bar{\QQ_{l}}}_{\overline{O_{0}}}$,
which yields $\mathrm{R}j_{!}\underline{\bar{\QQ_{l}}}_{O_{0}}\rightarrow\mathrm{IC}(O_{0})[-4]\rightarrow\mathrm{IC}(O_{1})[-2]$.
Thus, $\mathrm{IC}(O_{1})$ (resp. $\mathrm{IC}(O_{0})$) is the third
(resp. the fourth) perverse cohomology sheaf of $\mathrm{R}j_{!}\underline{\bar{\QQ_{l}}}_{O_{0}}$.
The long exact sequence in perverse cohomology then gives:\[
\ldots\rightarrow 0
\rightarrow \mathrm{IC}(O_{1})
\rightarrow\ ^{\mathrm{p}}\HH^3(\mathcal{L}_{\blacksquare})
\rightarrow 0
\rightarrow \mathrm{IC}(O_{0})
\rightarrow\ ^{\mathrm{p}}\HH^4(\mathcal{L}_{\blacksquare})
\rightarrow \mathrm{IC}(O_{1})(-1)\oplus\mathrm{IC}(O_{2})(-2)
\rightarrow 0
\rightarrow\ldots,
\]hence $^{\mathrm{p}}\HH^{3}(\mathcal{L}_{\blacksquare})\simeq\mathrm{IC}(O_{1})$
and $^{\mathrm{p}}\HH^{4}(\mathcal{L}_{\blacksquare})\simeq\mathrm{IC}(O_{0})\oplus\mathrm{IC}(O_{1})(-1)\oplus\mathrm{IC}(O_{2})(-2)$
by Proposition \ref{Prop/WeightDecomp} (all intersection complexes
in perverse degree 4 have weight 4). Let $K^{\bullet}=\ ^{\mathrm{p}}\HH^{3}(\mathcal{L}_{\blacksquare})[-3]\oplus\ ^{\mathrm{p}}\HH^{4}(\mathcal{L}_{\blacksquare})[-4]$.
We compare in the table below stalks of the (standard) cohomology
sheaves $\HH^{i}(\mathcal{L}_{\blacksquare})$ and $\HH^{i}(K^{\bullet})$.
Since these do not coincide, we conclude that $\mathcal{L}_{\blacksquare}$
is not semisimple.

\[
\begin{tabular}{c|c|c|c|c}
 & & $x\in O_0$ & $x\in O_1$ & $x\in O_2$ \\
\hline
\hline
\multirow{4}{*}{$\iota_x^*K^{\bullet}$} & $\HH_{\mathrm{\acute{e}t}}^0$ & $\bar{\QQ_l}$ & $\bar{\QQ_l}$ & $\bar{\QQ_l}$ \\
 & $\HH_{\mathrm{\acute{e}t}}^1$ & $0$ & $\bar{\QQ_l}$ & $\bar{\QQ_l}$ \\
 & $\HH_{\mathrm{\acute{e}t}}^2$ & $0$ & $\bar{\QQ_l}(-1)$ & $\bar{\QQ_l}(-1)$ \\
 & $\HH_{\mathrm{\acute{e}t}}^4$ & $0$ & $0$ & $\bar{\QQ_l}(-2)$ \\
\hline
\hline
\multirow{3}{*}{$\iota_x^*\mathcal{L}_{\blacksquare}$} & $\HH_{\mathrm{\acute{e}t}}^0$ & $\bar{\QQ_l}$ & $0$ & $0$ \\
 & $\HH_{\mathrm{\acute{e}t}}^2$ & $0$ & $\bar{\QQ_l}(-1)$ & $\bar{\QQ_l}(-1)$ \\
 & $\HH_{\mathrm{\acute{e}t}}^4$ & $0$ & $0$ & $\bar{\QQ_l}(-2)$
\end{tabular}
\]\end{proof}

\begin{rmk}

The computation above also shows that certain perverse constitutents
of $\mathcal{L}_{\blacksquare}$ (or the associated function) do not
lie in the Ringel-Hall algebra. Let us switch to $\KK=\CC$ and let
$(Q,n)$ and $\rr$ be as in Proposition \ref{Prop/SemisimpFail}.
Then $\chi_{\mathrm{IC}(O_{1})}=\mathbf{1}_{O_{1}}+\mathbf{1}_{O_{2}}$
is the function associated to the third perverse cohomology sheaf
of $\mathcal{L}_{\blacksquare}$. However, a quick computation in
Geiss, Leclerc and Schröer's spherical Hall algebra \cite{GLS16}
shows that its component of rank vector $\rr$ is generated by $\mathbf{1}_{O_{2}}$
and $\mathbf{1}_{O_{0}}+\mathbf{1}_{O_{1}}+\mathbf{1}_{O_{2}}$ as
a $\QQ$-vector space. So $\chi_{\mathrm{IC}(O_{1})}$ does not lie
in the spherical Hall algebra. This owes to the fact that $\mathcal{L}_{\blacksquare}$
is not semisimple, as the contributions of $\mathrm{IC}(O_{1})$ to
$^{\mathrm{p}}\mathcal{H}^{3}(\mathcal{L}_{\blacksquare})$ and $^{\mathrm{p}}\mathcal{H}^{4}(\mathcal{L}_{\blacksquare})$
cancel out.

\end{rmk}

\subsection{A tentative cohomological Hall algebra for $A_{2}$ with multiplicities}

In this section, we attempt to build a CoHA product on $\HH_{\bullet}^{\mathrm{BM}}\left(\mathfrak{M}_{\Pi_{(A_{2},\alpha)}}\right)$
``by hand'', taking inspiration from the structure of the preprojective
CoHA when $\alpha=1$. For quivers without multiplicities, the multiplication
of the preprojective CoHA is built using functorial properties of
Borel-Moore homology with respect to locally complete intersection
pullback and proper pushforward \cite{SV13b,YZ18a}. However, when
$n>1$, the stack of flags of $\Pi_{(Q,n)}$-modules is no longer
proper over $\mathfrak{M}_{\Pi_{(Q,n)}}$, so we cannot construct
a CoHA with those techniques. Instead, we compute the Borel-Moore
homology of the preprojective stack $\mathfrak{M}_{\Pi_{(A_{2},\alpha)}}$
and try to build a multiplication explicitly, using results of Davison
and Hennecart on the structure of the preprojective CoHA of $A_{2}$
\cite{Dav22,Hen24}. Throughout, we work over $\KK=\CC$.

Let us first compute $\HH_{\bullet}^{\mathrm{BM}}\left(\mathfrak{M}_{\Pi_{(A_{2},\alpha)}}\right)$.

\begin{prop} \label{Prop/BMhomPreprojStA2}

Let $\alpha\geq1$, $(Q,\nn):=(A_{2},\alpha)$ and $\rr\in\ZZ_{\geq0}^{Q_{0}}$.
Let $\langle\bullet,\bullet\rangle$ denote the Euler form of $Q$.
Then there is an isomorphism of $\ZZ\times\ZZ_{\geq0}^{Q_{0}}$-graded
mixed Hodge structures:\[
\HH_{-\bullet}^{\mathrm{BM}}\left(\mathfrak{M}_{\Pi_{(A_{2},\alpha)},\rr}\right)\otimes\mathbb{L}^{-\alpha\langle\rr,\rr\rangle}
\simeq
\bigoplus_{\substack{r\geq0 \\ q_0+\ldots+q_{\alpha-1}=r}}
\left(
\HH^{\bullet}(\mathrm{BGL}_{r_1-r})\otimes\HH^{\bullet}(\mathrm{BGL}_{r_2-r})\otimes\left(\bigotimes_{i=0}^{\alpha-1}\HH^{\bullet}(\mathrm{BGL}_{q_i})\right)
\right)
.
\]In particular, $\HH_{-\bullet}^{\mathrm{BM}}\left(\mathfrak{M}_{\Pi_{(A_{2},\alpha)}}\right)$
is pure, of Tate type.

\end{prop}

\begin{proof}

We compute $\HH_{\mathrm{c}}^{\bullet}\left(\mathfrak{M}_{\Pi_{(A_{2},\alpha)},\rr}\right)$
by stratifying $\mu_{(A_{2},\alpha),\rr}^{-1}(0)$ along orbits of
$\GL_{\alpha,\rr}\circlearrowleft R(A_{2},\alpha;\rr)$. There are
finitely many orbits, indexed by the Smith normal form \cite{Sta16}
of $x\in R(A_{2},\alpha;\rr)=\Hom_{\mathcal{O}_{\alpha}}(\mathcal{O}_{\alpha}^{\oplus r_{1}},\mathcal{O}_{\alpha}^{\oplus r_{2}})$:\[
x\sim
\left(
\begin{array}{cccccc}
I_{q_0} & \ldots & 0 & 0 & \ldots & 0 \\
\vdots & \ddots & \vdots & \vdots & \ddots & \vdots \\
0 & \ldots & t^{\alpha-1}I_{q_{\alpha-1}} & 0 & \ldots & 0 \\
0 & \ldots & 0 & 0 & \ldots & 0 \\
\vdots & \ddots & \vdots & \vdots & \ddots & \vdots \\
0 & \ldots & 0 & 0 & \ldots & 0
\end{array}
\right)
.
\]Consider $\mathbb{O}_{x}=\GL_{\alpha,\rr}\cdot x$, the orbit of $x$
in $R(A_{2},\alpha;\rr)$ and $\pi:\mu_{(A_{2},\alpha),\rr}^{-1}(0)\rightarrow R(A_{2},\alpha;\rr)$
the projection $(x,y)\mapsto x$. Then by Proposition \ref{Prop/MomMapExSeq},
$\pi^{-1}(\mathbb{O}_{x})$ is the conormal bundle to $\mathbb{O}_{x}$,
hence $\pi$ restricts to an affine fibration $\pi^{-1}(\mathbb{O}_{x})\rightarrow\mathbb{O}_{x}$
of relative dimension $\dim(\GL_{\alpha,\rr})_{x}-\alpha\langle\rr,\rr\rangle$.
By Lemma \ref{Lem/AffFib} and a variant of Lemma \ref{Lem/GrpChg},
we obtain $\HH_{\mathrm{c}}^{\bullet}\left([\pi^{-1}(\mathbb{O}_{x})/\GL_{\alpha,\rr}]\right)\otimes\mathbb{L}^{\otimes\alpha\langle\rr,\rr\rangle}\simeq\HH_{\mathrm{c}}^{\bullet}(\mathrm{B}(\GL_{\alpha,\rr})_{x})\otimes\mathbb{L}^{\otimes\dim(\GL_{\alpha,\rr})_{x}}\simeq\HH^{\bullet}(\mathrm{B}(\GL_{\alpha,\rr})_{x})^{\vee}$
(mind that we dualise $\HH^{\bullet}(\mathrm{B}(\GL_{\alpha,\rr})_{x})$
without changing cohomological degrees). The claim now follows from
Lemma \ref{Lem/Strat} and Lemma \ref{Lem/CohStabiliser} below. \end{proof}

\begin{lem} \label{Lem/CohStabiliser}

Given $x\in R(A_{2},\alpha;\rr)$ with Smith normal form as in the
proof of Proposition \ref{Prop/BMhomPreprojStA2}, we have:\[
\HH^{\bullet}(\mathrm{B}(\GL_{\alpha,\rr})_{x})\simeq\HH^{\bullet}(\mathrm{BGL}_{r_1-r})\otimes\HH^{\bullet}(\mathrm{BGL}_{r_2-r})\otimes\left(\bigotimes_{i=0}^{\alpha-1}\HH^{\bullet}(\mathrm{BGL}_{q_i})\right)
.
\]

\end{lem}

\begin{proof}

Consider $(g_{1},g_{2})\in(\GL_{\alpha,\rr})_{x}$. Let us decompose
$g_{1}$ and $g_{2}$ in blocks following the Smith normal form of
$x$: $g_{1}=(g_{1}^{i,j})$ and $g_{2}=(g_{2}^{k,l})$, where $i,j$
(resp. $k,l$) range from $0$ to $\alpha$ and $g_{1}^{i,j}$ (resp.
$g_{2}^{k,l}$) has size $q_{i}\times q_{j}$, with the convention
$q_{\alpha}=r_{1}-r$ (resp. $q_{\alpha}=r_{2}-r$). Then the equation
$g_{2}x=xg_{1}$ reads:\[
\left\{
\begin{array}{l}
t^ig_1^{i,j}=t^jg_2^{k,l}\text{, for }i,j,k,l\leq\alpha-1, \\
g_1^{i,\alpha}=0\text{, for }i\leq\alpha-1, \\
g_2^{\alpha,j}=0\text{, for }j\leq\alpha-1.
\end{array}
\right.
\]This implies that $g_{1}$ (resp. $g_{2}$) is block upper-triangular
(resp. block lower-triangular) modulo $t$, $g_{1}^{i,i}=g_{2}^{i,i}$
modulo $t$ for $0\leq i\leq\alpha-1$ and $g_{1}^{i,i},g_{2}^{i,i}$
are invertible for all $0\leq i\leq\alpha$. Now, let $\GL_{\alpha,\rr}^{1}$
be the subgroup of matrices which are congruent to $\mathrm{Id}$
modulo $t$. There is a short exact sequence of algebraic groups:\[
1\rightarrow (\GL_{\alpha,\rr})_{x}\cap\GL_{\alpha,\rr}^{1}\rightarrow (\GL_{\alpha,\rr})_{x}\rightarrow\overline{(\GL_{\alpha,\rr})_{x}}\rightarrow 1.
\]As an algebraic variety, $(\GL_{\alpha,\rr})_{x}\cap\GL_{\alpha,\rr}^{1}$
is an affine space, so the quotient map $(\GL_{\alpha,\rr})_{x}\rightarrow\overline{(\GL_{\alpha,\rr})_{x}}$
is an affine fibration\footnote{The fibration is Zariski-locally trivial because $(\GL_{n,\rr})_{x}\cap\GL_{n,\rr}^{1}$
is a special group, see \cite{Ser58}.}. Moreover, the above discussion shows that $\overline{(\GL_{\alpha,\rr})_{x}}$
is a product of an affine space and $\GL_{r_{1}-r}\times\GL_{r_{2}-r}\times\prod_{i}\GL_{q_{i}}$.
The claim now follows from a variant of Lemma \ref{Lem/DepthChg}.
\end{proof}

From Proposition \ref{Prop/BMhomPreprojStA2}, we can already deduce:

\begin{cor} \label{Cor/AbstractPBWiso}

There exists an isomorphism of $\ZZ\times\ZZ_{\geq0}^{Q_{0}}$-graded
mixed Hodge structures:\[
\mathcal{A}_{\Pi_{(A_2,\alpha)}}:=
\bigoplus_{\rr\in\ZZ_{\geq0}^{Q_0}}\HH_{-\bullet}^{\mathrm{BM}}\left(\mathfrak{M}_{\Pi_{(A_{2},\alpha)},\rr}\right)\otimes\mathbb{L}^{\otimes(-\alpha\langle\rr,\rr\rangle)}
\simeq
\mathrm{Sym}\left(\bigoplus_{\rr\in\ZZ_{\geq0}^{Q_0}\setminus\{0\}}A_{(A_2,\alpha),\rr}(\mathbb{L}^{\otimes (-1)})\otimes\HH^{\bullet}(\mathrm{B}\mathbb{G}_{\mathrm{m}})\right),
\]where $\mathrm{Sym}$ is the operator on $\ZZ\times\ZZ_{\geq0}^{Q_{0}}$-graded
mixed Hodge structures which categorifies the plethystic exponential
(see \cite[\S 3.2.]{DM20}).

\end{cor}

Indeed, the dual isomorphism follows from Theorem \ref{Thm/Ch2ExpFmlKacPol}
and the fact that both mixed Hodge structures are pure, of Tate type
(see Section \ref{Subsect/CohAlgVar}). By the discussion on Smith
normal forms, we also obtain:\[
A_{(A_2,\alpha),\rr}=
\left\{
\begin{array}{ll}
1, & \text{if }\rr=\epsilon_i,\ i=1,2; \\
\alpha, & \text{if }\rr=\epsilon_1+\epsilon_2; \\
0, & \text{otherwise}.
\end{array}
\right.
\]In other words, the polynomials $A_{(A_{2},\alpha),\rr}$ are categorified
by a $\ZZ_{\geq0}^{Q_{0}}$-graded pure Hodge structure, sitting inside
$\mathcal{A}_{\Pi_{(A_{2},\alpha)}}$ and concentrated in cohomological
degree 0. In light of the PBW isomorphism for preprojective CoHAs
\cite{DHSM22}, it would be interesting to realise the isomorphism
of Corollary \ref{Cor/AbstractPBWiso} using a CoHA multiplication
on $\mathcal{A}_{\Pi_{(A_{2},\alpha)}}$. Let us briefly recall how
this works when $\alpha=1$.

When there are no multiplicities, $\mathcal{A}_{\Pi_{A_{2}}}$ can
be endowed with the aforementioned Hall algebra product \cite{SV13b,YZ18a}.
Note that $\mathcal{A}_{\Pi_{A_{2}}}$ is concentrated in non-negative
cohomological degrees. The part $\mathcal{A}_{\Pi_{A_{2}}}^{0}$ of
cohomological degree 0 is a subalgebra, called the BPS algebra \cite{Dav20}.
As a $\QQ$-vector space, $\mathcal{A}_{\Pi_{A_{2}}}^{0}$ is spanned
by the irreducible components of $\mu_{A_{2},\dd}^{-1}(0)$, $\dd\in\ZZ_{\geq0}^{Q_{0}}$.
These components are all lagrangian subvarieties of $R(\overline{Q},\dd)=\mathrm{T}^{*}R(Q,\dd)$,
the conormal bundles of the orbits of $\GL_{\dd}\circlearrowleft R(Q,\dd)$
\cite[Exmp. 4.10.]{S12b}. In \cite{Hen24}, Hennecart showed that
the characteristic cycle map induces an isomorphism of algebras $\mathrm{CC}:\mathbf{H}_{A_{2}}\overset{\sim}{\longrightarrow}\mathcal{A}_{\Pi_{A_{2}}}^{0}$
between the (spherical) Ringel-Hall algebra and $\mathcal{A}_{\Pi_{A_{2}}}^{0}$.
Then by \cite[Prop. 10.20.]{Lus91}, we obtain $\mathcal{A}_{\Pi_{A_{2}}}^{0}\simeq U(\mathfrak{n})$,
where $\mathfrak{n}$ is the positive part of $\mathfrak{sl}_{3}$.
Finally, Davison showed \cite[Prop. 3.7, 3.14.]{Dav22} that $\mathcal{A}_{\Pi_{A_{2}}}\simeq U(\mathfrak{n}[u])$.
His proof relies on a coproduct built on $\mathcal{A}_{\Pi_{A_{2}}}$
using the sheaf-theoretic version of the CoHA, and an action of the
Heisenberg algebra on $\mathcal{A}_{\Pi_{A_{2}}}$ involving Chern
classes of tautological bundles.

We can generalise part of this structure to higher multiplicities
($\alpha>1$), using our description of $\mathcal{A}_{\Pi_{(A_{2},\alpha)}}^{0}$
from Proposition \ref{Prop/BMhomPreprojStA2}. Let us first establish
an isomorphism of graded vector spaces between the Ringel-Hall algebra
and $\mathcal{A}_{\Pi_{(A_{2},\alpha)}}^{0}$.

\begin{prop} \label{Prop/CCiso}

For $\rr\in\ZZ_{\geq0}^{Q_{0}}$, consider the $\QQ$-vector space
$\mathbf{H}_{(A_{2},\alpha),\rr}$ of constructible functions on $R(A_{2},\alpha;\rr)$
which are constant along $\GL_{\alpha,\rr}$-orbits. Then the characteristic
cycle map induces an isomorphism of $\QQ$-vector spaces:\[
\mathrm{CC}:
\mathbf{H}_{(A_{2},\alpha),\rr}
\overset{\sim}{\longrightarrow}
\HH_{\mathrm{top}}^{\mathrm{BM}}\left(\mathfrak{M}_{\Pi_{(A_{2},\alpha)},\rr}\right)
,
\]where $\HH_{\mathrm{top}}^{\mathrm{BM}}\left(\mathfrak{M}_{\Pi_{(A_{2},\alpha)},\rr}\right)$
is the equivariant Borel-Moore homology group of degree $2\dim\left(\mathfrak{M}_{\Pi_{(A_{2},\alpha)},\rr}\right)$.

\end{prop}

\begin{proof}

Since $\GL_{\alpha,\rr}$ acts on $R(A_{2},\alpha;\rr)$ with finitely
many orbits, these form a Whitney stratification of $R(A_{2},\alpha;\rr)$,
by Proposition \ref{Prop/WhitneyStrat}. Moreover, there is a partial
order on orbits defined by: $\mathbb{O}'\leq\mathbb{O}$ if, and only
if, $\mathbb{O}'\subseteq\overline{\mathbb{O}}$. Let us call the
orbits $\mathbb{O}_{1},\ldots,\mathbb{O}_{n}$ so that $i\leq j$
whenever $\mathbb{O}_{j}\leq\mathbb{O}_{i}$. Now:\begin{align*}
& \mathbf{H}_{(A_{2},n),\rr}=\bigoplus_{i=1}^n\QQ\cdot\mathbf{1}_{\mathbb{O}_{i}}, \\
& \HH_{\mathrm{top}}^{\mathrm{BM}}\left(\mathfrak{M}_{\Pi_{(A_{2},\alpha)},\rr}\right)=\bigoplus_{i=1}^n\QQ\cdot [\overline{\mathrm{T}_{\mathbb{O}_{i}}^*R(A_{2},\alpha;\rr)}].
\end{align*}We claim that the matrix of $\mathrm{CC}$ in those bases is lower-unitriangular.
Indeed, for $1\leq i\leq n$, $U_{i}=R(Q,n;\rr)\setminus\bigcup_{j>i}\overline{\mathbb{O}_{j}}$
is open and contains $\mathbb{O}_{i}$ as a closed submanifold. By
Proposition \ref{Prop/CCcomp}, we obtain:\[
\mathrm{CC}(\mathbf{1}_{\mathbb{O}_{i}})\vert_{U_{i}}=\mathrm{CC}(\mathbf{1}_{\mathbb{O}_{i}}\vert_{U_{i}})=[\overline{\mathrm{T}_{\mathbb{O}_{i}}^*U_i}].
\]So $\mathrm{CC}(\mathbf{1}_{\mathbb{O}_{i}})\in[\overline{\mathrm{T}_{\mathbb{O}_{i}}^{*}R(A_{2},\alpha;\rr)}]+\sum_{j>i}\QQ\cdot[\overline{\mathrm{T}_{\mathbb{O}_{j}}^{*}R(A_{2},\alpha;\rr)}]$,
which finishes the proof. \end{proof}

The structure of the Ringel-Hall algebra $\mathbf{H}_{(A_{2},\alpha)}$
is also well-understood, thanks to work of Geiss, Leclerc and Schröer
\cite{GLS16}. Let us briefly recall some of their results. Given
$(f_{1},f_{2})\in\mathbf{H}_{(A_{2},\alpha),\rr_{1}}\times\mathbf{H}_{(A_{2},\alpha),\rr_{2}}$
and $\rr=\rr_{1}+\rr_{2}$, the Hall product is defined by:\[
f_1*f_2=\left(\pi_{(\rr_1,\rr_2),\alpha}\right)_!\left( f_1\boxtimes f_2\right)
,
\]where $\pi_{(\rr_{1},\rr_{2}),\alpha}:\tilde{\mathcal{F}}_{(\rr_{1},\rr_{2}),\alpha}\rightarrow R(A_{2},\alpha;\rr)$
is the first projection from the variety of pairs\[
\tilde{\mathcal{F}}_{(\rr_{1},\rr_{2}),\alpha}
:=
\left\{
(x,M)\in R(A_{2},\alpha;\rr)\times\mathcal{F}_{(\rr_{1},\rr_{2}),\alpha} \vert\ x(M)\subseteq M
\right\}
,
\]$\mathcal{F}_{(\rr_{1},\rr_{2}),\alpha}$ is the variety of flags
of locally free modules\footnote{Note that Geiss, Leclerc and Schröer integrate $f\boxtimes g$ along
the variety of flags of all modules \cite{GLS16}. In our case, since
$f,g$ are supported on the variety of locally free representations,
this gives the same product.} $M\subseteq M_{\rr}:=\bigoplus_{i\in Q_{0}}\mathcal{O}_{\alpha}^{\oplus r_{i}}$
of ranks $\rr_{1},\rr$, and $(f\boxtimes g)(x,M):=f(x\vert_{M})g(x\vert_{M_{\rr}/M})$.
Here, we abuse notations and denote by $f(x\vert_{M})$ the value
of $f$ on the orbit of $\GL_{\alpha,\rr}\circlearrowleft R(A_{2},\alpha;\rr)$
corresponding to the restriction of $x$ to $M$.

The Ringel-Hall algebra $\mathbf{H}_{(A_{2},\alpha)}$ is also endowed
with a comultiplication $\Delta$, which promotes it to a cocommutative
bialgebra \cite[Prop. 4.5.]{GLS16}. Given $f\in\mathbf{H}_{(A_{2},\alpha),\rr}$,
$\Delta(f)\in\bigoplus_{\rr_{1}+\rr_{2}=\rr}\mathbf{H}_{(A_{2},\alpha),\rr_{1}}\otimes\mathbf{H}_{(A_{2},\alpha),\rr_{2}}$
is defined by:\[
\Delta(f)(x_1,x_2)=f(x_1\oplus x_2),
\]where we again abuse notations and denote by $f(x_{1}\oplus x_{2})$
the value of $f$ at the orbit corresponding to the direct sum of
representations associated to $x_{1},x_{2}$. We can now state and
prove:

\begin{prop} \label{Prop/RingelHallA2mult}

As a bialgebra, $\mathbf{H}_{(A_{2},\alpha)}$ is isomorphic to $U(\tilde{\mathfrak{n}}[u])$,
where $\tilde{\mathfrak{n}}$ is the trivial extension of the upper
half of $\mathfrak{sl}_{3}$ by an abelian Lie algebra of dimension
$\alpha-1$.

\end{prop}

\begin{proof}

For $0\leq i\leq\alpha$, let $\mathbb{O}_{i}$ be the orbit of the
matrix $t^{i}\in R(A_{2},\alpha;\epsilon_{1}+\epsilon_{2})$. Then
\[
\tilde{\mathfrak{n}}:=\mathbf{H}_{(A_{2},\alpha),\epsilon_1}\oplus\mathbf{H}_{(A_{2},\alpha),\epsilon_2}\oplus\bigoplus_{i=0}^{\alpha-1}\QQ\cdot\mathbf{1}_{\mathbb{O}_{i}}
\] is the subspace of primitive elements of $\mathbf{H}_{(A_{2},\alpha)}$
with respect to $\Delta$. Indeed, $\Delta(f)=f\otimes1+1\otimes f$
if, and only, if $f$ is supported on orbits of indecomposable representations
of $(A_{2},\alpha)$. Therefore, $\tilde{\mathfrak{n}}$ is a Lie
subalgebra of $\mathbf{H}_{(A_{2},\alpha)}$.

Let us also call $\mathbf{1}_{\epsilon_{i}},\ i=1,2$ the function
supported on $R(A_{2},\alpha;\epsilon_{i})$ with value $1$. By \cite[Prop. 4.9.]{GLS16},
$\QQ\cdot\mathbf{1}_{\epsilon_{1}}\oplus\QQ\cdot\mathbf{1}_{\epsilon_{2}}\oplus\QQ\cdot\mathbf{1}_{\mathbb{O}_{0}}$
is isomorphic to the positive part of $\mathfrak{sl}_{3}$ as a Lie
algebra. Since $\tilde{\mathfrak{n}}$ is a Lie subalgebra of $\mathbf{H}_{(A_{2},\alpha)}$,
the Lie bracket vanishes for all other combinations of the above basis
elements, by examining their $\ZZ_{\geq0}^{Q_{0}}$-grading. Therefore,
$\tilde{\mathfrak{n}}$ is the trivial extension of the positive part
of $\mathfrak{sl}_{3}$ by an abelian Lie algebra of dimension $\alpha-1$.

The rest of the proof follows the arguments of \cite[Prop. 4.7.]{GLS16}:
the only group-like element of $\mathbf{H}_{(A_{2},\alpha)}$ is the
unit, so $\mathbf{H}_{(A_{2},\alpha)}$ is an irreducible cocommutative
bialgebra \cite[Lem. 8.0.1.]{Swe69}, hence a Hopf algebra \cite[Thm. 9.2.2.]{Swe69}
and thus there is an isomorphism $\mathbf{H}_{(A_{2},\alpha)}\simeq U(\tilde{\mathfrak{n}})$
\cite[Thm. 13.0.1.]{Swe69}. \end{proof}

Combining Propositions \ref{Prop/CCiso} and \ref{Prop/RingelHallA2mult},
we obtain a multiplication on $\mathcal{A}_{\Pi_{(A_{2},\alpha)}}^{0}\simeq\mathbf{H}_{(A_{2},\alpha)}$.
More precisely, for $x\in R(A_{2},\alpha;\rr)$ with Smith normal
form indexed by $(q_{0},\ldots,q_{\alpha-1})$, the lagrangian cycle
$[\overline{\mathrm{T}_{\mathbb{O}_{x}}^{*}R(A_{2},\alpha;\rr)}]$
gives a basis of\[
\HH^{0}(\mathrm{BGL}_{r_{1}-r})\otimes\HH^{0}(\mathrm{BGL}_{r_{2}-r})\otimes\left(\bigotimes_{i=0}^{\alpha-1}\HH^{0}(\mathrm{BGL}_{q_{i}})\right)\simeq\QQ
,
\]under the isomorphism:\[
\HH_{-\bullet}^{\mathrm{BM}}\left(\mathfrak{M}_{\Pi_{(A_{2},\alpha)},\rr}\right)\otimes\mathbb{L}^{-\alpha\langle\rr,\rr\rangle}
\simeq
\bigoplus_{\substack{r\geq0 \\ q_0+\ldots+q_{\alpha-1}=r}}
\left(
\HH^{\bullet}(\mathrm{BGL}_{r_1-r})\otimes\HH^{\bullet}(\mathrm{BGL}_{r_2-r})\otimes\left(\bigotimes_{i=0}^{\alpha-1}\HH^{\bullet}(\mathrm{BGL}_{q_i})\right)
\right)
.
\]Then $\tilde{\mathfrak{n}}$ is identified with the subspace of $\mathcal{A}_{\Pi_{(A_{2},\alpha)}}^{0}$
generated by $[\overline{\mathrm{T}_{\mathbb{O}_{i}}^{*}R(A_{2},\alpha;\epsilon_{1}+\epsilon_{2})}]$,
$0\leq i\leq\alpha-1$, which is isomorphic to $\bigoplus_{\rr}A_{(A_{2},\alpha),\rr}(\mathbb{L}^{\otimes(-1)})$
as a graded mixed Hodge structure. Under these identifications, the
isomorphism:\[
\mathcal{A}_{\Pi_{(A_2,\alpha)}}^0
\simeq
\mathrm{Sym}\left(\bigoplus_{\rr\in\ZZ_{\geq0}^{Q_0}\setminus\{0\}}A_{(A_2,\alpha),\rr}(\mathbb{L}^{\otimes (-1)})\right),
\] from Corollary \ref{Cor/AbstractPBWiso} can be realised as the PBW
isomorphism $\mathrm{Sym}(\tilde{\mathfrak{n}})\overset{\sim}{\longrightarrow}U(\tilde{\mathfrak{n}})$.
Therefore, $\tilde{\mathfrak{n}}$ can be regarded as a candidate
for a BPS Lie algebra of $(A_{2},\alpha)$.

Concerning $\mathcal{A}_{\Pi_{(A_{2},\alpha)}}$ in higher cohomological
degrees, we were so far unable to produce an explicit formula for
a multiplication. However, mimicking \cite{Dav22}, one can embed
$\tilde{\mathfrak{n}}[u]$ in:\[
\begin{split}
\mathcal{A}_{\Pi_{(A_2,\alpha)},\epsilon_1}\oplus\mathcal{A}_{\Pi_{(A_2,\alpha)},\epsilon_2}\oplus\mathcal{A}_{\Pi_{(A_2,\alpha)},\epsilon_1+\epsilon_2}
\simeq &
\ \HH^{\bullet}(\mathrm{B}\mathbb{G}_{\mathrm{m}})\oplus\HH^{\bullet}(\mathrm{B}\mathbb{G}_{\mathrm{m}}) \\
& \ \oplus\left(\bigoplus_{i=0}^{\alpha-1}\HH^{\bullet}(\mathrm{B}\mathbb{G}_{\mathrm{m}})\oplus\left(\HH^{\bullet}(\mathrm{B}\mathbb{G}_{\mathrm{m}})\otimes\HH^{\bullet}(\mathrm{B}\mathbb{G}_{\mathrm{m}})\right)\right) \\
\simeq & \ \QQ[u]\oplus\QQ[u]\oplus\left(\bigoplus_{i=0}^{\alpha-1}\QQ[u]\oplus\QQ[u_1,u_2]\right),
\end{split}
\]as the unique $\QQ[u]$-submodule of $\mathcal{A}_{\Pi_{(A_{2},\alpha)},\epsilon_{1}}\oplus\mathcal{A}_{\Pi_{(A_{2},\alpha)},\epsilon_{2}}\oplus\mathcal{A}_{\Pi_{(A_{2},\alpha)},\epsilon_{1}+\epsilon_{2}}$
containing $\tilde{\mathfrak{n}}\hookrightarrow\mathcal{A}_{\Pi_{(A_{2},\alpha)},\epsilon_{1}}^{0}\oplus\mathcal{A}_{\Pi_{(A_{2},\alpha)},\epsilon_{2}}^{0}\oplus\mathcal{A}_{\Pi_{(A_{2},\alpha)},\epsilon_{1}+\epsilon_{2}}^{0}$.
Here $\mathcal{A}_{\Pi_{(A_{2},\alpha)},\epsilon_{1}}\oplus\mathcal{A}_{\Pi_{(A_{2},\alpha)},\epsilon_{2}}\oplus\mathcal{A}_{\Pi_{(A_{2},\alpha)},\epsilon_{1}+\epsilon_{2}}$
is seen as a direct sum of regular $\QQ[u]$-modules and $\QQ[u_{1},u_{2}]$
(where $u$ acts by multiplication by $u_{1}+u_{2}$), using the above
decomposition. We record the following heuristic conjecture for future
works:

\begin{cj}

There is a geometrically meaningful (co)multiplication on $\mathcal{A}_{\Pi_{(A_{2},\alpha)}}$
which restricts to the previous (co)multiplication on $\mathcal{A}_{\Pi_{(A_{2},\alpha)}}^{0}$.
Moreover, there is an isomorphism of bialgebras $\mathcal{A}_{\Pi_{(A_{2},\alpha)}}\simeq U(\tilde{\mathfrak{n}}[u])$.

\end{cj}

\pagebreak{}

\phantomsection \label{Section/indexNotations}

\addcontentsline{toc}{section}{Glossary of notations}

\printindex[notations]

\pagebreak{}

\phantomsection \label{Section/indexTerms}

\addcontentsline{toc}{section}{Index}

\printindex[terms]

\pagebreak{}

\phantomsection 

\bibliographystyle{plain}
\addcontentsline{toc}{section}{\refname}\bibliography{0D__EPFL_Articles_Bibliographie}

\begin{thebibliography}{100}

\bibitem{AMRV22}
Tarig Abdelgadir, Anton Mellit, and Fernando Rodriguez-Villegas.
\newblock The {T}utte polynomial and toric {N}akajima quiver varieties.
\newblock {\em Proceedings of the Royal Society of Edinburgh},
  152(5):1323--1339, 2022.

\bibitem{Ach13}
Pramod Achar.
\newblock Equivariant mixed {H}odge modules.
\newblock Available at \url{https://www.math.lsu.edu/~pramod/docs/emhm.pdf},
  2013.

\bibitem{AA16}
Avraham Aizenbud and Nir Avni.
\newblock Representation growth and rational singularities of the moduli space
  of local systems.
\newblock {\em Inventiones Mathematicae}, 204(1):245--316, 2016.

\bibitem{AA18}
Avraham Aizenbud and Nir Avni.
\newblock Counting points of schemes over finite rings and counting
  representations of arithmetic lattices.
\newblock {\em Duke Mathematical Journal}, 167(14):2721--2743, 2018.

\bibitem{AHR20}
Jarod Alper, Jack Hall, and David Rydh.
\newblock A {L}una {\'e}tale slice theorem for algebraic stacks.
\newblock {\em Annals of Mathematics}, 191(3):675--738, 2020.

\bibitem{AS18}
Enrico Arbarello and Giulia Sacc\`{a}.
\newblock Singularities of moduli spaces of sheaves on {K}3 surfaces and
  {N}akajima quiver varieties.
\newblock {\em Advances in Mathematics}, 329:649--703, 2018.

\bibitem{ARS97}
Maurice Auslander, Idun Reiten, and Sverre Smal{\o}.
\newblock {\em Representation theory of {A}rtin algebras}.
\newblock Cambridge University Press, 1997.

\bibitem{AKOV16}
Nir Avni, Benjamin Klopsch, Uri Onn, and Christopher Voll.
\newblock Similarity classes of integral p-adic matrices and representation
  zeta functions of groups of type {A}2.
\newblock {\em Proceedings of the London Mathematical Society},
  112(2):267--350, 2016.

\bibitem{AOPV09}
Nir Avni, Uri Onn, Amritanshu Prasad, and Leonid Vaserstein.
\newblock Similarity classes of 3$\times$3 matrices over a local principal
  ideal ring.
\newblock {\em Communications in Algebra}, 37(8):2601--2615, 2009.

\bibitem{BGL87}
Dagmar Baer, Werner Geigle, and Helmut Lenzing.
\newblock The preprojective algebra of a tame hereditary {A}rtin algebra.
\newblock {\em Communications in Algebra}, 15(1--2):425--457, 1987.

\bibitem{BBS13}
Kai Behrend, Jim Bryan, and Bal{\'a}zs Szendr{\H{o}}i.
\newblock Motivic degree zero {D}onaldson-{T}homas invariants.
\newblock {\em Inventiones Mathematicae}, 192(1):111--160, 2013.

\bibitem{BBDG18}
Alexander Beilinson, Joseph Bernstein, Pierre Deligne, and Ofer Gabber.
\newblock {\em Faisceaux pervers}.
\newblock Soci{\'e}t{\'e} math{\'e}matique de France, 2018.

\bibitem{BDHK18}
Gergely B{\'e}rczi, Brent Doran, Thomas Hawes, and Frances Kirwan.
\newblock Geometric invariant theory for graded unipotent groups and
  applications.
\newblock {\em Journal of topology}, 11(3):826--855, 2018.

\bibitem{BDHK20}
Gergely B{\'e}rczi, Brent Doran, Thomas Hawes, and Frances Kirwan.
\newblock Projective completions of graded unipotent quotients.
\newblock {\em \url{https://arxiv.org/abs/1607.04181}}, 2020.

\bibitem{BGP73}
I.~N. Bernstein, I.~M. Gelfand, and V.~A. Ponomarev.
\newblock Coxeter functors and {G}abriel's theorem.
\newblock {\em Russian mathematicals surveys}, 28(2):17--32, 1973.

\bibitem{BL94}
Joseph Bernstein and Valery Lunts.
\newblock {\em Equivariant sheaves and functors}.
\newblock Lecture Notes in Mathematics. Springer, 1994.

\bibitem{Bir73}
Garrett Birkhoff.
\newblock {\em Lattice theory}.
\newblock American Mathematical Society, 1973.

\bibitem{Bjo80}
Anders Bj{\"o}rner.
\newblock Shellable and {C}ohen-{M}acaulay partially ordered sets.
\newblock {\em Transactions of the American Mathematical Society},
  260(1):159--183, 1980.

\bibitem{BGV16}
Raf Bocklandt, Federica Galluzzi, and Francesco Vaccarino.
\newblock The {N}ori-{H}ilbert scheme is not smooth for 2-{C}alabi-{Y}au
  algebras.
\newblock {\em Journal of Noncommutative Geometry}, 10:745--774, 2016.

\bibitem{Bor88}
Richard Borcherds.
\newblock Generalized {K}ac-{M}oody algebras.
\newblock {\em Journal of Algebra}, 115:501--512, 1988.

\bibitem{Boz15}
Tristan Bozec.
\newblock Quivers with loops and perverse sheaves.
\newblock {\em Mathematische Annalen}, 362(3):773--797, 2015.

\bibitem{Boz16}
Tristan Bozec.
\newblock Quivers with loops and generalized crystals.
\newblock {\em Compositio Mathematica}, 152(10):1999--2040, 2016.

\bibitem{BCS20}
Tristan Bozec, Damien Calaque, and Sarah Scherotzke.
\newblock Relative critical loci and quiver moduli.
\newblock {\em \url{https://arxiv.org/abs/2006.01069}}, 2020.

\bibitem{BS19a}
Tristan Bozec and Olivier Schiffmann.
\newblock Counting absolutely cuspidals for quivers.
\newblock {\em Mathematische Zeitschrift}, 292(1):133--149, 2019.

\bibitem{BSV20}
Tristan Bozec, Olivier Schiffmann, and Eric Vasserot.
\newblock On the number of points of nilpotent quiver varieties over finite
  fields.
\newblock {\em Annales Scientifiques de l'{\'E}cole Normale Sup{\'e}rieure},
  53(6):1501--1544, 2020.

\bibitem{BD19}
Christopher Brav and Tobias Dyckerhoff.
\newblock Relative {C}alabi-{Y}au structures.
\newblock {\em Compositio Mathematica}, 155:372--412, 2019.

\bibitem{Bri17a}
Michel Brion.
\newblock Some structure theorems for algebraic groups.
\newblock In Mahir~Bilen Can, editor, {\em Algebraic Groups: Structures and
  Actions}, volume~94, pages 53--126, 2017.

\bibitem{Bud21}
Nero Budur.
\newblock Rational singularities, quiver moment maps and representations of
  surface groups.
\newblock {\em International Mathematics Research Notices},
  2021(15):11782--11817, 2021.

\bibitem{BZ19}
Nero Budur and Ziyu Zhang.
\newblock Formality conjecture for {K}3 surfaces.
\newblock {\em Compositio Mathematica}, 155(5):902--911, 2019.

\bibitem{CR22}
Angela Carnevale and Tobias Rossmann.
\newblock Linear relations with disjoint supports and average size of kernels.
\newblock {\em Journal of the London Mathematical Society}, 2022.

\bibitem{COW21}
Francesca Carocci, Giulio Orecchia, and Dimitri Wyss.
\newblock {BPS}-invariants from $p$-adic integrals.
\newblock {\em \url{https://arxiv.org/abs/2112.12103}}, 2021.

\bibitem{CLNS18}
Antoine Chambert-Loir, Johannes Nicaise, and Julien Sebag.
\newblock {\em Motivic integration}.
\newblock Springer, 2018.

\bibitem{Cha20}
Jean-Yves Charbonnel.
\newblock Projective dimension and commuting variety of a reductive {L}ie
  algebra.
\newblock {\em \url{https://arxiv.org/abs/2006.12942}}, 2020.

\bibitem{CG97}
Neil Chriss and Victor Ginzburg.
\newblock {\em Representation theory and complex geometry}.
\newblock Birkh{\"a}user, 1997.

\bibitem{CL05}
Raf Cluckers and Fran{\c{c}}ois Loeser.
\newblock Ax-{K}ochen-{E}r{\v{s}}ov theorems for p-adic integrals and motivic
  integration.
\newblock In {\em Geometric methods in algebra and number theory}. Springer,
  2005.

\bibitem{CL10}
Raf Cluckers and Fran{\c{c}}ois Loeser.
\newblock Constructible exponential functions, motivic {F}ourier transform and
  transfer principle.
\newblock {\em Annals of Mathematics}, 171(2):1011--1065, 2010.

\bibitem{Con00}
Brian Conrad.
\newblock {\em Grothendieck duality and base change}.
\newblock Springer, 2000.

\bibitem{CB92}
William Crawley-Boevey.
\newblock Lectures on representations of quivers.
\newblock Available at
  \url{https://www.math.uni-bielefeld.de/~wcrawley/quivlecs.pdf}, 1992.

\bibitem{CB99a}
William Crawley-Boevey.
\newblock {DMV} lectures on representations of quivers, preprojective algebras
  and deformations of quotient singularities.
\newblock Available at
  \url{https://www.math.uni-bielefeld.de/~wcrawley/dmvlecs.pdf}, 1999.

\bibitem{CB99b}
William Crawley-Boevey.
\newblock Preprojective algebras, differential operators and a {C}onze
  embedding for deformations of {K}leinian singularities.
\newblock {\em Commentarii Mathematici Helvetici}, 74:548--574, 1999.

\bibitem{CB01}
William Crawley-Boevey.
\newblock Geometry of the moment map for representations of quivers.
\newblock {\em Compositio Mathematica}, 126(3):257--293, 2001.

\bibitem{CB02}
William Crawley-Boevey.
\newblock Decomposition of {M}arsden-{W}einstein reductions for representations
  of quivers.
\newblock {\em Compositio Mathematica}, 2002.

\bibitem{CB03a}
William Crawley-Boevey.
\newblock Normality of {M}arsden-{W}einstein reductions for representations of
  quivers.
\newblock {\em Mathematische Annalen}, 325(1):55--79, 2003.

\bibitem{CBH98}
William Crawley-Boevey and Martin Holland.
\newblock Noncommutative deformations of {K}leinian singularities.
\newblock {\em Duke Mathematical Journal}, 92(3):605--635, 1998.

\bibitem{CBS06}
William Crawley-Boevey and Peter Shaw.
\newblock Multiplicative preprojective algebras, middle convolution and the
  {D}eligne-{S}imspon problem.
\newblock {\em Advances in Mathematics}, 201:180--208, 2006.

\bibitem{CBVB04}
William Crawley-Boevey and Michel van~den Bergh.
\newblock Absolutely indecomposable representations and {K}ac-{M}oody {L}ie
  algebras.
\newblock {\em Inventiones Mathematicae}, 155(3):537--559, 2004.

\bibitem{Dav17}
Ben Davison.
\newblock The critical {C}o{HA} of a quiver with potential.
\newblock {\em Quarterly Journal of Mathematics}, 2017.

\bibitem{Dav18}
Ben Davison.
\newblock Purity of critical cohomology and {K}ac's conjecture.
\newblock {\em Mathematical Research Letters}, 25(2):469--488, 2018.

\bibitem{Dav20}
Ben Davison.
\newblock {BPS} {L}ie algebras and the less perverse filtration on the
  preprojective {C}o{HA}.
\newblock {\em \url{https://arxiv.org/abs/2007.03289}}, 2020.

\bibitem{Dav21a}
Ben Davison.
\newblock Purity and 2-{C}alabi-{Yau} categories.
\newblock {\em \url{https://arxiv.org/abs/2106.07692}}, 2021.

\bibitem{Dav22}
Ben Davison.
\newblock Affine {BPS} algebras, {W}-algebras and the cohomological {H}all
  algebra of {$\mathbb{A}^2$}.
\newblock {\em \url{https://arxiv.org/abs/2209.05971}}, 2022.

\bibitem{Dav23a}
Ben Davison.
\newblock A boson-fermion correspondence in cohomological {D}onaldson-{T}homas
  theory.
\newblock {\em Glasgow Mathematical Journal}, 65:S28--S52, 2023.

\bibitem{Dav23c}
Ben Davison.
\newblock The integrality conjecture and the cohomology of preprojective
  stacks.
\newblock {\em Journal f{\"{u}}r die reine und angewandte Mathematik},
  804:105--154, 2023.

\bibitem{DHSM22}
Ben Davison, Lucien Hennecart, and Sebastian Schlegel-Mejia.
\newblock {BPS} {L}ie algebras for totally negative 2-{C}alabi-{Y}au categories
  and nonabelian {H}odge theory for stacks.
\newblock {\em \url{https://arxiv.org/abs/2212.07668}}, 2022.

\bibitem{DHSM23}
Ben Davison, Lucien Hennecart, and Sebastian Schlegel-Mejia.
\newblock {BPS} algebras and generalised {K}ac-{M}oody algebras from
  2-{C}alabi-{Y}au categories.
\newblock {\em \url{https://arxiv.org/abs/2303.12592}}, 2023.

\bibitem{DM15b}
Ben Davison and Sven Meinhardt.
\newblock Motivic {D}onaldson-{T}homas invariants for the one-loop quiver with
  potential.
\newblock {\em Geometry \& Topology}, 19(5):2535--2555, 2015.

\bibitem{DM20}
Ben Davison and Sven Meinhardt.
\newblock Cohomological {D}onaldson-{T}homas theory of a quiver with potential
  and quantum enveloping algebras.
\newblock {\em Inventiones Mathematicae}, 221:777--871, 2020.

\bibitem{Del70}
Pierre Deligne.
\newblock Th{\'{e}}orie de {H}odge {I}.
\newblock In {\em Actes du congr{\`{e}}s international des
  math{\'{e}}maticiens}, volume~1, pages 425--430, 1970.

\bibitem{Del71}
Pierre Deligne.
\newblock Th{\'{e}}orie de {H}odge: {II}.
\newblock {\em Publications math{\'{e}}matiques de l'{IHES}}, 40:5--57, 1971.

\bibitem{Del74b}
Pierre Deligne.
\newblock La conjecture de {W}eil: {I}.
\newblock {\em Publications math{\'{e}}matiques de l'{IHES}}, 43:273--307,
  1974.

\bibitem{Del74a}
Pierre Deligne.
\newblock Th{\'e}orie de {H}odge, {III}.
\newblock {\em Publications math{\'e}matiques de l'{IHES}}, 44:5--77, 1974.

\bibitem{Del80}
Pierre Deligne.
\newblock La conjecture de {W}eil: {II}.
\newblock {\em Publications math{\'{e}}matiques de l'{IHES}}, 52:137--252,
  1980.

\bibitem{Den87}
Jan Denef.
\newblock On the degree of {I}gusa's local zeta function.
\newblock {\em American Journal of Mathematics}, 109(6):991--1008, 1987.

\bibitem{Dim92}
Alexandru Dimca.
\newblock {\em Singularities and topology of hypersurfaces}.
\newblock Springer Science \& Business Media, 1992.

\bibitem{Dim04}
Alexandru Dimca.
\newblock {\em Sheaves in topology}.
\newblock Springer Science \& Business Media, 2004.

\bibitem{DR79}
Vlastimil Dlab and Claus Ringel.
\newblock The preprojective algebra of a modulated graph.
\newblock In Vlastimil Dlab and Peter Gabriel, editors, {\em Representation
  Theory II}, volume 832 of {\em Lecture Notes in Mathematics}, 1979.

\bibitem{Dol03}
I.~Dolgachev.
\newblock {\em Lectures on {I}nvariant {T}heory}.
\newblock Cambridge University Press, 2003.

\bibitem{EG98a}
Dan Edidin and William Graham.
\newblock Equivariant intersection theory.
\newblock {\em Inventiones Mathematicae}, 131:595--634, 1998.

\bibitem{ELM04}
Lawrence Ein, Robert Lazarsfeld, and Mircea Musta{\c{t}}{\u{a}}.
\newblock Contact loci in arc spaces.
\newblock {\em Compositio Mathematica}, 140:1229--1244, 2004.

\bibitem{Elk78}
Ren{\'e}e Elkik.
\newblock Singularit{\'e}s rationnelles et d{\'e}formations.
\newblock {\em Inventiones Mathematicae}, 47:139--147, 1978.

\bibitem{Fan14}
Zhaobing Fan.
\newblock Geometric approach of {H}all algebra of representations of quivers
  over local rings.
\newblock {\em \url{https://arxiv.org/abs/1012.5257}}, 2014.

\bibitem{Gab72}
Peter Gabriel.
\newblock Unzerlegbare {D}arstellungen {I}.
\newblock {\em Manuscripta Mathematica}, 6:71--103, 1972.

\bibitem{GLS16}
Christof Geiss, Bernard Leclerc, and Jan Schr{\"o}er.
\newblock {Q}uivers with relations for symmetrizable {C}artan matrices {III}:
  {C}onvolution algebras.
\newblock {\em Representation Theory of the American Mathematical Society},
  20(13):375--413, 2016.

\bibitem{GLS17a}
Christof Geiss, Bernard Leclerc, and Jan Schr{\"o}er.
\newblock {Q}uivers with relations for symmetrizable {C}artan matrices {I}:
  {F}oundations.
\newblock {\em Inventiones Mathematicae}, 209:61--158, 2017.

\bibitem{GLS18a}
Christof Geiss, Bernard Leclerc, and Jan Schr{\"o}er.
\newblock Quivers with relations for symmetrizable {C}artan matrices {IV}:
  {C}rystal graphs and semicanonical functions.
\newblock {\em Selecta Mathematica}, 24:3283--3348, 2018.

\bibitem{GLS18b}
Christof Geiss, Bernard Leclerc, and Jan Schr{\"o}er.
\newblock {Q}uivers with relations for symmetrizable {C}artan matrices {V}:
  {C}aldero-{C}hapoton formulas.
\newblock {\em Proceedings of the London Mathematical Society},
  117(1):125--148, 2018.

\bibitem{GLS17b}
Christoff Geiss, Bernard Leclerc, and Jan Schr{\"o}er.
\newblock {Q}uivers with relations for symmetrizable {C}artan matrices {II}:
  {C}hange of symmetrizers.
\newblock {\em International Mathematics Research Notices}, 2018(9):2866--2898,
  2017.

\bibitem{GP79}
I.~M. Gelfand and V.~A. Ponomarev.
\newblock Model algebras and representations of graphs.
\newblock {\em Functional Analysis and its Applications}, 13:157--166, 1979.

\bibitem{Gin06}
Victor Ginzburg.
\newblock Calabi-{Y}au algebras.
\newblock {\em \url{https://arxiv.org/abs/math/0612139}}, 2006.

\bibitem{Gla19}
Itay Glazer.
\newblock On rational singularities and counting points of schemes over finite
  rings.
\newblock {\em Algebra \& Number theory}, 13(2):485--500, 2019.

\bibitem{Gre95}
James Green.
\newblock Hall algebras, hereditary algebras and quantum groups.
\newblock {\em Inventiones Mathematicae}, 1995.

\bibitem{GWZ20a}
Michael Groechenig, Dimitri Wyss, and Paul Ziegler.
\newblock Mirror symmetry for moduli spaces of {H}iggs bundles via p-adic
  integration.
\newblock {\em Inventiones Mathematicae}, 221(2):505--596, 2020.

\bibitem{GWZ23b}
Michael Groechenig, Dimitri Wyss, and Paul Ziegler.
\newblock {BPS} functions from non-archimedean integration.
\newblock Talk of Dimitri Wyss at the conference Categorified enumerative
  geometry and geometric representation theory, Lausanne, September 2023.

\bibitem{GWZ23a}
Michael Groechenig, Dimitri Wyss, and Paul Ziegler.
\newblock Donaldson-{T}homas invariants via p-adic integrals.
\newblock Talk of Paul Ziegler at the conference Non-archimedean methods in
  arithmetic and geometry, Les Diablerets, February 2023.

\bibitem{Gro65}
Alexander Grothendieck.
\newblock {\'E}l{\'e}ments de g{\'e}om{\'e}trie alg{\'e}brique: {IV.} {\'e}tude
  locale des sch{\'e}mas et des morphismes de sch{\'e}mas, {S}econde partie.
\newblock {\em Publications math{\'e}matiques de l'IHES}, 24:5--231, 1965.

\bibitem{Gro66}
Alexander Grothendieck.
\newblock {\'E}l{\'e}ments de g{\'e}om{\'e}trie alg{\'e}brique: {IV.} {\'e}tude
  locale des sch{\'e}mas et des morphismes de sch{\'e}mas, {T}roisi{\`e}me
  partie.
\newblock {\em Publications math{\'e}matiques de l'IHES}, 28:5--255, 1966.

\bibitem{HHJ24}
Eloise Hamilton, Victoria Hoskins, and Joshua Jackson.
\newblock Affine non-reductive {GIT} and moduli of representations of quivers
  with multiplicities.
\newblock {\em \url{https://arxiv.org/abs/2404.06560}}, 2024.

\bibitem{Hau10}
Tam{\'a}s Hausel.
\newblock Kac's conjecture from {N}akajima quiver varieties.
\newblock {\em Inventiones Mathematicae}, 181(1):21--37, 2010.

\bibitem{HLRV13b}
Tam{\'a}s Hausel, Emmanuel Letellier, and Fernando Rodriguez-Villegas.
\newblock Positivity for {K}ac polynomials and {DT}-invariants of quivers.
\newblock {\em Annals of Mathematics}, 2013.

\bibitem{HLRV24}
Tam{\'a}s Hausel, Emmanuel Letellier, and Fernando Rodriguez-Villegas.
\newblock Locally free representations of quivers over {F}robenius algebras.
\newblock {\em Selecta Mathematica}, 30(20), 2024.

\bibitem{HRV08}
Tam{\'a}s Hausel and Fernando Rodriguez-Villegas.
\newblock Mixed {H}odge polynomials of character varieties.
\newblock {\em Inventiones Mathematicae}, 174(3):555--624, 2008.

\bibitem{HWW23}
Tam{\'a}s Hausel, Michael~Lennox Wong, and Dimitri Wyss.
\newblock Arithmetic and metric aspects of open de {R}ham spaces.
\newblock {\em Proceedings of the London Mathematical Society},
  127(4):958--1027, 2023.

\bibitem{Hen24}
Lucien Hennecart.
\newblock On geometric realizations of the unipotent enveloping algebra of a
  quiver.
\newblock {\em Advances in Mathematics}, 441:109536, 2024.

\bibitem{HSS21}
Hans-Christian Herbig, Gerald Schwarz, and Christopher Seaton.
\newblock When does the zero fiber of the moment map have rational
  singularities?
\newblock {\em \url{https://arxiv.org/abs/2108.07306}}, 2021.

\bibitem{Hir64}
Heisuke Hironaka.
\newblock Resolution of singularities of an algebraic variety over a field of
  characteristic zero: {I}.
\newblock {\em Annals of Mathematics}, 79(1):109--203, 1964.

\bibitem{HTT08}
Ryoshi Hotta, Kiyoshi Takeuchi, and Toshiyuki Tanisaki.
\newblock {\em D-modules, perverse sheaves and representation theory}.
\newblock Birkh{\"{a}}user, 2008.

\bibitem{Hua00}
J.~Hua.
\newblock Counting {R}epresentations of {Q}uivers over {F}inite {F}ields.
\newblock {\em Journal of Algebra}, 2000.

\bibitem{HLS23}
Hualin Huang, Zengqiang Lin, and Xiuping Su.
\newblock Components of {AR}-quivers for string algebras of type
  {$\tilde{\mathbb{C}}$} and a conjecture by {G}eiss-{L}eclerc-{S}chr{\"{o}}er.
\newblock {\em Journal of Algebra}, 632:331--362, 2023.

\bibitem{Huy16}
Daniel Huybrechts.
\newblock {\em Lectures on {K3} surfaces}.
\newblock Cambridge University Press, 2016.

\bibitem{HL10}
Daniel Huybrechts and Manfred Lehn.
\newblock {\em The geometry of moduli spaces of sheaves}.
\newblock Cambridge University Press, 2010.

\bibitem{Igu00}
Jun ichi Igusa.
\newblock {\em An {I}ntroduction to the {T}heory of {L}ocal {Z}eta
  {F}unctions}.
\newblock American Mathematical Society, 2000.

\bibitem{Kac80}
Victor Kac.
\newblock Infinite root systems, representations of graphs and invariant
  theory.
\newblock {\em Invent. Math.}, 1980.

\bibitem{Kac82}
Victor Kac.
\newblock Infinite root systems, representations of graphs and invariant theory
  {II}.
\newblock {\em Journal of Algebra}, 1982.

\bibitem{Kac83}
Victor Kac.
\newblock Root systems, representations of quivers and invariant theory.
\newblock In Francesco Gherardelli, editor, {\em Invariant {T}heory}, volume
  996 of {\em Lecture Notes in Mathematics}, pages 74--108. Springer, 1983.

\bibitem{Kac95}
Victor Kac.
\newblock {\em Infinite-dimensional {L}ie algebras}.
\newblock Cambridge University Press, 1995.

\bibitem{KLS06}
Dmitry Kaledin, Manfred Lehn, and Christoph Sorger.
\newblock Singular symplectic moduli spaces.
\newblock {\em Inventiones Mathematicae}, 164(3):591--614, 2006.

\bibitem{KKS09}
Seok-Jin Kang, Masaki Kashiwara, and Olivier Schiffmann.
\newblock Geometric construction of crystal bases for quantum generalized
  {K}ac-{M}oody algebras.
\newblock {\em Advances in Mathematics}, 222:996--1015, 2009.

\bibitem{KS06}
Seok-Jin Kang and Olivier Schiffmann.
\newblock Canonical bases for quantum generalized {K}ac-{M}oody algebras.
\newblock {\em Advances in Mathematics}, 200:455--478, 2006.

\bibitem{KS23a}
Daniel Kaplan and Travis Schedler.
\newblock Multiplicative preprojective algebras are 2-{C}alabi-{Y}au.
\newblock {\em Algebra and Number Theory}, 17(4):831--883, 2023.

\bibitem{KS97}
Masaki Kashiwara and Yoshihisa Saito.
\newblock Geometric construction of crystal bases.
\newblock {\em Duke Mathematical Journal}, 89(1):9--36, 1997.

\bibitem{KS90}
Masaki Kashiwara and Pierre Schapira.
\newblock {\em Sheaves on manifolds}, volume 292 of {\em Grundlehren der
  mathematischen Wissenschaften}.
\newblock Springer, 1990.

\bibitem{Kel08}
Bernhard Keller.
\newblock Calabi-{Y}au triangulated categories.
\newblock {\em Trends in representation theory of algebras and related topics},
  pages 467--489, 2008.

\bibitem{Kel11}
Bernhard Keller.
\newblock Deformed {C}alabi-{Y}au completions.
\newblock {\em Journal f{\"u}r die reine und angewandte Mathematik}, 2011.

\bibitem{KW21}
Bernhard Keller and Yu~Wang.
\newblock An introduction to relative {C}alabi-{Y}au structures.
\newblock {\em \url{https://arxiv.org/abs/2111.10771}}, 2021.

\bibitem{KW01}
Reinhardt Kiehl and Rainer Weissauer.
\newblock {\em Weil conjectures, perverse sheaves and l'adic {F}ourier
  transform}, volume~42 of {\em Ergebnisse der Mathematik und ihrer
  Grenzgebiete}.
\newblock Springer, 2001.

\bibitem{Kin94}
A.~D. King.
\newblock Moduli of representations of finite-dimensional algebras.
\newblock {\em Quarterly Journal of Mathematics}, 45(2):515--530, 1994.

\bibitem{KS11}
Maxim Kontsevich and Yan Soibelman.
\newblock Cohomological {H}all algebra, exponential {H}odge structures and
  motivic {D}onaldson-{T}homas invariants.
\newblock {\em Communications in Number Theory and Physics}, 5(2):231--252,
  2011.

\bibitem{Kro89}
Peter Kronheimer.
\newblock The construction of {ALE} spaces as hyper-{K\"{a}}hler quotients.
\newblock {\em Journal of Differential Geometry}, 29:665--683, 1989.

\bibitem{KN90}
Peter Kronheimer and Hiraku Nakajima.
\newblock Yang-{M}ills instantons on {ALE} gravitational instantons.
\newblock {\em Mathematische Annalen}, 288:263--307, 1990.

\bibitem{LO08}
Yves Laszlo and Martin Olsson.
\newblock The six operations for sheaves on {A}rtin stacks {II}: adic
  coefficients.
\newblock {\em Publications {M}ath{\'{e}}matiques de l'{IHES}}, 2008.

\bibitem{LBP90}
Lieven Le~Bruyn and Claudio Procesi.
\newblock Semisimple representations of quivers.
\newblock {\em Transactions of the American Mathematical Society},
  317(2):585--598, 1990.

\bibitem{LRV23}
Emmanuel Letellier and Fernando Rodriguez-Villegas.
\newblock E-series of character varieties of non-orientable surfaces.
\newblock {\em Annales de l'Institut Fourier}, 73(4):1385--1420, 2023.

\bibitem{Li12}
Fang Li.
\newblock Modulation and natural valued quiver of an algebra.
\newblock {\em Pacific Journal of Mathematics}, 256(1):105--128, 2012.

\bibitem{Lun73}
Domingo Luna.
\newblock Slices {\'e}tales.
\newblock {\em Bull. Soc. Math. France}, 33:81--105, 1973.

\bibitem{Lus85a}
George Lusztig.
\newblock Character sheaves ({I}).
\newblock {\em Advances in Mathematics}, 1985.

\bibitem{Lus91}
George Lusztig.
\newblock Quivers, perverse sheaves and quantized enveloping algebras.
\newblock {\em Journal of the American Mathematical Society}, 4(2):365--421,
  1991.

\bibitem{Lus10}
George Lusztig.
\newblock {\em Introduction to quantum groups}.
\newblock Springer Science \& Business Media, 2010.

\bibitem{Mac95}
Ian Macdonald.
\newblock {\em Symmetric functions and {H}all polynomials}.
\newblock Oxford University Press, 1995.

\bibitem{Mac74}
Robert MacPherson.
\newblock Chern classes for singular algebraic varieties.
\newblock {\em Annals of Mathematics}, 1974.

\bibitem{MV22}
Joshua Maglione and Christopher Voll.
\newblock Flag {H}ilbert-{P}oincar{\'e} series of hyperplane arrangements and
  their {I}gusa zeta functions.
\newblock {\em \url{https://arxiv.org/abs/2103.03640}}, 2022.

\bibitem{MSS11}
Laurentiu Maxim, Morihiko Saito, and J{\"o}rg Sch{\"u}rmann.
\newblock Symmetric products of mixed {H}odge modules.
\newblock {\em Journal de math{\'e}matiques pures et appliqu{\'e}es},
  96(5):462--483, 2011.

\bibitem{Moz11a}
Sergey Mozgovoy.
\newblock Motivic {D}onaldson-{T}homas invariants and {M}c{K}ay correspondence.
\newblock {\em \url{https://arxiv.org/abs/1107.6044}}, 2011.

\bibitem{Moz19}
Sergey Mozgovoy.
\newblock Commuting matrices and volumes of linear stacks.
\newblock {\em \url{https://arxiv.org/abs/1901.00690}}, 2019.

\bibitem{MS20}
Sergey Mozgovoy and Olivier Schiffmann.
\newblock Counting {H}iggs bundles and type {A} quiver bundles.
\newblock {\em Compositio Mathematica}, 2020.

\bibitem{MFK94}
David Mumford, John Fogarty, and Frances Kirwan.
\newblock {\em Geometric invariant theory}.
\newblock Springer, 1994.

\bibitem{Mus01}
Mircea Musta\c{t}\v{a}.
\newblock Jet schemes of locally complete intersection canonical singularities.
\newblock {\em Inventiones Mathematicae}, 145(3):397--424, 2001.

\bibitem{Mus02}
Mircea Musta{\c{t}}{\u{a}}.
\newblock Singularities of pairs via jet schemes.
\newblock {\em Journal of the American Mathematical Society}, 15(3):599--615,
  2002.

\bibitem{Mus22a}
Mircea Musta{\c{t}}{\u{a}}.
\newblock Bernstein-{S}ato polynomials for general ideals vs principal ideals.
\newblock {\em Proceedings of the American Mathematical Society},
  150(9):3655--3662, 2022.

\bibitem{Nak94}
Hiraku Nakajima.
\newblock Instantons on {ALE} spaces, quiver varieties and {K}ac-{M}oody
  algebras.
\newblock {\em Duke Mathematical Journal}, 76(2):365--416, 1994.

\bibitem{Nak98}
Hiraku Nakajima.
\newblock Quiver varieties and {K}ac-{M}oody algebras.
\newblock {\em Duke Mathematical Journal}, 91(3):515--560, 1998.

\bibitem{Neu99}
J{\"u}rgen Neukirch.
\newblock {\em Algebraic number theory}.
\newblock Springer, 1999.

\bibitem{OBV15}
Eamonn O'Brien and Christopher Voll.
\newblock Enumerating classes and characters of p-groups.
\newblock {\em Transactions of the American Mathematical Society},
  367(11):7775--7796, 2015.

\bibitem{Ols16}
Martin Olsson.
\newblock {\em Algebraic spaces and stacks}, volume~62 of {\em Colloquium
  Publications}.
\newblock American Mathematical Society, 2016.

\bibitem{PS08}
Chris Peters and Jozef Steenbrink.
\newblock {\em Mixed {H}odge {S}tructures}.
\newblock Springer, 2008.

\bibitem{Pfe23}
Calvin Pfeifer.
\newblock A generic classification of locally free representations of affine
  {GLS} algebras.
\newblock {\em \url{https://arxiv.org/abs/2308.09587}}, 2023.

\bibitem{Rei08a}
Markus Reineke.
\newblock Moduli of representations of quivers.
\newblock {\em \url{https://arxiv.org/abs/0802.2147}}, 2008.

\bibitem{RS17}
Jie Ren and Yan Soibelman.
\newblock Cohomological {H}all algebras, semicanonical bases and
  {D}onaldson-{T}homas invariants for 2-dimensional {C}alabi-{Y}au categories
  (with an appendix by {B}en {D}avison).
\newblock {\em Algebra, geometry, and physics in the 21st century}, 2017.

\bibitem{Rin90a}
Claus Ringel.
\newblock Hall algebras and quantum groups.
\newblock {\em Inventiones Mathematicae}, 1990.

\bibitem{Rin98}
Claus Ringel.
\newblock The preprojective algebra of a quiver.
\newblock In Idun Reiten, Sverre Smal{\o}, and {\O}yvind Solberg, editors, {\em
  Algebras and modules II}, volume~24 of {\em Conference Proceedings, Canadian
  Mathematical Society}, pages 467--480, 1998.

\bibitem{Ros18}
Tobias Rossmann.
\newblock The average size of the kernel of a matrix and orbits of linear
  groups.
\newblock {\em Proceedings of the London Mathematical Society},
  117(3):574--616, 2018.

\bibitem{Ros20}
Tobias Rossmann.
\newblock The average size of the kernel of a matrix and orbits of linear
  groups, {II}: duality.
\newblock {\em Journal of Pure and Applied Algebra}, 224(4):106203, 2020.

\bibitem{Ros22}
Tobias Rossmann.
\newblock On the enumeration of orbits of unipotent groups over finite fields.
\newblock {\em \url{https://arxiv.org/abs/2208.04646}}, 2022.

\bibitem{RV19}
Tobias Rossmann and Christopher Voll.
\newblock Groups, graphs and hypergraphs: average size of kernels of generic
  matrices with support constraints.
\newblock {\em \url{https://arxiv.org/abs/1908.09589}}, 2019.

\bibitem{Sai88}
Morihiko Saito.
\newblock Modules de {H}odge polarisables.
\newblock {\em Publications of the RIMS, Kyoto University}, 24:849--995, 1988.

\bibitem{Sai89}
Morihiko Saito.
\newblock Introduction to mixed {H}odge modules.
\newblock {\em Ast{\'e}risque}, 179-180:145--162, 1989.

\bibitem{Sai90}
Morihiko Saito.
\newblock Mixed {H}odge modules.
\newblock {\em Publications of the RIMS, Kyoto University}, 26:221--333, 1990.

\bibitem{SU21}
Matthew Satriano and Jeremy Usatine.
\newblock A motivic change of variables formula for {A}rtin stacks.
\newblock {\em \url{https://arxiv.org/abs/2109.09800}}, 2021.

\bibitem{SU23c}
Matthew Satriano and Jeremy Usatine.
\newblock Motivic integration for singular {A}rtin stacks.
\newblock {\em \url{https://arxiv.org/abs/2309.11442}}, 2023.

\bibitem{S12b}
Olivier Schiffmann.
\newblock Lectures on canonical and crystal bases of {H}all algebras.
\newblock In {\em Geometric methods in representation theory}, volume~24 of
  {\em S{\'e}minaires et congr{\`e}s}, pages 143--258. Soc. Math. Fr., 2012.

\bibitem{S16}
Olivier Schiffmann.
\newblock Indecomposable vector bundles and stable {H}iggs bundles over a
  smooth projective curve.
\newblock {\em Annals of Mathematics}, 2016.

\bibitem{SV13b}
Olivier Schiffmann and {\'E}ric Vasserot.
\newblock Cherednik algebras, {W}-algebras and the equivariant cohomology of
  the moduli space of instantons on {A}2.
\newblock {\em Publications math{\'e}matiques de l'{IH\'ES}}, 118(1):213--342,
  2013.

\bibitem{SV20}
Olivier Schiffmann and Eric Vasserot.
\newblock On cohomological {H}all algebras of quivers: generators.
\newblock {\em Journal f{\"u}r die reine und angewandte Mathematik},
  760:59--132, 2017.

\bibitem{SV96}
Wilfried Schmid and Kari Vilonen.
\newblock Characteristic cycles of constructible sheaves.
\newblock {\em Inventiones Mathematicae}, 124:451--502, 1996.

\bibitem{Sch14}
Christian Schnell.
\newblock An overview of {M}orihiko {S}aito's theory of mixed {H}odge modules.
\newblock {\em \url{https://arxiv.org/abs/1405.3096}}, 2014.

\bibitem{Sch92}
Aidan Schofield.
\newblock General representations of quivers.
\newblock {\em Proceedings of the London Mathematical Society}, 3(1):46--64,
  1992.

\bibitem{Sch80}
Gerald Schwarz.
\newblock Lifting smooth homotopies of orbit spaces.
\newblock {\em Publications Math{\'e}matiques de l'IH{\'E}S}, 51:37--135, 1980.

\bibitem{Ser58}
Jean-Pierre Serre.
\newblock Espaces fibr{\'{e}}s alg{\'{e}}briques.
\newblock {\em S{\'{e}}minaire Claude Chevalley}, 3(1):1--37, 1958.

\bibitem{Spr98}
T.~A. Springer.
\newblock {\em Linear {A}lgebraic {G}roups}.
\newblock Birkh\"{a}user, 1998.

\bibitem{Sta07}
Richard Stanley.
\newblock {\em Combinatorics and commutative algebra}.
\newblock Springer Science \& Business Media, 2007.

\bibitem{Sta16}
Richard Stanley.
\newblock Smith normal forms in combinatorics.
\newblock {\em Journal of combinatorial theory}, 2016.

\bibitem{Swe69}
Moss Sweedler.
\newblock {\em Hopf algebras}.
\newblock Mathematics Lecture Notes. Benjamin, 1969.

\bibitem{SP}
{The {S}tacks {P}roject {A}uthors}.
\newblock \textit{Stacks Project}.
\newblock \url{https://stacks.math.columbia.edu}, 2024.

\bibitem{TV07}
Bertrand To{\"e}n and Michel Vaqui{\'e}.
\newblock Moduli of objects in dg-categories.
\newblock In {\em Annales scientifiques de l'Ecole normale sup{\'e}rieure},
  volume~40, pages 387--444, 2007.

\bibitem{Tot99}
Burt Totaro.
\newblock The {C}how ring of a classifying space.
\newblock In {\em Proceedings of symposia in pure mathematics}, volume~67,
  1999.

\bibitem{Vai84}
Israel Vainsencher.
\newblock Complete collineations and blowing up determinantal ideals.
\newblock {\em Mathematische Annalen}, 267:417--432, 1984.

\bibitem{VdB15}
Michel van~den Bergh.
\newblock Calabi-{Y}au algebras and superpotentials.
\newblock {\em Selecta Mathematica}, 2015.

\bibitem{VV23b}
Michela Varagnolo and {\'E}ric Vasserot.
\newblock Non symmetric quantum loop groups and critical convolution algebras.
\newblock {\em \url{https://arxiv.org/abs/2308.01809}}, 2023.

\bibitem{Ver22}
Tanguy Vernet.
\newblock Rational singularities for moment maps of totally negative quivers.
\newblock {\em \url{https://arxiv.org/abs/2209.14791}}, 2022.

\bibitem{Ver23}
Tanguy Vernet.
\newblock Positivity for toric {K}ac polynomials in higher depth.
\newblock {\em \url{https://arxiv.org/abs/2310.02912}}, 2023.

\bibitem{VZG08}
Willem Veys and W.~A. Zuniga-Galindo.
\newblock Zeta functions for analytic mappings, log-principalization of ideals
  and {N}ewton polyhedra.
\newblock {\em Transactions of the American Mathematical Society},
  360(4):2205--2227, 2008.

\bibitem{Vis05}
Angelo Vistoli.
\newblock {\em Fundamental {A}lgebraic {G}eometry: {G}rothendieck's {FGA}
  explained}, volume 123 of {\em Mathematical Surveys and Monographs}, chapter
  Grothendieck topologies, fibered categories and descent theory, pages 1--104.
\newblock American Mathematical Society, 2005.

\bibitem{Wei12}
Andr{\'{e}} Weil.
\newblock {\em Adeles and algebraic groups}.
\newblock Springer Science \& Business Media, 2012.

\bibitem{Wys17b}
Dimitri Wyss.
\newblock {\em Motivic and p-adic {L}ocalisation {P}henomena}.
\newblock PhD thesis, EPFL, 2017.

\bibitem{Yam}
Daisuke Yamakawa.
\newblock Notes on the preprojective algebras of
  {G}eiss-{L}eclerc-{S}chr{\"o}er.
\newblock Unpublished.

\bibitem{Yam10}
Daisuke Yamakawa.
\newblock Quiver varieties with multiplicities, {W}eyl groups of non-symmetric
  {K}ac-{M}oody algebras and {P}ainlev{\'e} equations.
\newblock {\em SIGMA. Symmetry, Integrability and Geometry: Methods and
  Applications}, 2010.

\bibitem{YZ18a}
Yaping Yang and Gufang Zhao.
\newblock The cohomological {H}all algebra of a preprojective algebra.
\newblock {\em Proceedings of the London Mathematical Society},
  116(5):1029--1074, 2018.

\bibitem{YZ20}
Yaping Yang and Gufang Zhao.
\newblock On two cohomological {H}all algebras.
\newblock {\em Proceedinges of the Royal Society of Edinburgh Section A:
  Mathematics}, 150(3):1581--1607, 2020.

\bibitem{YZ22}
Yaping Yang and Gufang Zhao.
\newblock The cohomological {H}all algebras of a preprojective algebra with
  symmetrizer.
\newblock {\em Algebras and Representation Theory}, 2022.

\bibitem{Yas17}
Takehiko Yasuda.
\newblock The wild {M}c{K}ay correspondence and $p$-adic measures.
\newblock {\em Journal of the European Mathematical Society},
  19(12):3709--3743, 2017.

\bibitem{Yos01}
K{\=o}ta Yoshioka.
\newblock Moduli spaces of sheaves on abelian surfaces.
\newblock {\em Mathematische Annalen}, 321(4):817--884, 2001.

\end{thebibliography}

\pagebreak
\end{document}